\RequirePackage{fix-cm}
\documentclass[smallextended]{svjour3}       
\smartqed  
\usepackage{graphicx}
\usepackage{amsfonts}
\usepackage{color}

\newcommand\Red[1]{\textcolor{red}{#1}}
\definecolor{darkblue}{rgb}{0,0,0.7}
\newcommand\Blue[1]{\textcolor{darkblue}{\textbf{#1}}}

\newtheorem{algorithm}{Algorithm}
\usepackage{comment}
\usepackage{mathptmx}      

\begin{document}

\title{Iterative regularization for ensemble data assimilation in reservoir models
}


\author{Marco A. Iglesias  
}


\institute{
M. A. Iglesias \at
              University of Nottingham \\
             \email{marco.iglesias@nottingham.ac.uk}   
}

\date{Received: date / Accepted: date}

\maketitle

\begin{abstract}

We propose the application of iterative regularization for the development of ensemble methods for solving Bayesian inverse problems. In concrete, we construct (i) a variational iterative regularizing ensemble Levenberg-Marquardt method (IR-enLM) and (ii) a derivative-free iterative ensemble Kalman smoother (IR-ES). The aim of these methods is to provide a robust ensemble approximation of the Bayesian posterior. The proposed methods are based on fundamental ideas from iterative regularization methods that have been widely used for the solution of deterministic inverse problems \cite{Iterative}. In this work we are interested in the application of the proposed ensemble methods for the solution of Bayesian inverse problems that arise in reservoir modeling applications. The proposed ensemble methods use key aspects of the regularizing Levenberg-Marquardt scheme developed by Hanke \cite{Hanke} and that we recently applied for history matching in \cite{LM}. Unlike standard methods where the stopping criteria and regularization parameters are typically selected heuristically, in the proposed ensemble methods the discrepancy principle is applied for (i) the selection of the regularization parameters and (ii) the early termination of the scheme. The discrepancy principle is key for the theory of iterative regularization and the purpose of the present work is to apply this principle for the development of ensemble methods defined as iterative updates of solutions to linear ill-posed inverse problems. 

The regularizing and convergence properties of iterative regularization methods for deterministic inverse problems have long been established. However, the approximation properties of the proposed ensemble methods in the context of Bayesian inverse problems is an open problem. In the case where the forward operator is linear and the prior is Gaussian, we show that the tunable parameters of the proposed IR-enLM and IR-ES can be chosen so that the resulting schemes coincide with the standard randomized maximum likelihood (RML) and the ensemble smoother (ES), respectively. Therefore, the proposed methods sample from the posterior in the linear-Gaussian case. Similar to RML and ES methods, in the nonlinear case, one may not conclude that the proposed methods produce samples from the posterior. The present work provides a numerical investigation of the performance of the proposed ensemble methods at capturing the posterior. In particular, we aim at understanding the role of the tunable parameters that arise from the application of iterative regularization techniques. The numerical framework for our investigations consists of using a state-of-the art MCMC method for resolving the Bayesian posterior from synthetic experiments. The resolved posterior via MCMC then provides a gold standard against to which compare the proposed IR-enLM and IR-ES. Our numerical experiments show clear indication that the regularizing properties of the regularization methods applied for the computation of each ensemble have significant impact of the approximation properties of the proposed ensemble methods at capturing the Bayesian posterior. Furthermore, we provide a comparison of the proposed regularizing methods with respect to some unregularized standard methods that have been typically used in the literature. Our numerical experiments showcase the advantage of using iterative regularization for obtaining more robust and stable approximation of the posterior than standard unregularized methods.

\end{abstract}

\section{Introduction}\label{Intro}

We use key elements of \textit{iterative regularization} techniques to develop robust ensemble methods for solving Bayesian inverse problems that arise in subsurface flow. A novel aspect of these ensemble methods is that the \textit{discrepancy principle} is used for (i) the selection of the regularization parameter that stabilizes the update of each ensemble member and (ii) the stopping criteria that avoids data overfitting. These strategies have been used for theoretically establishing the convergence and regularizing properties of some well-known iterative regularization methods aimed at solving deterministic nonlinear ill-posed inverse problems. In this work we extend these strategies for the solution of Bayesian inverse problems where the aim is to capture the Bayesian posterior. In particular, we use the aforementioned elements from iterative regularization to improve the robustness of RML-based and Kalman-based methodologies. The ensemble methods proposed in the present manuscript constitute a proof-of-concept which demonstrates the potential advantages of incorporating fundamental components of iterative regularization such as the aforementioned discrepancy principle. This principle, and more generally, iterative regularization, can be potentially combined with existing frameworks for ensemble data assimilation. It is our conjecture that when standard methods are combined with systematic approaches that take into account the mathematical structure of the underlying inverse problem, further increase in robustness, stability and accuracy for capturing the Bayesian posterior may be obtained. While the present work is focused in large-scale inverse problems that arise in reservoir flow modeling, the proposed methods can be applied for the solution of generic PDE-constrained inverse problems.

\subsection{Background of the proposed methods}\label{sec0}
In a recent publication \cite{LM} we studied the application of iterative regularization methods for computing \textit{inverse estimates} of unknown geologic properties by means of history matching. More concretely, we developed an application of the regularizing Levenberg-Marquardt (LM) scheme \cite{Hanke} for solving history matching problems posed as the minimization of
\begin{eqnarray}\label{eq:A}
\Phi(u)\equiv\vert\vert \Gamma^{-1/2}(y-G(u))\vert\vert_{Y}^{2} 
\end{eqnarray}
where $u$ is the unknown geologic property, $y$ denotes production data, $\Gamma$ is the measurements error covariance. $G:X\to Y$ is the forward operator that arises from the reservoir model; $G$ maps the parameter space $X$ (of admissible geologic properties) to the observation space $Y$. In typical reservoir modeling applications, geologic properties need to be defined on thousands or even millions of gridblocks. Due to the resulting large size of the space $X$, the computation of a minimizer of (\ref{eq:A}) is unstable (ill-posed) in the sense that an arbitrarily small data misfit (\ref{eq:A}) may not necessarily correspond to an estimates $u$ that is close to the optimal. This ill-posedness, that arises from the mathematical structure of $G$, requires regularization. In the approach proposed in \cite{LM}, stable approximations to the minimizer of (\ref{eq:A}) were generated with an application of the regularizing LM scheme. In that application, the regularizing LM scheme was initialized with the prior mean $\overline{u}$. Additionally, it used the prior error covariance $C$ for the regularization that was built into the iterative scheme. 

The regularizing LM scheme uses the \textit{discrepancy principle} \cite{Groetsch} to impose a limit on how close the model predictions must fit the data. More precisely, for some $\tau >1$, the regularizing LM scheme terminates whenever at the $m$ iteration we find that
\begin{eqnarray}\label{discrepancy}
\vert\vert \Gamma^{-1/2} (y-G(u_{m+1}))\vert\vert_{Y}\leq \tau\eta\leq \vert\vert \Gamma^{-1/2} (y-G(u_{m}))\vert\vert_{Y}
\end{eqnarray}
where $u_{m}$ is the LM estimate at the $m$-th iteration and the \textit{noise level} $\eta$ is defined by
\begin{eqnarray}\label{eq:noise}
\eta \equiv \vert\vert \Gamma^{-1/2}(y-G(u^{\dagger}))\vert\vert_{Y}
\end{eqnarray}
where $u^{\dagger}$ denotes the truth properties. In other words, the noise level $\eta$ is the weighted data misfit for the truth $u^{\dagger}$. The discrepancy principle is the key to ensure the regularizing properties of iterative regularization methods \cite{Iterative} and in particular of the  regularizing LM scheme that we applied in \cite{LM} for computing inverse estimates of geologic properties by means of history matching. Although the truth $u^{\dagger}$ is unknown, it is reasonable to assume that knowledge of the noise level $\eta$ is available (see discussion in \cite[Section 3]{LM}).

While \cite{LM} provides numerical evidence of the robustness and accuracy of the regularizing LM scheme for recovering $u^{\dagger}$, a minimizer of (\ref{eq:1.1}) may not be unique. In other words, different approximations to a minimizer of (\ref{eq:1.1}) may be found. For example, initializing the regularizing LM scheme with a different (from the prior mean $\overline{u}$) initial guess may lead to different approximation to the minimizer of (\ref{eq:1.1}). Similarly, a small perturbation of the production data $y$ may result in a significantly different, yet stable, estimate of the geologic properties. Both uncertainties in prior knowledge and observational noise give rise to uncertainty in the computation of inverse estimates of the geologic properties that we often perform with least-squares. The quantification of this uncertainty is essential for the optimal management and decision-making involved in reservoir applications. In this paper we extend the approach of \cite{LM} and apply fundamental ideas from iterative regularization techniques to develop ensemble methods that, within the Bayesian framework \cite{Andrew}, quantify uncertainty of these inverse estimates. More precisely, we provide numerical evidence that this methods produce useful information of the posterior that arises from the Bayesian inverse problem of conditioning the geologic properties to data from the reservoir dynamics.


\subsection{The Bayesian posterior and its ensemble approximation}
The most standard approach to quantify inverse estimates of the geologic properties is the Bayesian framework. The aim is to merge uncertainties, both in prior knowledge and observational noise, with the mathematical model that describes the reservoir flow. The prior uncertainty in geologic properties is incorporated in terms of a {\em prior} probability distribution ${\mathbb P}(u)$. The data $y$ and the geologic properties $u$ are related via the forward operator $G$ by
\begin{eqnarray}\label{eq:1.1}
y=G(u)+\xi
\end{eqnarray}
where $\xi$ is a vector of random noise. The {\em likelihood} of the measured data $y$ given a particular instance of $u$ is denoted by ${\mathbb P}(y|u)$. In the Bayesian framework, the uncertainty of the inverse estimates of $u$ given $y$ is quantified with the conditional {\em posterior} probability ${\mathbb P}(u|y)$, which from Bayes' rule, is given by 
\begin{equation}\label{eq:1.2}
\frac{{\mathbb P}(u|y)}{{\mathbb P}(u)} \propto {\mathbb P}(y|u).
\end{equation}
Since the forward operator $G$ that arises from typical reservoir models is highly nonlinear and/or the prior distribution of geologic properties ${\mathbb P}(u)$ is not Gaussian, the posterior ${\mathbb P}(u|y)$ cannot be described with a few parameters. In general, the characterization of the posterior may be conducted by means of sampling. Unfortunately, this approach often requires millions of forward model evaluations which is only feasible for small problems where a relatively coarse grid is used to discretize the geologic properties \cite{EmeRey,Oliver2}. The full characterization of the posterior by means of sampling is therefore impractical. However, sampling can be used for the purpose of benchmarking and assessing methods that can be actually used in practice \cite{EmeRey,Oliver2,Evaluation}. That is the case of \textit{ensemble methods} which have become the paradigm for capturing features of the posterior ${\mathbb P}(u|y)$ within a reasonable amount of computing resources and time. The general objective of these methods is to provide fast, accurate and robust estimates of the posterior (and/or its mapping under the forward operator) by means of an ensemble $\{u^{(j)}\}_{j=1}^{N_{e}}$ of $N_{e}$ inverse estimates of the geologic properties of the reservoir. From this ensemble we may compute statistical information of the unknown $u$ as well as quantities of interest related to the model predictions $G(u)$. Furthermore, an ensemble of reservoir properties that accurately captures the posterior can be used to quantify the uncertainty in future production scenarios (e.g. drilling new wells) and/or additional subsurface applications (e.g. injection of CO$_2$ in a depleted reservoir). Moreover, efficient ensemble methods may be combined with optimal control modules that closes the loop of reservoir management and monitoring \cite{Binghuai}. It is clear that ensemble methods that provide a robust and accurate approximation of the Bayesian posterior are fundamental for geologic uncertainty quantification in subsurface modeling applications.

Fully rigorous Bayesian sampling is very expensive and often impractical;
to overcome this difficulty there exist two main frameworks to generate an ensemble approximation of the posterior which, although not rigorously justifiable
in general, are widely used within the reservoir modeling community: (i) the randomized maximum likelihood (RML) method and (ii) Kalman-based techniques. The RML method consist of an ensemble of realizations computed by minimizing a randomized version of the cost functional that defines the maximum  a posteriori (MAP) estimate \cite{Oliver}. In other words, each ensemble member of RML involves an optimization problem that may be addressed with similar methods to the ones used for computing the MAP \cite{svdRML,EmeRey,Oliver2,Evaluation}.  When a relatively small number of measurements is available, the computation of the derivative of the forward operator (or sensitivity) may be feasible provided the adjoint of the derivative can be computed. In this case, standard Gauss-Newton or Levenberg-Marquardt methods can be potentially applied for the minimization of each ensemble member of RML \cite{Evaluation,Li}. However, in some cases, large amount of data need to be assimilated and the computation of the sensitivity matrix becomes computationally prohibitive. This limitation of Newton-type methods can be addressed by the application of quasi-Newton methods where only the action of the sensitivity is needed. In \cite{Gao,Tavakoli} and \cite{EmeRey}, for example, the LBFGS was applied for the computation of the MAP and the ensemble of RML, respectively. In some cases, a SVD decomposition of the sensitivity matrix of the MAP may be used as a reparameterization of the search space of the iterative scheme (e.g. LM or LBFGS) used for the optimization. The effect of this reparameterization is twofold. On the one hand, a truncated SVD parameterization reduces the parameter space and could potentially alleviate the ill-posedness of the inverse problem. On the other hand, the aforementioned SVD decomposition may be computed by means of a Lanczos algorithm with the aid of an adjoint method, thus avoiding the explicit computation of the sensitivity matrix \cite{svdRML,Tavakoli}.

Ensemble Kalman-based methods use the Kalman formula for generating an ensemble of inverse estimates of the geologic properties \cite{EnKF_US}. Each ensemble member may be updated in iterative or noniterative and in sequential (filtering) or all-at-once (smoothing) basis \cite{EnKFReview}.  As stated before, variational methods often require an adjoint code for the efficient computation of the sensitivity matrix (or its multiplication by a vector). In contrast, ensemble Kalman-based methods are derivate-free approaches where the ensemble updates are based on simple computations involving covariances and crosscovariances computed directly from ensemble usually initialized with samples from the prior. Thus, ensemble Kalman methods can be used in a black-box fashion. However, it is widely known the lack of stability (or roughness) of ensemble Kalman methods when the ensemble size is small with respect to the number of parameters or measurements. Therefore, in the context of reservoir applications, attention has been recently given to the regularization of ensemble Kalman based methods by means of localization \cite{Emerick3,OliverLocal} and multiple data assimilation \cite{Emerick20133}. For a recent review and comparison of ensemble methods we refer the reader to the work of \cite{EmeRey}. It is also important to mention the work of \cite{EnKF_US} that uses Kalman updates within an iterative ensemble scheme to solve PDE-constrained parameter identification problems. By a subspace property proven in \cite{EnKF_US}, the space of admissible solutions is defined as the subspace generated by the prior ensemble. Thus, the regularization results from defining the search space on a compact subspace. However, the numerical results of \cite{EnKF_US} suggest that an early termination of the scheme is additionally required to further stabilize the iterative scheme.

Methods that combine ideas from both RML and Kalman-based techniques have been developed in \cite{YanChen,YanChenLM,IterativeEnKF,Gu}. The essence of these methods is to preserve the properties of RML to capture the posterior, while using a derivative-free ensemble-based optimization method. For instance, in the LM iteration step for the minimization of each ensemble member in RML, the LM-enRML approach of \cite{YanChenLM} replaces the sensitivity matrix with an adjoint-free approximation computed from the ensemble. Similarly, LM-enRML uses the prior ensemble to approximate the prior error covariance that appears in the LM scheme. The results reported in \cite{YanChenLM} indicate that LM-enRML provides a good approximation of the posterior for a small problem where the posterior can be computed analytically. However, the approximation properties of LM-enRML for capturing the posterior for large problems is still an open problem. 

\subsection{Contribution of this work}

In this paper we contribute to both mainstreams of ensemble methods previously described. On the one hand, we extend our implementation of the regularizing LM of \cite{LM} and develop a variational Iterative Regularizing ensemble LM method (IR-enLM). The aim of this method is to generate a stabilized ensemble of inverse estimates computed from a randomized least-squares minimization. On the other hand, we apply fundamental ideas from iterative regularization and construct an Iterative Regularizing ensemble (Kalman) Smoother (IR-ES). For these proposed ensemble methods the computation of each ensemble member can be posed as the solution of a deterministic nonlinear ill-posed inverse problem that we address by means of iterative regularization methods. The resulting ensemble, and the main goal of these methods is to provide a robust approximation of the Bayesian posterior. However, the regularization of each ensemble member and the resulting approximation of the Bayesian posterior are two independent aspects that combined yield the proposed techniques. The main objective of the present work is to explore the connection between these two aspects and determine the capabilities of these ensemble methods for capturing the posterior. This objective requires understanding the role of the tunable parameters that arise from the application of iterative regularization methods. While such role has been established both theoretical and numerical the context of deterministic inverse problems, here we provide an extensive numerical study to understand these tunable parameters in the context of Bayesian inverse problems that arise in reservoir modeling applications. Our aims is to obtain a range of tunable parameters that corresponds to accurate and computationally feasible approximations of the posterior. 


Although the proposed methods are based on iterative regularization for which a convergence theory is available in the context of deterministic inverse problems \cite{Iterative} the corresponding numerical analysis of the approximation for the solution of Bayesian inverse problems is still an open problem beyond the scope of the present manuscript. In the trivial case where the forward operator is linear and the prior is Gaussian, we provide specific conditions for which the proposed methods provide the proper sampling of the posterior. For the nonlinear case, we use the framework of \cite{Evaluation} to assess the performance of the proposed method for recovering the mean and variance of the posterior. In particular, we use a state-of-the-art MCMC method for functions \cite{David} that enable us to sample posteriors that result from synthetic data from moderate size experiments.

We reiterate that a key objective of the present work is to develop and promote the use of ideas from iterative regularization for the design of ensemble approximation of the posterior that arises from reservoir modeling applications. Thus, the proposed IR-enLM and IR-ES are a proof-of-concept of iterative regularization ideas for ensemble approximations to the Bayesian posterior. We do not suggest to {\em replace} existing methods but rather to demonstrate a methodology to {\em enhance} those methods by the careful choice of regularization parameters and stopping criteria; we show the potential advantages of taking into account methodologies that have been rigorously constructed for the solution of deterministic inverse problems. In particular, the proposed IR-enLM and IR-ES are based on the discrepancy principle for the early termination and selection of parameters. Further applications of iterative regularization can be developed in combination with state-of-the-art frameworks such as the SVD parameterization of \cite{svdRML} and the LM approaches of \cite{YanChenLM}. Those applications, however, are beyond the scope of the present work.

The paper is organized as follows. Relevant aspects of variational iterative regularization are discussed in Section \ref{sec:var}. In particular, a brief description of RML is presented in subsection \ref{sec:RML}. The proposed IR-enLM algorithm is introduces in subsection \ref{sec:IR-enLM}. The linear-Gaussian case is addressed in subsection \ref{linear}. Differences between IR-enLM and a standard implementation of RML is discussed in \ref{sec:RML-LM}. In Section \ref{sec:IR-ESA} we study Kalman-based ensemble methods. More concretely, \ref{sec:ES} we briefly introduce the standard ES. A regularized version of this ES algorithm is introduce in \ref{sec:R-ES}. The proposed IR-ES is then presented in subsection \ref{sec:IR-ES}. Numerical examples of the proposed methods are presented in Section \ref{Numerics}. In subsection \ref{sec:fm} we describe the generation of synthetic data for the experiments. The resulting posteriors are resolved with the methodology described in subsection \ref{sec:pos}. The numerical experiments of the implementation of IR-enLM and IR-ES are displayed in subsections \ref{sec:numIR-enLM} and \ref{sec:numIR-ES}, respectively. Comparison with standard unregularized methods is presented in \ref{comp}. Final remarks are provided in Section \ref{Conclusions}. The forward models under consideration are described in the Appendix.

\section{Variational iterative regularization}\label{sec:var}

In this section we define a variational ensemble method that quantifies uncertainty of inverse estimates of geologic properties. The proposed method is analogous to the randomized maximum likelihood (RML) method of \cite{Oliver} which we review in the following subsection. For simplicity, throughout the rest of this document we consider that the prior distribution of the geologic properties is the Gaussian measure ${\mathbb P}(u)=N(\overline{u},C)$ where as before,  $C$ is the prior covariance and $\overline{u}$ is the prior mean. In addition, we consider that the observational noise $\xi$ in (\ref{eq:1.1}) is centered Gaussian with covariance $\Gamma$. Under these assumptions, (\ref{eq:1.2}) becomes
\begin{equation}\label{eq:1.2B}
{\mathbb P}(u|y) \propto \exp \{-J(u,y)\}
\end{equation}
where
\begin{eqnarray}\label{eq:sa}
J(u,y)\equiv  \frac{1}{2}\vert\vert \Gamma^{-1/2}(y-G(u))\vert\vert_{Y}^{2} +\frac{1}{2}\vert\vert C^{-1/2}(u-\overline{u})\vert\vert_{X}^2 
\end{eqnarray}

\subsection{The randomized maximum likelihood (RML) method}\label{sec:RML}
From (\ref{eq:1.2B})-(\ref{eq:sa}) is easy to see that the posterior distribution ${\mathbb P}(u|y)$ is maximized for $u\in X$ that minimizes
\begin{eqnarray}\label{eq:sa2}
J(u,y)\equiv  \frac{1}{2}\vert\vert \Gamma^{-1/2}(y-G(u))\vert\vert_{Y}^{2} +\frac{1}{2}\vert\vert C^{-1/2}(u-\overline{u})\vert\vert_{X}^2 
\end{eqnarray}
Under certain conditions, such a $u$ exist \cite{Andrew} and is often referred as the maximum a posteriori (MAP) estimate. Since the operator $G$ is nonlinear for the case of interest for the present application, the MAP estimator may not be unique and so multiple modes of the posterior may exist. In this case, the computation of one minimizer of (\ref{eq:sa}) does not provide sufficient information to fully characterize the posterior. In the context of Bayesian subsurface inverse problems \cite{Oliver}, RML has been proposed to approximate the posterior with an ensemble $\{u_{RML}^{(j)}\}_{j=1}^{N_{e}}$ of $N_{e}$ realizations  defined by
\begin{eqnarray}\label{eq:RML}
 u_{RML}^{(j)} = {\rm argmin}_{u} J_{RML}^{(j)}(u)
\end{eqnarray}
where 
\begin{eqnarray}\label{eq:RMLB}
 J_{RML}^{(j)} = \frac{1}{2}\vert\vert \Gamma^{-1/2}(y^{(j)}-G(u))\vert\vert_{Y}^2+\frac{1}{2}\vert\vert C^{-1/2}(u-u^{(j) })\vert\vert_{X}^{2} 
\end{eqnarray}
and
\begin{eqnarray}\label{eq:RML2}
u^{(j)}\sim N(\overline{u},C),\qquad y^{(j)} =y+\xi^{(j)},\qquad \xi^{(j)} \sim N(0,\Gamma). 
\end{eqnarray}
In the case where $G(u)$ is a linear operator (and since we assume that $\mathbb{P}(u)$ is Gaussian) the ensemble obtained with (\ref{eq:RML}) are samples of the posterior distribution \cite{Oliver}. In the nonlinear case, however, the analysis of the approximation properties of RML is still an open problem. Nevertheless, multiple methods for computing or approximating the minimizer of (\ref{eq:RML}) have been proposed in recent years \cite{svdRML,YanChenLM,IterativeEnKF}. As stated in the preceding section, the main focus of these RML-based ensemble methods has been the numerical efficiency and computational feasibility that is essential in practice. While some of these methods may be more efficient and suitable than others, they are all expected to provide an approximation of the posterior similar to the one provided by the ensemble defined in (\ref{eq:RML}). Even though most RML-based methods have been assessed in terms of their capabilities for data-fitting, some recent publications have established the numerical efficacy of some RML implementations for capturing the Bayesian posterior. In concrete, \cite{EmeRey,Evaluation,Oliver2} have used MCMC to fully characterize the posterior which was, in turn, used to assess the performance of ensemble methods including some implementations of RML. For problems where the dimension of the parameter space is very small, it has been found that some of those variational implementations exhibits the best performance at approximating some moments of the posterior \cite{EmeRey,Oliver2}. However, for larger problems, the work in \cite{Evaluation} reported an implementation of RML that overestimated the variance of the posteriors from some test problems which were, in turn, fully resolved with a state-of-the-art MCMC method for functions. This suboptimal behavior of RML can be attributed to (i) the larger size of experiments considered in \cite{Evaluation} and (ii) the optimization method used for the minimization of (\ref{eq:sa}). More specifically, \cite{Evaluation} uses a standard ``unregularized'' LM proposed used in \cite{Li} for minimizing (\ref{eq:sa}) (see also \cite{Oliver} for this standard application of LM). 

In the general nonlinear case, it is quite clear that the RML ensemble approximation of the Bayesian posterior will strongly depend on the numerical technique that employed for the minimization of (\ref{eq:RML}). On the other hand, this minimization that we perform to compute each ensemble member is nothing but a regularized version of a deterministic nonlinear ill-posed inverse problems. Therefore, it should come as no surprise that (\ref{eq:RML}) may be subject to the numerical instabilities discussed in \cite{LM} for the computation of the MAP estimate. More concretely, for each ensemble member, the stable computation of (\ref{eq:RML}) requires that the ``prior term'' $\frac{1}{2}\vert\vert C^{-1/2}(u-u^{(j)})\vert\vert_{X}^{2} $ provides enough regularization to the ill-posed problem of minimizing $\frac{1}{2}\vert\vert \Gamma^{-1/2}(y^{(j)}-G(u))\vert\vert_{Y}^2$. The stabilization of (\ref{eq:RML}) thus relies on the proper selection of the operators $\Gamma$ and $C$. These, however, are typically chosen according to measurement and prior geological information, respectively. Therefore, for some choices of $C$ and $\Gamma$, numerical instabilities in the computation of (\ref{eq:RML}) could potentially arise unless some additional form of regularization is applied. In \cite{LM} we showed synthetic experiments where the computation of the MAP estimator (i.e. minimizer of (\ref{eq:sa})) with a standard unregularized LM method did not ensure the proper regularization of the history matching problem. That is, even though data were successfully fitted, the corresponding estimates of the geologic properties were significantly  different from the truth. While recovering the truth was the main focus of \cite{LM}, in the present work we are interested in the solution to the Bayesian inverse problem, i.e.  the posterior distribution. In Section \ref{Numerics} of the present document we therefore show that, when the standard unregularized LM method of \cite{Oliver} is implemented to minimize (\ref{eq:RML}) in RML, suboptimal approximations of the Bayesian posterior may be obtained. This motivates our application of iterative regularization for the development of a robust and stable variational ensemble approximations of the posterior. It is worth mentioning that the aforementioned instabilities that may arise from the minimization of the MAP (or alternatively RML) has been also observed and addressed in some implementations of incremental variational data assimilation methods for numerical weather prediction applications \cite{TELA:TELA527}.

It is important to remark that some recent implementations of RML use a truncated parameterization of the SVD of the sensitivity matrix \cite{svdRML} within a LM algorithm. This non-standard LM approach has been developed for computational efficiency because it avoids the explicit computations of the sensitivity that is needed in standard LM methods. However, the truncation of the aforementioned SVD parametrization is an additional for of regularization \cite{Hansen}. This truncation has the potential disadvantage of removing eigenvalues of the sensitivity matrix that may be relevant for the proper characterization of petrophysical properties in complex geologies. It is also worth mentioning the recent work of \cite{YanChenLM} where an RML-based approach was used in conjunction with an LM method that, as indicated in Section \ref{Intro} uses ensemble approximations of the sensitivity and prior error covariance. While the approach of \cite{YanChenLM} uses the standard stoping criteria and regularizing parameters, it also incorporates truncated SVD of the variables of interest, thereby inducing a regularization. Moreover, \cite{YanChenLM} reported the use localization as an additional form of regularization. In the following section we present an approach where neither truncation of the spectrum nor localization is used. Instead, basic ideas from iterative regularization are applied to stabilize the computations of the ensemble. Nonetheless, as indicated earlier, the ideas from iterative regularization that we develop below may be potentially to these existing methods where some type of parameterizations (e.g. TSVD basis) and approximations (e.g. of sensitivities) are performed.


\subsection{An Iterative Regularizing ensemble LM (IR-enLM) method}\label{sec:IR-enLM}

We propose a variational ensemble method that aims at providing stable and robust ensemble approximations of the Bayesian posterior. Each member $u_{IR}^{(j)}$ of the ensemble is a stable approximation of a minimizer of
\begin{eqnarray}\label{eq:IR}
\Phi^{(j)}(u)\equiv \vert\vert \Gamma^{-1/2}(y^{(j)}-G(u^{(j)}))\vert\vert_{Y},\qquad j\in\{1,\dots.N_{e}\}
\end{eqnarray}
computed with the regularizing LM scheme, initialized with a random sample from the prior, i.e. $u_{0}^{(j)}\equiv u^{(j)}\sim N(\overline{u},C)$. The data $y^{(j)}$ in (\ref{eq:IR}) is a perturbation of the original data $y$ as defined in (\ref{eq:RML2}). Note that (\ref{eq:IR}) excludes the prior term of (\ref{eq:RML}). For the sake of clarity, in the subsequent lines we briefly outline key aspects of the regularizing LM scheme applied to the computation of stable minimizers of (\ref{eq:IR}). For full details on the implementation of this method for history matching, the reader is referred to \cite{LM}. The theory of the regularizing properties and the convergence of the regularizing LM scheme are found in \cite{Hanke}. This includes the conditions on the forward model $G$ for which the theory ensure convergence to stable solutions (see also \cite{LM}).

The regularizing LM scheme applied to each ensemble member $j\in\{1,\dots.N_{e}\}$ at each iteration level $m$ involves the computation of $u_{m+1}^{(j)}=u_{m}^{(j)}+\Delta u_{m}^{(j)}$, where the increment $\Delta u_{m}^{(j) }$ is defined as the minimizer of 
\begin{eqnarray}\label{eq:1.4}
J_{LM}^{(j)}(w)=\frac{1}{2}\vert\vert \Gamma^{-1/2}(y^{(j)}-G(u_{m}^{(j)})-DG(u_{m}^{(j)})w)\vert\vert_{Y}^{2}+\frac{1}{2}\alpha_{m}^{(j)}\vert\vert C^{-1/2}w\vert\vert_{X}^2
\end{eqnarray}
which from standard arguments can be expressed as
\begin{eqnarray}\label{eq:1.7}
\Delta u_{m}^{(j)}(\alpha_{m}^{(j)})=C \, DG^{*}(u_{m}^{(j)})[DG(u_{m}^{(j)}) \, C \, DG^{\ast}(u_{m}^{(j)})+\alpha_{m}^{(j)} \Gamma]^{-1}[y^{(j)}-G(u_{m}^{(j)})]
\end{eqnarray}
The regularizing LM scheme of Hanke has two main components that ensure the regularization of the minimization of (\ref{eq:IR}): (i) the selection of the regularization parameter $\alpha_{m}^{(j)}$ and (ii) the stopping criteria. These are both based on the discrepancy principle. In concrete, according to the theory of Hanke \cite{Hanke}, $\alpha_{m}^{(j)}$ must satisfy
\begin{eqnarray}\label{eq:1.6}
\vert\vert \Gamma^{-1/2} (y^{(j)}-G(u_{m}^{(j)})-DG(u_{m}^{(j)})\Delta u_{m}^{(j)}(\alpha_{m}^{(j)}))\vert\vert_{Y}^{2}
\ge \rho^{2} \vert\vert \Gamma^{-1/2} (y^{(j)}-G(u_{m}^{(j)}))\vert\vert_{Y}^{2} \nonumber\\
\end{eqnarray}
for some $\rho\in (0,1)$. From (\ref{eq:1.7}) the previous expression can be written as 
\begin{eqnarray}\label{eq:1.8}
\alpha_{m}^{(j)}\vert\vert  \Gamma^{1/2}[DG(u_{m}^{(j)}) \, C \, DG^{\ast}(u_{m}^{(j)})+\alpha_{m}^{(j)}\Gamma]^{-1}[y^{(j)}-G((u_{m}^{(j)})]\vert\vert_{Y}\nonumber\\\ge \rho\vert\vert  \Gamma^{-1/2}(y^{(j)}-G(u_{m}^{(j)}))\vert\vert_{Y}
\end{eqnarray}
The existence of such $\alpha_{m}^{(j)}$ has been shown in \cite{Iterative} and its actual computation can be carried out with a simple iterative scheme \cite{LM} (see also Algorithm \ref{IR-enLM} below). The selection of $\alpha_{m}^{(j)}$ based on (\ref{eq:1.8}) ensures the regularization of the minimizer of (\ref{eq:1.4}). However, to fully stabilize the computation of a minimizer of (\ref{eq:IR}), the iterative regularizing LM scheme is terminated according to the following discrepancy principle
\begin{eqnarray}\label{discrepancyB}
\vert\vert \Gamma^{-1/2} (y^{(j)}-G(u_{k+1}^{(j)}))\vert\vert_{Y}\leq \tau\eta^{(j)}\leq \vert\vert \Gamma^{-1/2} (y^{(j)}-G(u_{k}^{(j)}))\vert\vert_{Y}
\end{eqnarray}
which is the application of (\ref{discrepancy}) for minimizing (\ref{eq:IR}). In the previous expression $\eta^{(j)}$ is the noise level corresponding to the data set $y^{(j)}$. In other words, 
\begin{eqnarray}\label{eq:1.9}
\eta^{(j)}\equiv \vert\vert \Gamma^{-1/2}(y^{(j)}-G(u^{\dagger}))\vert\vert_{Y}
\end{eqnarray}

Since we do not have access to the truth, (and therefore to $G(u^{\dagger})$) and estimate of (\ref{eq:1.9}) is required. Let us recall from our definition (\ref{eq:RML2}) that $y^{(j)}=y+\xi^{(j)}$ and so (\ref{eq:1.9}) becomes
\begin{eqnarray}\label{eq:1.9B}
\eta^{(j)}\equiv \vert\vert \Gamma^{-1/2}(y^{(j)}-G(u^{\dagger}))\vert\vert_{Y}\leq \vert\vert \Gamma^{-1/2}(y-G(u^{\dagger}))\vert\vert_{Y}+\vert\vert \Gamma^{-1/2}\xi^{(j)}\vert\vert_{Y}\nonumber\\
=\eta +\vert\vert \Gamma^{-1/2}\xi^{(j)}\vert\vert_{Y}
\end{eqnarray}
where we have used the definition for the noise level (\ref{eq:noise}). 
On the other hand, from the definition of the perturbed observations (\ref{eq:RML2}) we have
\begin{eqnarray*}
\frac{1}{N_{e}}\sum_{j=1}^{N_{e}} \Gamma^{-1/2}(y^{(j)}-G(u^{\dagger})) =\frac{1}{N_{e}}\sum_{j=1}^{N_{e}} \Gamma^{-1/2}(y+\xi^{(j)}-G(u^{\dagger}))= \Gamma^{-1/2}(y-G(u^{\dagger}))
\end{eqnarray*}
where we have assumed that the finite ensemble of data perturbations has mean zero. From the previous expression it follows that
\begin{eqnarray}\label{eq:1.9B1}
\eta\equiv \vert\vert  \Gamma^{-1/2}(y-G(u^{\dagger})) \vert\vert_{Y}\leq \frac{1}{N_{e}}\sum_{j=1}^{N_{e}} \vert\vert \Gamma^{-1/2}(y^{(j)}-G(u^{\dagger})) \vert\vert_{Y}
\end{eqnarray}
From (\ref{eq:1.9B1}) and (\ref{eq:1.9B}) we therefore find
\begin{eqnarray}\label{eq:1.9B2}
\eta\leq \frac{1}{N_{e}}\sum_{j=1}^{N_{e}} \eta^{(j)}\leq\eta + \frac{1}{N_{e}}\sum_{j=1}^{N_{e}} \vert\vert \Gamma^{-1/2}\xi^{(j)}\vert\vert_{Y}
\end{eqnarray}
Note that an estimate of $\eta^{(j)}$ in the interval
\begin{eqnarray}\label{eq:1.9B3}
 \Big[\eta,\eta+ \vert\vert \Gamma^{-1/2}\xi^{(j)}\vert\vert_{Y}\Big]
\end{eqnarray}
will be consistent with (\ref{eq:1.9B2}). For the present work we propose the midpoint of (\ref{eq:1.9B3}) as an estimate of $\eta^{(j)}$. In other words, 
\begin{eqnarray}\label{estima}
\eta^{(j)}\equiv \eta+ \frac{1}{2}\vert\vert \Gamma^{-1/2}\xi^{(j)}\vert\vert_{Y}
\end{eqnarray}
Since $\xi^{(j)}$ is the (known) data perturbation (\ref{eq:RML2}), our estimate of $\eta^{(j)}$ can be computed. Additionally, as discussed in Section \ref{Intro} we assume that an estimate of the noise level $\eta$ is available. For the computation of the noise level $\eta$ we refer the reader to the discussion in \cite[Section 3]{LM}. 

We now combine the previous ideas in the following iterative method that generates an ensemble of inverse estimates of geologic properties.
\begin{algorithm}[iteratively regularized ensemble LM]\label{IR-enLM} {~~}\\
Let us consider $u_{0}^{(j)}\equiv u^{(j)}\sim N(\overline{u},C)$, $y^{(j)}$ according to (\ref{eq:RML2}). Let $\rho<1$ and $\tau>1/\rho$. Use the following implementation of the regularizing LM scheme to compute an approximation to (\ref{eq:IR}) for $j\in\{1,\dots,N_{e}\}$. For $m=1,\dots$
\begin{itemize}
\item[(1)] \textbf{Stopping rule (Discrepancy Principle)}. If
\begin{eqnarray}\label{discrepancy2}
\vert\vert \Gamma^{-1/2}(y^{(j)}-G(u_m^{(j)}))\vert\vert_{Y}\leq \tau \big(\eta +\frac{1}{2}\vert\vert \Gamma^{-1/2}\xi^{(j)}\vert\vert_{Y}\big)
\end{eqnarray}
stop. Output: $u_{m}^{(j)}$. 
\item[(2)] \textbf{Selection of $\alpha_{m}^{(j)}$}. Let $\alpha_{m,n}^{(j)}=2^{n}\alpha_{m,0}^{(j)}$ for $n\geq 0$ with $\alpha_{m,0}^{(j)}=1$. Let $N\in \mathbb{N}\cup\{0\}$ be the minimum such that $\alpha_{m}^{(j)}\equiv \alpha_{m,N}^{(j)}$ satisfies
\begin{eqnarray}\label{eq:1.10}
\rho^{2}\vert\vert \Gamma^{-1/2}(y^{(j)}-G(u_{m}^{(j)}))\vert\vert_{Y}^2\leq (\alpha_{m}^{(j)})^2 \vert\vert  \Gamma^{1/2}[DG(u_{m}^{(j)}) \, C \, DG^{\ast}(u_{m}^{(j)})+\alpha_{m}^{(j)}\Gamma]^{-1}[y^{\eta}-G(u_{m}^{(j)})]\vert\vert_{Y}^2 \nonumber\\
\end{eqnarray}
\item[(3)]\textbf{Update}. Define 
\begin{eqnarray}\label{eq:1.11}
u_{m+1}^{(j)}=u_{m}^{(j)} +C \, DG^{*}(u_{m}^{(j)})[DG(u_{m}^{(j)}) \, C \, DG^{\ast}(u_{m}^{(j)})+\alpha_{m,N}^{(j)} \Gamma]^{-1}[y^{(j)}-G(u_{m}^{(j)})]~~~
\end{eqnarray}
\end{itemize}
\end{algorithm}
Note that in the right-hand side of (\ref{discrepancy2}) we have used our estimate of the noise level ( \ref{estima}). Also note that ensemble members are independent from one another. Thus, the computation of the ensemble members in Algorithm \ref{IR-enLM} can be easily parallelized. In addition, the iterative selection of $\alpha_{m}^{(j)}$ (step 2 of Algorithm \ref{IR-enLM}) generates an increasing sequence which, as proven in \cite{LM}, ensures that (\ref{eq:1.10}) is satisfied for some $\alpha_{m,N}^{(j)}$ for $N$ finite. 

We reiterate that the proposed IR-enLM algorithm generates an ensemble of minimizers of (\ref{eq:IR}) that aim at capturing the Bayesian posterior $P(u\vert y)$. The main goal of the present work is to use numerical experiments to understand the effect of the tunable parameters $\rho$ and $\tau$ in the associated approximation properties of the proposed method.

\subsection{The tunable parameters of IR-enLM} \label{tunable}

Both $\rho$ and $\tau$ are tunable parameters of the proposed IR-enLM which arise from the application of the regularizing LM scheme of \cite{Hanke}. The requirements $\rho\in (0,1)$ and $\tau>1/\rho$ are part of the set of sufficient conditions for the convergence of the regularizing LM scheme to a stable solution of the deterministic (least-square) inverse problem in the small noise limit. Therefore, these conditions will ensure the convergence to a stable solution of the deterministic (least-square) inverse problem associated to the computation of each ensemble in expression  (\ref{eq:IR}). However, we emphasize that the convergence analysis of the ensemble approximation of Bayesian posterior is nonexistent. In particular, the aforementioned conditions on $\rho$ and $\tau$ have not been studied in the context of solving the Bayesian inverse problem. In Section \ref{Numerics} of the present manuscript we will provide extensive numerical studies to understand role of $\rho$ and $\tau$ in terms of approximating the mean and variance of synthetic posteriors. In the paragraph below we provide a brief discussion to gain intuition on the effect of the tunable parameters on the computation of each ensemble member of IR-enLM.

The parameter $\rho$ in (\ref{eq:1.8}) appears from the application of the discrepancy principle. This controls the output on the linearized data misfit (\ref{eq:1.6}) and prevents data overfitting. For a thorough discussion of the discrepancy principle applied to the selection of $\alpha_{m}^{(j)}$ we refer the reader to \cite{LM}. However, it is important to emphasize that the parameter $\rho$ in (\ref{eq:1.8}) controls the increment $\Delta u_{m}^{(j)}$ on each step of the regularizing  LM scheme applied for the minimization of (\ref{eq:IR}). More precisely, the larger $\rho$ the larger $\alpha_{m}^{(j)}$. Alternatively, we observe from (\ref{eq:1.4}) that larger $\alpha_{m}^{(j)}$'s are associated to smaller minimizers (or increments) $\Delta u_{m}^{(j)}$. Thus, we note that the Tikhonov term (multiplied by $\alpha_{m}^{(j)}$) in (\ref{eq:1.4}) is essential to control the large values that $\Delta u_{m}^{(j)}$ may potentially take because of the ill-posedness of the linearized inverse problem of minimizing (\ref{eq:1.4}) without the regularization term.  Note, however, that the control on the increment $\Delta u_{m}^{(j)}$ given by (\ref{eq:1.8}) is proportional to the data misfit, which is in turn, decreases with the number of iterations. Therefore, the regularizing LM scheme ensures smaller steps of the LM at the beginning of the iterations. For smaller values of $\rho$ (with $\rho\in (0,1)$) the increment $\Delta u_{m}^{(j)}$ can be potentially large. However, in this case (of a smaller $\rho$) larger values of $\tau$ are expected in the discrepancy principle (\ref{discrepancy2}) which, in turn, will ensure stability by the early termination of the scheme. In the latter case the accuracy of the scheme may be compromised. In other words, the smaller $\rho$ the larger $\tau>1/\rho$  and the earlier we may have to stop the IR-enLM according to (\ref{discrepancy2}), leading to potential inaccurate inverse estimates. 

\subsection{Computational cost of IR-enLM} 

On each iteration level and for each ensemble member of IR-enLM, the main computational cost is the explicit computation of $DG(u_{m}^{(j)}) \, C \, DG^{\ast}(u_{m}^{(j)})$. This is however, computationally feasible provided that the adjoint $DG^{\ast}(u_{m}^{(j)}$ is available and that a small number of observations are assimilated. In that case, the cost of constructing $DG(u_{m}^{(j)}) \, C \, DG^{\ast}(u_{m}^{(j)})$ is proportional to the number of observations times the computational cost of solving the adjoint problem of the linearized forward model (i.e. the cost of computing $DG^{\ast}(u_{m}^{(j)})$ which depends on the particular reservoir model under consideration. In addition, IR-enLM requires the evaluation of forward model at each iteration. Therefore, the cost of an $N_{e}$ size IR-enLM is around $N_{e}J$ forward model evaluations plus $N_{e}JN_{M}$ adjoint solves, where $J$ is the number of iterations to achieve convergence. We remark that the working assumption of small number of measurements is essential for the computational feasibility of the proposed IR-enLM. However, this is not a restrictive assumption for other iterative regularization methods. In \cite{LM} we discuss other iterative regularization methods that avoid the explicit computation of $DG(u_{m}^{(j)}) \, C \, DG^{\ast}(u_{m}^{(j)})$ and that can be potentially used in the ensemble-base framework proposed here. However, investigating the use of additional iterative regularization methods is beyond the scope of this manuscript.

\subsection{The linear case}\label{linear}
When $G$ is a linear operator, the tunable parameters in the regularizing LM scheme for history matching can be chosen so that the resulting approximation coincides with the MAP estimator \cite{LM}. Analogously, in the linear case, the following proposition provides conditions of the tunable parameters so that the proposed IR-enLM  Algorithm \ref{IR-enLM} coincides with RML (\ref{eq:RML}), which as proven in \cite{Oliver}, generates samples of the posterior. 

\begin{proposition}\label{propo1}
Let $G$ be a linear operator $G(u)\equiv Gu$. Consider $u_{0}^{(j)}$ and $y^{(j)}$ as in Algorithm \ref{IR-enLM}. Assume that 
\begin{eqnarray}\label{eq:1.12}
\vert\vert\Gamma^{-1/2}(y^{(j)}-Gu_{0}^{j})\vert\vert_{Y} > \Big[\eta +\frac{1}{2}\vert\vert \Gamma^{-1/2}\xi^{(j)}\vert\vert_{Y}\Big]
\end{eqnarray}
for all $j\in \{1,\dots,N_{e}\}$. Let $\rho\in \mathbb{R}$ be such that
\begin{eqnarray}\label{eq:1.13}
\rho < \frac{1}{\vert\vert [G \, C \,G^{\ast}+\Gamma]^{1/2} \Gamma^{-1/2}\vert\vert^{2}}
\end{eqnarray}
and 
\begin{eqnarray}\label{eq:1.14}
\tau> \max\Big\{ \vert\vert [G \, C \,G^{\ast}+\Gamma]^{-1/2} \Gamma^{1/2}\vert\vert^{2}, \frac{1}{\rho}\Big\}  \frac{ \max_{j}\vert\vert\Gamma^{-1/2}(y^{(j)}-Gu_{0}^{j})\vert\vert_{Y}}{\eta +\frac{1}{2}\min_{j}\vert\vert \Gamma^{-1/2}\xi^{(j)}\vert\vert_{Y}}\nonumber\\
\end{eqnarray}
Then, IR-enLM produces samples of the posterior distribution. 
\end{proposition}
\begin{proof}
Note that if assumption (\ref{eq:1.12}) does not hold, then the proposed algorithm will terminate and the prior ensemble yields the solution. However, from this assumption and (\ref{eq:1.14}) we find
\begin{eqnarray}\label{eq:1.15}
\tau> \frac{1}{\rho}\frac{ \max_{j}\vert\vert\Gamma^{-1/2}(y^{(j)}-Gu_{0}^{j})\vert\vert_{Y}}{\eta +\frac{1}{2}\min_{j}\vert\vert \Gamma^{-1/2}\xi^{(j)}\vert\vert_{Y}}\ge  \frac{1}{\rho} \frac{ \vert\vert\Gamma^{-1/2}(y^{(j)}-Gu_{0}^{j})\vert\vert_{Y}}{\eta +\frac{1}{2}\vert\vert \Gamma^{-1/2}\xi^{(j)}\vert\vert_{Y}}>\frac{1}{\rho}.
\end{eqnarray}
Additionally, with a similar argument to the one used in \cite[Proposition 1]{LM} we can show that $\rho$ in (\ref{eq:1.13}) satisfies $\rho<1$ and that the initial guess $\alpha_{0,0}^{(j)}=1$ satisfies inequality in (\ref{eq:1.10}). Therefore, $N=0$, $\alpha_{0}^{(j)}=1$ and so expression (\ref{eq:1.11}) becomes 
\begin{eqnarray}\label{eq:1.16}
u_{1}^{(j)}=u_{0}^{(j)} +C \, G^{*}[G \, C \, G^{\ast}+ \Gamma]^{-1}[y^{(j)}-G u_{0}^{(j)}]
\end{eqnarray}
which is the minimizer of (\ref{eq:RMLB}) with $G(u)=Gu$. In other words, $\{u_{1}^{(j)}\}_{j=1}^{N_{e}}$ is an ensemble generated with the RML method described above. Therefore, as proven in \cite{Oliver}, $\{u_{1}^{(j)}\}_{j=1}^{N_{e}}$ are samples of the posterior. We now show that $u_{1}^{(j)}$ satisfies the stopping criteria with $\tau$ given by (\ref{eq:1.14}). Note that
\begin{eqnarray}\label{eq:1.21}
\vert\vert \Gamma^{-1/2}(y^{(j)}-Gu_1^{(j)})\vert\vert_{Y} =\vert\vert \Gamma^{1/2}[G \, C \,G^{\ast}+\Gamma]^{-1} (y^{(j)}-Gu_o^{(j)})\vert\vert_{Y}\nonumber\\
\leq \vert\vert [G \, C \,G^{\ast}+\Gamma]^{-1/2}\Gamma^{1/2}\vert\vert^2 \vert\vert \Gamma^{-1/2}(y^{(j)}-Gu_o^{(j)})\vert\vert_{Y}\nonumber\\
\leq  \vert\vert [G \, C \,G^{\ast}+\Gamma]^{-1/2}\Gamma^{1/2}\vert\vert^2 \max_{j}
\vert\vert \Gamma^{-1/2}(y^{(j)}-Gu_o^{(j)})\vert\vert_{Y}\nonumber\\<\tau  \Big(\eta +\frac{1}{2}\min_{j}\vert\vert \Gamma^{-1/2}\xi^{(j)}\vert\vert_{Y}\Big)\leq \tau  \Big(\eta +\frac{1}{2}\vert\vert \Gamma^{-1/2}\xi^{(j)}\vert\vert_{Y}\Big)
\end{eqnarray}
which implies that $\{u_{1}^{(j)}\}_{j=1}^{N_{e}}$ satisfies our stoping criteria (\ref{discrepancy2}). Therefore, in the linear case, each ensemble of the proposed algorithm terminates after one iteration and the resulting ensemble members are samples of the posterior distribution. $\Box$
\end{proof}
\begin{remark}
While the linear case is not relevant for most reservoir applications, the previous result provides conditions on the tunable parameters under which the consistency of the IR-enLM for sampling the (Gaussian) posterior can be established. Note that the conditions (\ref{eq:1.13}) and (\ref{eq:1.14}) are consistent to the conditions required by the regularizing LM scheme to ensure stability (i.e. $\rho \in (0,1)$ and $\tau>1/\rho)$. However, we emphasize that the conditions (\ref{eq:1.13}) and (\ref{eq:1.14}) do not ensure that, in the nonlinear case, the ensemble IR-enLM provides the proper sampling of the posterior. In Section \ref{Numerics} we provide a numerical study, on a nonlinear reservoir model, to understand the approximation properties of IR-enLM and its dependence on the choice of the tunable parameters described earlier.
\end{remark}

\subsection{Differences between IR-enLM and the standard unregularized LM method for RML}\label{sec:RML-LM}

We note that when $G$ is linear, the ensemble generated with RML can be obtained explicitly by formula (\ref{eq:1.16}). In this case we were also able to shown in Proposition \ref{propo1} an equivalence between RML and the proposed IR-enLM method. In the nonlinear case relevant for the present application such equivalence is nonexistent. Moreover, any comparison between RML and the proposed IR-enLM depends on the particular implementation for the minimization of (\ref{eq:RMLB}). For the present work we consider what we refer as the ``standard unregularized LM scheme for RML'' (i.e. for minimizing (\ref{eq:RMLB})). This ``unregularized'' method does not add any additional form of regularization, like for example,  a truncated SVD parameterization which, as stated before, has the effect of regularizing the LM step. We emphasize, however, that this standard unregularized LM is applied to the minimization of (\ref{eq:RMLB}) which is already regularized. The main point of this section and the associated numerical experiments of subsection \ref{comp} is to show that the intrinsic regularization in (\ref{eq:RMLB}) combined with the standard methods of optimization may generate unstable ensemble members which are, in turn, detrimental to the robustness for recovering/approximating the Bayesian posterior.

The standard unregularized LM method applied for generating ensembles of RML is defined by $u_{m+1}^{(j)}=u_{m}^{(j)}+\Delta u^{(j)}$ where the step $\Delta u^{(j)}$ satisfies
\begin{eqnarray}\label{eq:1.25}
\Big[ DG^{\ast}(u_{m}^{(j)})\Gamma^{-1}DG(u_{m}^{(j)})+C^{-1}+\lambda_{m}^{(j)} C^{-1}\Big]\Delta u^{(j)}= DG^{\ast}(u_{m}^{(j)})\Gamma^{-1}[y^{(j)}-G(u_{m}^{(j)})-C^{-1}(u_{m}^{(j)}-u^{(j)})]\nonumber\\
\end{eqnarray}
We reiterate that this unregularized version of the LM method excludes truncated SVD parametrization and further localization/inflation of the associated matrices in (\ref{eq:1.25}). 

Most applications of both regularized and unregularized LM approaches for RML-based methods \cite{svdRML,Oliver,YanChen,Li,Tavakoli} consistently choose the initial regularization LM parameter $\lambda=\lambda_{0}$ in (\ref{eq:1.25}) as $\Lambda_{0}\equiv \min\{\sqrt{J(u_{0})/N_{D}},J(u_{0})/N_{D}\}$. For $m\ge 0$, $\lambda_{m+1}$ is then chosen according to
\begin{eqnarray}\label{eq:1.26}
\lambda_{m+1}=\left\{\begin{array}{cc}
\lambda_{m}/\kappa& \textrm{if}~~J(u_{m+1})< J(u_{m})\\
\kappa\lambda_{m} & \textrm{if}~~J(u_{m+1})\ge J(u_{m})\end{array}\right.
\end{eqnarray}
where $\kappa=10$ is the typical selection \cite{svdRML,Oliver,YanChen,Li,Tavakoli}. Moreover, all these approaches use the following stopping criteria
\begin{eqnarray}\label{eq:1.21B}
\frac{\vert J(u_{m+1})-J(u_{m})\vert }{J(u_{m+1})}\leq \epsilon_{0},\qquad \frac{\vert\vert u_{m+1}-u_{m}\vert\vert_{X} }{ \vert\vert u_{m+1} \vert\vert_{X} }\leq \epsilon_{1}\label{eq:3.9D}
\end{eqnarray}
In addition, it has been often claim that a successful application of these methods requires to produce estimates that satisfy
\begin{eqnarray}\label{eq:1.22}
J_{RML}^{(j)}(u^{(j)})\leq  N_{D}+5\sqrt{2N_{D}}
\end{eqnarray}
where $N_{D}$ is the number of measurements (i.e. the dimension of $Y$). While  (\ref{eq:1.22}) has been proven under the linear assumption on $G$, for the general case, an analogous estimate is nonexistent. It is important to reiterate that, a sufficient decrease in the data misfit may not be associated with a decrease of the error with respect to the truth. 

For a thorough discussion of the technical differences between the standard unregularized LM method and the regularizing LM scheme we refer the reader to the discussion of \cite[section 3.3]{LM}. In subsection \ref{comp} we display experiments where the standard unregularized LM implementation of RML leads to ensembles that exhibit suboptimal performance in capturing the posterior distribution. This lack of stability and the resulting uncontrolled estimates of this scheme arise from (i) the selection of the regularization parameter $\lambda_{m}^{(j)}$  and (ii) the stropping criteria that do not prevent overfitting of the data. In contrast, we observe that the regularizing LM scheme implemented within our IR-enLM uses the discrepancy principle to control selection of the LM parameter as well as the early termination of the algorithm. Therefore, even though (\ref{eq:1.21B}) and (\ref{eq:1.22}) have been shown to produce acceptable numerical results in cases where the problem has been further regularized (e.g. like with truncated SVD \cite{svdRML}), the general application of the selection of regularization parameters and stopping criteria should be taken with caution.

\section{Ensemble Kalman-based methods}\label{sec:IR-ESA}

Using the derivative of the forward operator in variational models often results in more accurate inverse estimates of the subsurface properties compared to the ones generated with ensemble Kalman-based methods. However, as indicated earlier, Kalman-based methods are often easier to implement and possibly the only choice when adjoint codes are not available. Based on ideas from iterative regularization methods, in this section we construct an ensemble Kalman smoother suitable for an all-at-once formulation consistent with our forward operator $G$ that comprises all data (\ref{eq:1.1}). Nevertheless, the ideas presented here can be developed in a sequential formulation by redefining the forward problem. 

\subsection{The standard unregularized ensemble smoother}\label{sec:ES}
The standard smoother of \cite{evensen2009data} consist of generating an ensemble 
\begin{eqnarray}\label{eq:1.27}
u^{(j,a)} =u^{(j,f)}+C^{uw}(C^{ww} +\Gamma   )^{-1}(y^{(j)}-w^{(j,f)})
\end{eqnarray}
where $u^{(j,f)}\sim N(\overline{u},C)$, $y^{(j)}$ is defined as in (\ref{eq:RML2}), and
\begin{eqnarray}\label{eq:1.28}
w^{(j,f)}\equiv G(u^{(j,f)}),\qquad \overline{w}^{f}=\frac{1}{N_{e}}\sum_{j=1}^{N_{e}}w^{(j,f)},\qquad \overline{u}^{f}=\frac{1}{N_{e}}\sum_{j=1}^{N_{e}}u^{(j,f)} \nonumber\\
C^{ww}=\frac{1}{N_{e}}\sum_{j=1}^{N_{e}}(w^{(j,f)}-\overline{w}^{f})(w^{(j,f)}-\overline{w}^{f})^T,\qquad C^{uw} = \frac{1}{N_{e}}\sum_{j=1}^{N_{e}} (u^{(j,f)}-\overline{u}^{f})(w^{(j,f)}-\overline{w}^{f})^T\nonumber\\
\end{eqnarray}

Similar to straightforward applications of EnKF, the implementation of ES often provides poor data misfit and/or very large/rough values of the ensemble members. Moreover, these standard Kalman-based methods usually underestimates the variance due to the collapse of the ensemble members towards the ensemble mean \cite{Evaluation,EmeRey}. As described in Section \ref{Intro}, the aforementioned behavior has become more evident when a small ensemble size is considered compared to the observations or the dimensions of the parameter space. Several forms of regularization to alleviate this ill-behavior have been proposed in terms of covariance localization, covariance inflation and/or redefining the filter in a square-root fashion. In the present section we demonstrate how an alternative form of regularization can be achieved by introducing ideas from iterative regularization to the definition of the ES. 

\subsection{The regularizing ES}\label{sec:R-ES}

Let us note that model predictions can also be updated according to the following formula
\begin{eqnarray}\label{eq:1.29}
w^{(j,a) }=w^{(j,f)}+C^{ww}(C^{ww} +\Gamma   )^{-1}(y^{(j)}-w^{(j,f)})
\end{eqnarray}
Therefore, if we define
\begin{eqnarray}\label{eq:1.30}
z=\left(\begin{array}{c}
u\\
w\end{array}\right),\qquad \Xi(z)=\left(\begin{array}{c}
u\\
G(u)\end{array}\right)\qquad Z\equiv X\times Y
\end{eqnarray}
we can now express (\ref{eq:1.27}) and (\ref{eq:1.29}) as 
\begin{eqnarray}\label{eq:1.31}
z^{(j,a) }=z^{(j,f)}+C^{f} H^{T} \Big(  HC^fH^{T}+\Gamma\Big)^{-1}(y^{(j)}-Hz^{(j,f)})
\end{eqnarray}
 where $H=(0,I)$ and 
\begin{eqnarray}\label{eq:1.32}
C^f=\left(\begin{array}{cc}
C^{uu}& C^{uw}\\
(C^{uw})^T&C^{ww}\end{array}\right)
\end{eqnarray}
From the expression (\ref{eq:1.31}) it follows that 
\begin{eqnarray}\label{eq:1.33}
z^{a }\equiv \frac{1}{N_{e}}\sum_{j=1}^{N_{e}}z^{(j,a) }=z^{f}+C^{f} H^{T} \Big(  HC^fH^{T}+\Gamma\Big)^{-1}(y-Hz^{f})
\end{eqnarray}
where $z^{f }\equiv \frac{1}{N_{e}}\sum_{j=1}^{N_{e}}z^{(j,f) }$. It is straightforward to show that (\ref{eq:1.33}) is equivalent to 
\begin{eqnarray}\label{eq:1.34}
z^{a}=\textrm{argmin}_{z}\Big(\vert\vert \Gamma^{-\frac{1}{2}}(y-Hz)\vert\vert_{Y}^2+\vert\vert (C^{f})^{-\frac{1}{2}} (z-z^{f})\vert\vert_{Z}^2\Big)
\end{eqnarray}
If we now define,
\begin{eqnarray}\label{eq:1.35}
w^{a}\equiv (C^{f})^{-\frac{1}{2}}(z^{a}-z^{f}),\qquad w\equiv (C^{f})^{-\frac{1}{2}}(z-z^{f}),\nonumber\\
d\equiv y-Hz^{f},\qquad L\equiv H (C^{f})^{\frac{1}{2}}
\end{eqnarray}
and substitute these expression in (\ref{eq:1.34}) we obtain
\begin{eqnarray}\label{eq:1.36}
w^{a}=\textrm{argmin}_{w}\Big(\vert\vert \Gamma^{-\frac{1}{2}}(d-Lw)\vert\vert_{Y}^2+\vert\vert  w\vert\vert_{Z}^2\Big)
\end{eqnarray}
which we recognize as the Tikhonov regularization method (with Tikhonov-parameter $\alpha=1$) applied to the following linear inverse problem:
\begin{eqnarray}\label{eq:ip}
 \textrm{Given~~$d$, ~~find  ~~$w\in Z$ ~~ such~  that ~~  $Lw=d$}
\end{eqnarray}
In subsection \ref{comp} we provide examples to show that the intrinsic choice of $\alpha=1$ in (\ref{eq:1.36}) does not necessarily provide enough regularization to this linear inverse problem. Therefore, this unregularized standard ES produces rough estimates with large data misfit like the ones that have been consistently reported in the literature \cite{Evaluation,EmeRey}. Moreover, this lack of regularization is, in turn, detrimental to the characterization of the posterior obtained from the ensemble approximation. 

The essence of the propose regularizing ES is to include a regularization parameter $\alpha$ so that the resulting problem (\ref{eq:1.36}) takes the form
\begin{eqnarray}\label{eq:1.37}
w^{a}(\alpha)=\textrm{argmin}_{w}\Big(\vert\vert \Gamma^{-\frac{1}{2}}(d-Lw)\vert\vert_{Y}^2+\alpha\vert\vert  w\vert\vert_{Z}^2\Big)
\end{eqnarray}
More importantly, we propose this $\alpha$ to satisfy a discrepancy principle similar to the one used in the regularizing LM scheme (\ref{eq:1.6}). More precisely, we require
\begin{eqnarray}\label{eq:1.38}
\vert\vert \Gamma^{-1/2}(d-Lw^{a}(\alpha))\vert\vert_{Y}\ge \rho\vert\vert \Gamma^{-1/2}d\vert\vert_{Y}
\end{eqnarray}
for some choice of $\rho\in (0,1)$. Let us now define the ``true'' parameters 
\begin{eqnarray}\label{eq:1.39}
w^{\dagger}\equiv (C^{f})^{-\frac{1}{2}}(z^{\dagger}-z^{f}),\qquad z^{\dagger}\equiv \left[\begin{array}{c} u^{\dagger}\\ G(u^{\dagger})\end{array}\right]
\end{eqnarray}
where $u^{\dagger}$ denotes the true geologic property. Furthermore, we define the true data misfit
\begin{eqnarray}\label{eq:1.40}
d^{\dagger}\equiv Lw^{\dagger} =G(u^{\dagger})-Hz^{f}.
\end{eqnarray}
From (\ref{eq:1.35}) and (\ref{eq:1.40}) we find
\begin{eqnarray}\label{eq:1.41}
\vert\vert \Gamma^{-1/2}(d-d^{\dagger})\vert\vert_{Y}=\vert\vert \Gamma^{-1/2}(y-G(u^{\dagger}))\vert\vert_{Y}\equiv \eta
\end{eqnarray}
where we have used the definition of the noise level (\ref{eq:noise}). Therefore, if we assume that 
\begin{eqnarray}\label{eq:1.42}
\eta \leq \rho \vert\vert \Gamma^{-1/2}d\vert\vert_{Y},
\end{eqnarray}
our choice of $\alpha$ in (\ref{eq:1.38}) yields
\begin{eqnarray}\label{eq:1.43}
\vert\vert \Gamma^{-1/2}(d-d^{\dagger})\vert\vert_{Y}= \eta\leq \rho \vert\vert\Gamma^{-1/2}d\vert\vert_{Y}\leq 
\vert\vert \Gamma^{-1/2}(d-Lw^{a}(\alpha))\vert\vert_{Y}
\end{eqnarray}
which is a discrepancy principle applied to the inverse problem (\ref{eq:ip}). Recall that a similar application of the discrepancy principle is used for the selection of the regularizing LM scheme (see expression (16) in LM). Therefore, by means of (\ref{eq:1.43}) our selection of $\alpha$ controls the ensemble update, so that the resulting mean $w^{a}$ does not produce a data fit better than the noise level. Thus, we seek to avoid data overfitting which may, in turn, cause the lack of stability previously discussed. 

In order to produce an ensemble whose transformed mean satisfies (\ref{eq:1.38}) we define
\begin{eqnarray}\label{eq:1.44}
z^{(j,a) }=z^{(j,f)}+C^{f} H^{T} \Big(  HC^fH^{T}+\alpha\Gamma\Big)^{-1}(y^{(j)}-Hz^{(j,f)})
\end{eqnarray}
whose mean satisfies
\begin{eqnarray}\label{eq:1.45}
z^{a}=\textrm{argmin}_{z}\frac{1}{2}\Big(\vert\vert \Gamma^{-\frac{1}{2}}(y-Hz)\vert\vert^2+\alpha\vert\vert (C^{f})^{-\frac{1}{2}} (z-z^{f})\vert\vert^2\Big)
\end{eqnarray}
We now use (\ref{eq:1.35}) to rewrite (\ref{eq:1.38}) as
\begin{eqnarray}\label{eq:1.46}
\rho \vert\vert \Gamma^{-1/2}(y-Hz^{f})\vert\vert_{Y}\leq  \vert\vert \Gamma^{-1/2}(y-Hz^{a}))\vert\vert_{Y}
\end{eqnarray}
which in terms of (\ref{eq:1.45}) and (\ref{eq:1.30})-(\ref{eq:1.32}) can be written as follows 
\begin{eqnarray}\label{eq:1.47}
\rho\vert\vert \Gamma^{-1/2}(y-\overline{w}^{f})\vert\vert_{Y}\leq \alpha\vert\vert \Gamma^{1/2}(C^{ww,f} +\alpha\Gamma   )^{-1}(y-\overline{w}^{f})\vert\vert_{Y}
\end{eqnarray}
which shows that the actual computation of $\alpha$ is analogous to the one in (\ref{eq:1.10}) for the IR-enLM algorithm defined in the preceding section. Finally, we use (\ref{eq:1.35}) to rewrite the assumptions (\ref{eq:1.42}) as
\begin{eqnarray}\label{eq:1.46}
\eta \leq \rho \vert\vert \Gamma^{-1/2}(y-Hz^{f})\vert\vert_{Y}=\rho\vert\vert \Gamma^{-1/2}(y-\overline{w}^{f})\vert\vert_{Y}
\end{eqnarray}
This assumption motivates the stopping criteria of our Iterative Regularizing ES (IR-ES) that we describe in the following sections.

\subsection{The Iterative Regularizing ensemble Smoother (IR-ES)}\label{sec:IR-ES}

The previous regularizing ES may alleviate the ill-posedness of the linear inverse problem (\ref{eq:ip}). While the aim of the proposed ensemble methods is to avoid data overfitting, the aforementioned linearized inversion may not necessarily produce estimates with a reasonably small data misfit  that properly explore the regions of the posterior with large probability. Therefore, in order to allow for more accurate approximations of the posteriors without the risk of overfitting data, in the following algorithm we define an iterative regularized ES scheme terminated according to the discrepancy principle applied to the mean of the updated (analyzed) model predictions $w_{m}^{(j,a)}$ at a given iteration.

\begin{algorithm}[Iterative regularizing ensemble Kalman smoother]\label{IR-ES}~~
\newline  \textbf{Prior ensemble and perturbed noise.} Let $\rho<1$ and $\tau>1/\rho$. Generate
\begin{eqnarray}\label{eq:2.1}
u_{0}^{(j)}\sim \mu_0,\qquad y^{(j)}\equiv y+\xi^{(j)},\qquad \xi^{(j)}\sim N(0,\Gamma), ~~~j\in \{1,\dots,N_{e}\}
\end{eqnarray}
 \textbf{Iterative Smoother.} For $m=0,\dots$
\begin{itemize}
\item[(1)] \textbf{Prediction Step:} Evaluate 
\begin{eqnarray}\label{eq:2.3}
w_{m}^{(j,f)}= G(u_{m}^{(j)})\qquad j\in \{1,\dots,N_{e}\}
\end{eqnarray}
\item[(2)] \textbf{Stoping criteria:} Compute 
$\overline{w}_{m}^{f}=\frac{1}{N_{e}}\sum_{j=1}^{N_{e}}w_{m}^{(j,f)}$. If
\begin{eqnarray}\label{eq:2.2}
\vert\vert \Gamma^{-1/2}(y-\overline{w}_{m}^{f}))\vert\vert_{Y}\leq \tau \eta,
\end{eqnarray}
stop. Output: $\{u_{m}^{(j)} \}_{j=1}^{N_{e}}$. Compute
\begin{eqnarray}\label{cross1}
\overline{u}_{m}=\frac{1}{N_{e}}\sum_{j=1}^{N_{e}}u_{m}^{(j)}\qquad C_{m}^{uw} = \frac{1}{N_{e}}\sum_{j=1}^{N_{e}} (u_m^{(j)}-\overline{u}_m)(w_m^{(j,f)}-\overline{w}_m^{f})^T\\
 C_{m}^{ww}=\frac{1}{N_{e}}\sum_{j=1}^{N_{e}}(w_m^{(j,f)}-\overline{w}_m^{f})(w_m^{(j,f)}-\overline{w}_m^{f})^T
\label{cross2}
\end{eqnarray}
\item[(3)] \textbf{Analysis step:} Compute the updated ensembles
\begin{eqnarray}\label{eq:2.4}
u_{m+1}^{(j)} =u_{m}^{(j)}+C_{m}^{uw}(C_{m}^{ww} +\alpha_{m}\Gamma   )^{-1}(y^{(j)}-w_{m}^{(j,f)})
\end{eqnarray}
\begin{eqnarray}\label{eq:2.5}
w_{m+1}^{(j,a) }=G(u_{m}^{(j)})+C_{m}^{ww}(C_{m}^{ww} +\alpha_{m}\Gamma   )^{-1}(y^{(j)}-w_{m}^{(j,f)})
\end{eqnarray}
for $\alpha_{m}$ such that
\begin{eqnarray}\label{eq:2.6}
\alpha_{m}\vert\vert \Gamma^{1/2}(C_{m}^{ww,f} +\alpha_{m}\Gamma   )^{-1}(y^{\eta}-\overline{w}_{m}^{f})\vert\vert_{Y}\leq \rho\vert\vert \Gamma^{-1/2}(y^{\eta}-\overline{w}_{m}^{f})\vert\vert_{Y}
\end{eqnarray}
(compute $\alpha_{m}$ with the increasing sequence of Algorithm \ref{IR-enLM} having an initial guess $\alpha_{m}=1$).
\end{itemize}
\end{algorithm}
\begin{remark}\label{remaIR}
Note that the stoping criteria (\ref{eq:2.2}) is consistent with the assumption in (\ref{eq:1.46}). Indeed, if at a given iteration level $m$, $\overline{w}_{m}^{f}$ satisfies
\begin{eqnarray}\label{eq:1.46B}
\eta \tau< \vert\vert \Gamma^{-1/2}(y-\overline{w}_{m}^{f}))\vert\vert_{Y},
\end{eqnarray}
for some $\tau$ with $\tau>1/\rho>1$, then
\begin{eqnarray}\label{eq:1.46BBC}
\eta<\rho \vert\vert \Gamma^{-1/2}(y-\overline{w}_{m}^{f}))\vert\vert_{Y},
\end{eqnarray}
which is assumption (\ref{eq:1.46}) which, in turn, ensures we select the regularization parameter according to the discrepancy principle. Note that, in contrast to IR-enLM where the regularizing LM scheme was applied to each ensemble computed as the minimizer of (\ref{eq:IR}), in the IR-ES the regularization is conducted over the mean of IR-ES. Since the ensemble of perturbed observation is assumed centered, the resulting equation (\ref{eq:1.37}) only contains the unperturbed (i.e. original) data. Thus, the discrepancy principle is applied with the original noise level $\eta$. 

\end{remark}
Similar to IR-enLM, the proposed IR-ES aims at providing an ensemble approximation of the Bayesian posterior $P(u\vert y)$. However, the convergence properties of the proposed IR-ES are an open problem beyond the scope of this paper. Nevertheless, in the linear case (recall we assume a Gaussian prior), the proposed algorithm coincides with the standard ES which, in turn, generates samples of the posterior as the ensemble size goes to infinity. 
\begin{proposition}\label{prp:2}
Let $G$ be a linear operator $G(u)\equiv Gu$. Assume that the mean of the prior ensemble satisfies
\begin{eqnarray}\label{eq:1.47}
\vert\vert\Gamma^{-1/2}(y-G\overline{u}_{0})\vert\vert_{Y}>\eta
\end{eqnarray}
where $\eta$ is the noise level defined by (\ref{eq:noise}). Let $\rho\in \mathbb{R}$ be such that
\begin{eqnarray}\label{eq:1.48}
\rho < \frac{1}{\vert\vert [G \, C_{0}^{uu} \,G^{T}+\Gamma]^{1/2} \Gamma^{-1/2}\vert\vert^{2}}
\end{eqnarray}
and 
\begin{eqnarray}\label{eq:1.49}
\tau> \max\Big\{ \vert\vert [G \, C_{0}^{uu} \,G^{T}+\Gamma]^{-1/2} \Gamma^{1/2}\vert\vert^{2}, \frac{1}{\rho}\Big\}  \frac{ \vert\vert\Gamma^{-1/2}(y-G\overline{u}_{0})\vert\vert_{Y}}{\eta}\nonumber\\
\end{eqnarray}
Then, this selection of $\rho$ and $\tau$ ensure that the IR-ES method samples the posterior as $N_{e}\to \infty$.
\end{proposition}
\begin{proof}
From the linearity of $G$ if follows that 
\begin{eqnarray}\label{eq:1.50}
\overline{w}_{0}^{f}= G\overline{u}_{0},\qquad C_{0}^{uw} =C_{0}^{uu}G^{T},\qquad C_{0}^{ww}  =GC_{0}^{uu}G^{T}
\end{eqnarray}
where
\begin{eqnarray*}
C_{0}^{uu} =\frac{1}{N_{e}}\sum_{j=1}^{N_{e}} (u_0^{(j)}-\overline{u}_{0})(u_0^{(j)}-\overline{u}_0)^{T}
\end{eqnarray*}
With similar arguments to those used in the proof of Proposition \ref{propo1} (see also Proposition 1 in \cite{LM}) we can show that $\rho$ and $\tau$ in (\ref{eq:1.48})-(\ref{eq:1.49}) satisfy $\rho<1$, $\tau>1/\rho$. Moreover, these choices of $\rho$ and $\tau$ also imply that $\alpha_{0}=1$ satisfies (\ref{eq:2.6}). Therefore, (\ref{eq:2.4}) becomes
\begin{eqnarray*}
u_{1}^{(j)} =u_{0}^{(j)}+C_{0}^{uw}(C_{0}^{ww} +\alpha_{0}\Gamma   )^{-1}(y^{(j)}-Gu_{0}^{(j)})
\end{eqnarray*}
which from (\ref{eq:1.50}) can be written as 
\begin{eqnarray}\label{eq:1.51}
u_{1}^{(j)} =u_{0}^{(j)}+C_{0}^{uu}G^{T}(GC_{0}^{uu}G^{T} +\Gamma   )^{-1}(y^{(j)}-Gu_{0}^{(j)})
\end{eqnarray}
In addition, 
\begin{eqnarray*}
w_{1}^{(j,a) }=Gu_{0}^{(j)}+GC_{0}^{uu}G^{T}(GC_{0}^{uu}G^{T}+\Gamma   )^{-1}(y^{(j)}-Gu_{0}^{(j)})
\end{eqnarray*}
and therefore
\begin{eqnarray*}
\overline{w}_{1}^{a}= \frac{1}{N_{e}}\sum_{j=1}^{N_{e}}w_{1}^{(j,a)} =G\overline{u}_{0}+C_{0}^{uu}G^{T}(GC_{0}^{uu}G^{T} +\Gamma   )^{-1}(y-G\overline{u}_{0})
\end{eqnarray*}
Thus,
\begin{eqnarray*}
\vert\vert \Gamma^{-1/2}(y-\overline{w}_1^a)\vert\vert_{Y} =\vert\vert \Gamma^{-1/2}(y-G\overline{u}_1)\vert\vert_{Y} =\vert\vert \Gamma^{1/2}[G \, C_{0}^{uu} \,G^{\ast}+\Gamma]^{-1} (y-G\overline{u}_0)\vert\vert_{Y}\nonumber\\
\end{eqnarray*}
which from some simple computations implies
\begin{eqnarray*}
\vert\vert \Gamma^{-1/2}(y-\overline{w}_1^a)\vert\vert_{Y} 
\leq \vert\vert [G \, C_{0}^{uu} \,G^{\ast}+\Gamma]^{-1/2}\Gamma^{1/2}\vert\vert^2 \vert\vert \Gamma^{-1/2}(y-G\overline{u}_0)\vert\vert_{Y}\nonumber\\
\leq  \vert\vert [G \, C_{0}^{uu} \,G^{\ast}+\Gamma]^{-1/2}\Gamma^{1/2}\vert\vert^2 
\vert\vert \Gamma^{-1/2}(y-G\overline{u}_0)\vert\vert_{Y}< \tau \eta
\end{eqnarray*}
where in the last inequality we have used (\ref{eq:1.49}). Therefore, the ensemble (\ref{eq:1.51}) satisfies the stopping criteria and so it is the output of the IR-ES scheme. Moreover, we note that
$$\overline{u}_{1}=\overline{u}_{0}+C_{0}^{uu}G^{T}(GC_{0}^{uu}G^{T} +\Gamma   )^{-1}(y-G\overline{u}_{0})$$
and
$$C_{1}^{uu}=\frac{1}{N_{e}}\sum_{j=1}^{N_{e}}(u_{1}^{(j)}-\overline{u}_{1})(u_{1}^{(j)}-\overline{u}_{1})^{T}=C_{0}^{uu}-C_{0}^{uu}G^{T}(GC_{0}^{uu}G^{T}+\Gamma)^{-1}GC_{0}^{uu}.$$ Then, formally, as $N_{e}\to \infty$ we obtain that $C_{0}^{uu}\to C$ and so the mean and covariance of the updated ensemble converges to the mean and covariance of the posterior \cite{Andrew}. In other words, in the linear case, the proposed IR-ES method generates samples of the posterior distribution as $N_{e}\to \infty$. $\Box$ 
\end{proof}

\subsection{Computational cost of IR-ES}\label{rem3}
 For the present application, the cost of computing (\ref{cross1})-(\ref{cross2}) at each iteration is negligible compared to the cost of the evaluation of the forward model (\ref{eq:2.3}) (for each ensemble member). In addition, in the case where a relatively small number (around $10^3$) of observations are assimilated, the cost of inverting $C_{m}^{ww,f} +\alpha_{m}\Gamma $ is also negligible compared to (\ref{eq:2.3}). Therefore, the main computational cost of IR-ES per iteration and per ensemble is due to (\ref{eq:2.3}). The total cost of a $N_{e}$-size ensemble of IR-ES is approximately $N_{e} J$ forward model evaluations where $J$ is the number of iterations to converge. For the forward models considered in Section \ref{Numerics}, our numerical experiments indicate that $J$ is typically between 10 and 20 iterations. Thus, for large models, the computational efficiency of Algorithm \ref{IR-ES} may be comparable to the one of the standard ES computed with a large ensemble (e.g. $10^3$). On the other hand, large ensembles with standard ES methods provide better approximations of the posterior (see subsection \ref{comp}). One may then conclude that the proposed IR-ES is equivalent to a standard ES with a large ensemble. However, the advantage of the proposed method is that, at the $m$th iteration level, each ensemble update in the analysis step will be controlled by the parameter $\alpha_{m}$ which is chosen according to the discrepancy principle. It is then reasonable to expect small changes in the covariances and crosscovariances defined in (\ref{cross1})-(\ref{cross2}). In other words, we assume that changes in the updated ensemble members are sufficiently small (due the regularization) so that the ensemble of model predictions (\ref{eq:2.5}) approximates the forward model evaluated at the update ensemble (\ref{eq:2.3}). Then, in order to reduce the computational cost of Algorithm \ref{IR-ES}, for $m>0$ we  propose to replace (\ref{eq:2.3}) by
\begin{eqnarray}\label{eq:4.10}
w_{m}^{(j,f)}= \left\{\begin{array}{cc}
G(u_{m}^{(j)}) &  \textrm{if}~~ \textrm{mod}(m,M_{ES})=0\\
w_{m}^{(j,a)}& \textrm{otherwise}\end{array}\right.
\end{eqnarray}
Clearly, the aforementioned assumption on small changes in the ensemble updates is valid for only a few number of iterations $M_{ES}$ after which the forward model needs to be evaluated (\ref{eq:2.3}). Note that in Algorithm \ref{IR-ES} we always compute the evaluation of the forward model at the initial ensemble. However, the next evaluation of the forward model at the updated ensemble is done whenever $\textrm{mod}(m,M_{ES})=0$. Therefore, when we use (\ref{eq:4.10}) instead of (\ref{eq:2.3}) in Algorithm \ref{IR-ES}, the computational cost becomes to $N_{e}(1+\textrm{ floor} (J/M_{ES}))$
It is important to remark that (\ref{eq:4.10}) does not affect the results of Proposition \ref{prp:2} since for the linear case the algorithm stops after the first iteration. The effect of the additional parameter $M_{ES}$ will be investigated numerically in subsection \ref{sec:numIR-ES}. 

For previous exposition we have assumed that the prior distribution is Gaussian $N(\overline{u},C)$. For the numerical experiments of the subsequent section we use these priors to define our initial ensemble. Moreover, the corresponding covariances are used in the formulas of IR-enLM Algorithm \ref{IR-enLM} (Note that $C$ does not appear in IR-ES). However, it is important to emphasize that the Gaussian assumption on the prior distribution is not fundamental for the application the proposed methods. Note that an initial ensemble can be potentially sampled from a non-Gaussian prior. In addition, recall that IR-enLM is posed as the minimization of (\ref{eq:IR}) iteratively regularized by (\ref{eq:1.4}). If a non-Gaussian prior is considered, the second term in (\ref{eq:1.4}) can be replaced by other type of regularization that enforces the prior knowledge of such non-Gaussian prior. In that case, however the resulting solution of the minimizer of (\ref{eq:1.4}) may not be computed in a closed form as in (\ref{eq:1.7}).


\section{Numerical Results }\label{Numerics}

In this section we present numerical examples of the application of the proposed ensemble methods for capturing the Bayesian posterior. In concrete,  we consider the posterior distributions associated to two sets of synthetic data from prototypical oil-water reservoir models. These posteriors are fully resolved with a state-of-the-art MCMC for functions. Our MCMC results provide a gold-standard that we use to investigate the performance of IR-enLM and IR-ES at capturing aspects of the Bayesian posterior. Furthermore, we display the advantage of the proposed methods by comparing them with some standard unregularized approaches.

\subsection{The forward operator, the prior and the synthetic data}\label{sec:fm}
We consider a 2D incompressible oil-water reservoir model described by expressions (\ref{eq:2.7})-(\ref{eq:2.7B}) presented in the Appendix. The reservoir is defined on a squared domain discretized on a $60\times 60$ grid. A water flood is considered on a time interval of 3 years discretized in 30 time steps. The numerical discretization of the PDEs was conducted with the numerical schemes and the MATLAB implementation discussed in \cite{Evaluation}. The geologic property of interest is the absolute permeability of the reservoir. However, the approach can be extended to incorporate additional geologic properties, as well as other model parameters. For the present work we consider two different well constrains. Model A consist of production wells constrained to prescribed bottom hole pressure (BHP) and injectors operated with prescribed rates.  For Model B we consider production wells operated under prescribed total flow rate and injectors operated under BHP. For both models we define 30 measurement collection times (at each time step). In addition, we consider a well configuration of nine production wells $P_{1},\dots,P_{9}$ and four injection wells $I_{1},\dots,I_{4}$. Well locations are displayed in Figure \ref{Figure1} (top-middle). The forward operators $G_{A}$ and $G_{B}$ that arise from these two well models, reservoir dynamics and well locations are presented in the Appendix. The forward operator $G_{A}$ corresponds to the forward operator used in \cite{LM} for history matching with the regularizing LM scheme.

We recall that the prior distribution is assumed Gaussian. For the present experiments we consider a mean $\overline{u}=\log(5\times 10^{-13}\textrm{m}^2)$ constant over the domain of the reservoir. The prior covariance $C$ is a spherical covariance function \cite{Geos,Oliver} with maximum (resp. minimum) range of $10^3\textrm{m}$ (resp. $5\times 10^2\textrm{m}$). We consider an angle of $\pi/2$ along the direction of maximum correlation. In order to generate synthetic data we define the true permeability denoted  by $u^{\dagger}$ and displayed in Figure \ref{Figure1} (top-left). This permeability field is a draw from the prior distribution described before. In other words, we consider the best-case-scenario where our prior knowledge includes the truth. The subsequent procedure to generate synthetic data is analogous to both forward operators and so for the sake of clarity we base our discussion only on $G_{A}$.

For simplicity, we consider a diagonal measurement error covariance $\Gamma_{A}$. Note that the diagonal of $\Gamma_{A}$ corresponds to the variance of the data described by the forward operator $G_{A}$. Let us recall that $G_{A}$ has coordinates associated to either BHP (at the injection wells) or water rates (at the producers). For the $k$-th entry of $\Gamma_{A}$ that corresponds to a BHP measurement, we define the standard deviation (i.e. the square-root of $(\Gamma_{A})_{k,k}$) as $10\%$ of the corresponding $k$-th component of vector $G_{A}(u^{\dagger})$. For the entries associated to water rates (which are zero before water breakthrough), we select the square-root of $(\Gamma_{A})_{k,k}$ as $3\%$ (before breakthrough) or $7\%$ (after breakthrough) of the nominal value of the total flow rate at the associated production well and measurement time. With $\Gamma_{A}$ defined as above, synthetic data $y_{A}$ is generated by
\begin{eqnarray}\label{eq:1.1V2}
y_{A}=G_{A}(u^{\dagger})+\xi_{A}
\end{eqnarray}
with $\xi_{A}\sim N(0,\Gamma_{A})$. The noise level for this synthetic experiment can be computed as 
\begin{eqnarray}\label{eq:1.1V3}
\eta_{A}=\vert\vert \Gamma_{A}^{-1/2}\xi_{A}\vert\vert_{Y}
\end{eqnarray}
With the observational noise defined in (\ref{eq:1.1V2}), the resulting noise level is $12\%$ of the weighted norm of the noise-free observations $\vert\vert \Gamma_{A}^{-1/2}G_{A}(u^{\dagger})\vert\vert_{Y}$. In order to illustrate the capabilities of the proposed methods in a realistic scenario where the exact value of $\eta_{A}$ will be unknown (since the truth $u^{\dagger}$ is unknown), instead of the exact value (\ref{eq:1.1V3}) here we use the an estimate of $ \eta_{A}$. Since the covariance $\Gamma_{A}$ is diagonal, the discussion of \cite{LM} suggests that the estimate $ \eta_{A}=\sqrt{M}$ is appropriate where $M$ is the total number of measurements. For this case we have 13 wells measured at each of the 30 time steps and so  $M=390$.

\begin{figure}
\includegraphics[scale=0.25]{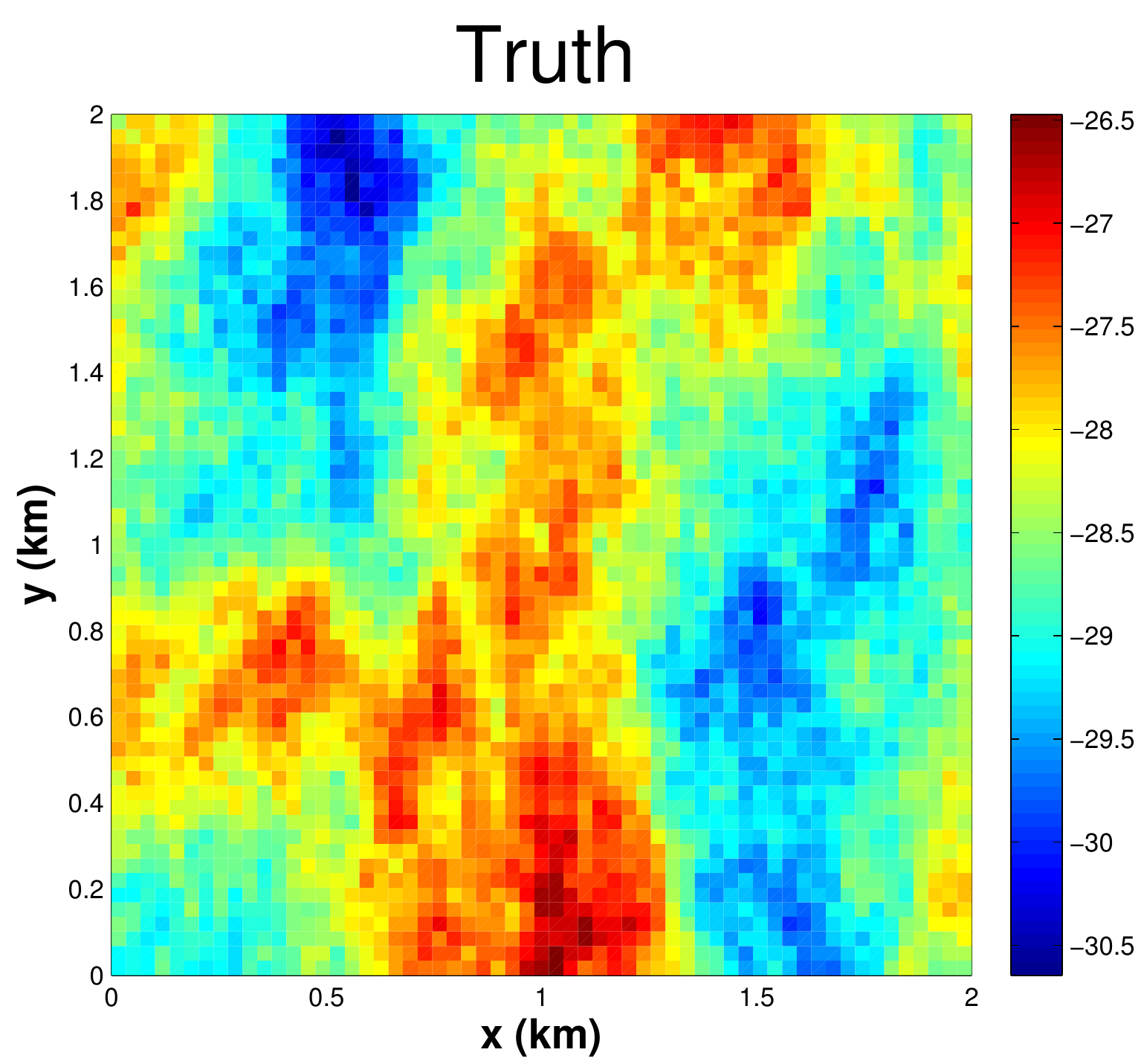}
\includegraphics[scale=0.25]{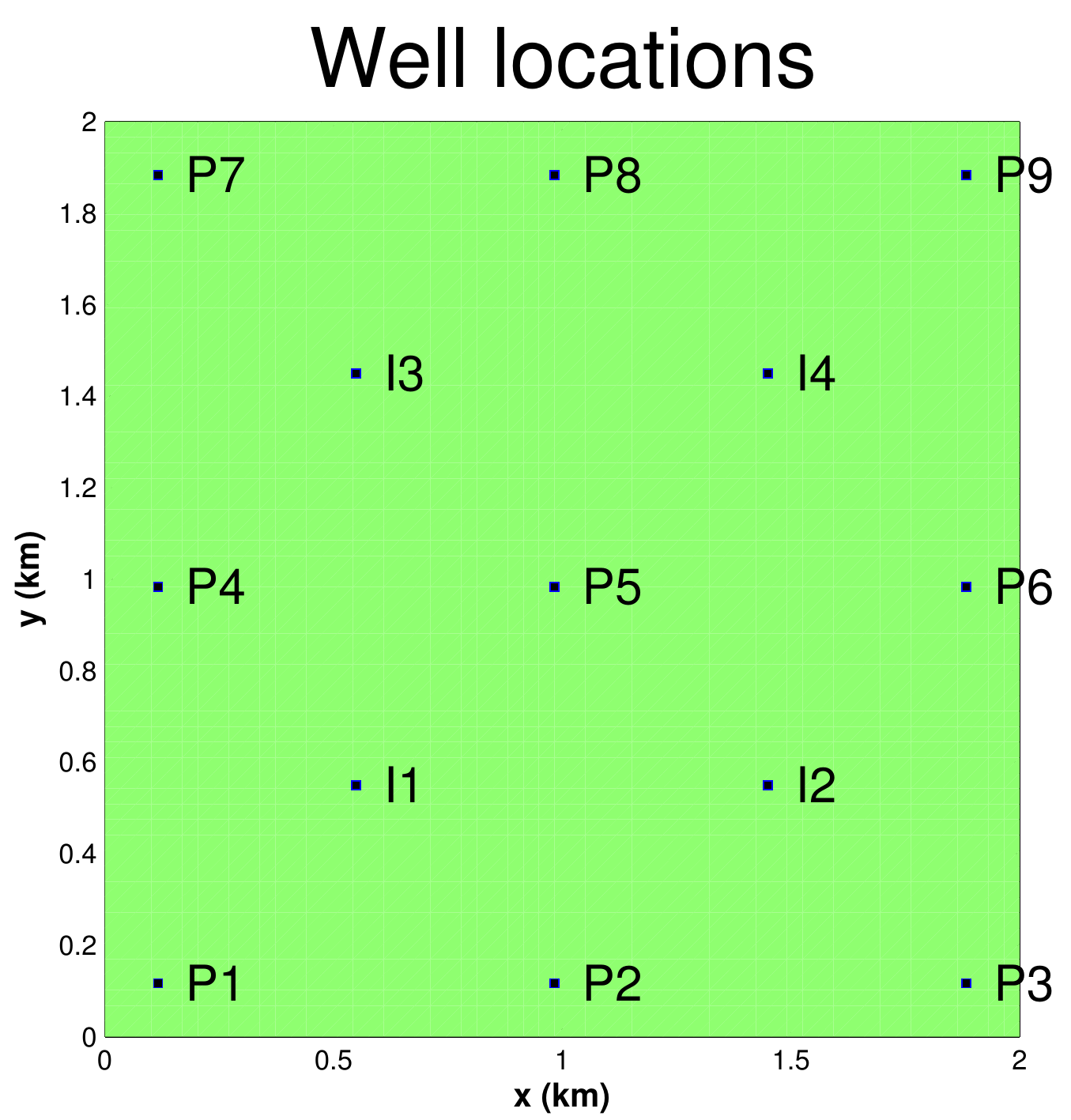}
\includegraphics[scale=0.25]{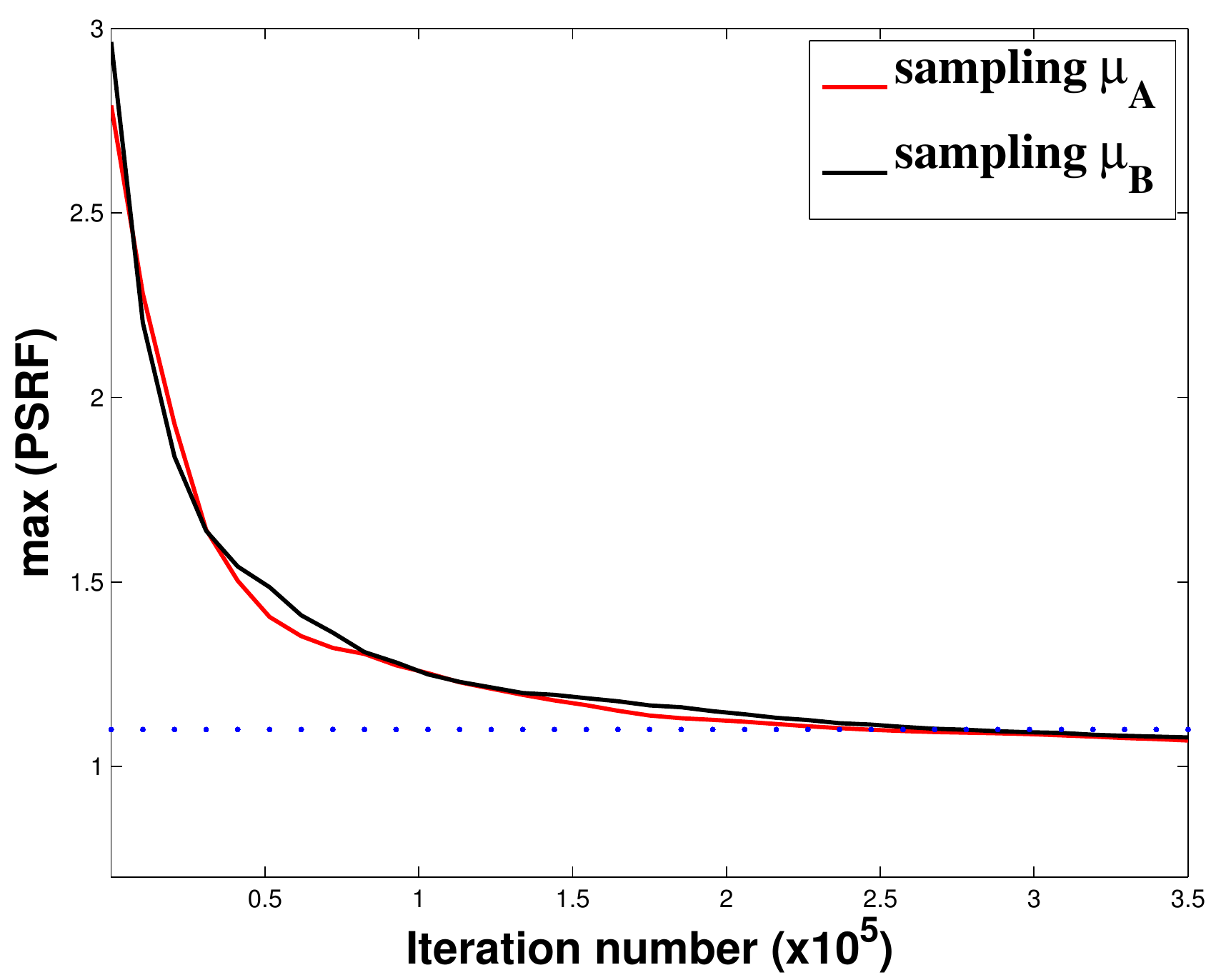}\\
\includegraphics[scale=0.25]{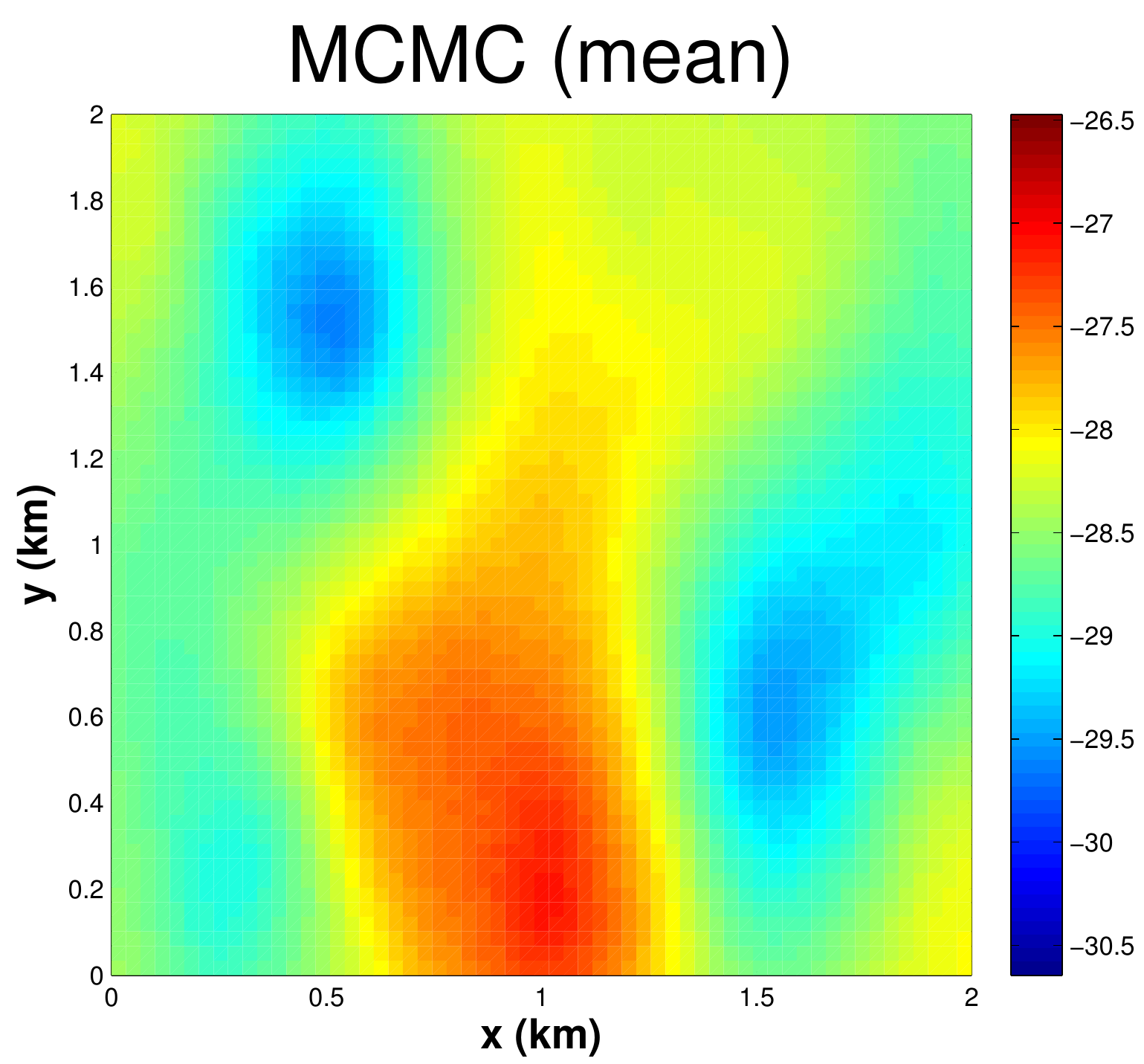}
\includegraphics[scale=0.25]{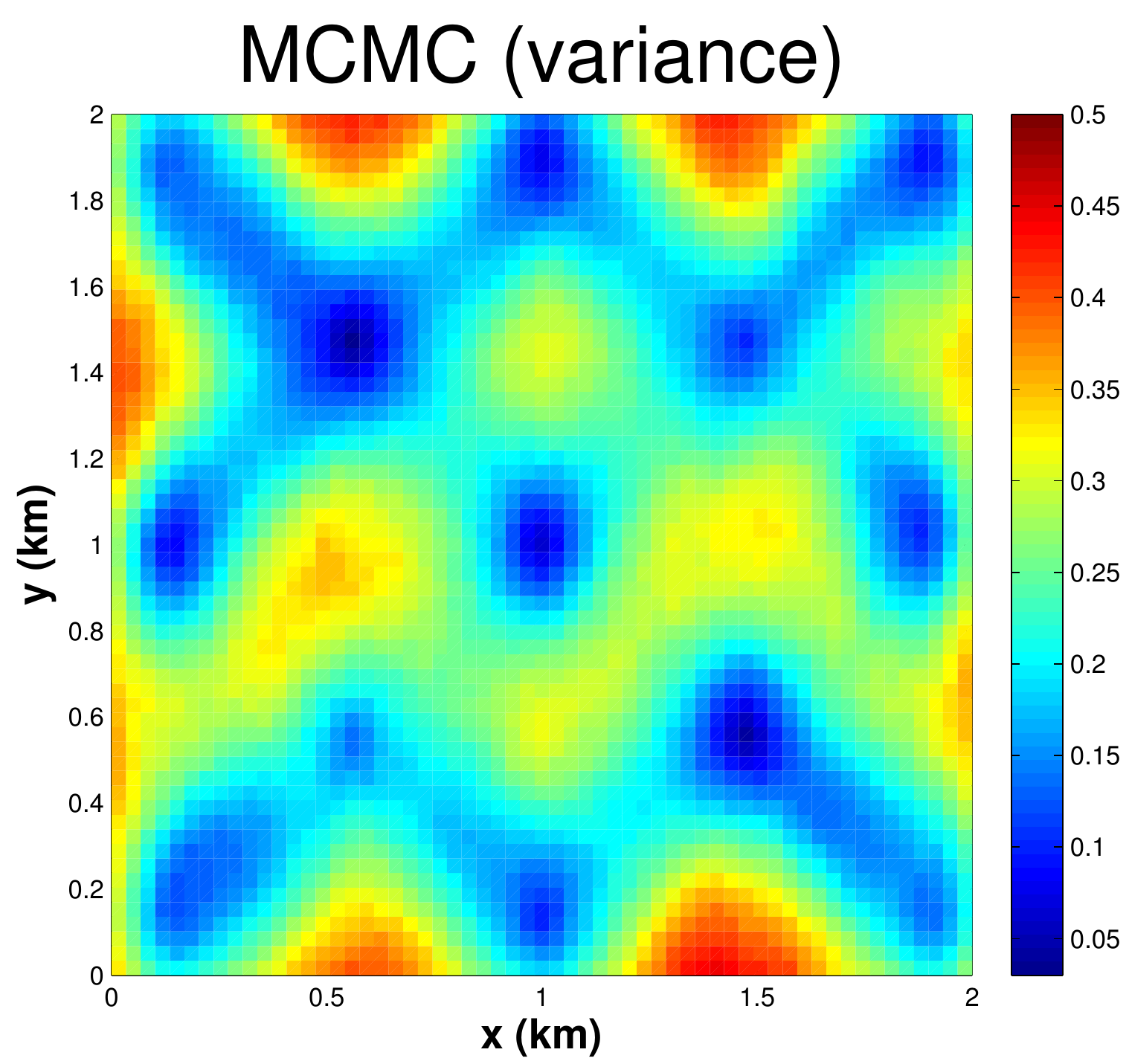}
\includegraphics[scale=0.25]{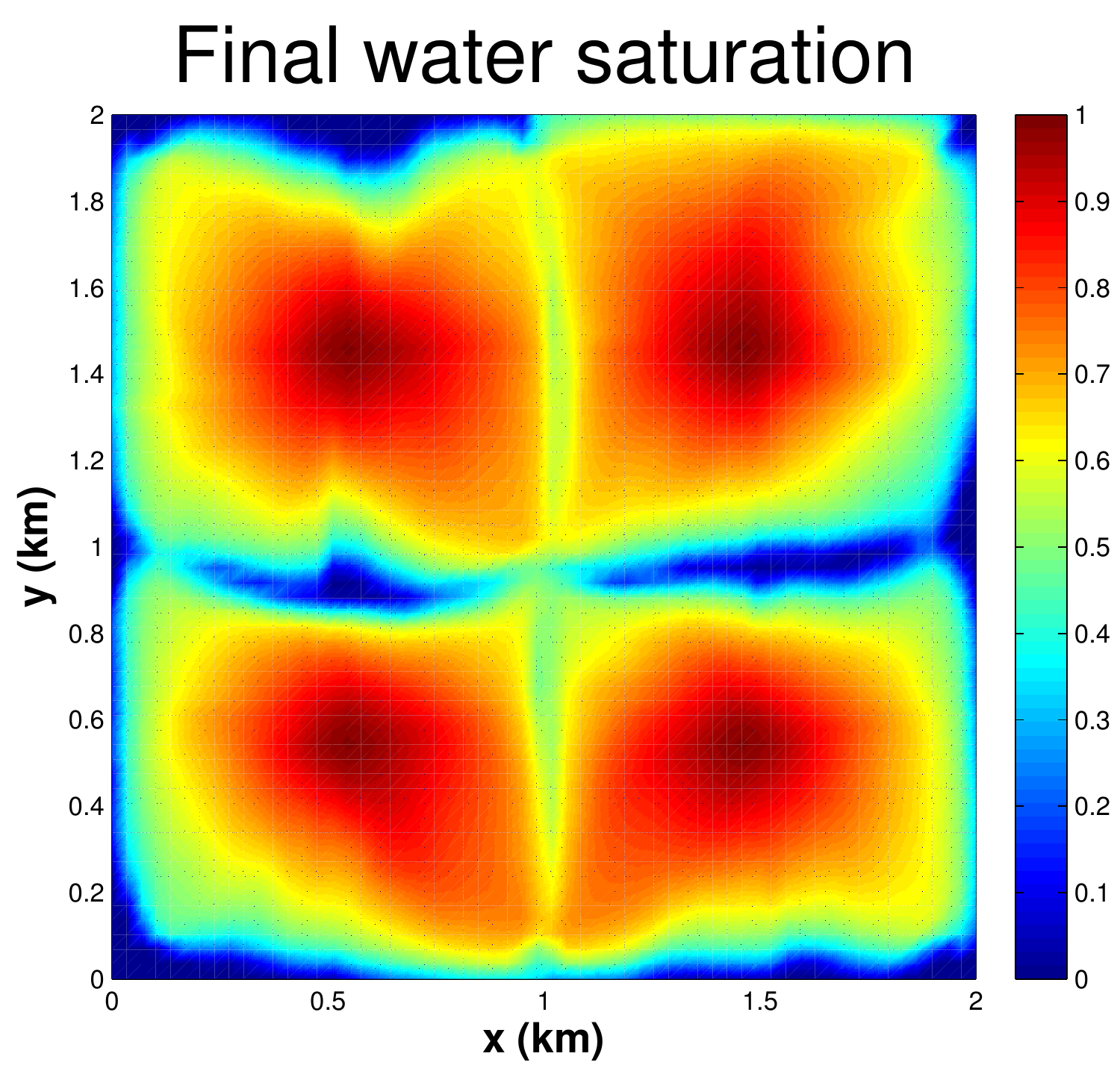}\\
\includegraphics[scale=0.25]{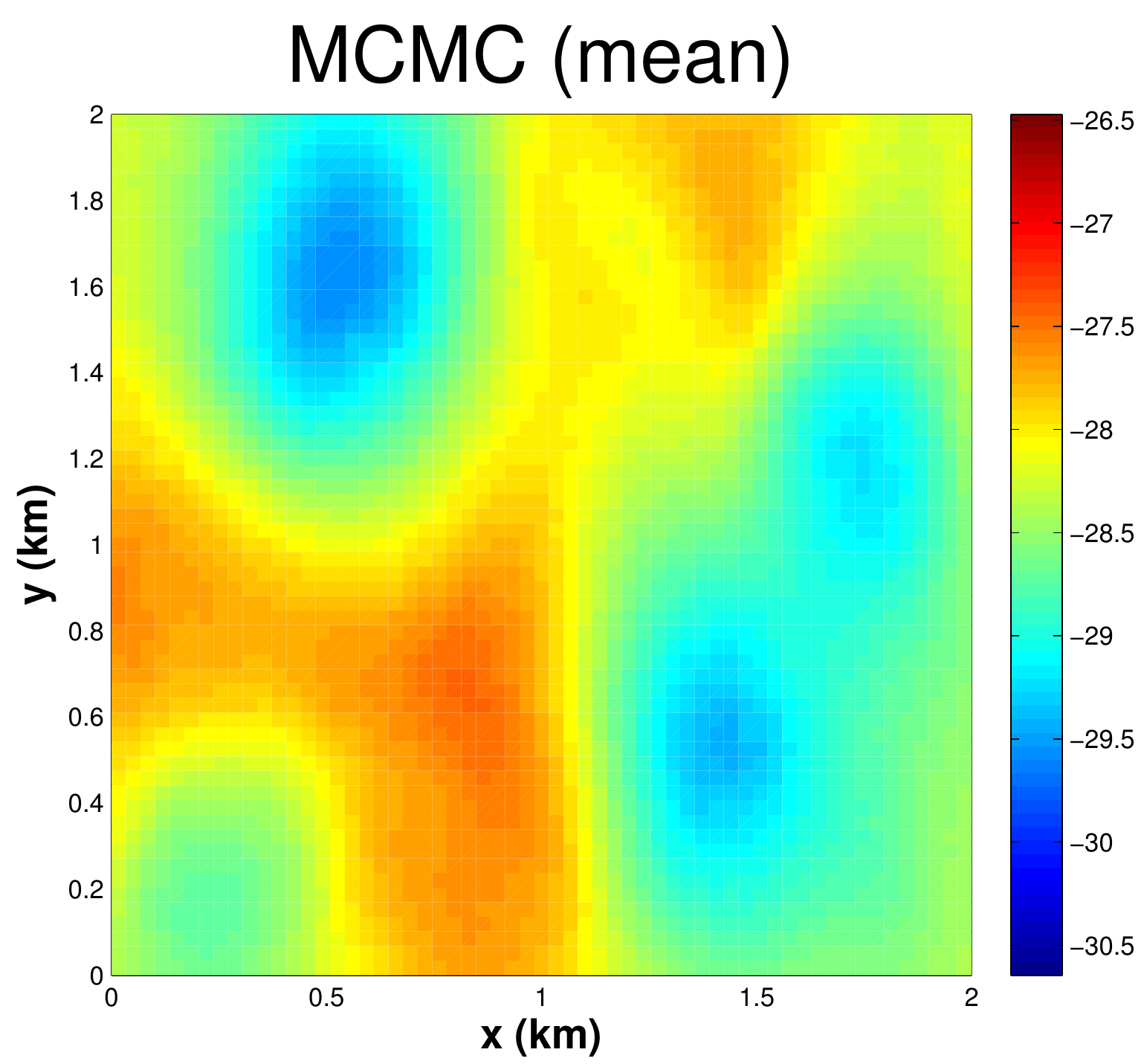}
\includegraphics[scale=0.25]{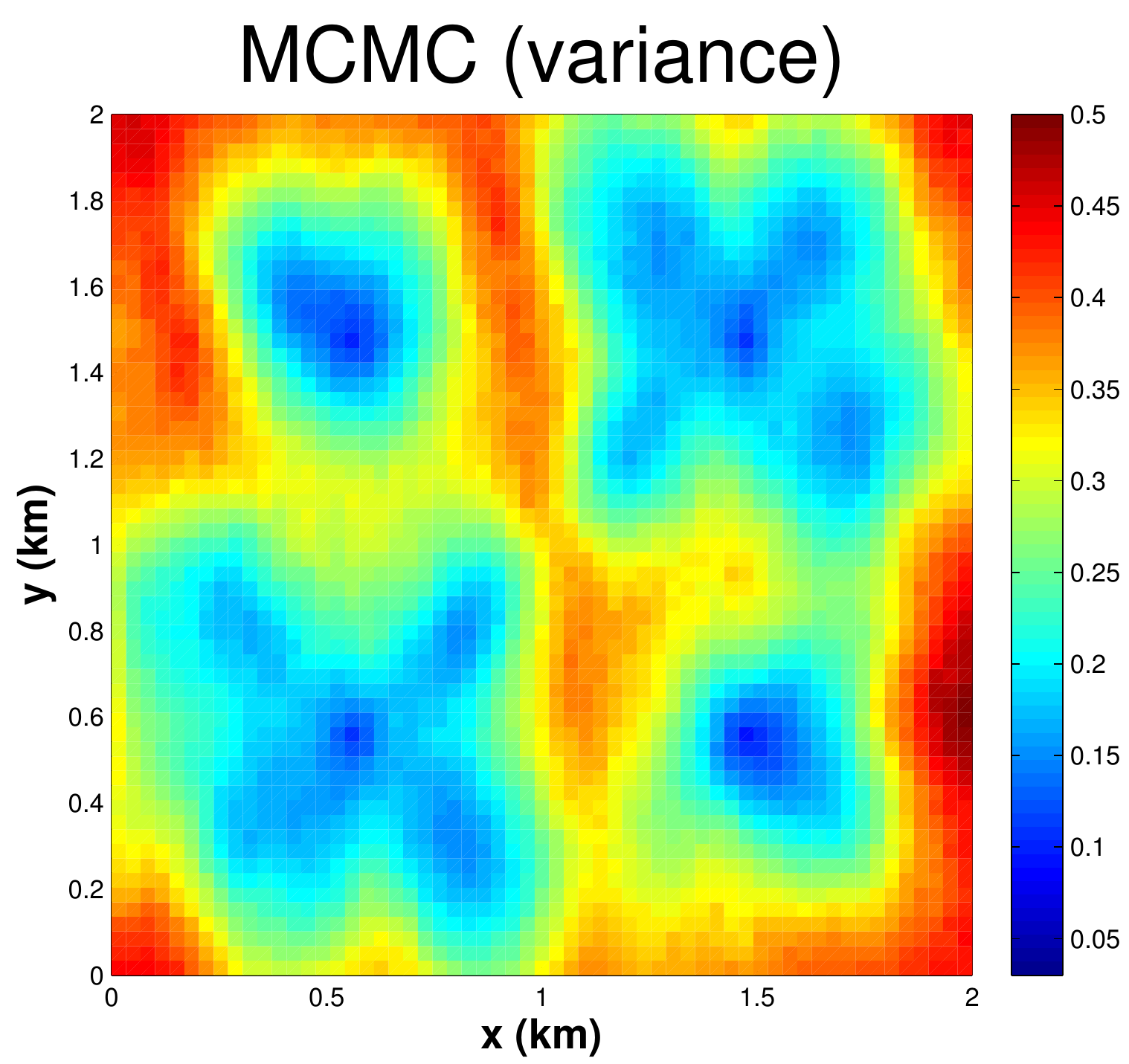}
\includegraphics[scale=0.25]{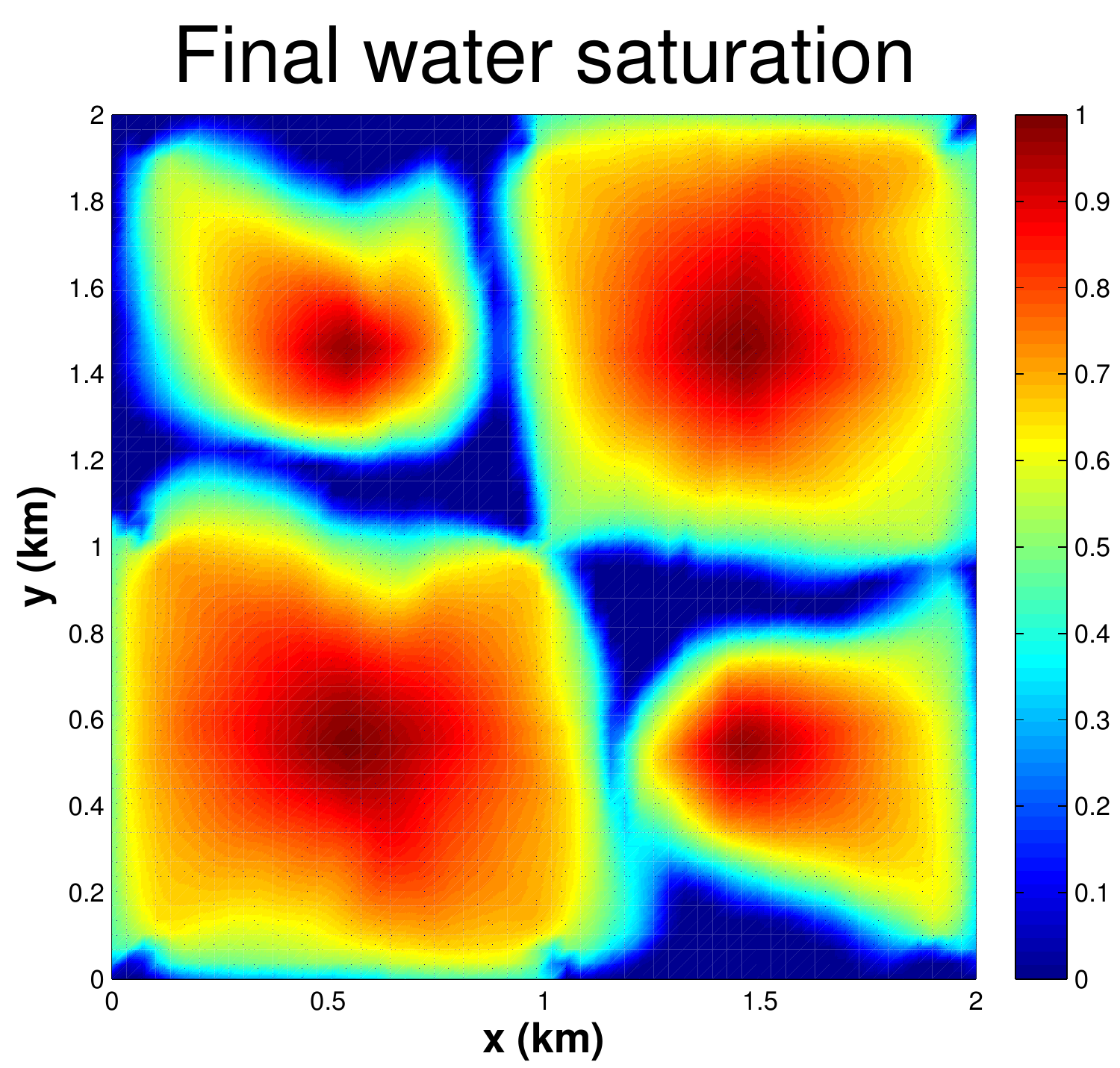}

\caption{Top-left: True log-permeability [$\log{\textrm{m}^2}$]. Top-middle: Well configuration. Top-Right: Gelman-Rubin diagnostic. Middle-left: mean of $\mu_{A}$. Middle-middle: variance of $\mu_A$. Middle-right: Final water saturation from Model A. Bottom-left: mean of $\mu_{B}$. Bottom-middle: variance of $\mu_B$. Bottom-right: Final water saturation from Model B. }  
\label{Figure1}
\end{figure}

The synthetic data $y_{B}$ for the forward model $G_{B}$ is generated with the procedure previously described for the analogous choice of $\Gamma_{B}$. For both forward operators we consider the same prior described above. However, the forward operators $G_{A}$ and $G_{B}$ and the corresponding synthetic data $y_{A}$ and $y_{B}$ give rise to two different posteriors $\mu_{A}=\mathbb{P}( u\vert y_{A})$ and $\mu_{B}=\mathbb{P}(u\vert y_{B})$ defined by (\ref{eq:1.2}). The aim of this section is to study the numerical performance of the proposed ensemble methods at capturing the mean and variance of these posteriors.

\subsection{Resolving the posteriors}\label{sec:pos}
In contrast to deterministic inverse problems where the aim is to recover the truth, in Bayesian inverse problems, the objective is to characterize the posterior distribution (\ref{eq:1.2}). Therefore, in order to assess the performance of the proposed ensemble-methods for approximating the Bayesian posterior, a fully resolved (accurate) posterior needs to be computed. However, as we discussed in Section \ref{Intro}, the application of standard MCMC methods for sampling the posterior that arises from subsurface flow models is usually restricted to small-size problems. While small-size problems, say from coarse 1D models, may be useful to establish Benchmarks \cite{EmeRey}, such small size may be detrimental to our ability of properly assessing the performance of the proposed methods under more stringent conditions that require proper regularization. In concrete, the well-known regularizing effect of discretization \cite{Iterative} may alleviate the ill-posedness intrinsic to this inverse problem. This, in turn, may overshadow the proper assessment of the regularizing effect of the proposed methods. Fortunately, grid invariant MCMC methods have been recently developed for sampling the Bayesian posterior that arises in large-scale PDE-constrained inverse problems \cite{David}. In particular, in \cite{Evaluation} the preconditioned Crank-Nicolson MCMC (pcn-MCMC) method has been applied to resolve the Bayesian posterior associated to inverse problems in reservoir models of moderate-size similar to the ones considered in the present work. The reader is refer to the work of \cite{Evaluation} for further details on the pcn-MCMC algorithm and its application for benchmarking Bayesian inverse problems in subsurface flow models.

By applying the  aforementioned pc-MCMC method we resolve the posteriors $\mu_{A}$ and $\mu_{B}$ that we introduced in the previous paragraphs and that we aim to approximate with the proposed ensemble methods. For the sampling of each of these posteriors we consider $80$ independent MCMC chains of length $3\times 10^{5}$ initialized with random draws from the prior distribution (recall the prior is the same for both $\mu_{A}$ and $\mu_{B}$). The convergence and mixing of these independent chains are evaluated with the Gelman-Rubin diagnostic based on the potential scale reduction factor (PSRF) \cite{Gelman}. In Figure \ref{Figure1} (top-right) we display this factor as a function of the MCMC iteration for the sampling of both posteriors $\mu_{A}$ and $\mu_{B}$. Approximate convergence (PSRF$<1.1$) is reached after $2.5\times 10^{5}$ iterations ensuring that our chains have run sufficiently long so that all the $80\times 3\times 10^5= 2.4\times 10^{7}$ samples from all these chain can be combined to characterize the target distribution. In the second (resp. third) row of Figure \ref{Figure1}  we display the mean (right) and the variance (middle) of the posterior distribution $\mu_{A}$ (resp. $\mu_{B}$). Some independent samples (from different MCMC chains) of $\mu_{A}$ are displayed in the top row of Figure \ref{Figure2}. While some of the main spatial features of the truth are replicated, sufficient variability can be also appreciated. We use our samples of the posterior $\mu_{A}$ to display in Figure \ref{Figure2} the water rates and BHP from some of the production and injection wells, respectively. In these figures, the red curve correspond to the truth (i.e. $G_{A}(u^{\dagger})$) and the vertical line separates the assimilation from the prediction time. Figure \ref{Figure3} shows analogous quantities for the posterior $\mu_{B}$ associated to the well Model B. 

In Figure \ref{Figure1} we display the mean (middle-left) and the variance (middle-middle) of the posterior $\mu_{A}$ characterized with pc-MCMC. Analogously, Figure \ref{Figure1} shows the mean (bottom-left) and the variance (bottom-middle) of the posterior $\mu_{B}$. Note that, even though the underlaying reservoir dynamics are the same for both forward operators, the difference in the well constraints has an important effect on the resulting posterior distribution. In particular we note significant differences in the variance of $\mu_{A}$ and $\mu_{B}$. Note that for Model A whose productions wells are operated under prescribed BHP, the associated well model (expression (\ref{eq:2.10})) that, in turn, defines the measurement operator depends on the permeability at the grid block containing the well. Therefore, these measurements may reduce the uncertainty of the permeability at those locations. In contrast, in Model B the production wells are operated under prescribed flow rate. In this case, we see from (\ref{eq:2.17}) that the measurement functional is independent of the log-permeability. It then comes as no surprise that larger variances are obtained in these locations. Furthermore, in Figure \ref{Figure1} we display the water saturation at the final time (for the assimilation period) associated to Model A (middle-right) and Model B (bottom-right), respectively. We note that, for Model B, larger variances are observed in the regions where the water front has not yet arrived.

\begin{figure}
\begin{center}
\includegraphics[scale=0.134]{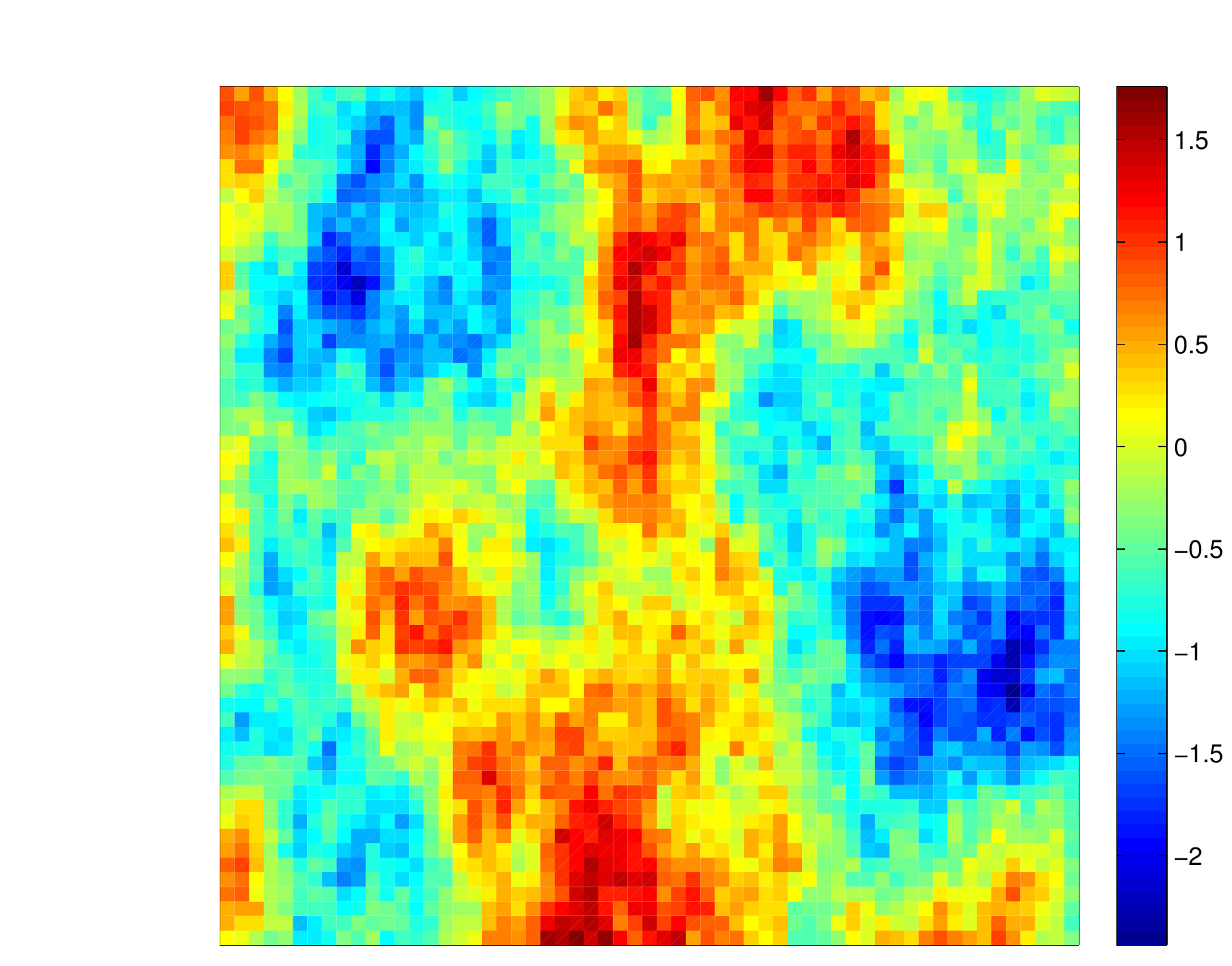}
\includegraphics[scale=0.134]{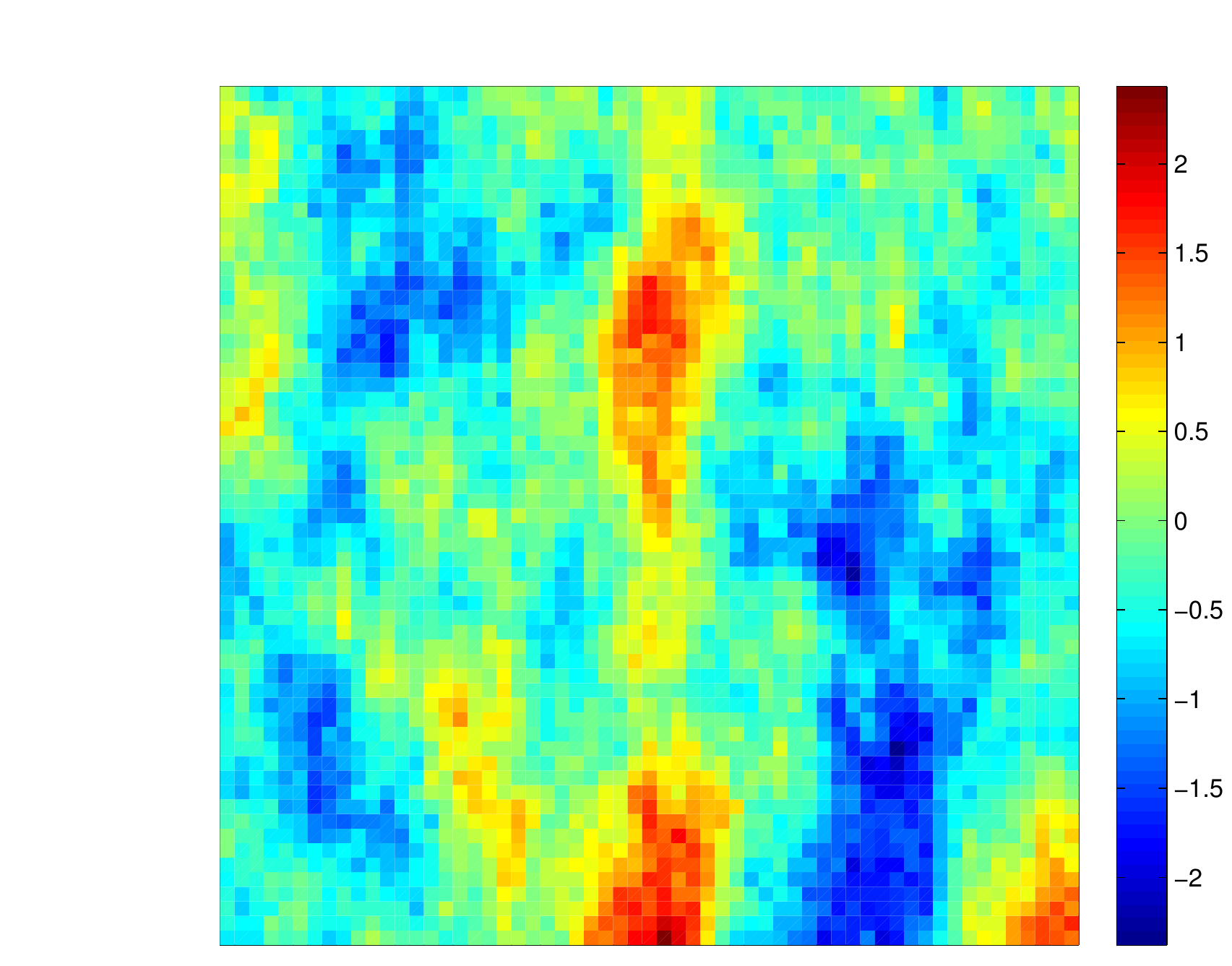}
\includegraphics[scale=0.134]{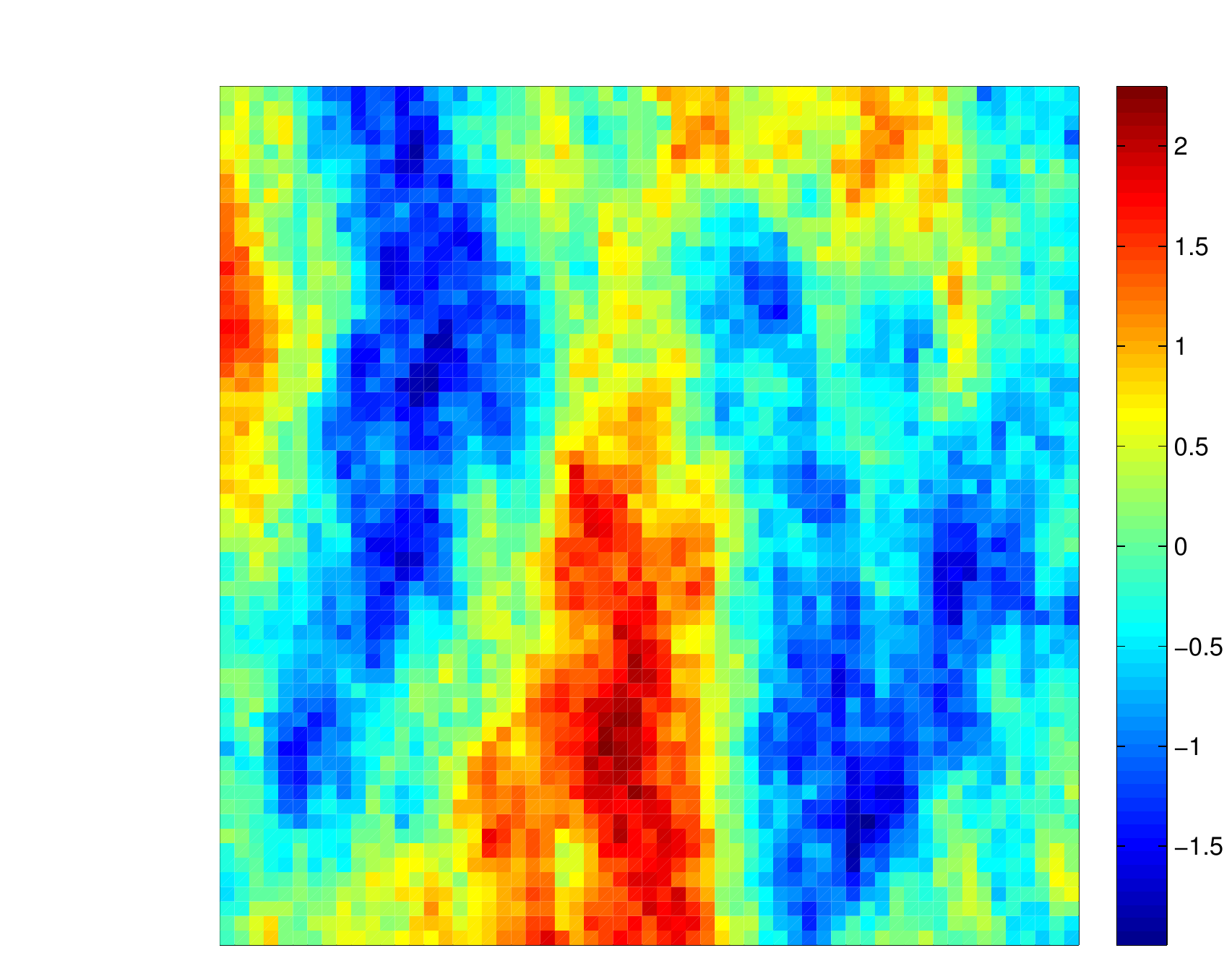}
\includegraphics[scale=0.134]{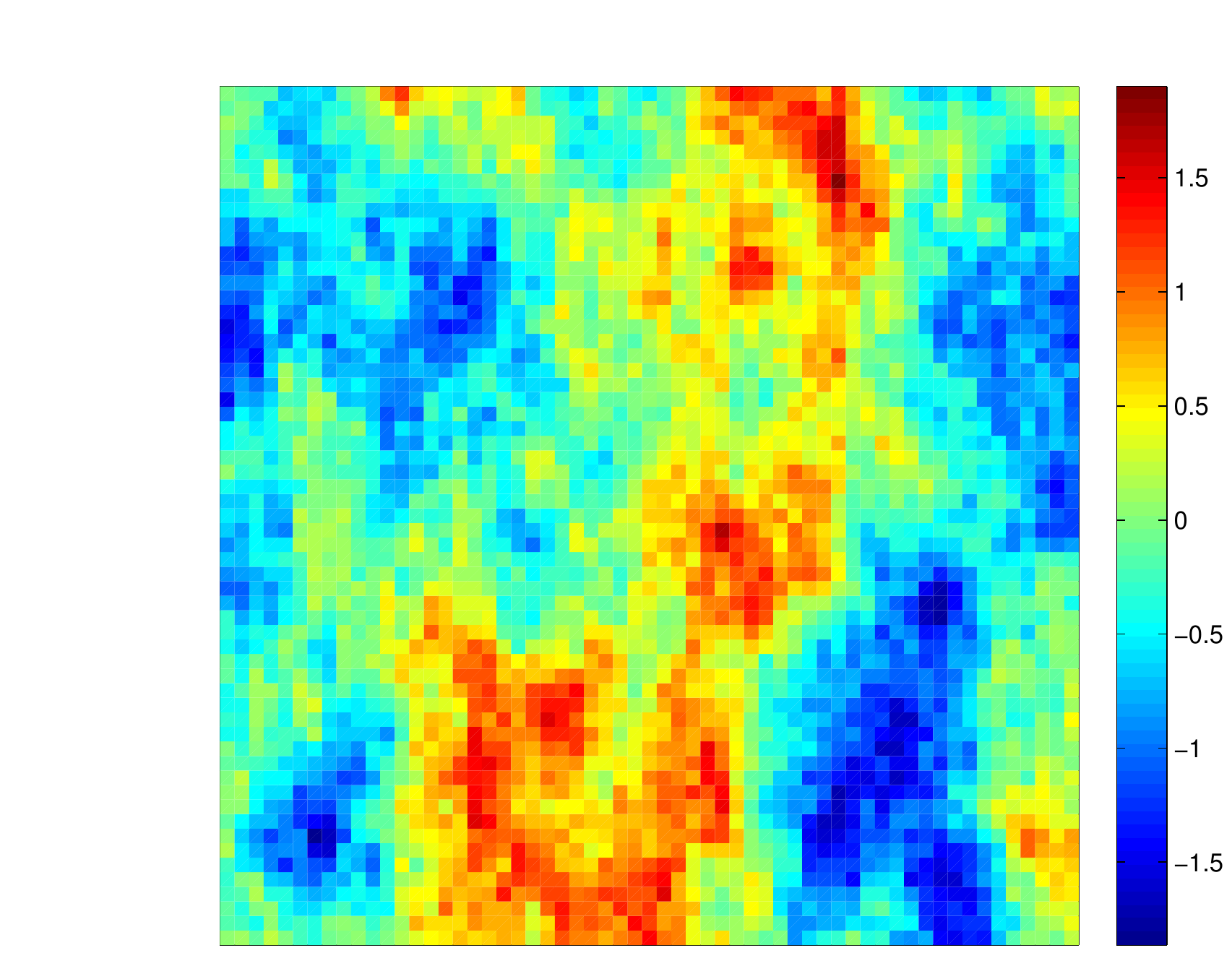}
\includegraphics[scale=0.134]{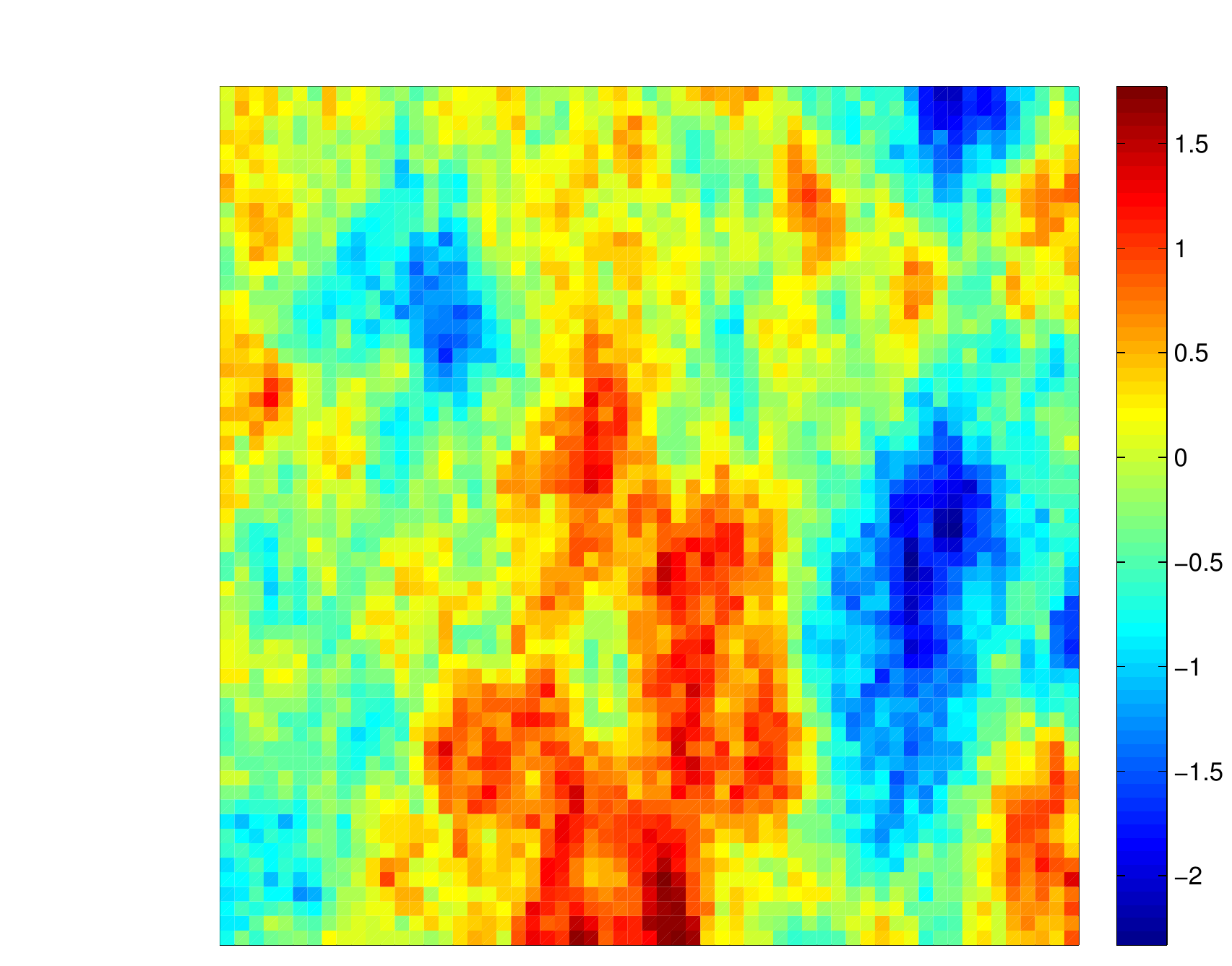}\\
\includegraphics[scale=0.165]{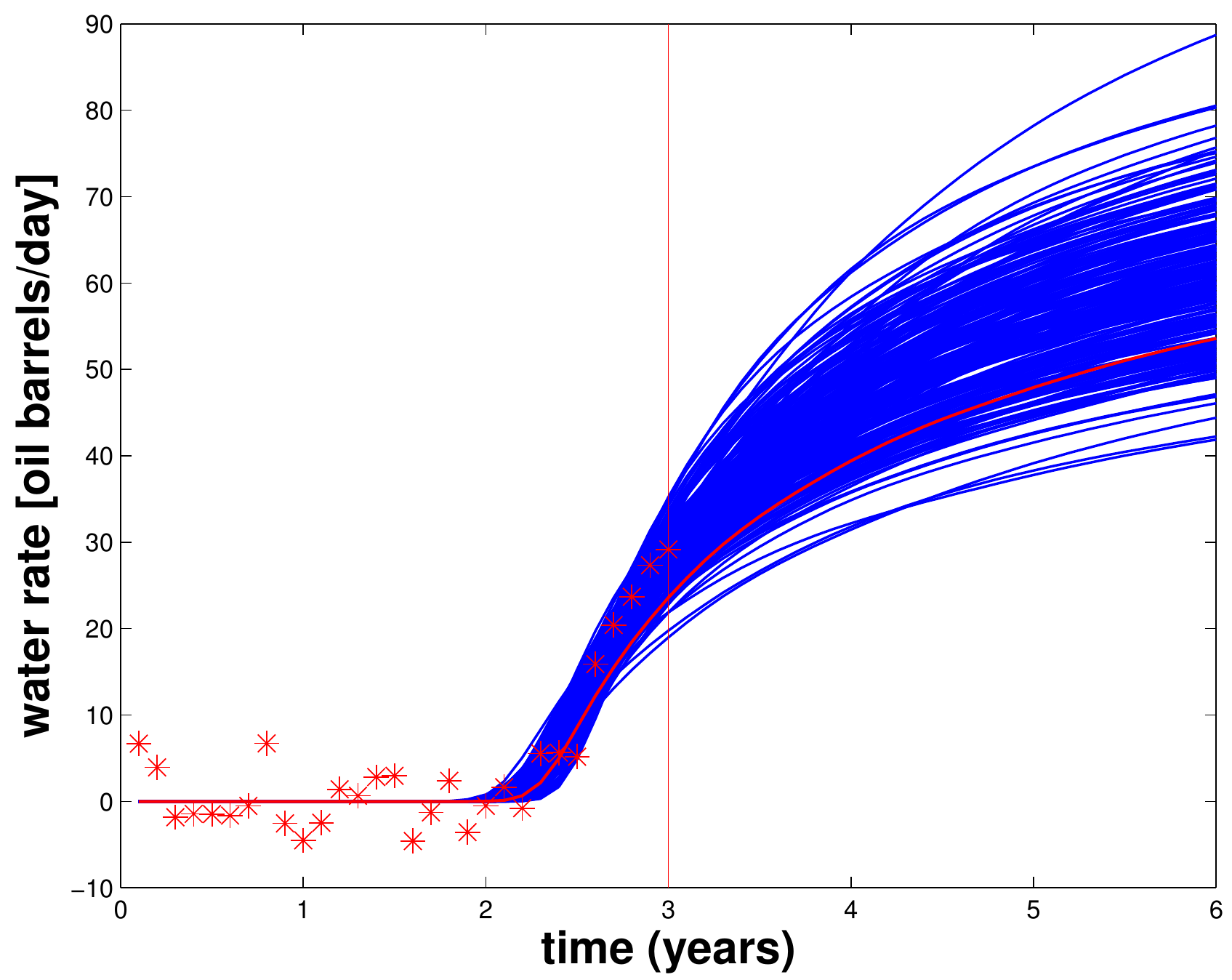}
\includegraphics[scale=0.165]{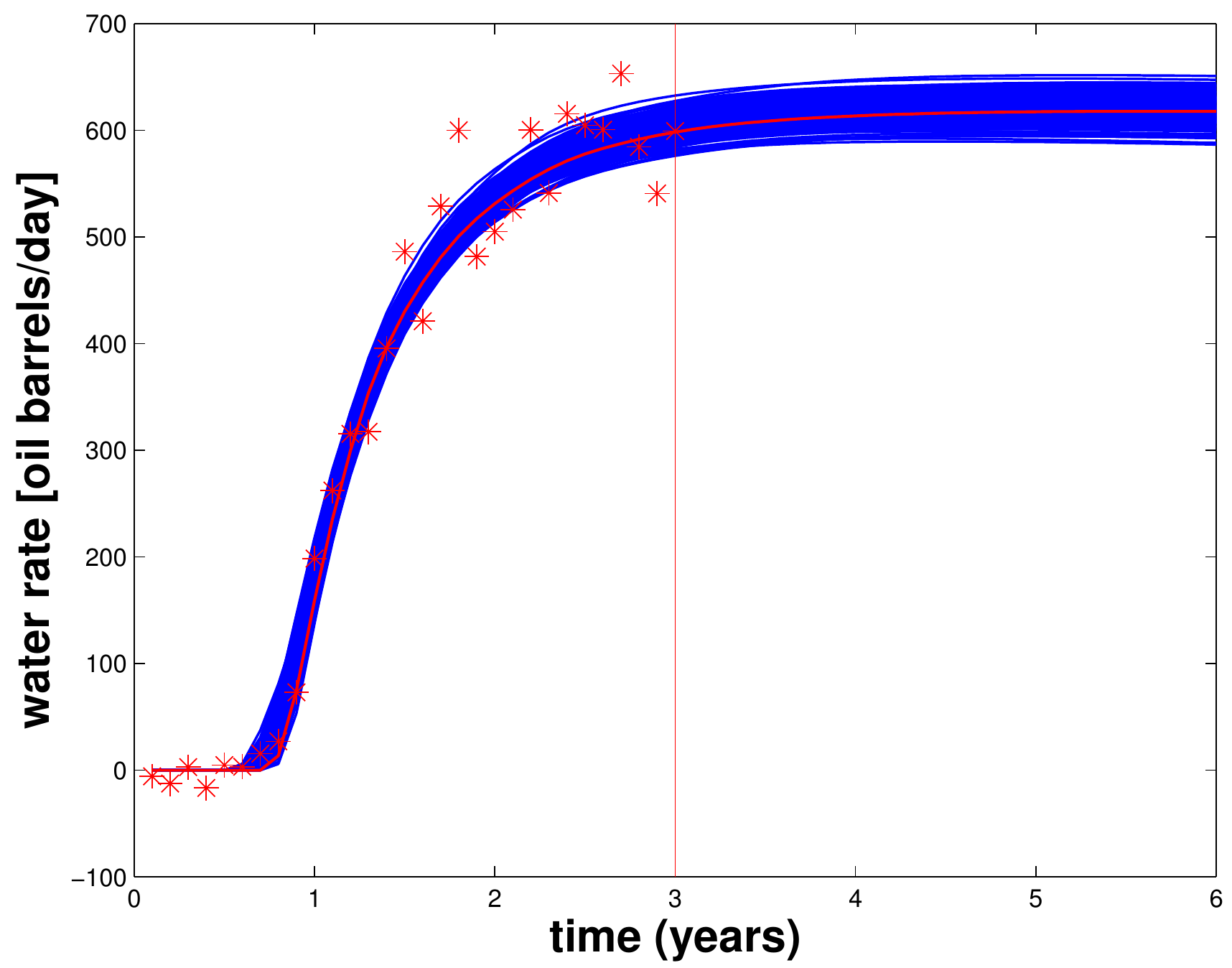}
\includegraphics[scale=0.165]{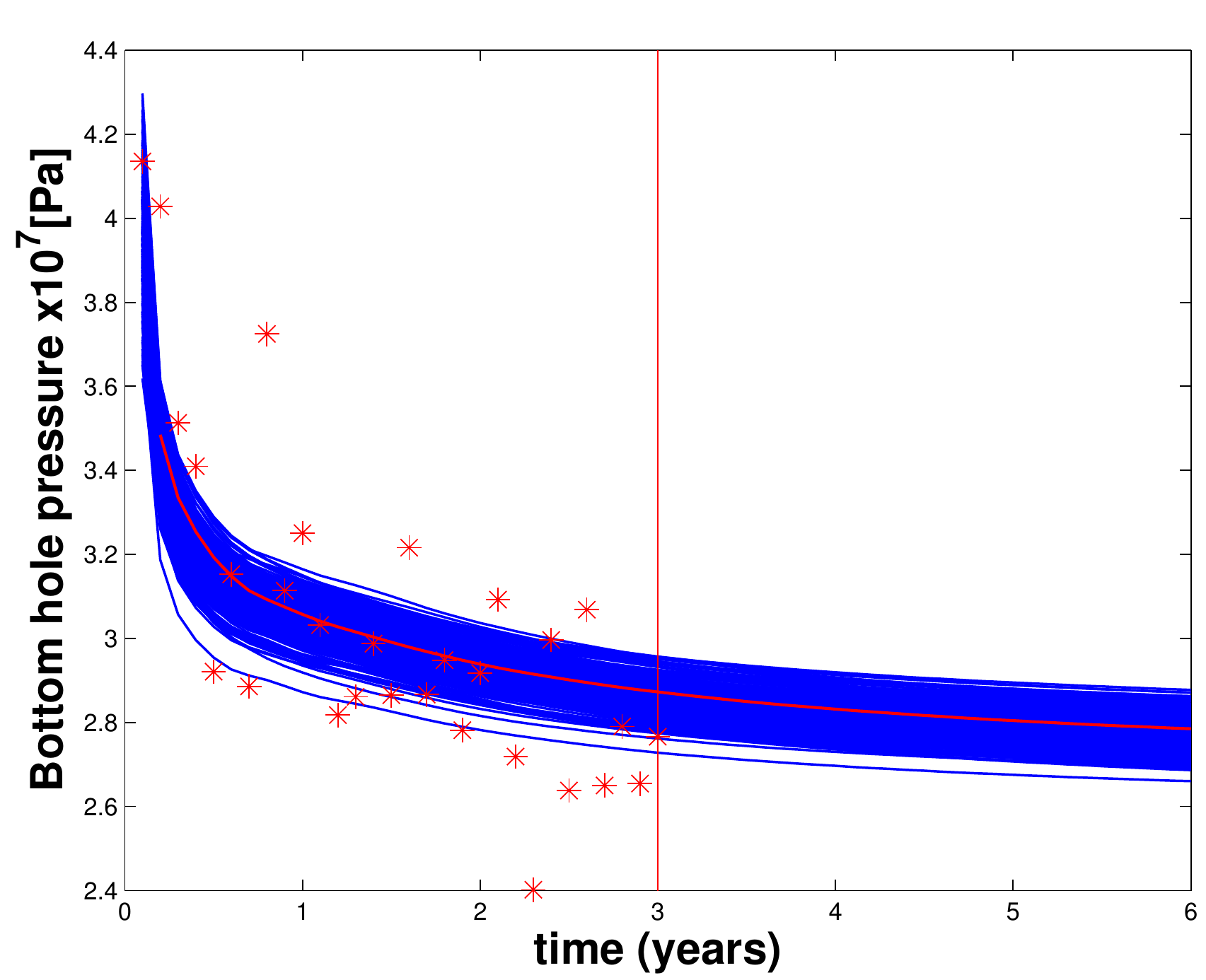}
\includegraphics[scale=0.165]{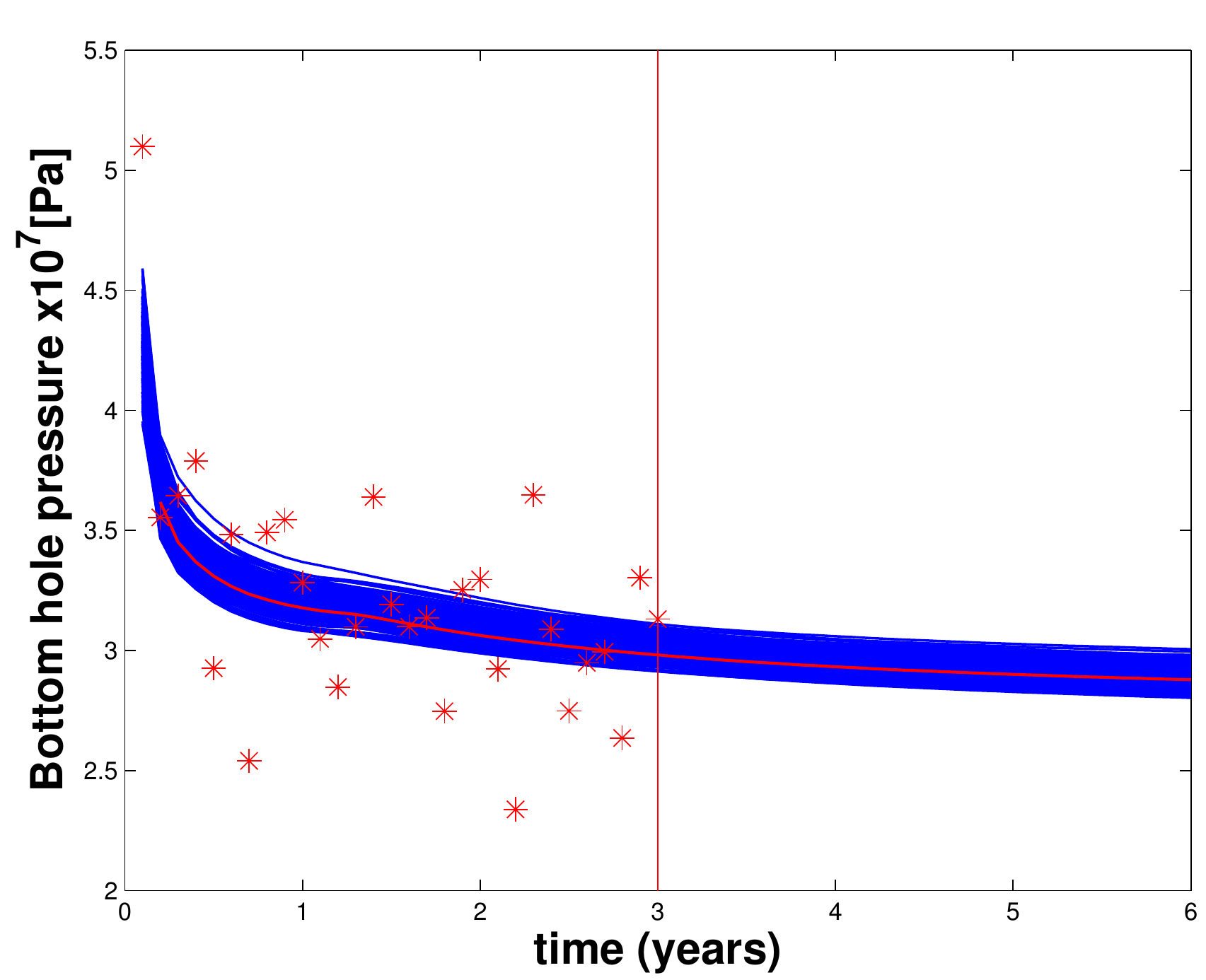}

\caption{ Top: Samples of the posterior distribution $\mu_{A}$ (characterized with pcn-MCMC) [$\log{\textrm{m}^2}$]. Bottom: (from left to right) water rates at wells $P_{1}$, $P_{2}$ and BHP from $I_{2}$ and $I_{3}$}  
\label{Figure2}
\end{center}
\end{figure}

\begin{figure}
\begin{center}
\includegraphics[scale=0.134]{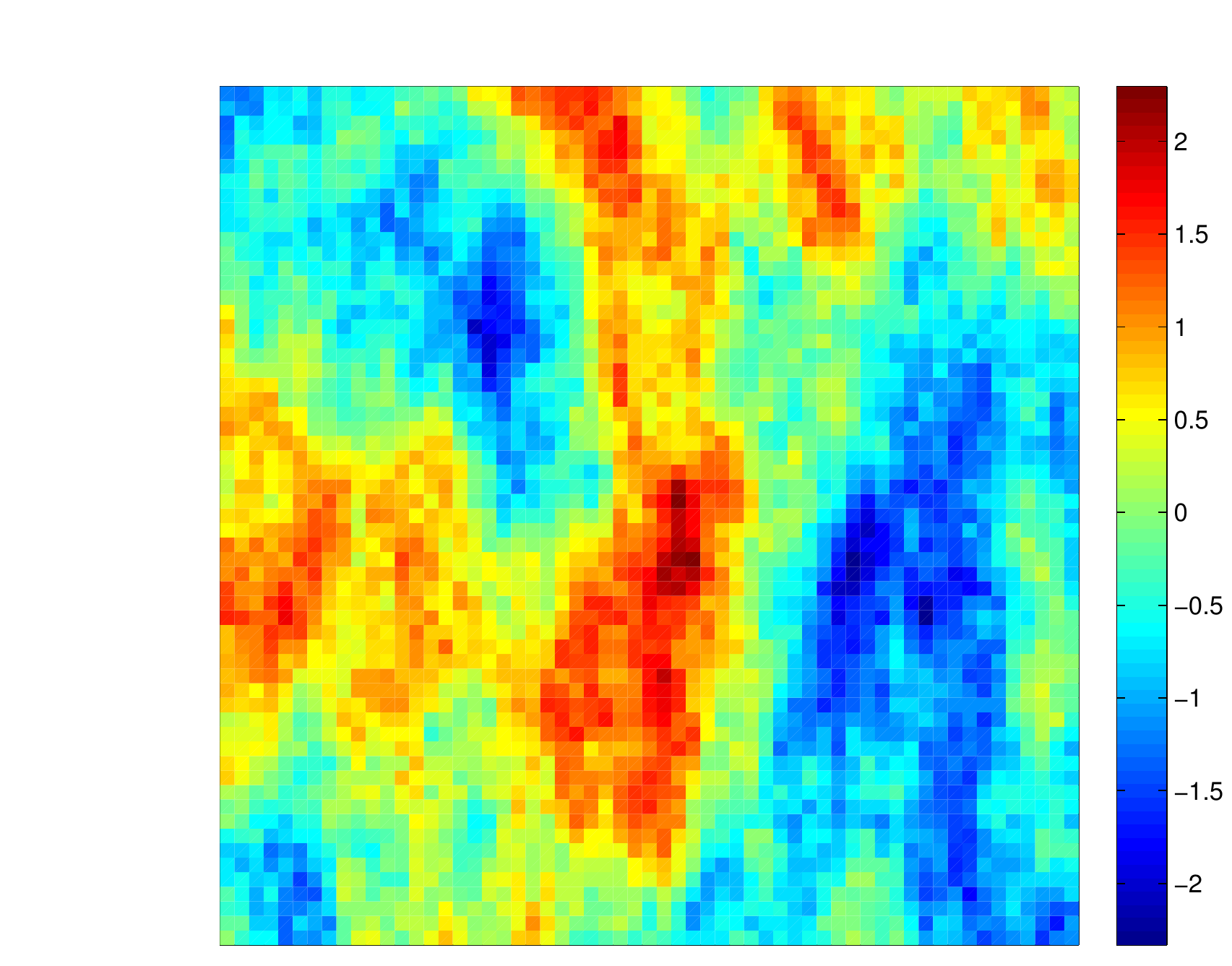}
\includegraphics[scale=0.134]{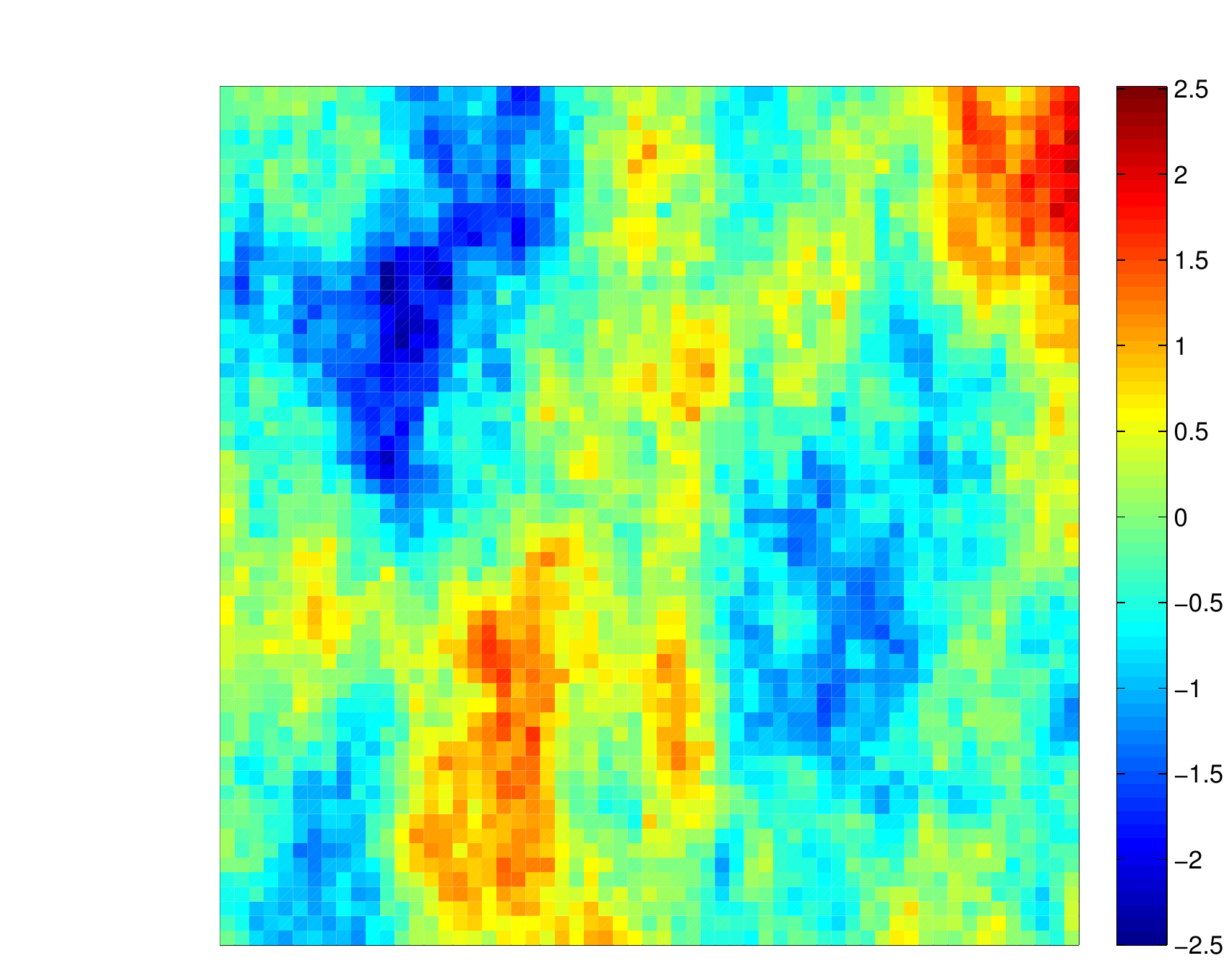}
\includegraphics[scale=0.134]{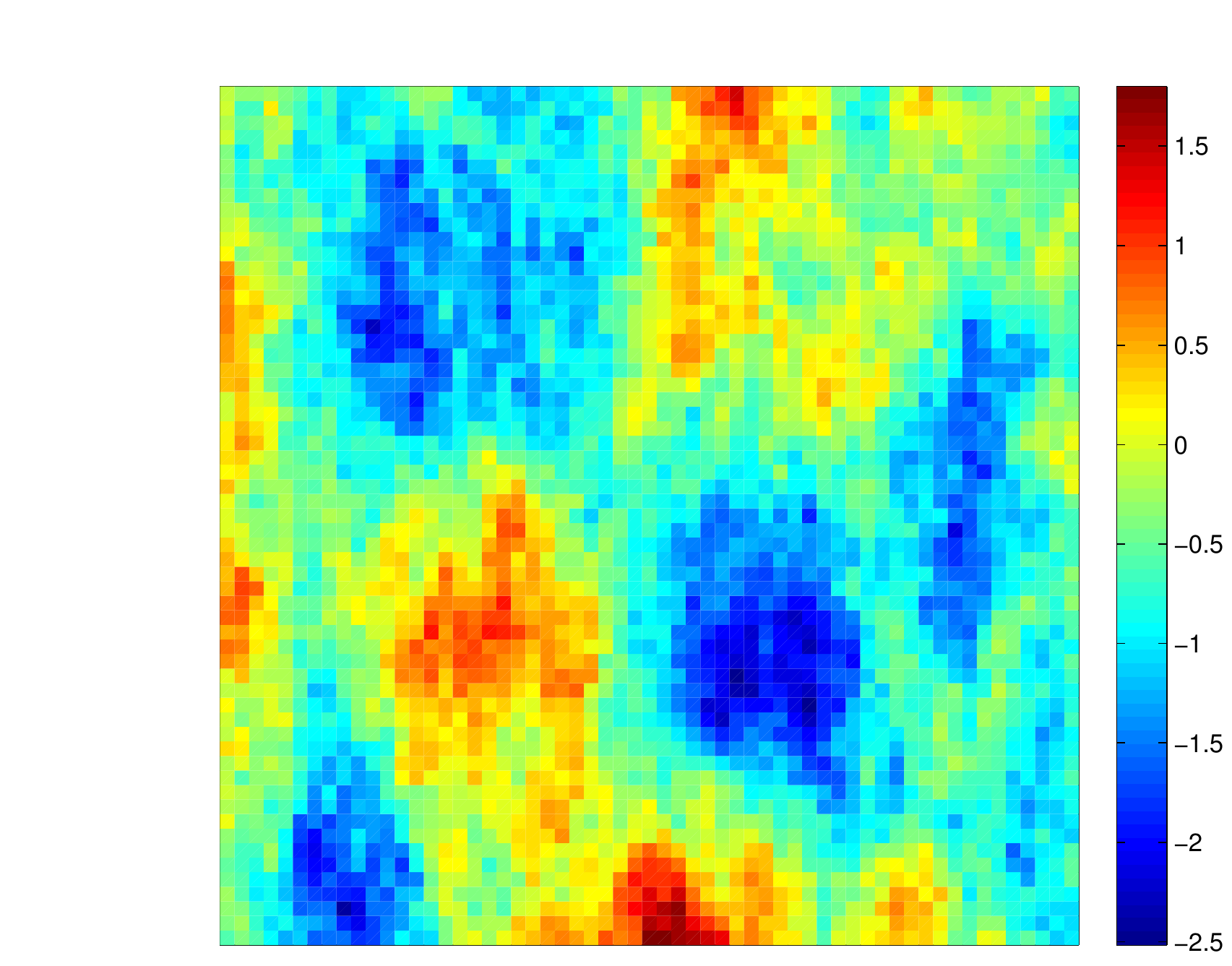}
\includegraphics[scale=0.134]{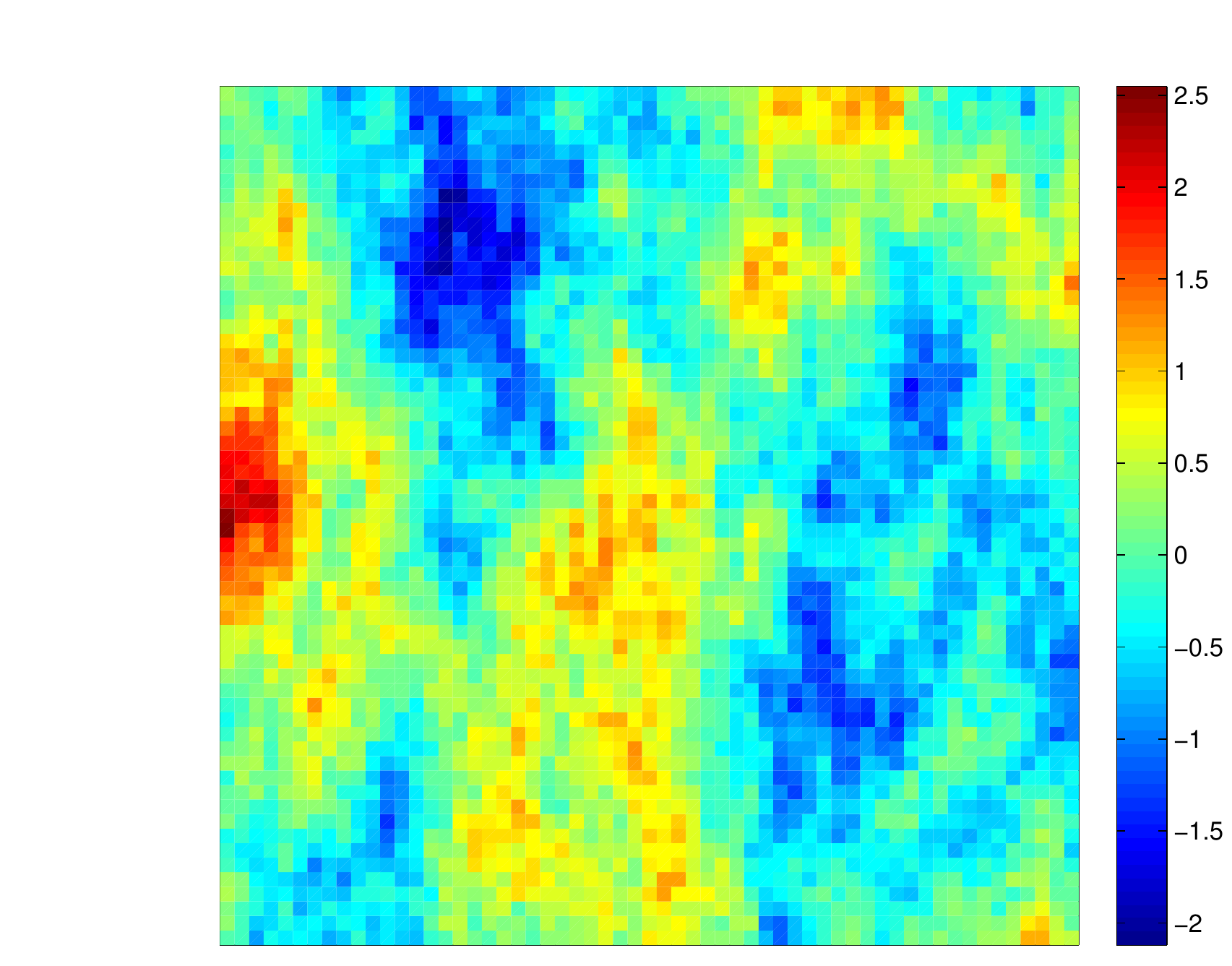}
\includegraphics[scale=0.134]{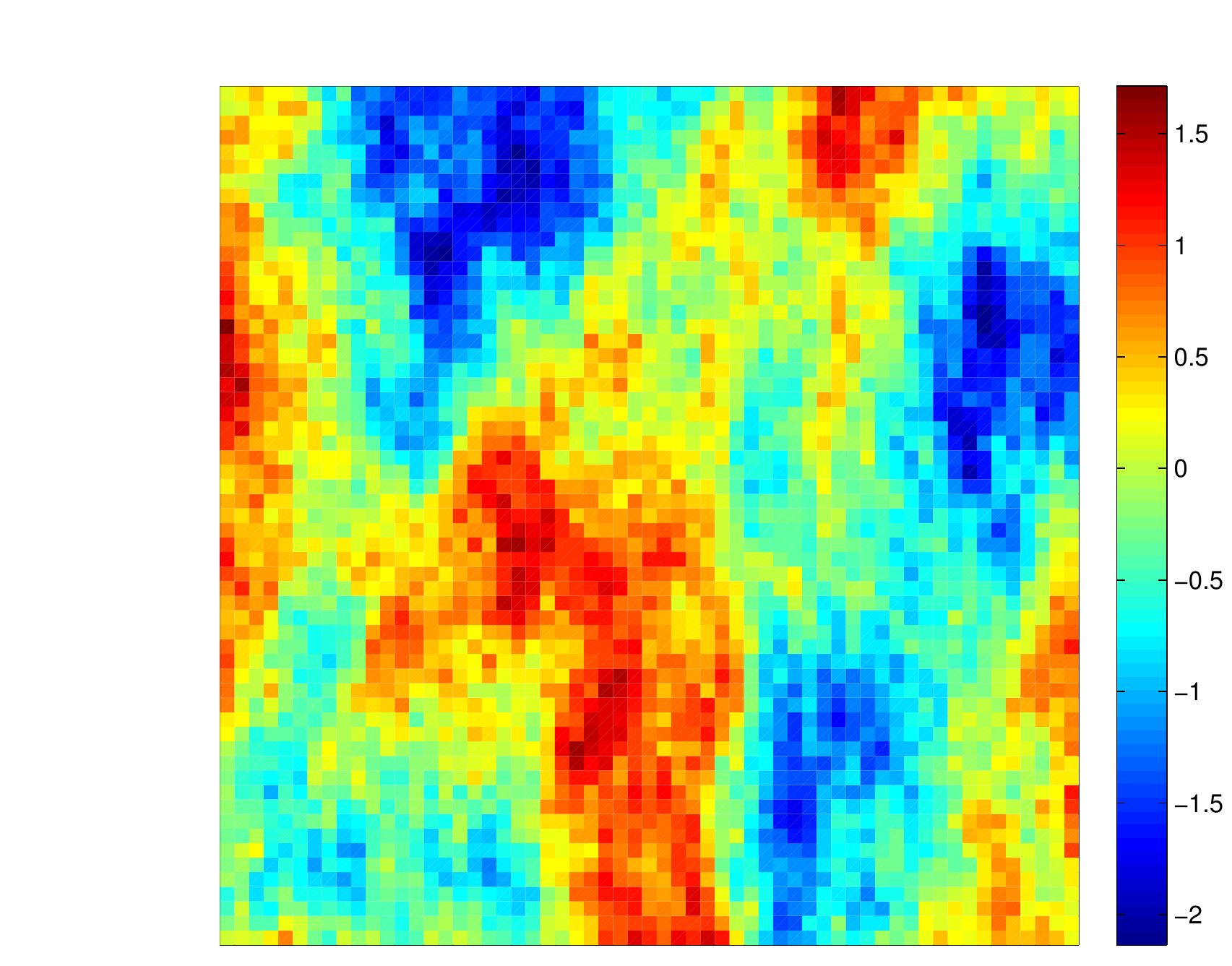}\\
\includegraphics[scale=0.165]{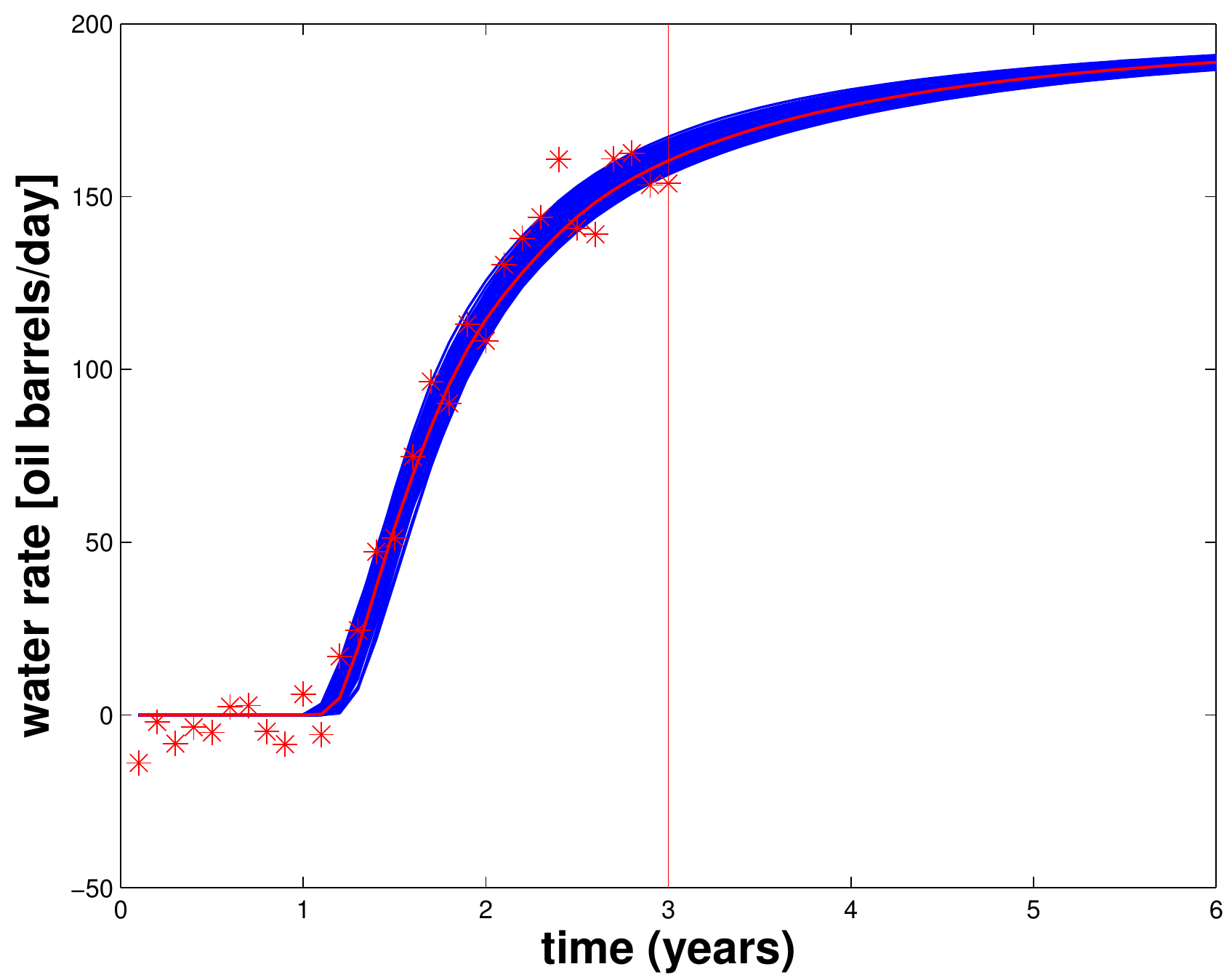}
\includegraphics[scale=0.165]{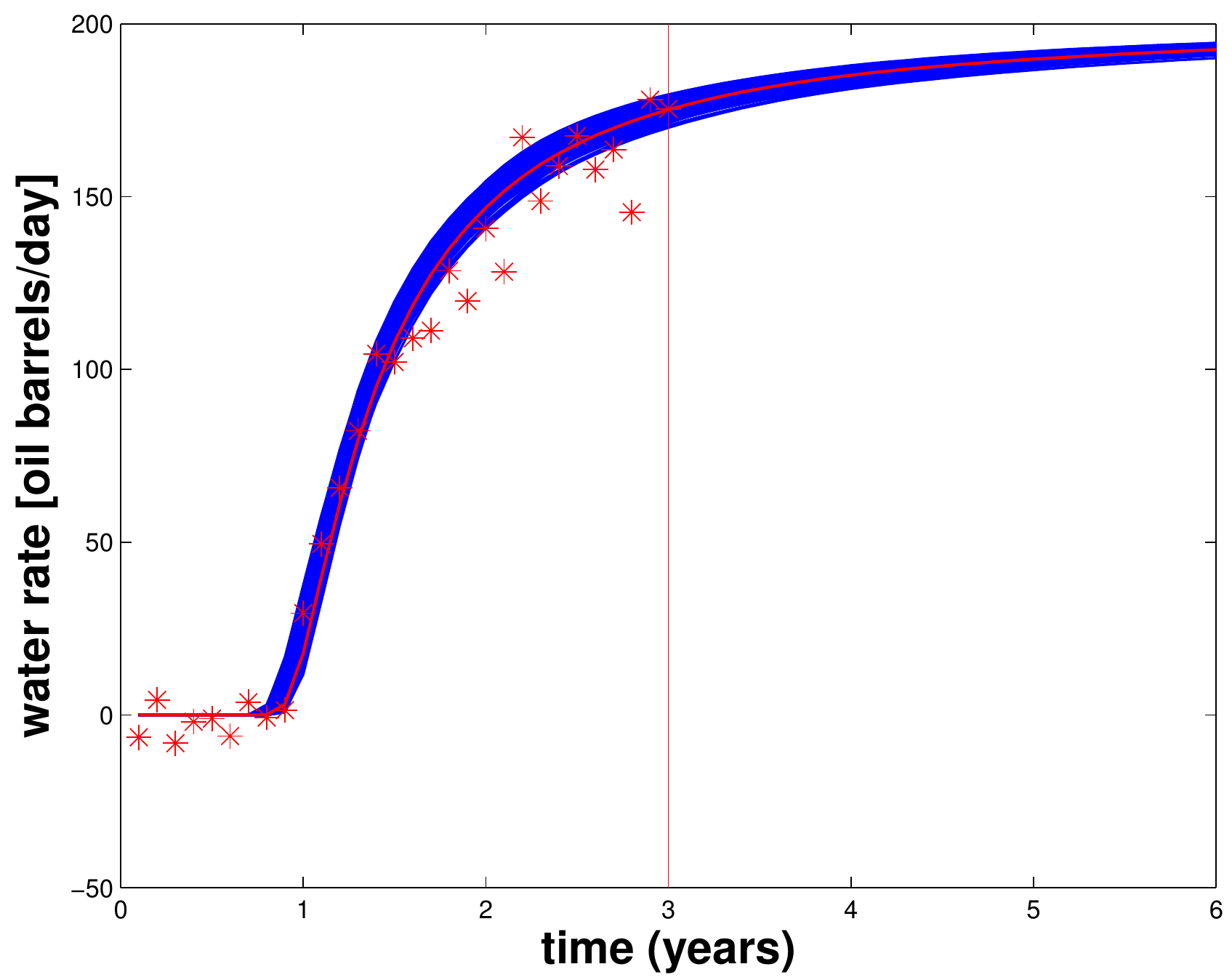}
\includegraphics[scale=0.165]{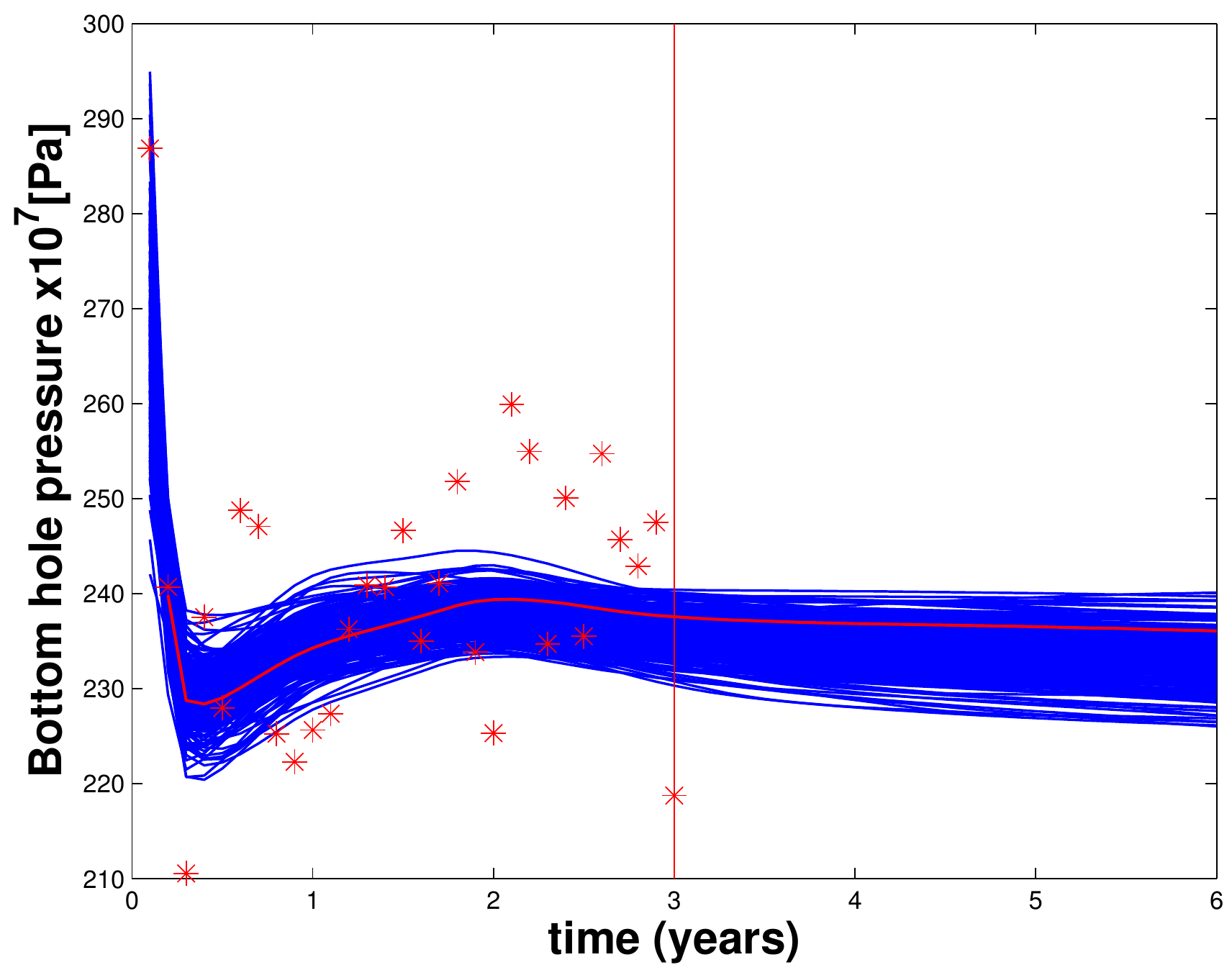}
\includegraphics[scale=0.165]{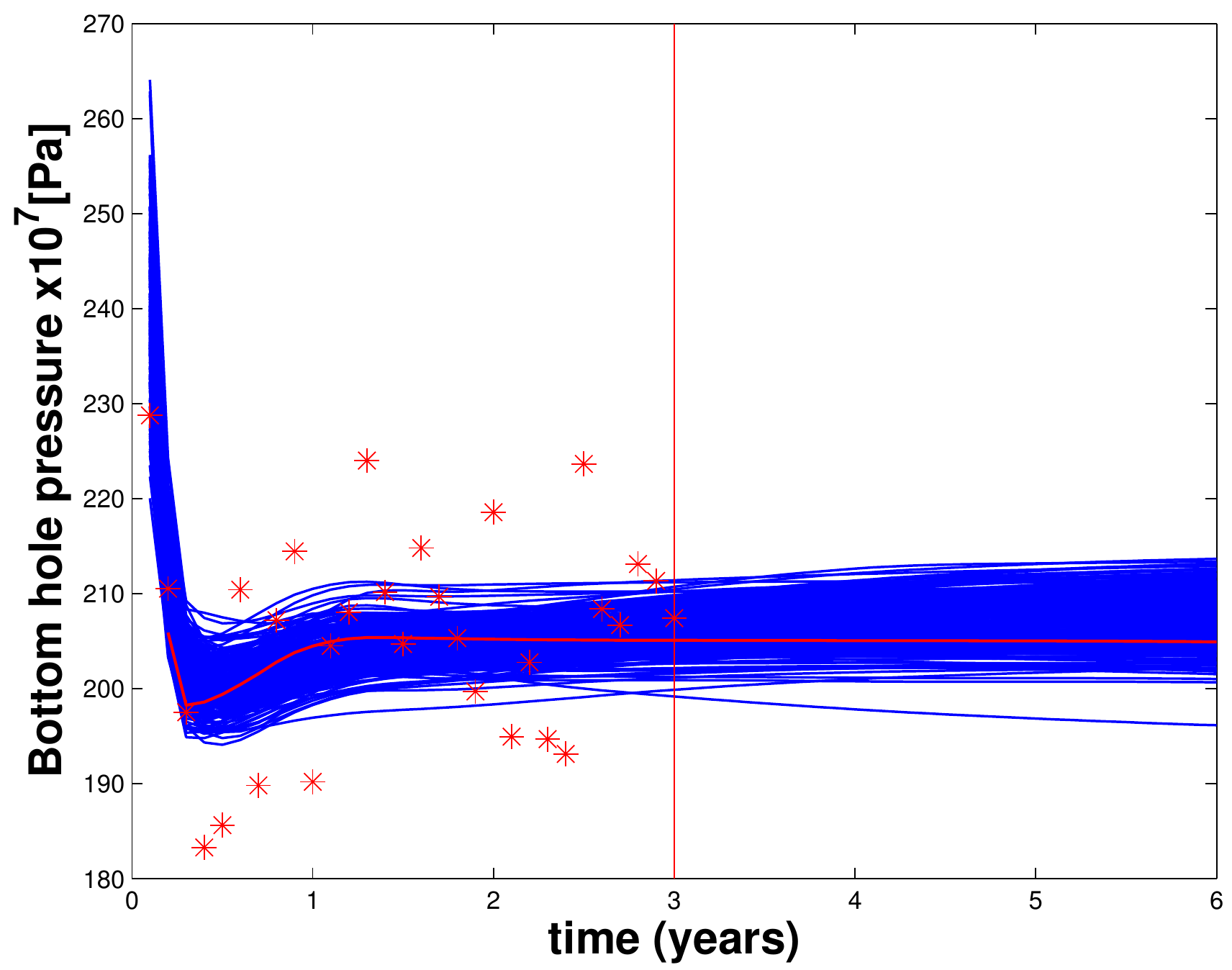}

\caption{ Top: Samples of the posterior distribution $\mu_{B}$ (characterized with pcn-MCMC) [$\log{\textrm{m}^2}$]. Bottom: (from left to right) water rates at wells $P_{1}$, $P_{2}$ and BHP from $I_{2}$ and $I_{3}$}  
\label{Figure3}
\end{center}
\end{figure}

\subsection{IR-enLM}\label{sec:numIR-enLM}

In this section we present a numerical investigation of the effect of the tunable parameters $\rho$ and $\tau$ in the approximation properties of the resulting ensemble obtained with IR-enLM (Algorithm \ref{IR-enLM}). The role of $\rho$ and $\tau$ in the regularizing LM scheme for the solution of deterministic inverse problems is well understood in terms of the theory of \cite{Hanke}. In addition, the work of \cite{LM} displays the application and validation of such theory applied to history matching problems where the aim was to recover the permeability of an oil-water reservoir. This subsection designed to understand the effect of $\rho$ and $\tau$ in the approximation properties of the proposed methods in the context of Bayesian inverse problems. More precisely, we investigate the effect of the tunable parameters in the accuracy and efficiency of IR-enLM for recovering the mean and the variance of the posteriors $\mu_{A}$ and $\mu_{B}$ computed with MCMC in the preceding subsection.

We apply IR-enLM to generate ensemble approximations to $\mu_{A}$ and $\mu_{B}$, respectively. We recall that for both cases, we have the same prior from which the initial ensemble for IR-enLM is drawn. Since the results depend on the initial ensemble, the results we report in this manuscript are averaged over 15 experiments corresponding to different initial ensembles of a fixed size. For each of these experiments, IR-enLM produces an ensemble of inverse estimates whose resulting mean $\hat{u}$ and variance $\hat{\sigma}$ are compared to the mean $u_{pos}$ and variance $\sigma_{pos}$ of the posterior. More precisely, we consider the relative errors
\begin{eqnarray}\label{eq:5.3}
\epsilon_{u}=\frac{\vert\vert (\hat{u}-\overline{u})-(u_{pos}-\overline{u})\vert\vert_{L^{2}(D)}}{\vert\vert (u_{pos}-\overline{u})\vert\vert_{L^{2}(D)}},\qquad \textrm{and}\qquad \epsilon_{\sigma}=\frac{\vert\vert \hat{\sigma}-\sigma_{pos}\vert\vert_{L^{2}(D)}}{\vert\vert \sigma_{pos}\vert\vert_{L^{2}(D)}},
\end{eqnarray}
We re-emphasize that the mean and the variance of the posterior (for $\mu_{A}$ and $\mu_{B}$, respectively) are computed from our pc-MCMC chains. 

We recall from our discussion in subsection \ref{tunable} that for a given $\rho\in (0,1)$, the theory of the regularizing LM scheme \cite{Hanke} requires that we select $\tau$ according to $\tau>1/\rho>1$ in order to ensure convergence to a stable computation of the minimizer of (\ref{eq:IR}). Larger values of $\rho$ (i.e. $\rho$ closer to one) yield small incremental steps and smaller $\tau$'s which enable the algorithm to progress for a long number of iterations obtaining more accurate minimizers of (\ref{eq:IR}). On the other hand, smaller values of $\rho$ and so larger values of $\tau$'s are associated with very early termination of the scheme (according to (\ref{discrepancy2})) and so the corresponding estimates of the minimizer of (\ref{eq:IR}) are less accurate although convergence is faster. These considerations in terms of the accuracy of the minimizer of (\ref{eq:IR}) are reflected in the approximation properties of IR-enLM as we can observe from Table \ref{Table1A} where we display the corresponding ensemble approximation of $\mu_{A}$ and $\mu_{B}$ obtained from IR-enLM for several choices of $\tau$, $\rho$ for an ensemble fixed size $N_{e}=50$. Indeed, Table \ref{Table1A} confirms that more accurate minimizers of (\ref{eq:IR}) obtained by selecting a smaller value of $\rho$ result in better ensemble approximations of the mean and variance of the posterior. In the 5th and 8th columns of Table \ref{Table1A} we display the average number of iterations for convergence (established by (\ref{discrepancy2})). For each $\rho$, in Table \ref{Table1A} we display different values of $\tau$ with the corresponding value $\tau=1/\rho$ highlighted blue. Note that as $\rho$ increases, the corresponding $\tau=1/\rho$ (in blue) produces more accurate solutions. For fixed $\rho=0.8$ and $N_{e}=50$, in Figure \ref{Figure4} we present a visual comparison of the mean (top row) and variance (top-middle row) produce by one ensemble of IR-enLM for approximating $\mu_{A}$ with different choices of $\tau$. For this ensemble and the describe choices of parameters, in the bottom-middle (resp. bottom-bottom) row of Figure \ref{Figure4} we display the box plots of water rate (resp. bottom-hole pressure) at some production (resp. injection) wells after 6 years of injection (i.e. 3 years of assimilation plus 3 years of forecast). Analogous visual results for the approximation of $\mu_{B}$ are displayed in Figure \ref{Figure5}.

For the larger values of $\rho$ in Table \ref{Table1A} we observe that  the optimal value of $\tau$ for reducing the error in the mean and variance is close to $\tau=1.0$ even though the application of the application of the regularizing LM scheme indicates that $\tau$ should be selected to satisfy $\tau>1/\rho$ to ensure stability. Note, for example, that for $\rho=0.7$, $\rho=0.8$ and $\rho=0.9$ the stabilization is ensure provided $\tau>1/\rho=1.45$, $\tau>1/\rho=1.25$ and $\tau>1/\rho=1.11$, respectively. For theses three cases, the errors in the mean and variance can be substantially reduced by considering $\tau=1.0$. Two potential causes for this improved accuracy with reduced $\tau$ are the following. First, we know from \cite{Hanke} that the condition on $\tau$ ($\tau>1/\rho>1$) is sufficient but not necessary for the convergence of the regularizing LM scheme. In other words, stable and more accurate approximate minimizers may be obtained for $1<\tau\leq 1/\rho$ thus contributing to more accurate ensemble approximations of the mean and variance of the posterior. Second, our application of the regularizing LM scheme for the computation of the minimizer of (\ref{eq:IR}) is based on the estimate of the noise level provided in (\ref{estima}). Other potential estimates of the noise level $\eta^{(j)}$ may result in $\tau$'s more consistent with the assumption $\tau>1/\rho$. While obtaining such an optimal estimate is beyond the scope of the present work, it is important to remark that our results reflect that changes in $\tau$ (with $1\leq \tau$) are associated with stable changes in the relative error of the mean and variance produced by IR-enLM. In addition, our numerical experiments suggest that for IR-enLM with $\rho\ge 0.7$, the value $\tau=1.0$ in (\ref{discrepancy2})) provides stable and accurate approximations of the mean and variance of the posterior. Note that this value is consistent with the discrepancy principle in the sense that the algorithm is stopped provided that the data misfit matches our estimate (\ref{estima}) of the noise level. Note that, in most cases, the error with respect to the mean can be further reduced by allowing the scheme to progress for more iterations (i.e. by selecting a smaller $\tau<1$). However, this decrease in the error of the mean is associated with an increase in the error of the variance. We therefore see that the error grow typical of the computation of solutions to ill-posed inverse problems is reflected in the properties of IR-enLM for approximating the Bayesian posterior. While for smaller values of $\rho$ ($\rho<0.7$) the value of $\tau$ seems to be consistent with $\tau>1/\rho$, the associated errors are substantial and therefore not recommended for further applications of IR-enLM. Note that $\rho\approx 0.8$ with $\tau=1.0$ is seems a reasonable compromise between accuracy and cost.

It is also worth mentioning that the aforementioned conclusions concerning $\tau$ apply for several choices of ensemble sizes as it can be observed in Table \ref{Table1B} (for $\rho=0.8$). Note that when we increase the ensemble size by a factor of ten, the relative error with respect to the mean and variance of $\mu_{A}$ (resp. $\mu_{B}$) are decreased by $\%30$ (resp. $\%20$) and $\%35$ (resp. $\%32$), respectively. Finally, we note that the IR-enLM was capable of obtaining a better approximation of $\mu_{A}$ (in terms of mean and variance) than the one of $\mu_{B}$. We recall that $\mu_{A}$ and $\mu_{B}$ correspond to two different Bayesian inverse problems. However, our results suggests that posteriors that quantity larger uncertainties may be more difficult to be captured with IR-enLM.

\begin{table}[H!]                                                                                  
\centering                                                                                         
\caption{Performance of IR-enLMfor different choices of $\rho$ and $\tau$ (for $N_{e}=50$)}      \label{Table1A}                                              
\begin{tabular}{cc|ccc||ccc}                                                         
\hline\noalign{\smallskip}
    & & &Approx. of $\mu_{A}$& & & Approx. of $\mu_{B}$ & \\
\noalign{\smallskip}\hline\noalign{\smallskip}
$\rho$ & $\tau$  &  $\epsilon_{u}$&  $\epsilon_{\sigma}$ & aver. iter. &  $\epsilon_{u}$&  $\epsilon_{\sigma}$ & aver. iter. \\
\noalign{\smallskip}\hline\noalign{\smallskip}   
\hline                                         \hline                     

0.5 & 3.000 & 0.682 & 0.788 & 2.936 & 0.636 & 0.873 & 2.229 \\  
0.5 & 2.800 & 0.669 & 0.780 & 3.024 & 0.620 & 0.872 & 2.287 \\  
0.5 & 2.500 & 0.651 & 0.756 & 3.152 & 0.598 & 0.860 & 2.464 \\  
0.5 & 2.200 & 0.596 & 0.729 & 3.384 & 0.568 & 0.860 & 2.640 \\  
\Blue{0.5} & \Blue{2.000 }& \Blue{0.573} &\Blue{ 0.718} &\Blue{ 3.492} &\Blue{ 0.550} &\Blue{ 0.867} & \Blue{2.742} \\  
0.5 & 1.800 & 0.488 & 0.765 & 3.644 & 0.531 & 0.919 & 2.853 \\  
0.5 & 1.600 & 0.408 & 0.791 & 3.840 & 0.516 & 0.962 & 2.978 \\  
0.5 & 1.400 & 0.341 & 0.765 & 4.100 & 0.491 & 1.012 & 3.145 \\  
\hline                                                          \hline                                                          
0.6 & 2.200 & 0.654 & 0.755 & 4.335 & 0.620 & 0.634 & 3.575 \\  
0.6 & 2.000 & 0.631 & 0.735 & 4.534 & 0.598 & 0.613 & 3.625 \\  
0.6 & 1.900 & 0.616 & 0.720 & 4.636 & 0.587 & 0.604 & 3.650 \\  
0.6 & 1.800 & 0.594 & 0.712 & 4.754 & 0.571 & 0.591 & 3.825 \\  
\Blue{0.6} & \Blue{1.700} &\Blue{ 0.564} &\Blue{ 0.702} &\Blue{ 4.862} &\Blue{ 0.543} & \Blue{0.573} &\Blue{ 3.975} \\  
0.6 & 1.600 & 0.515 & 0.695 & 4.990 & 0.514 & 0.581 & 4.075 \\  
0.6 & 1.500 & 0.452 & 0.716 & 5.108 & 0.489 & 0.591 & 4.200 \\  
0.6 & 1.300 & 0.328 & 0.681 & 5.428 & 0.421 & 0.628 & 4.525 \\  
\hline                                                          \hline                                                          
0.7 & 1.500 & 0.565 & 0.673 & 7.049 & 0.536 & 0.508 & 6.122 \\    
\Blue{0.7} & \Blue{1.450} & \Blue{0.544} &\Blue{ 0.665} &\Blue{ 7.138} &\Blue{ 0.522} &\Blue{ 0.498} &\Blue{ 6.206 }\\  
0.7 & 1.400 & 0.517 & 0.651 & 7.230 & 0.498 & 0.488 & 6.320 \\  
0.7 & 1.250 & 0.379 & 0.637 & 7.590 & 0.420 & 0.487 & 6.651 \\  
0.7 & 1.150 & 0.293 & 0.572 & 7.885 & 0.381 & 0.498 & 6.869 \\  
0.7 & 1.100 & 0.269 & 0.504 & 8.059 & 0.367 & 0.495 & 6.973 \\  
0.7 & 1.000 & 0.240 & 0.405 & 8.597 & 0.335 & 0.528 & 7.335 \\  
0.7 & 0.950 & 0.221 & 0.421 & 9.425 & 0.332 & 0.595 & 7.978 \\  
\hline                                                          \hline                                                          
0.8 & 1.500 & 0.619 & 0.706 & 10.906 & 0.569 & 0.529 & 9.789 \\ 
0.8 & 1.400 & 0.594 & 0.687 & 11.248 & 0.544 & 0.499 & 10.158 \\
0.8 & 1.300 & 0.555 & 0.648 & 11.611 & 0.511 & 0.466 & 10.541 \\
\Blue{0.8} &\Blue{ 1.250} &\Blue{ 0.518} & \Blue{0.621} &\Blue{ 11.842} &\Blue{ 0.490} & \Blue{0.447 }&\Blue{ 10.739} \\
0.8 & 1.200 & 0.466 & 0.586 & 12.026 & 0.463 & 0.430 & 10.938 \\
0.8 & 1.150 & 0.387 & 0.560 & 12.216 & 0.430 & 0.425 & 11.127 \\
0.8 & 1.000 & 0.249 & 0.386 & 13.088 & 0.329 & 0.425 & 11.755 \\
0.8 & 0.950 & 0.214 & 0.394 & 13.945 & 0.310 & 0.479 & 12.344 \\
\hline                                                          \hline                                                          
0.9 & 1.160 & 0.547 & 0.594 & 25.960 & 0.491 & 0.438 & 23.907 \\
0.9 & 1.140 & 0.524 & 0.564 & 26.150 & 0.482 & 0.430 & 24.118 \\
0.9 & 1.120 & 0.497 & 0.538 & 26.324 & 0.472 & 0.419 & 24.331 \\
\Blue{0.9 }& \Blue{1.110} & \Blue{0.482} &\Blue{ 0.523} &\Blue{ 26.425} &\Blue{ 0.467} &\Blue{ 0.415} &\Blue{ 24.430} \\
0.9 & 1.100 & 0.466 & 0.510 & 26.515 & 0.462 & 0.412 & 24.510 \\
0.9 & 1.050 & 0.375 & 0.443 & 26.949 & 0.420 & 0.386 & 25.022 \\
0.9 & 1.000 & 0.277 & 0.371 & 27.476 & 0.362 & 0.381 & 25.511 \\
0.9 & 0.950 & 0.217 & 0.348 & 28.364 & 0.304 & 0.419 & 26.239 \\
\hline                        
\end{tabular}    
\end{table}

\begin{table}[H!]                                                                                  
\centering                                                                                         
\caption{Performance of IR-enLMfor different choices of $N_{e}$ and $\tau$ (for $\rho=0.8$)}       \label{Table1B}                                             
\begin{tabular}{cc|ccc||ccc}                                                         
\hline\noalign{\smallskip}
    & & &Approx. of $\mu_{A}$& & & Approx. of $\mu_{B}$ & \\
\noalign{\smallskip}\hline\noalign{\smallskip}
$\tau$ & $N_{e}$  &  $\epsilon_{u}$&  $\epsilon_{\sigma}$ & aver. iter. &  $\epsilon_{u}$&  $\epsilon_{\sigma}$ & aver. iter. \\
\noalign{\smallskip}\hline\noalign{\smallskip}
\hline                                                             
\hline                                                             
1.3 & 25.000 & 0.577 & 0.699 & 11.776 & 0.538 & 0.548 & 10.588 \\  
1.25 & 25.000 & 0.540 & 0.668 & 11.980 & 0.519 & 0.532 & 10.784 \\ 
1.2 & 25.000 & 0.502 & 0.642 & 12.120 & 0.495 & 0.515 & 10.992 \\  
1.1 & 25.000 & 0.361 & 0.568 & 12.540 & 0.430 & 0.500 & 11.340 \\  
\Red{1} & \Red{25.000} &\Red{ 0.300} &\Red{ 0.449 }& \Red{13.152} &\Red{ 0.372} &\Red{ 0.512} & \Red{11.792} \\    
0.95 & 25.000 & 0.267 & 0.463 & 13.936 & 0.359 & 0.546 & 12.320 \\ 
\hline                                                             
\hline

1.3 & 50.000 & 0.555 & 0.648 & 11.611 & 0.511 & 0.466 & 10.541 \\
1.25 & 50.000 & 0.518 &0.621 & 11.842 & 0.490 & 0.447& 10.739 \\
1.2 & 50.000 & 0.466 & 0.586 & 12.026 & 0.463 & 0.430 & 10.938 \\
1.1 & 50.000 & 0.315 & 0.503 & 12.444 & 0.393 & 0.422 & 11.317 \\
\Red{1} & \Red{50.000} & \Red{0.249} & \Red{0.386} & \Red{13.088} &\Red{ 0.329 }& \Red{0.425} & \Red{11.755} \\
0.95 & 50.000 & 0.214 & 0.394 & 13.945 & 0.310 & 0.479 & 12.344 \\
  \hline                                                             \hline                                                 
1.3 & 75.000 & 0.545 & 0.617 & 11.569 & 0.502 & 0.437 & 10.551 \\  
1.25 & 75.000 & 0.506 & 0.585 & 11.768 & 0.480 & 0.416 & 10.761 \\                       
1.2 & 75.000 & 0.455 & 0.549 & 11.945 & 0.455 & 0.402 & 10.947 \\  
1.1 & 75.000 & 0.300 & 0.467 & 12.357 & 0.383 & 0.387 & 11.325 \\  
\Red{1 }& \Red{75.000 }& \Red{0.231} &\Red{ 0.339} & \Red{12.988} & \Red{0.319} &\Red{ 0.386} & \Red{11.767 }\\    
0.95 & 75.000 & 0.196 & 0.357 & 13.920 & 0.300 & 0.438 & 12.345 \\ 
\hline                                                             
\hline                                               
              
1.3 & 100.000 & 0.544 & 0.612 & 11.602 & 0.496 & 0.423 & 10.541 \\ 
1.25 & 100.000 & 0.507 & 0.585 & 11.833 & 0.475 & 0.403 & 10.739 \\
1.2 & 100.000 & 0.451 & 0.548 & 12.021 & 0.448 & 0.384 & 10.938 \\ 
1.1 & 100.000 & 0.295 & 0.459 & 12.441 & 0.374 & 0.376 & 11.317 \\ 
\Red{1} & \Red{100.000} & \Red{0.226} &\Red{ 0.336} & \Red{13.084} &\Red{ 0.307} & \Red{0.378} & \Red{11.755} \\                              
0.95 & 100.000 & 0.188 & 0.346 & 13.960 & 0.285 & 0.432 & 12.344 \\

\hline                   
1.3 & 250.000 & 0.535 & 0.580 & 11.569 & 0.487 & 0.394 & 10.541 \\ 
1.25 & 250.000 & 0.496 & 0.548 & 11.768 & 0.466 & 0.374 & 10.739 \\
1.2 & 250.000 & 0.444 & 0.512 & 11.945 & 0.437 & 0.355 & 10.938 \\ 
1.1 & 250.000 & 0.284 & 0.425 & 12.357 & 0.363 & 0.347 & 11.317 \\ 
\Red{1} &\Red{ 250.000 }& \Red{0.210} &\Red{ 0.290} & \Red{12.988 }& \Red{0.293} & \Red{0.348} & \Red{11.755 }\\   
0.95 & 250.000 & 0.171 & 0.308 & 13.920 & 0.270 & 0.405 & 12.344 \\
\hline                                                             
\hline                                                             
\end{tabular}                                                     
\end{table}

\begin{figure} 
\begin{center}
\includegraphics[scale=0.15]{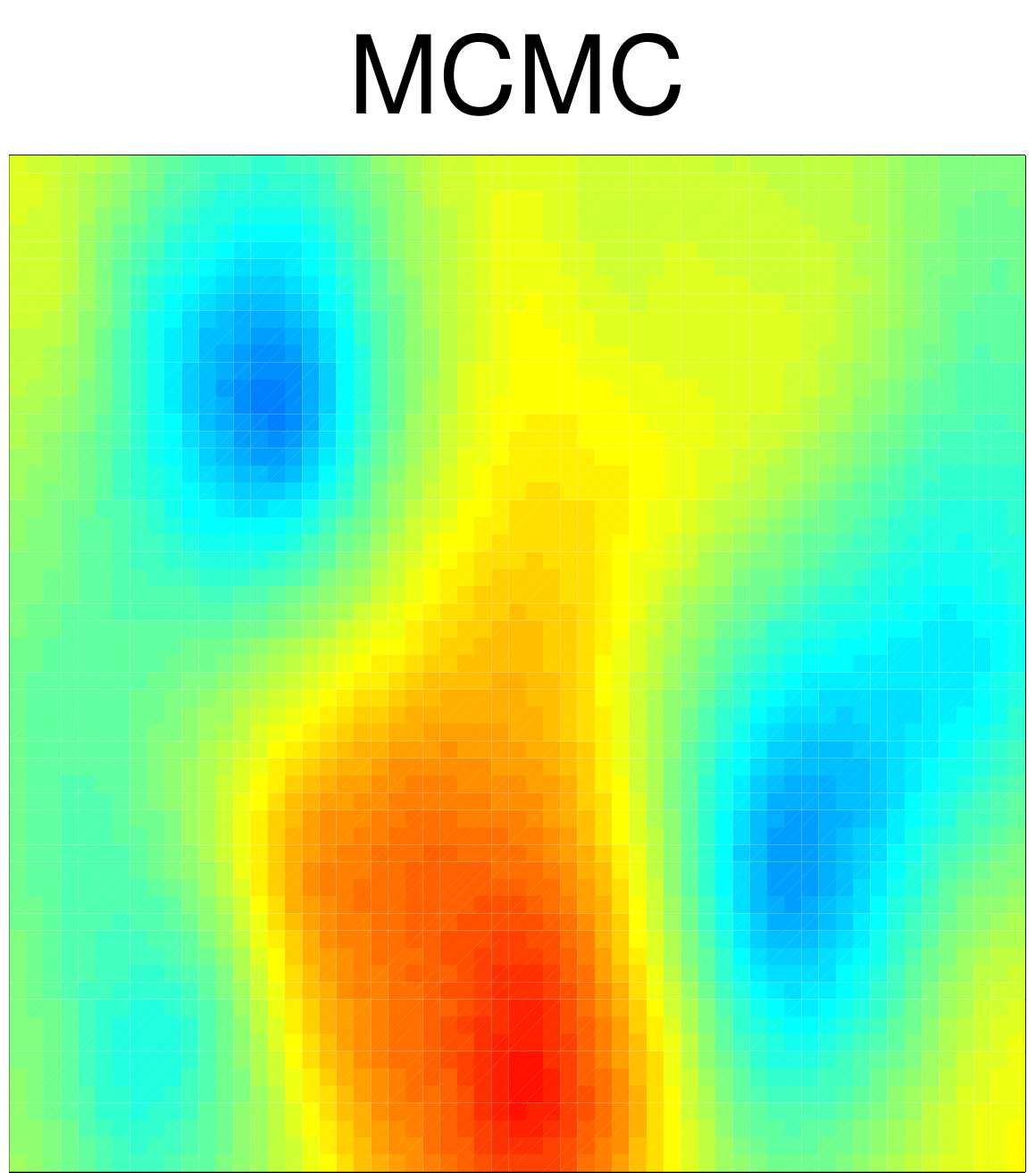}~
\includegraphics[scale=0.15]{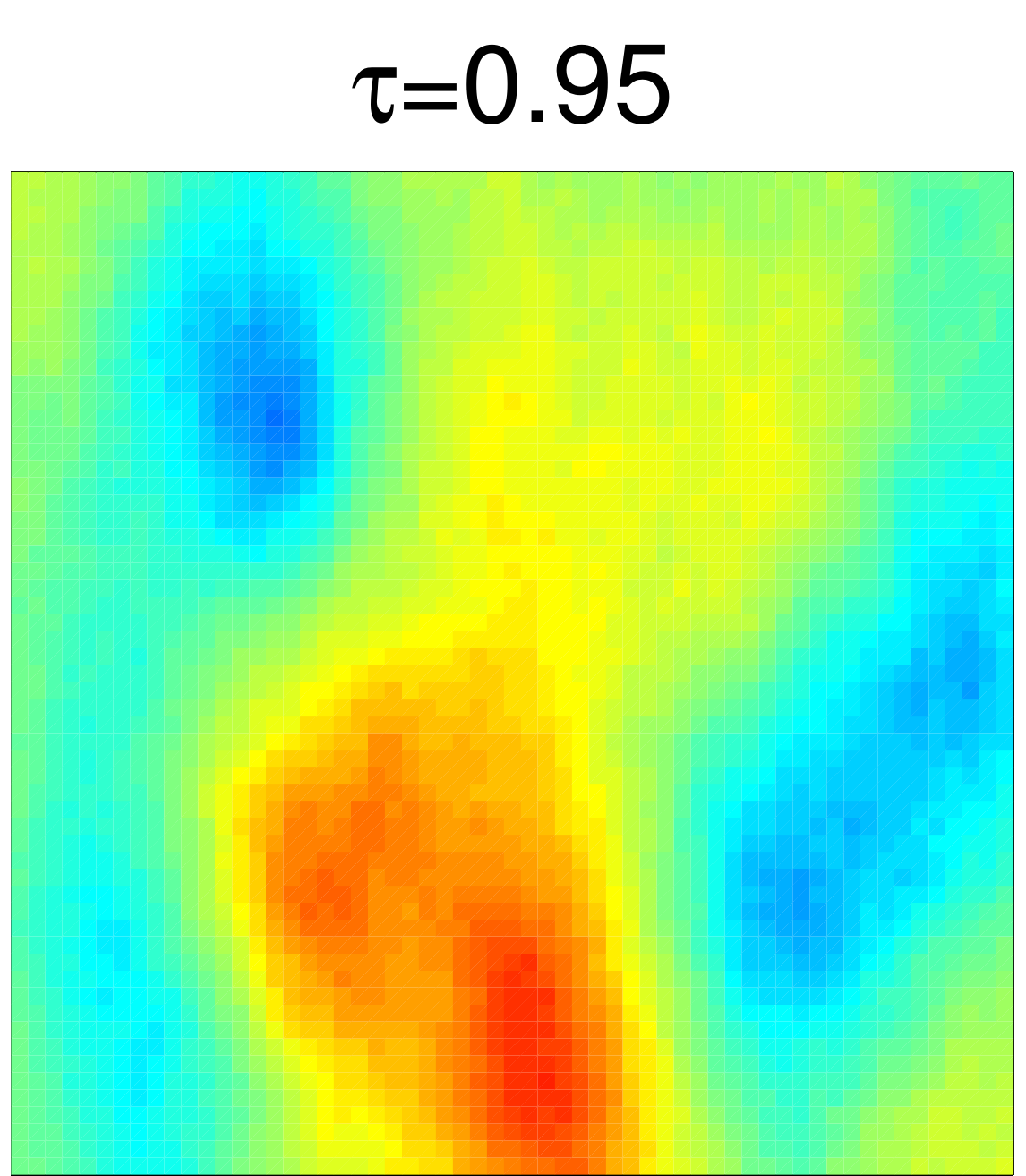}~
\includegraphics[scale=0.15]{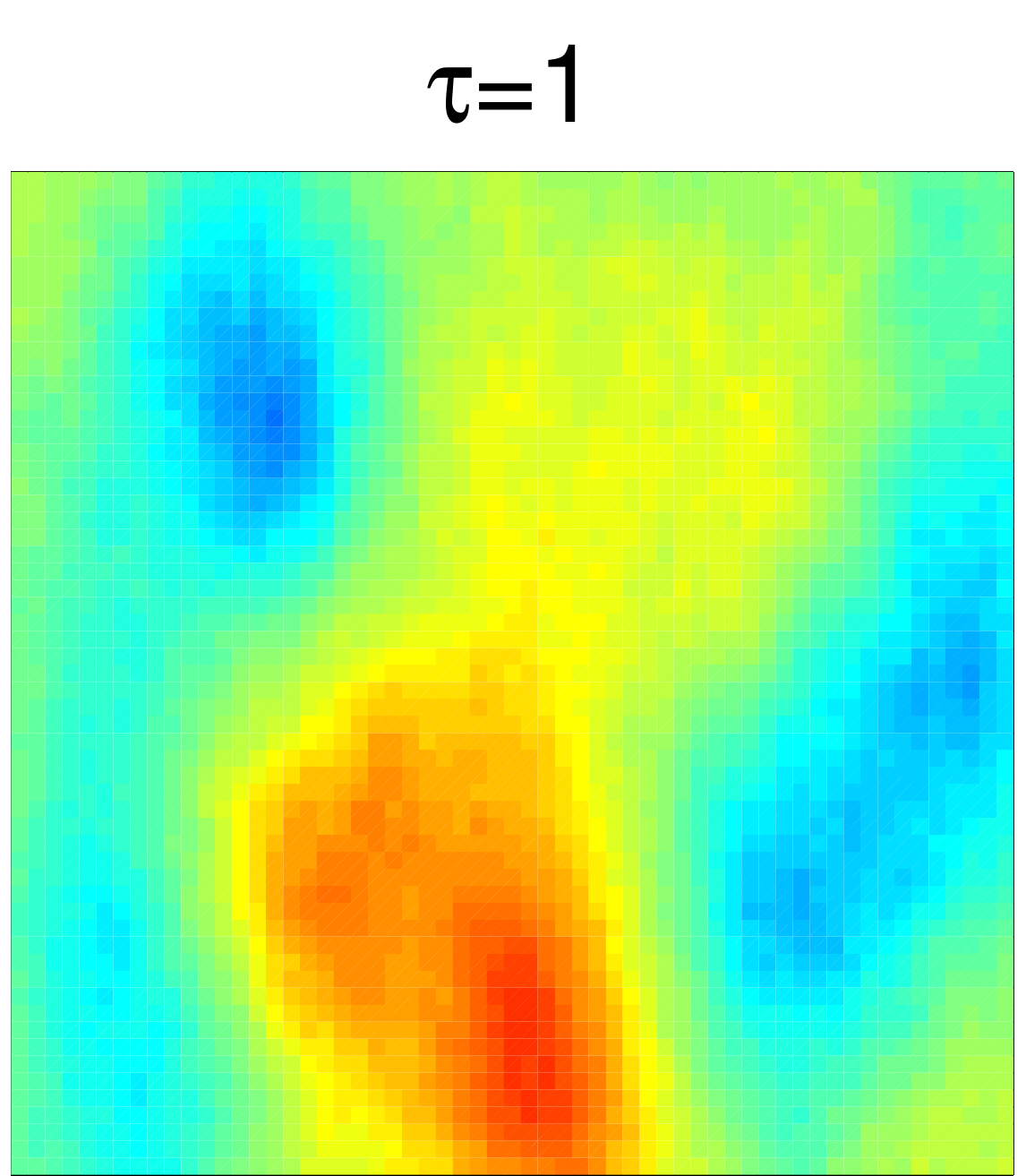}~
\includegraphics[scale=0.15]{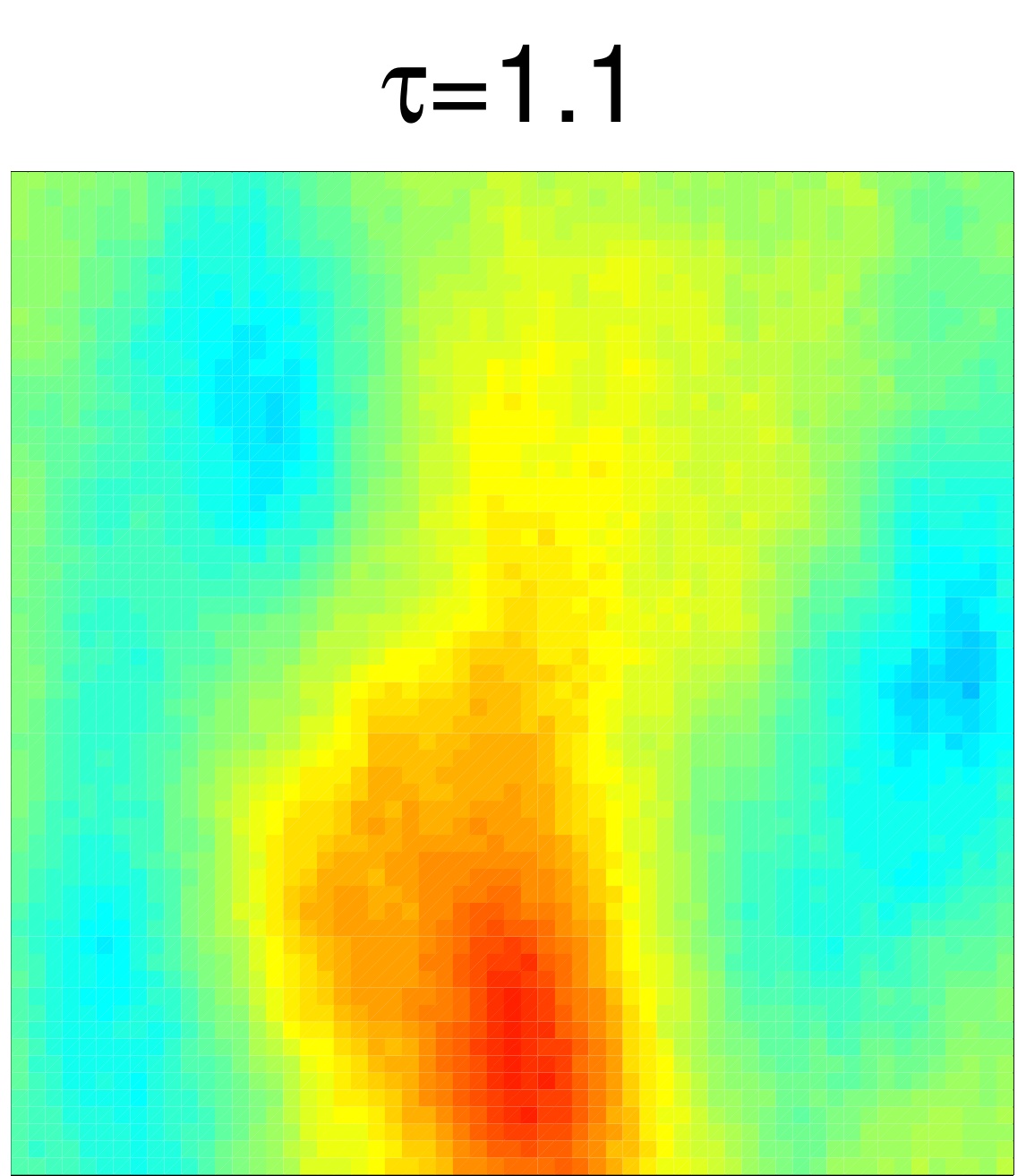}~
\includegraphics[scale=0.15]{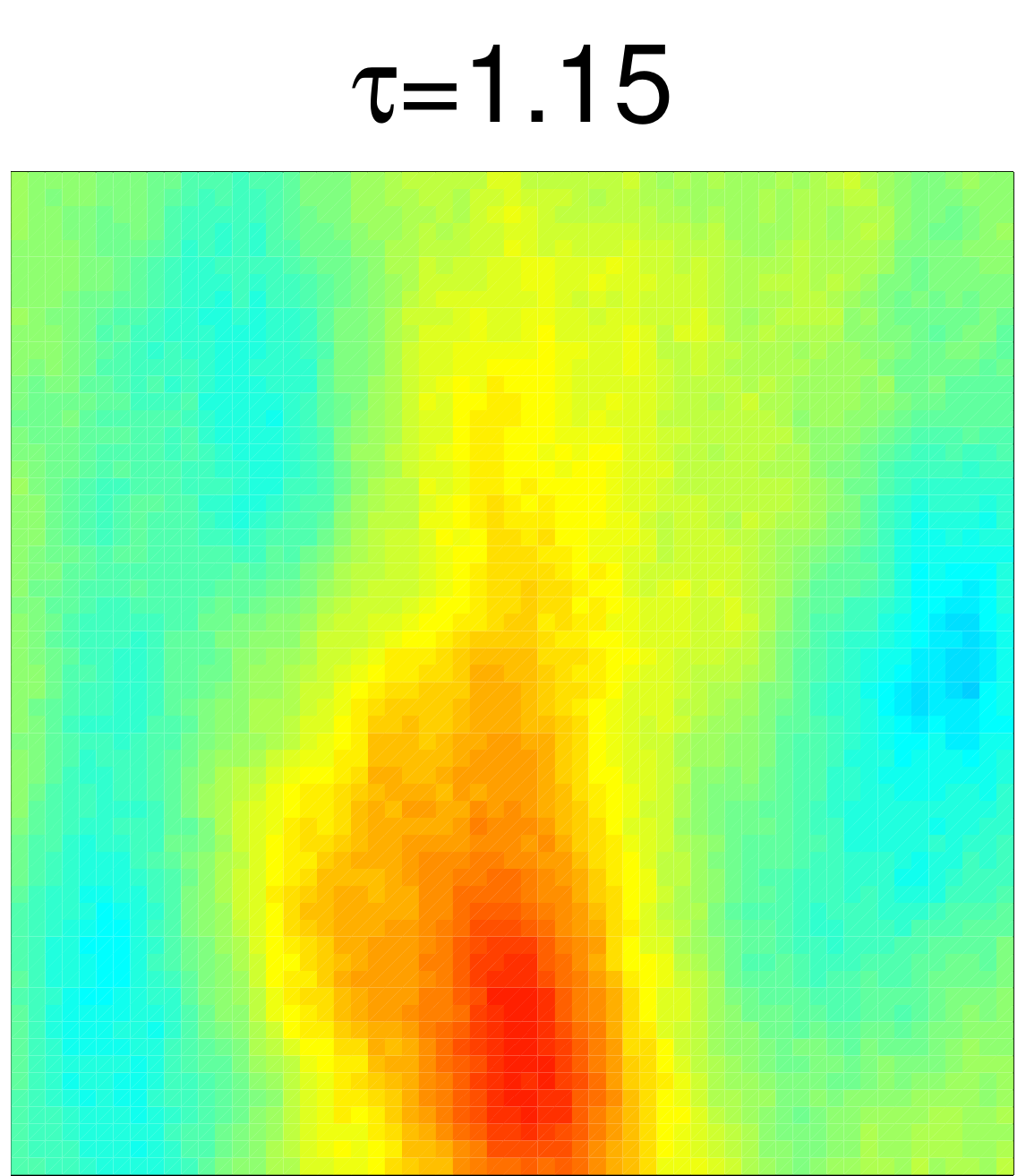}~
\includegraphics[scale=0.15]{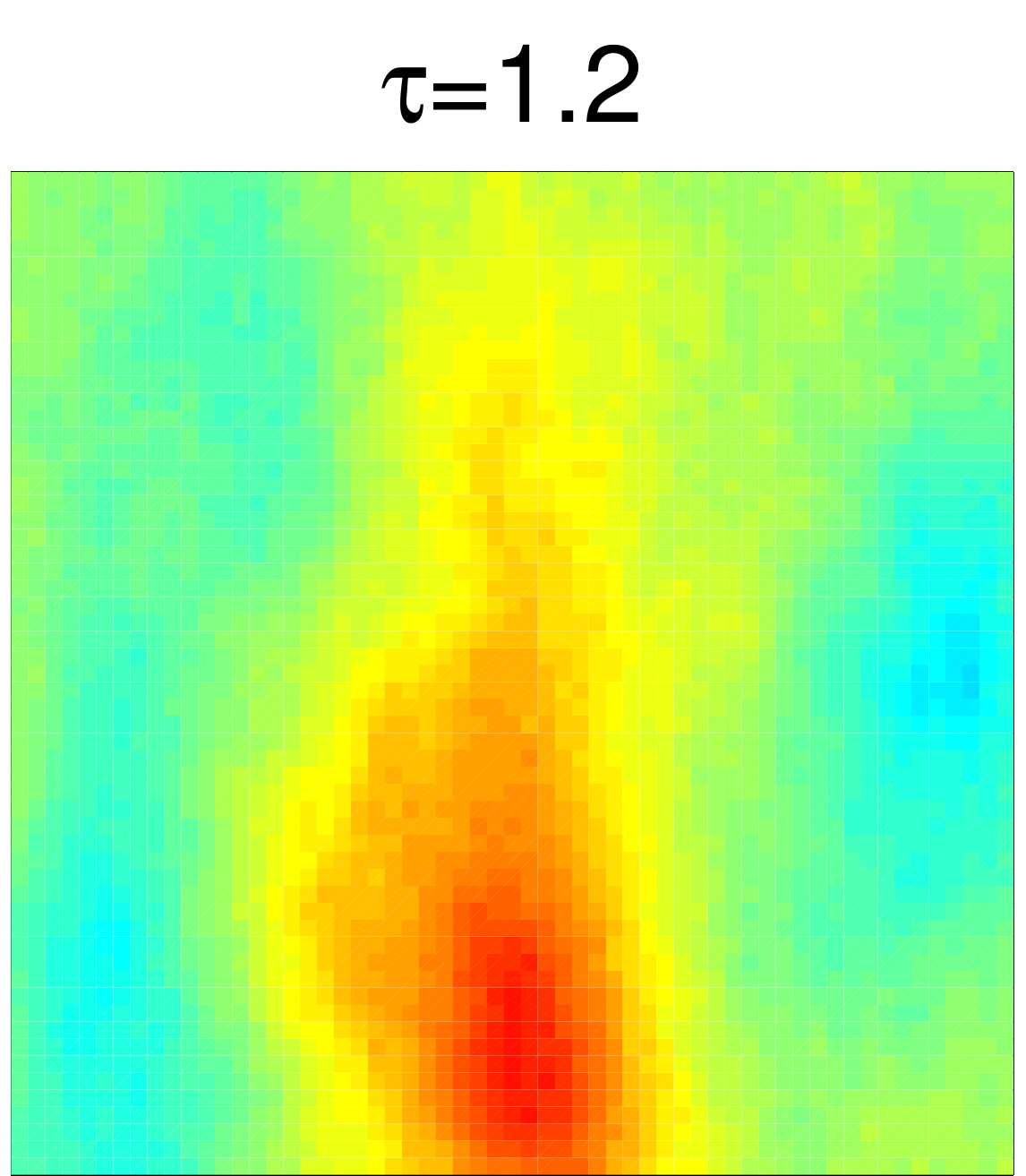}\\
\includegraphics[scale=0.15]{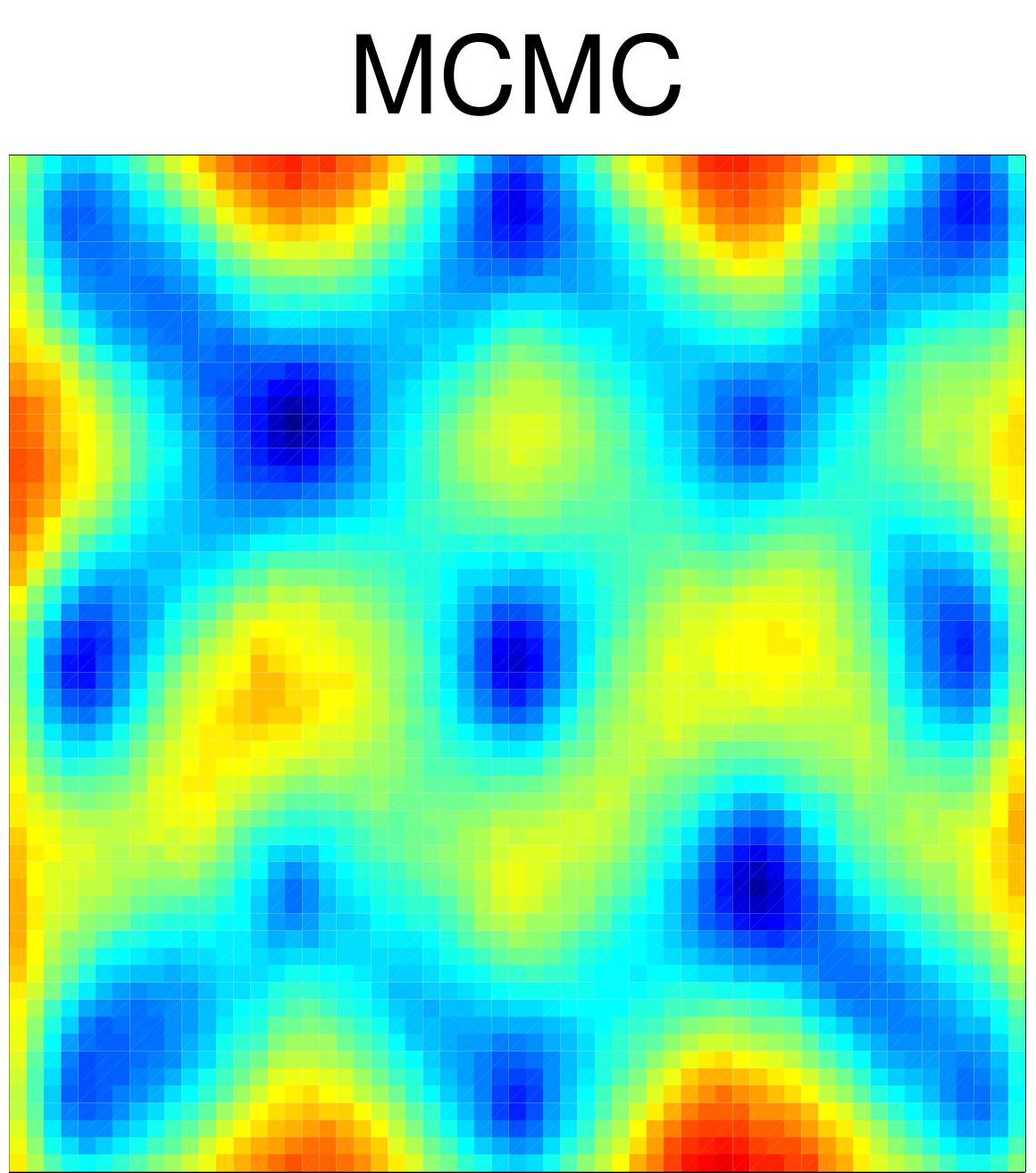}~
\includegraphics[scale=0.15]{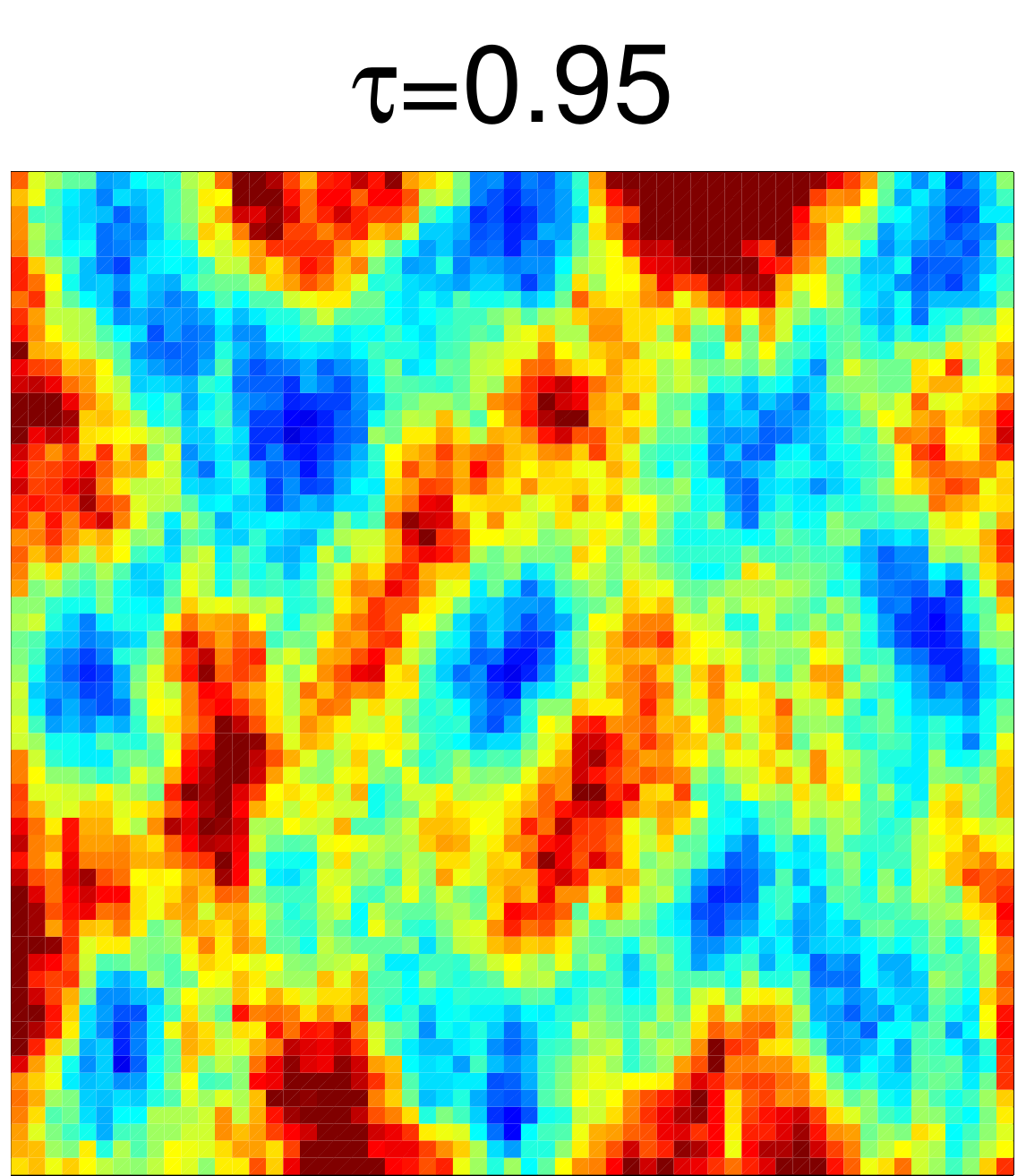}~
\includegraphics[scale=0.15]{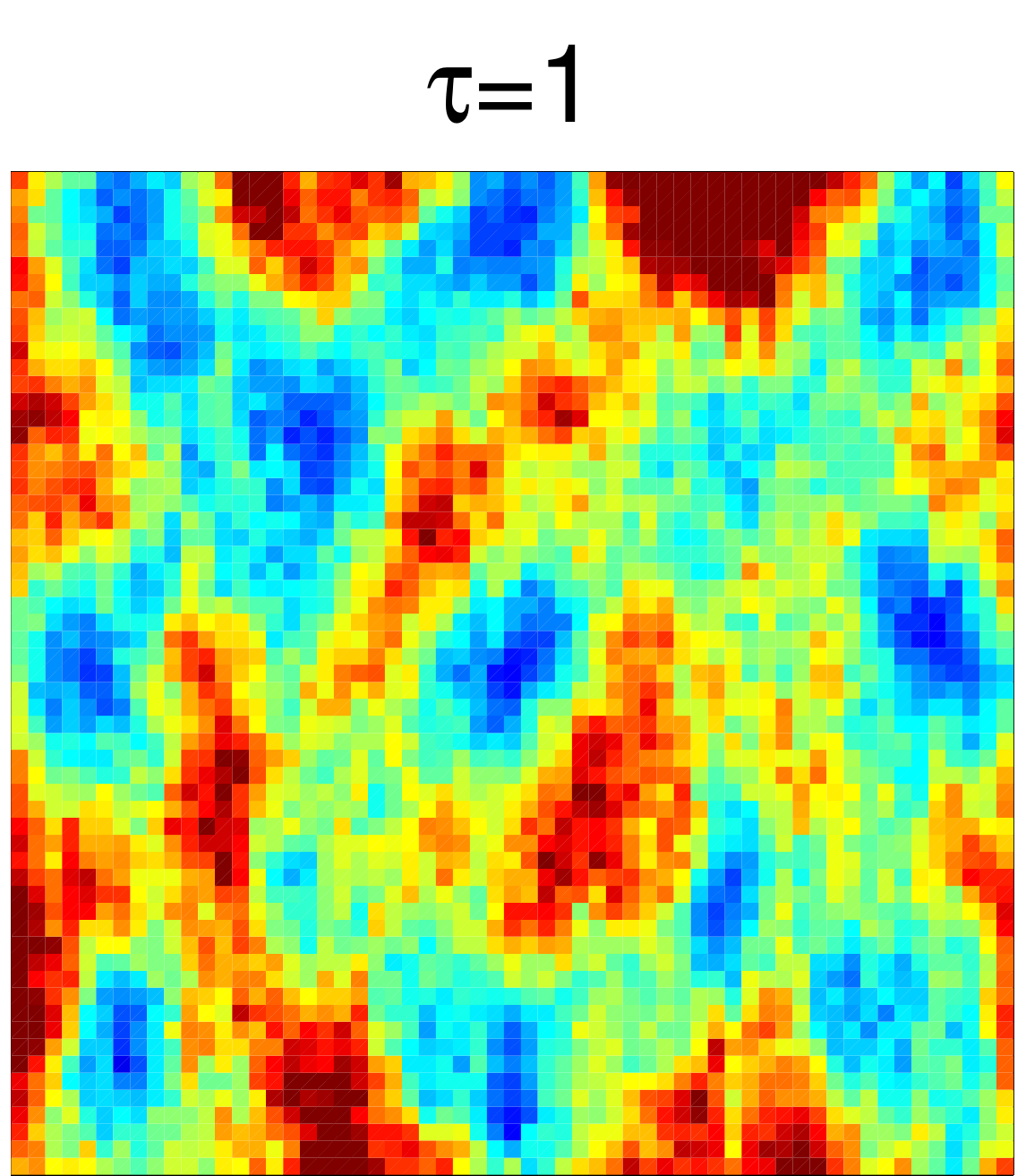}~
\includegraphics[scale=0.15]{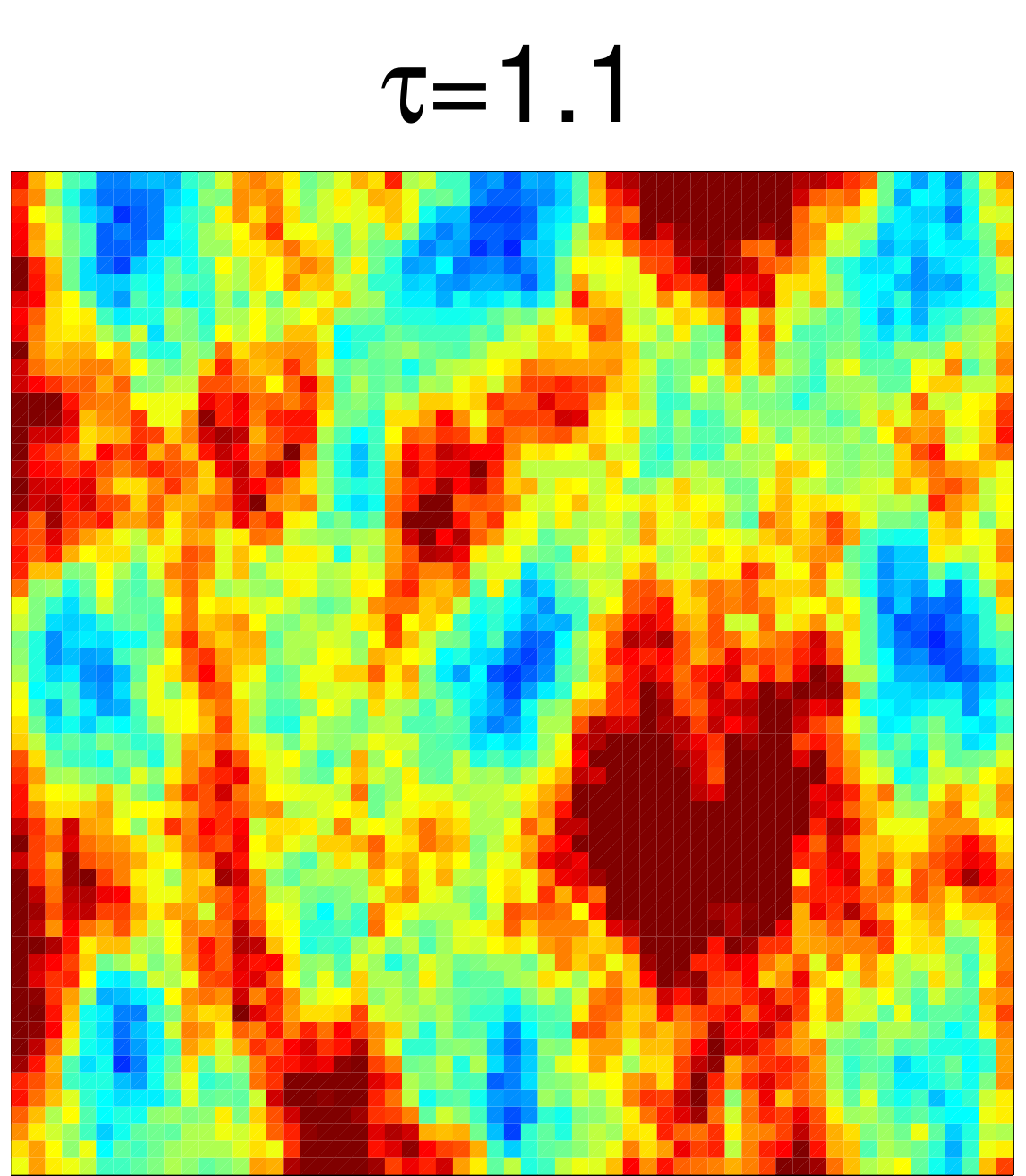}~
\includegraphics[scale=0.15]{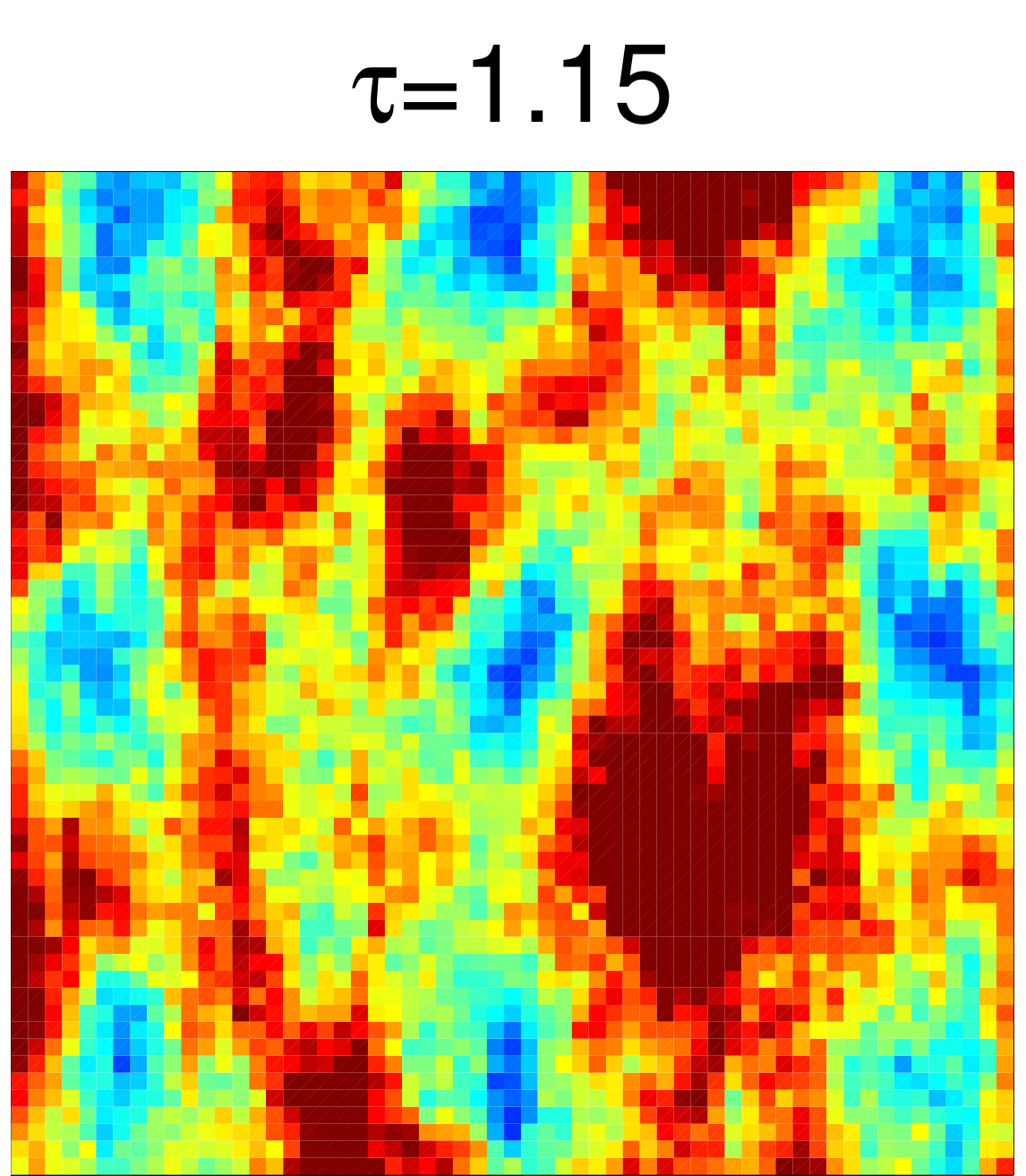}~
\includegraphics[scale=0.15]{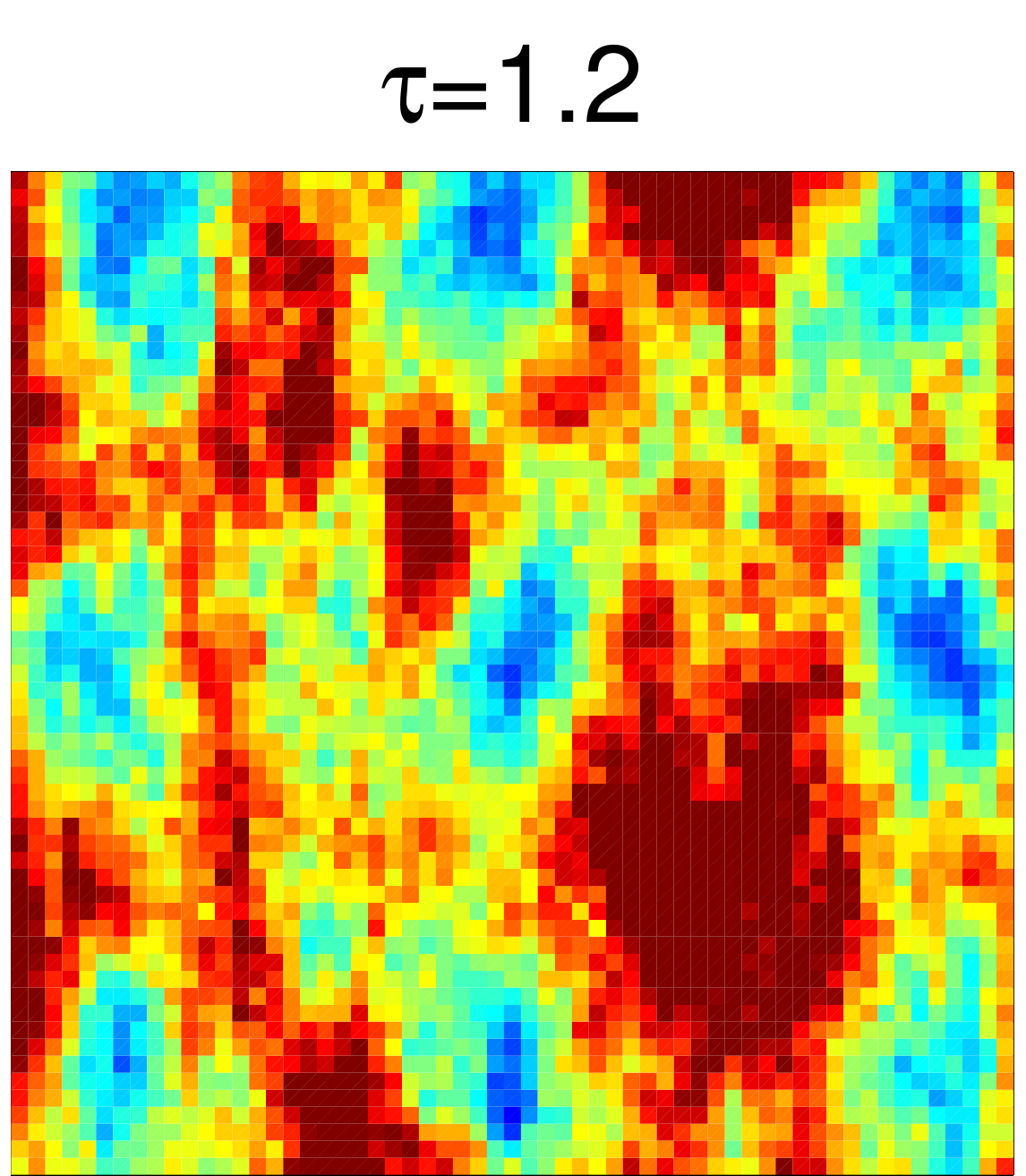}\\
\includegraphics[scale=0.225]{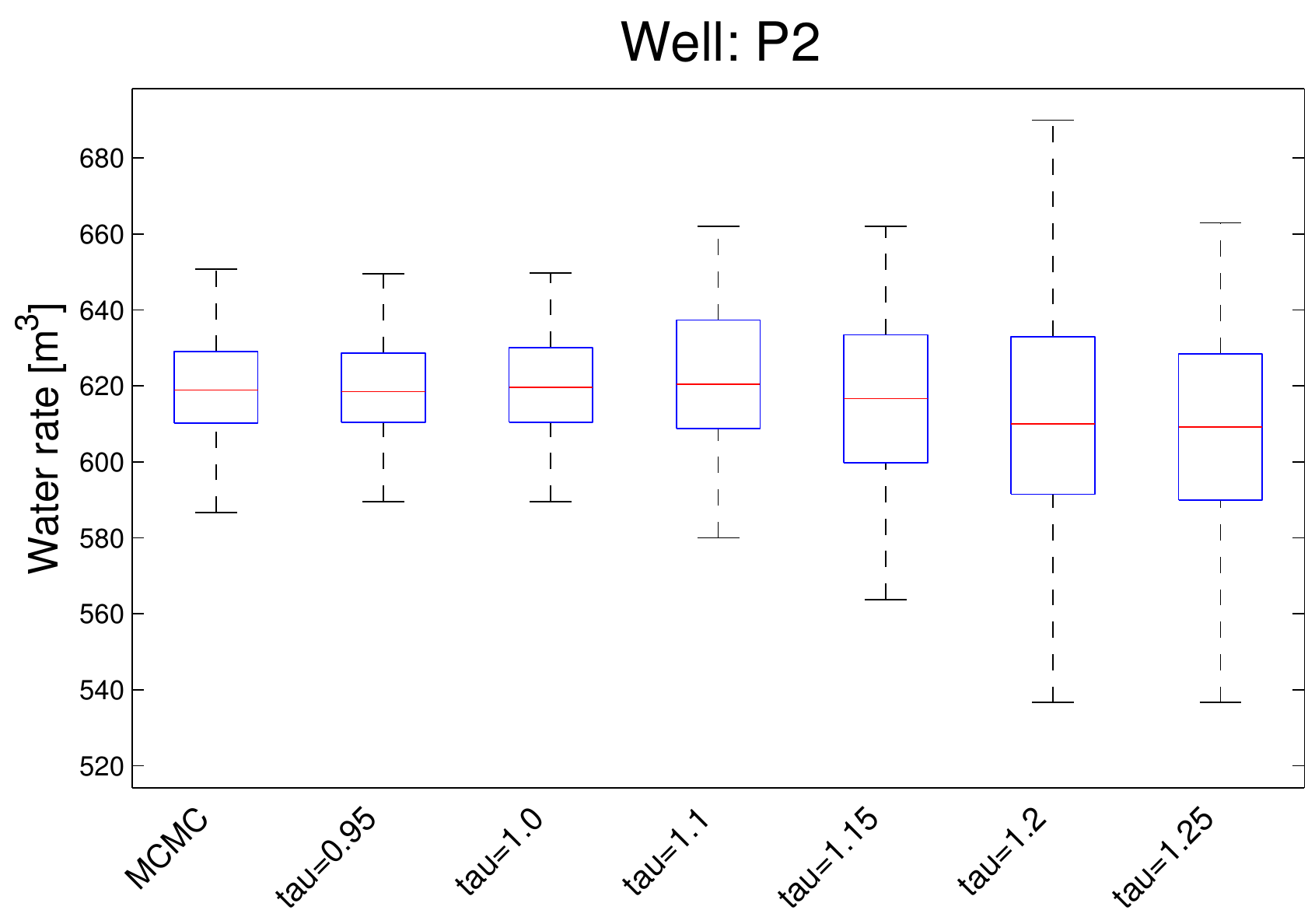}
\includegraphics[scale=0.225]{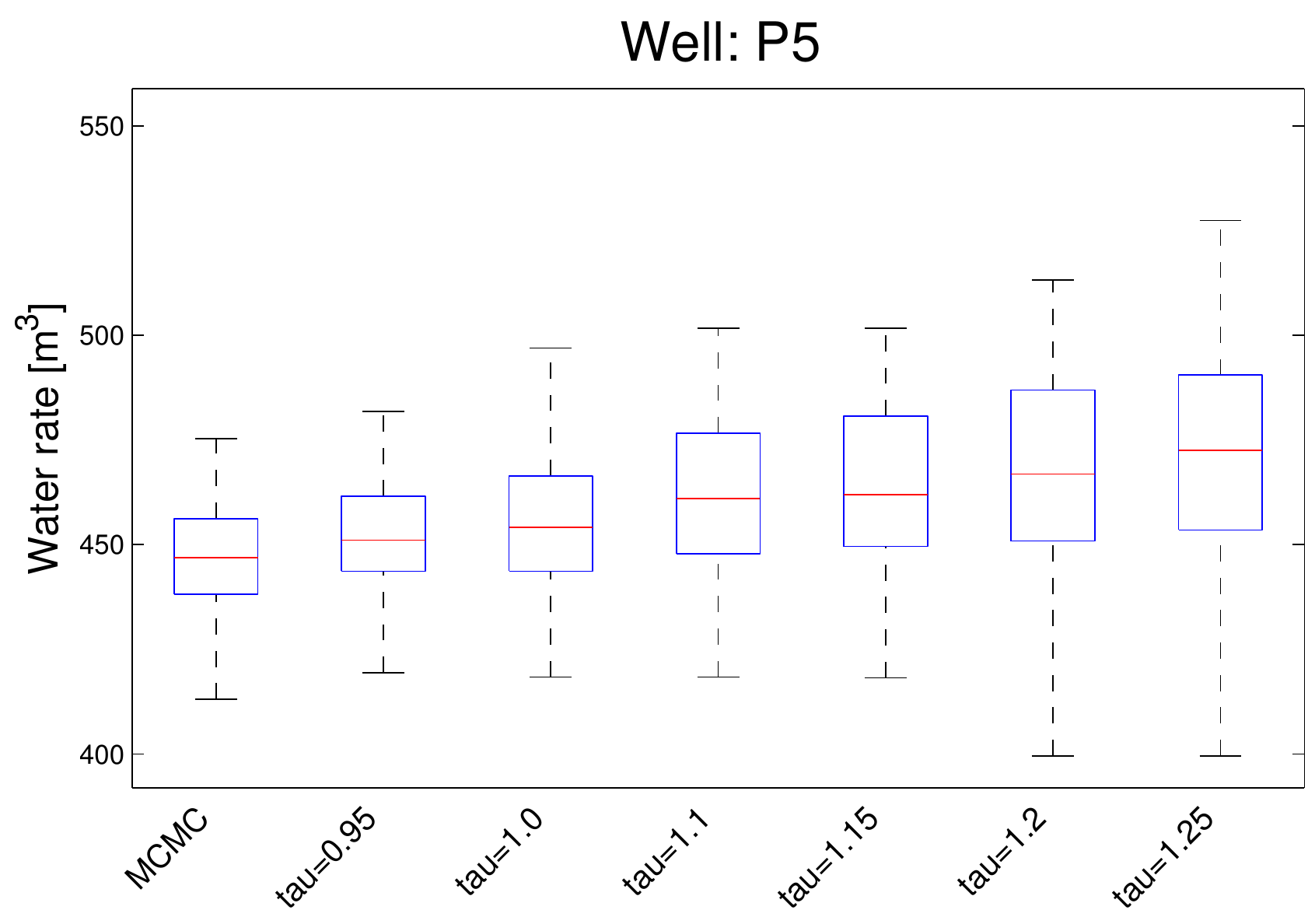}
\includegraphics[scale=0.225]{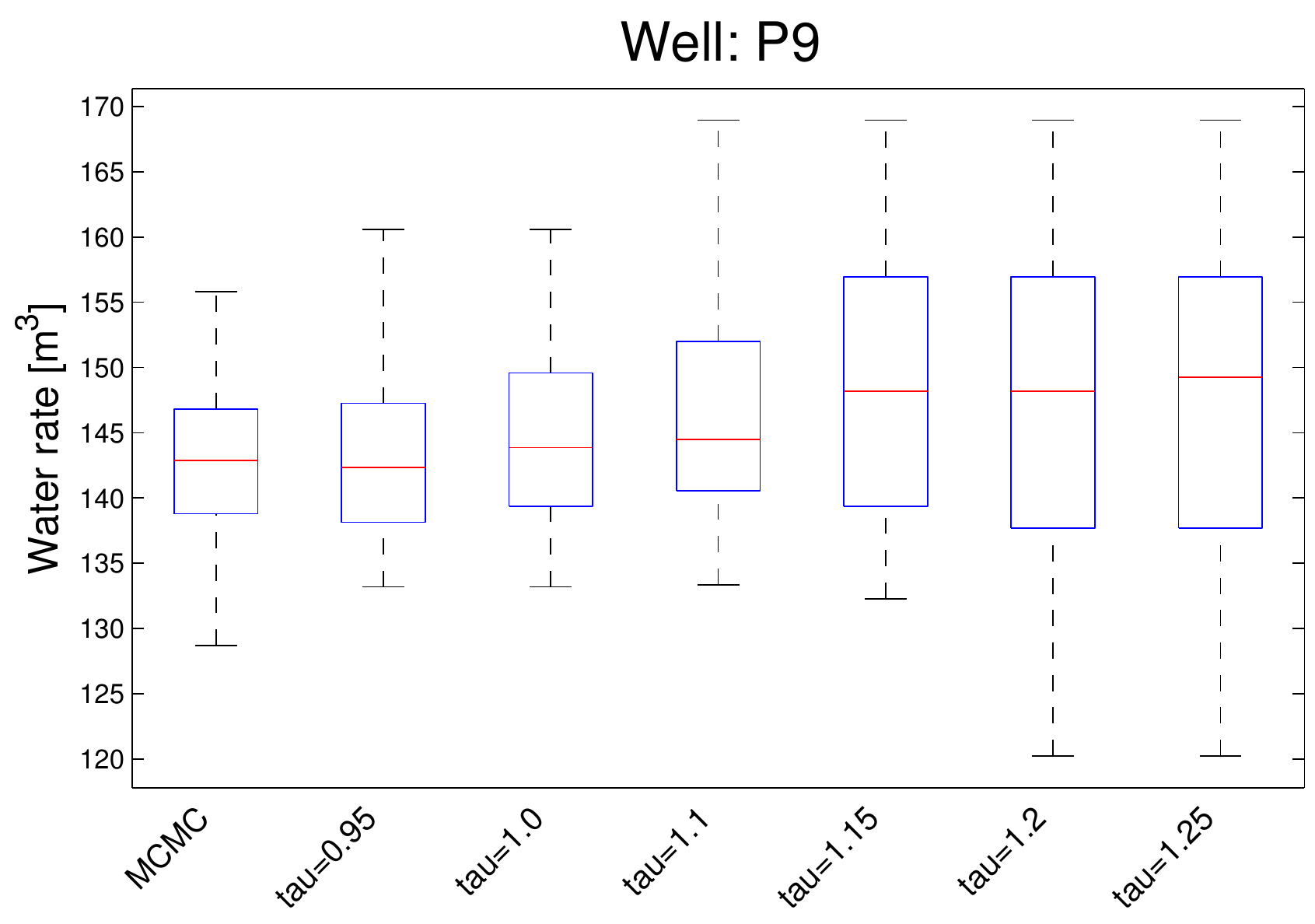}\\
\includegraphics[scale=0.225]{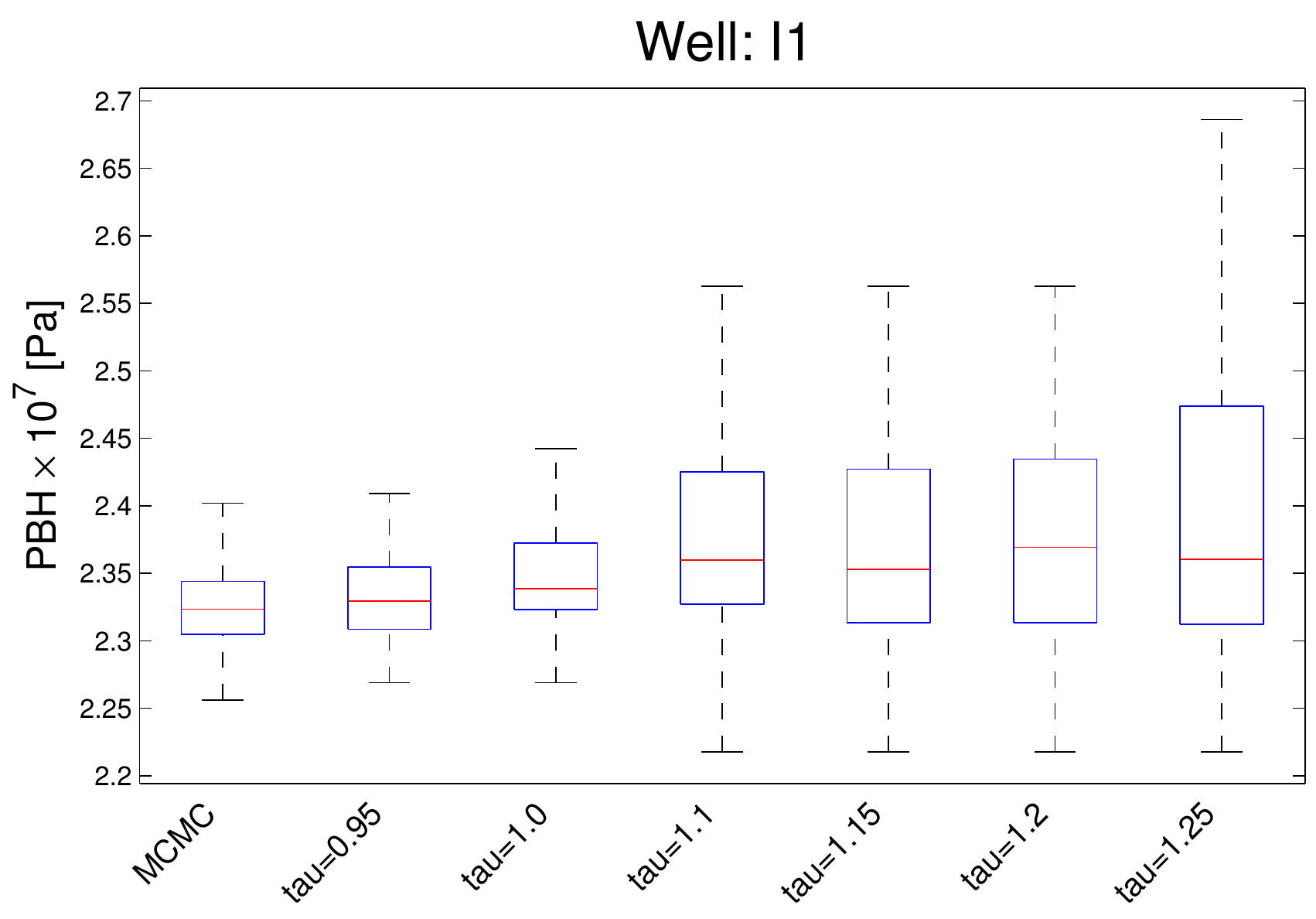}
\includegraphics[scale=0.225]{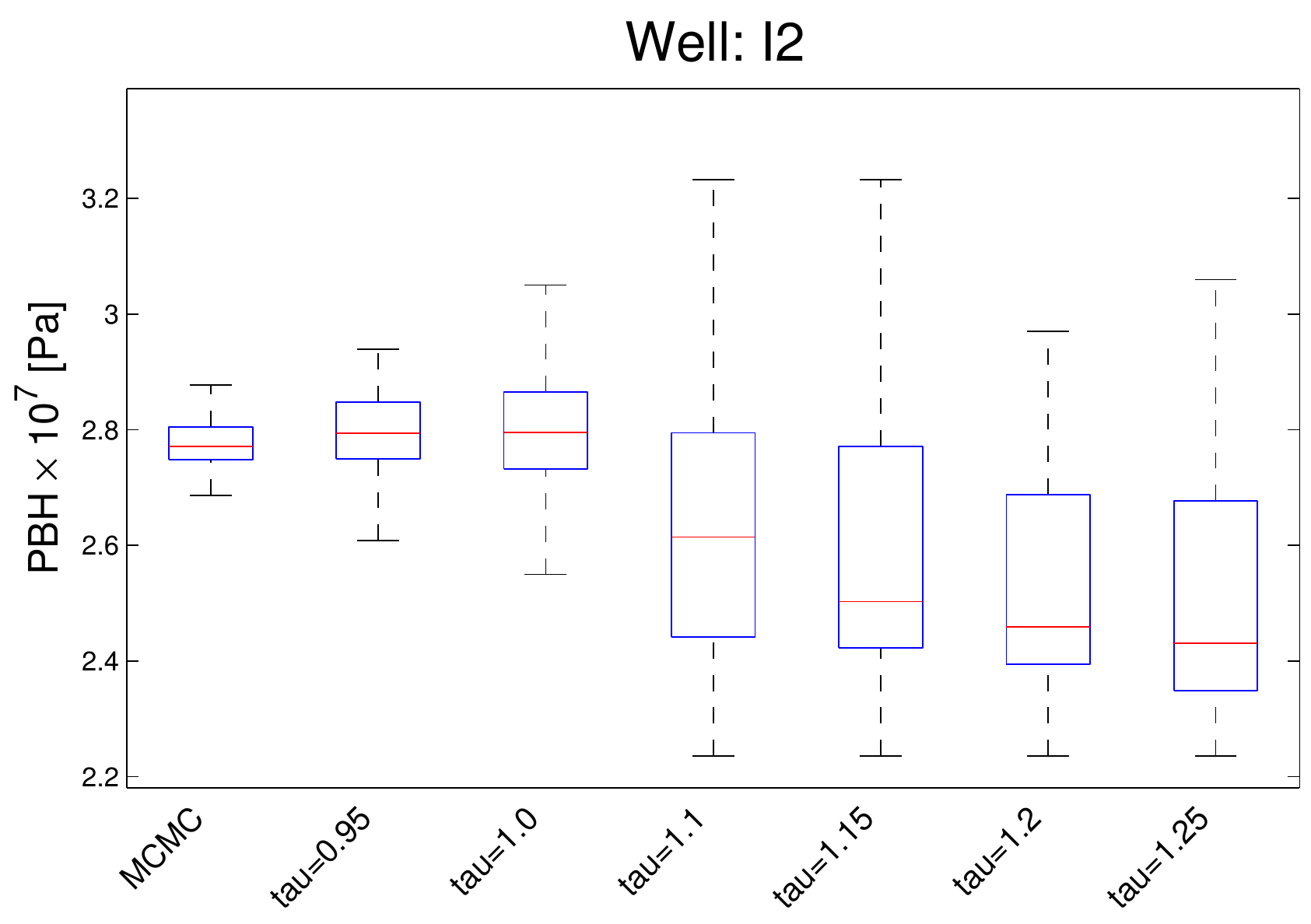}
\includegraphics[scale=0.225]{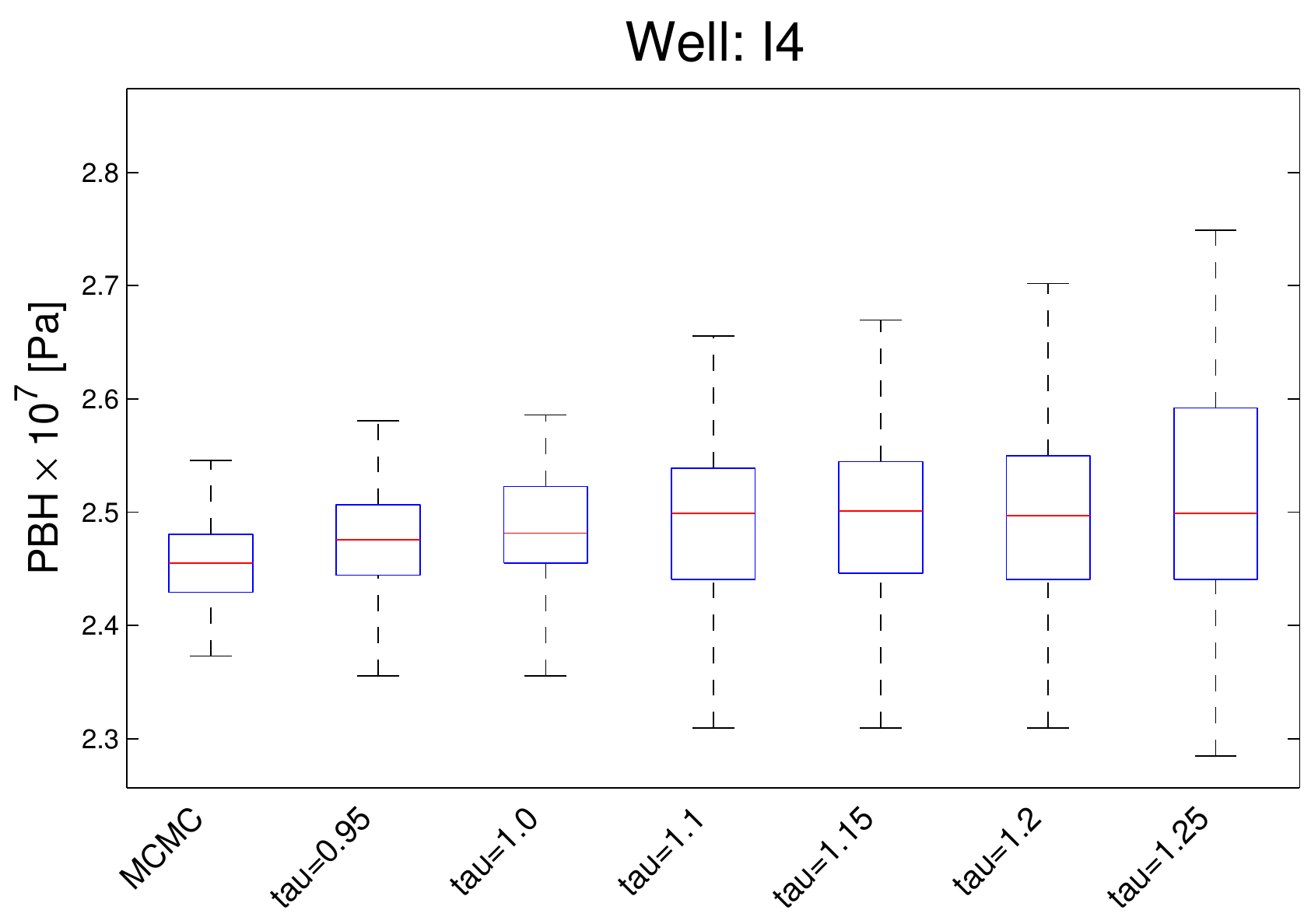}
\caption{Top row (mean of $\mu_A$) and Top-middle row (variance of $\mu_A$) from left to right: pcn-MCMC, IR-enLM approximations with $\rho=0.8$, $N_{e}=50$ and $\tau=0.95$, $\tau=1.0$, $\tau=1.1$, $\tau=1.15$, $\tau=1.2$. (color scales are the same as in the second row of Figure \ref{Figure1}). Middle-bottom and bottom row: Box plots of water rates from production wells $P_{2},P_{5}, P_{9}$ and PBH from injections wells $I_{1},I_{2},I_{4}$ after 6 years of water flood simulated from $\mu_{A}$ (with MCMC ) and the ensemble approximation with IR-enLM for $\rho=0.8$ and different choices of $\tau$}\label{Figure4}
\end{center}
\end{figure}

\begin{figure}
\begin{center}
\includegraphics[scale=0.15]{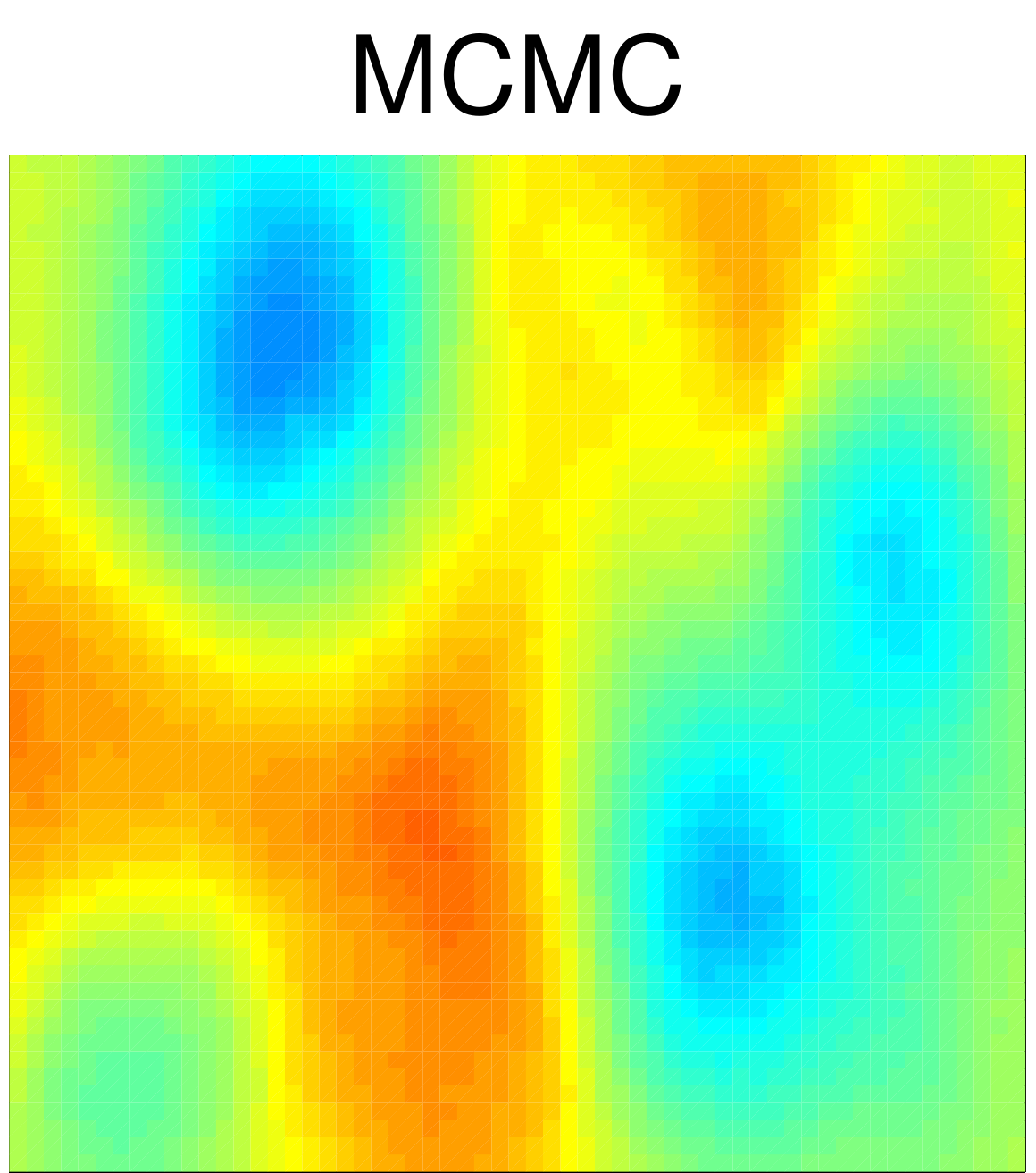}~
\includegraphics[scale=0.15]{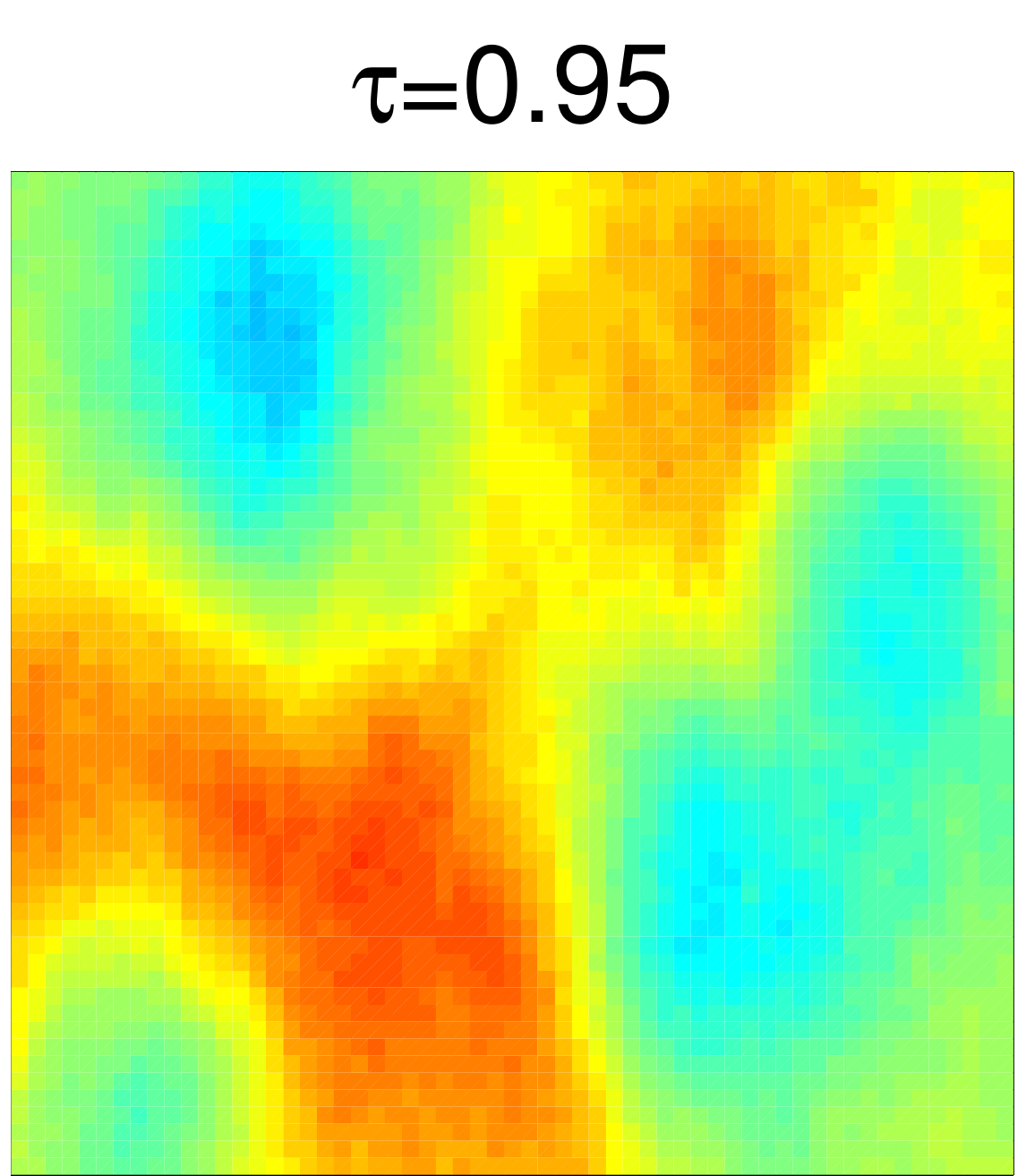}~
\includegraphics[scale=0.15]{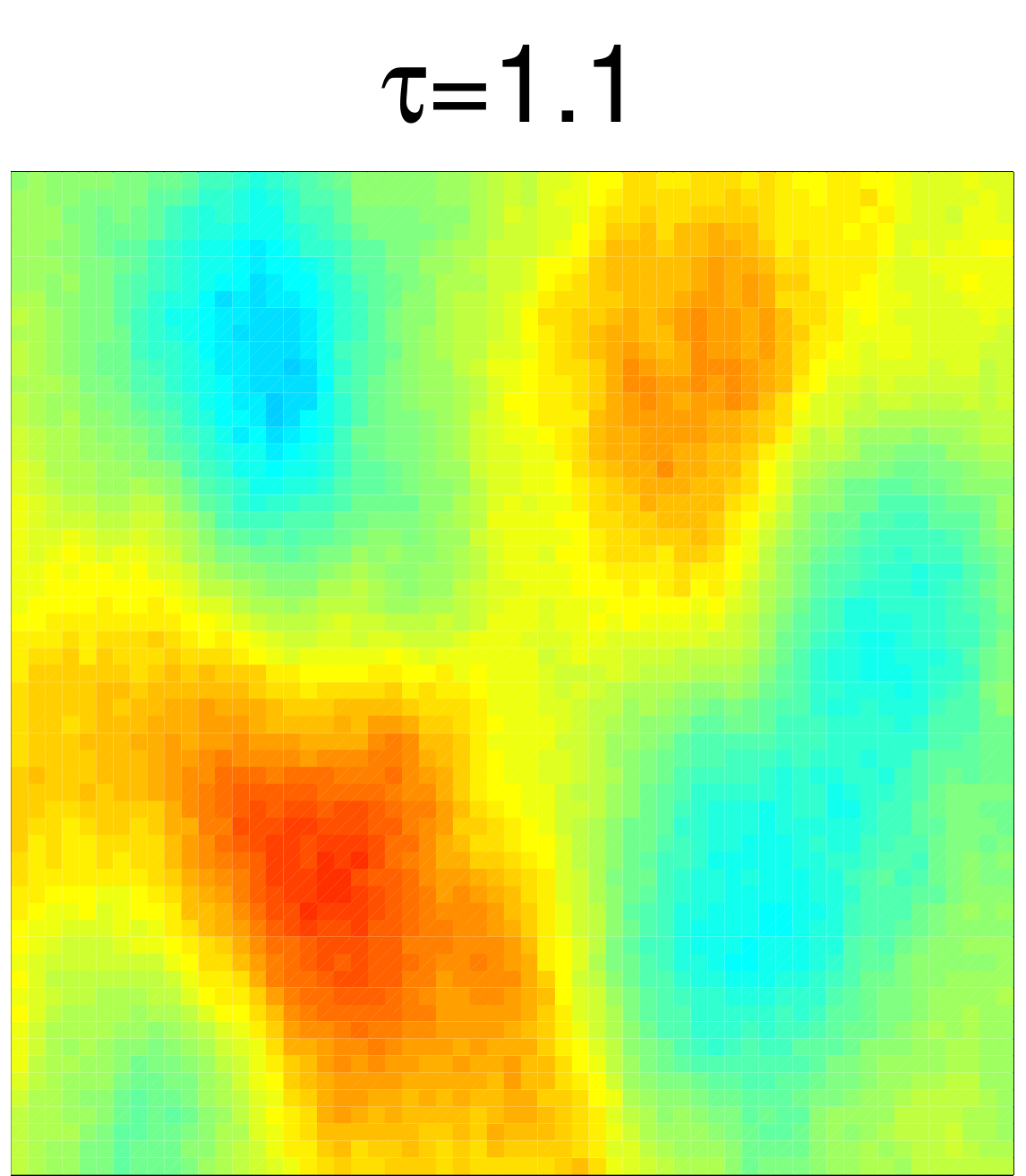}~
\includegraphics[scale=0.15]{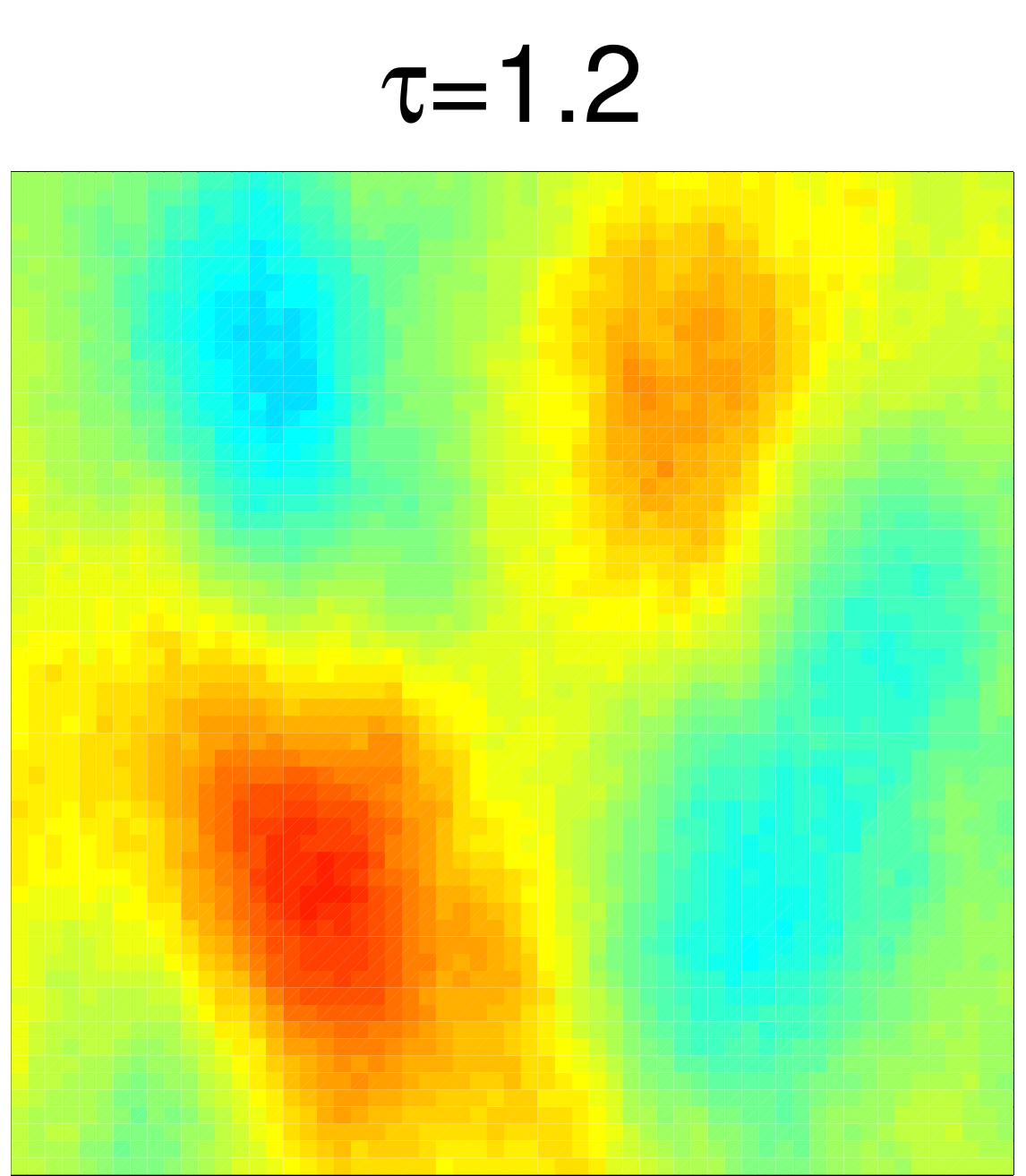}~
\includegraphics[scale=0.15]{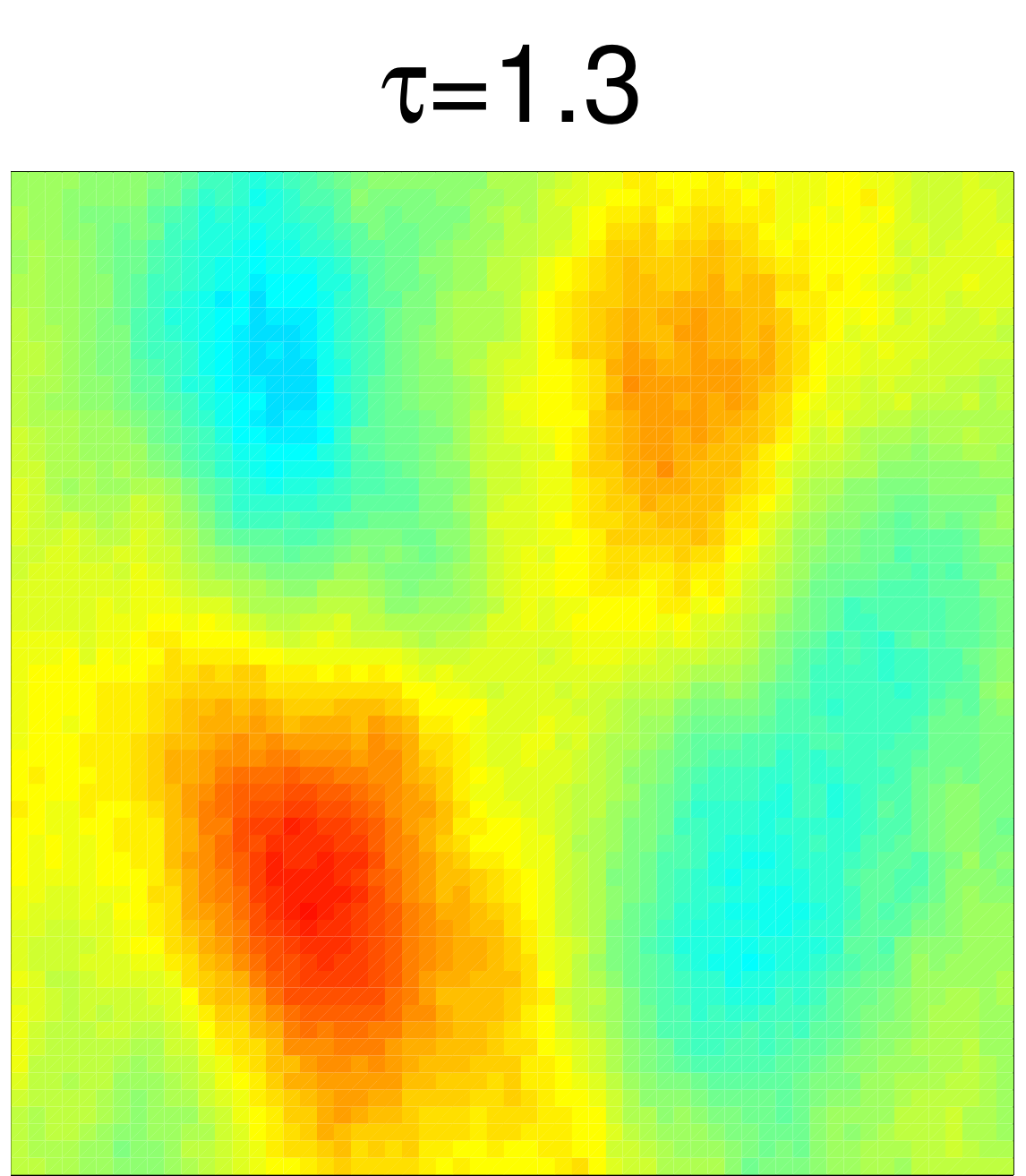}~
\includegraphics[scale=0.15]{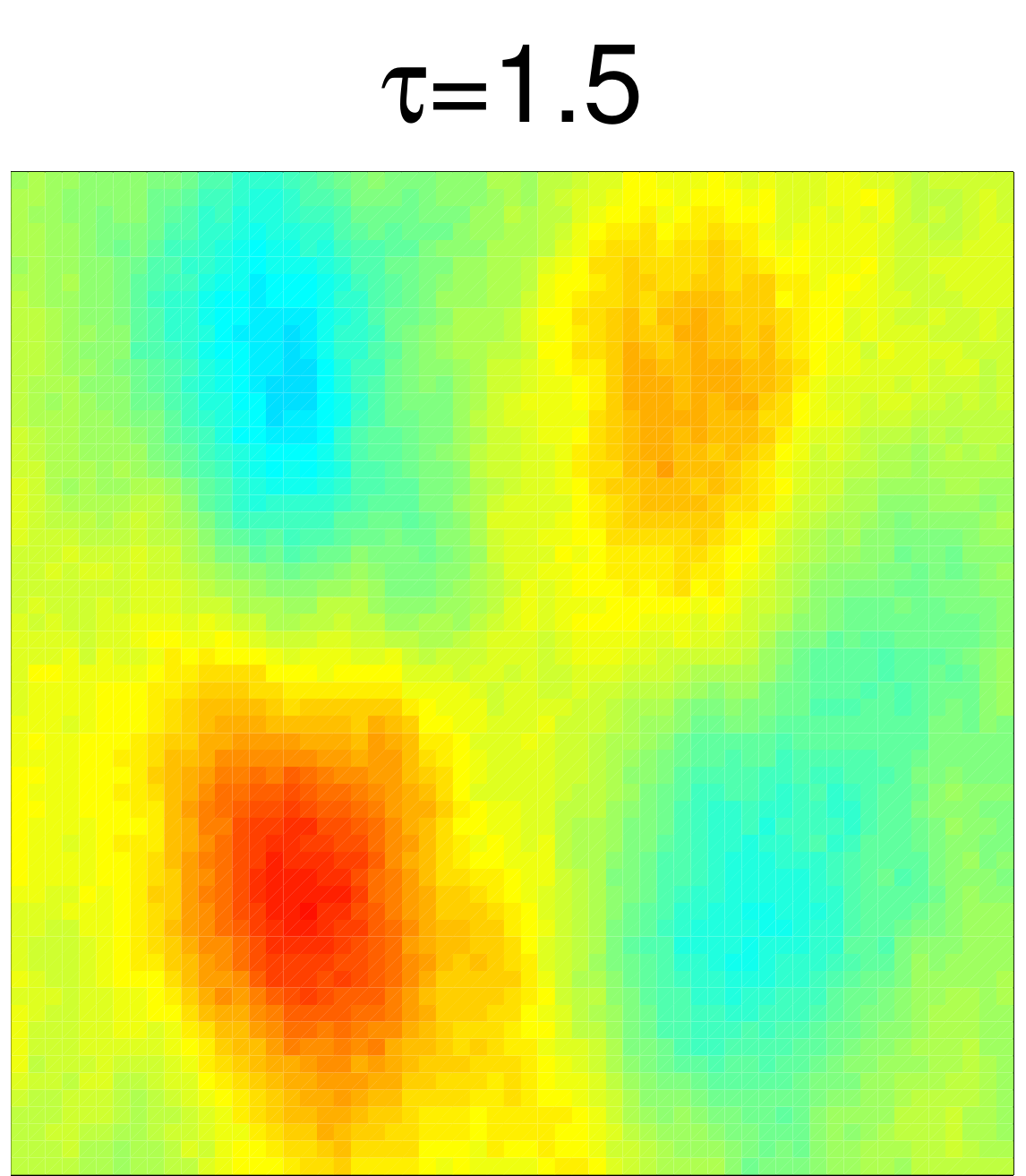}\\
\includegraphics[scale=0.15]{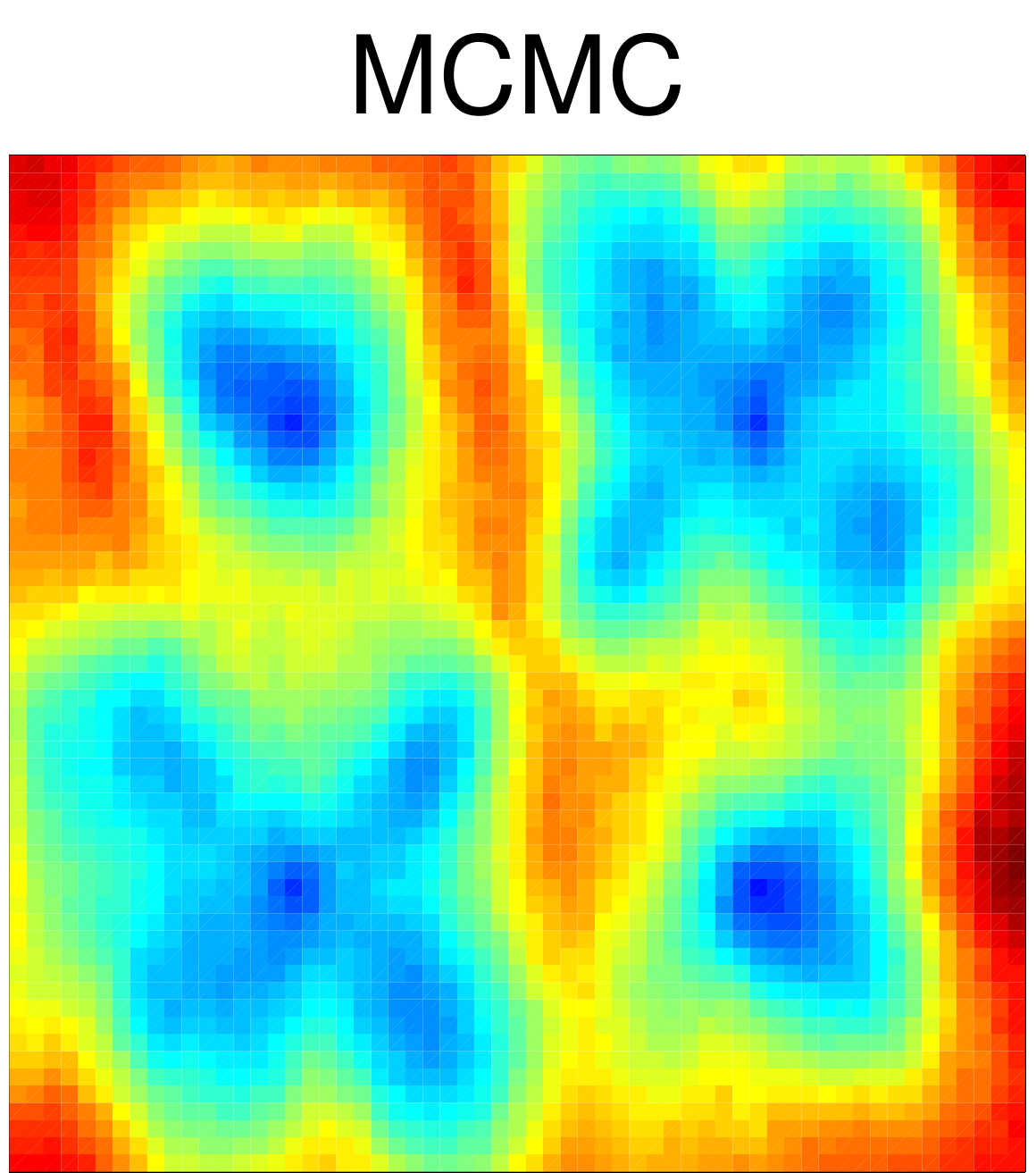}~
\includegraphics[scale=0.15]{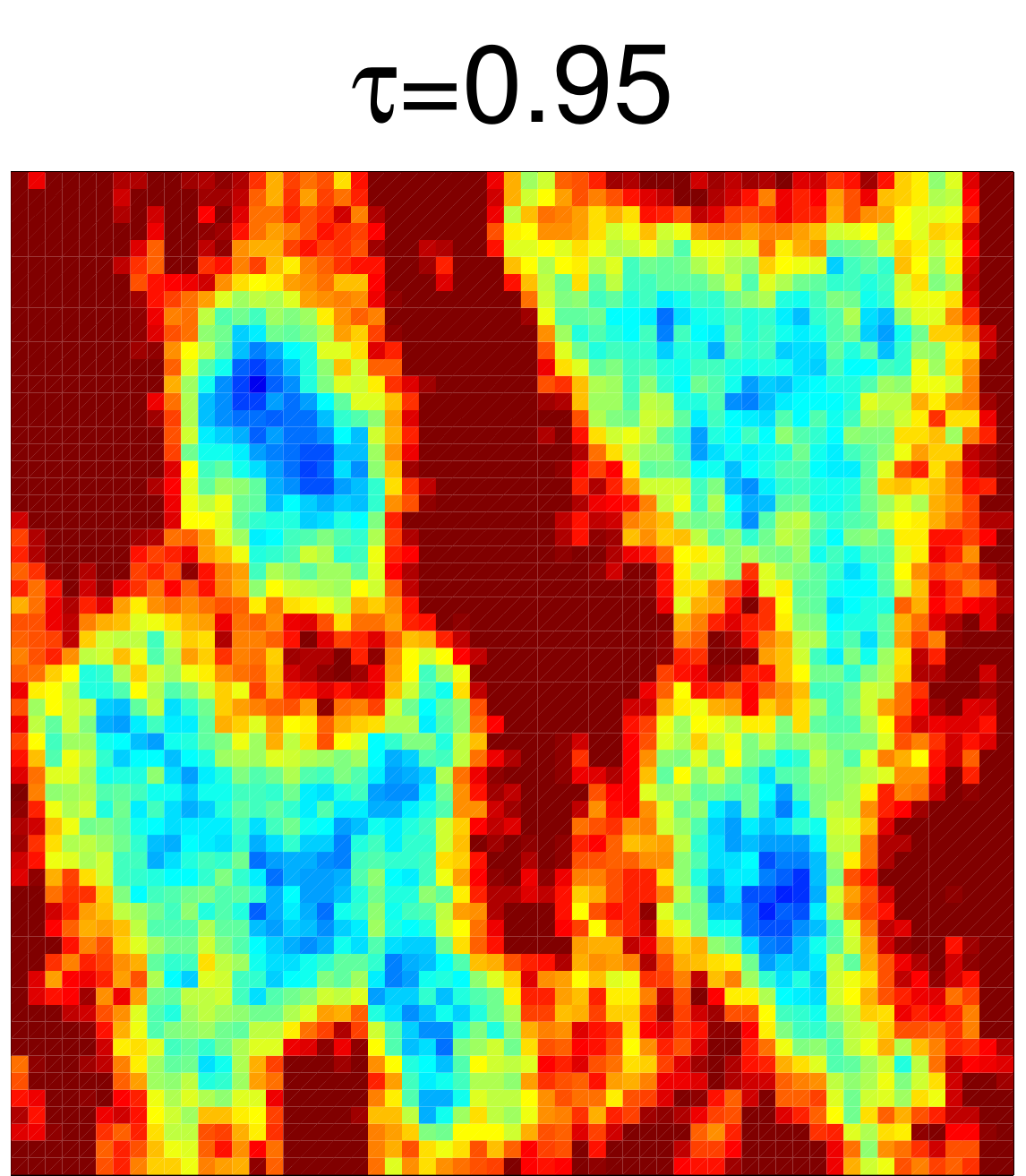}~
\includegraphics[scale=0.15]{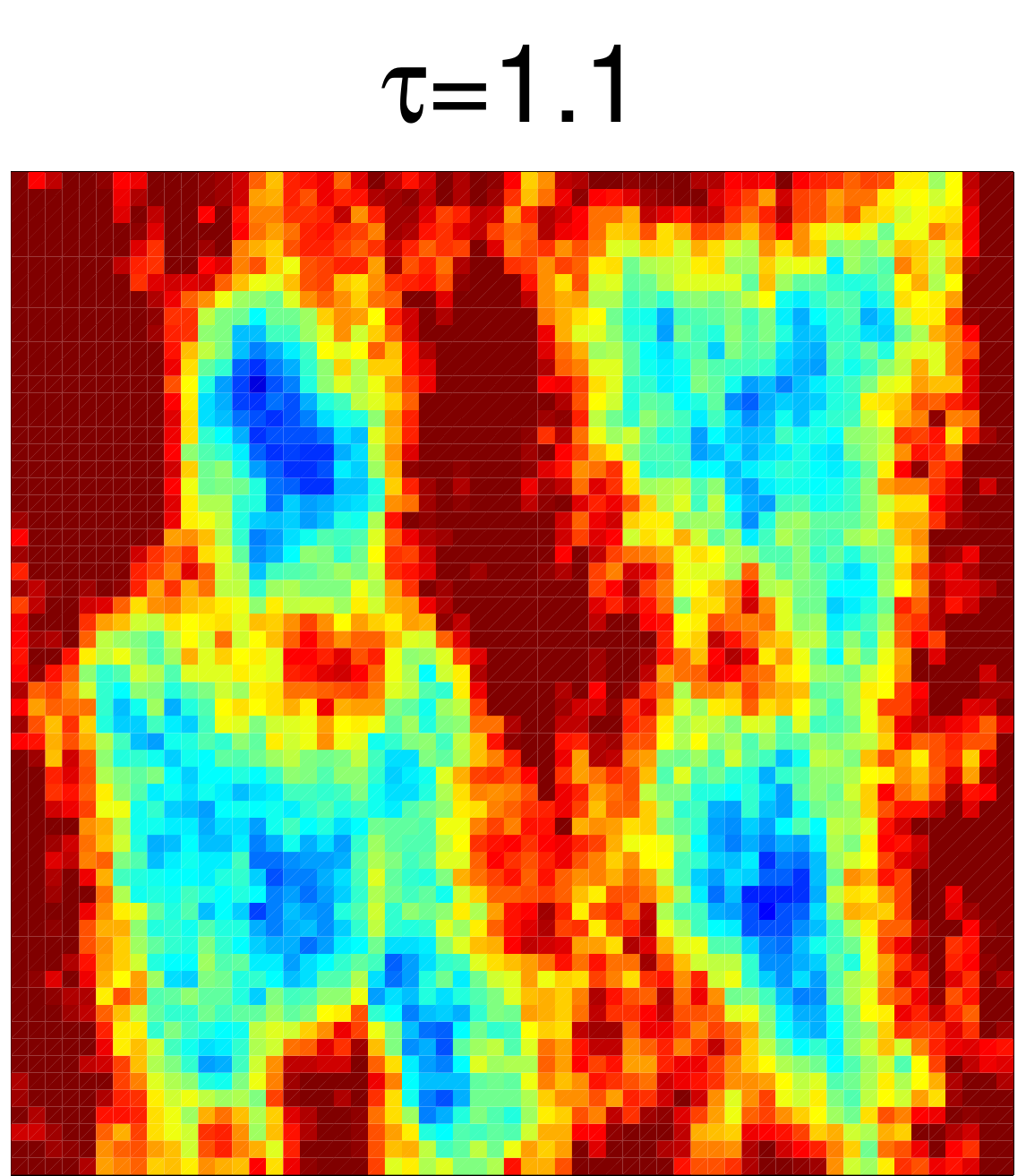}~
\includegraphics[scale=0.15]{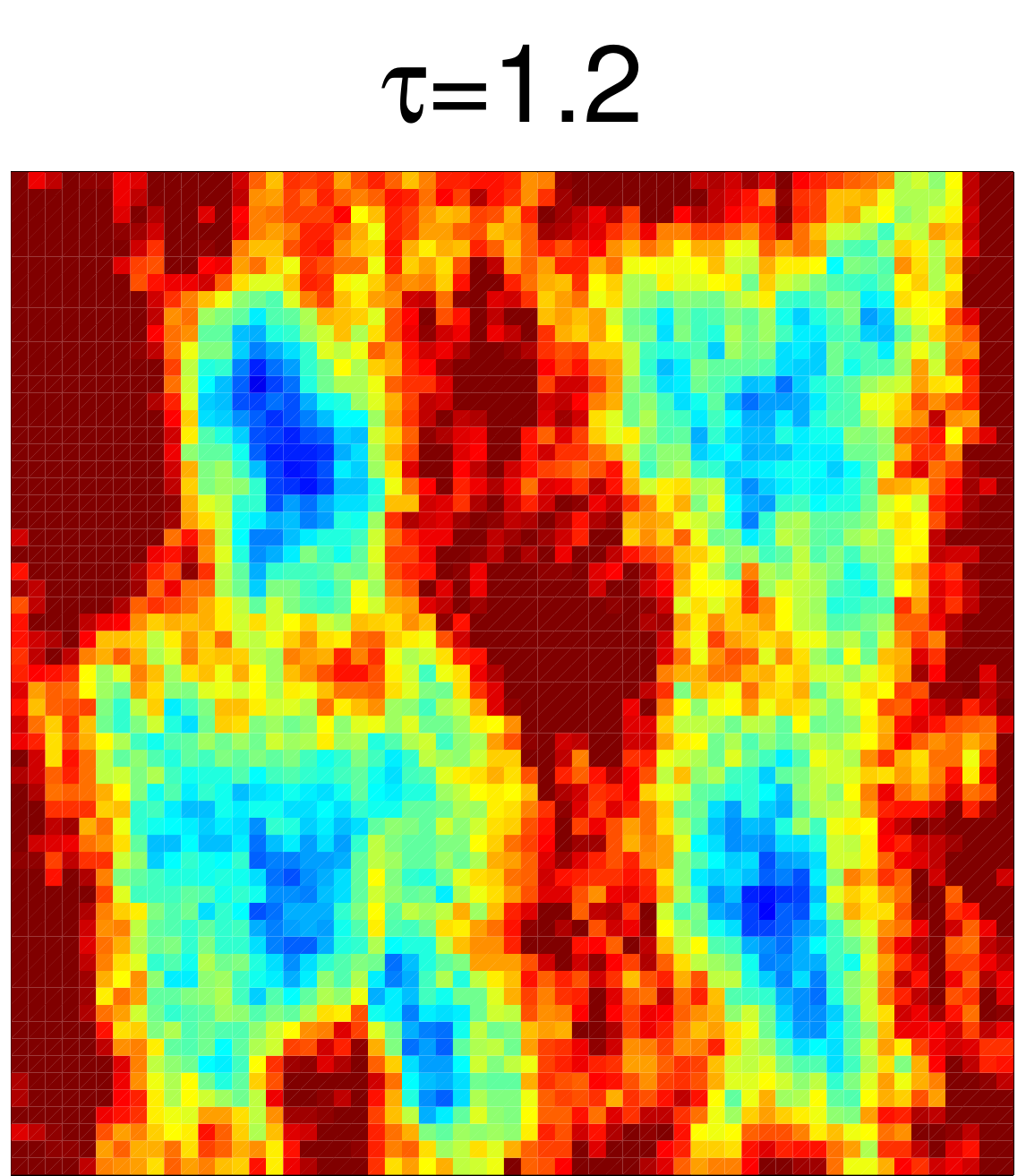}~
\includegraphics[scale=0.15]{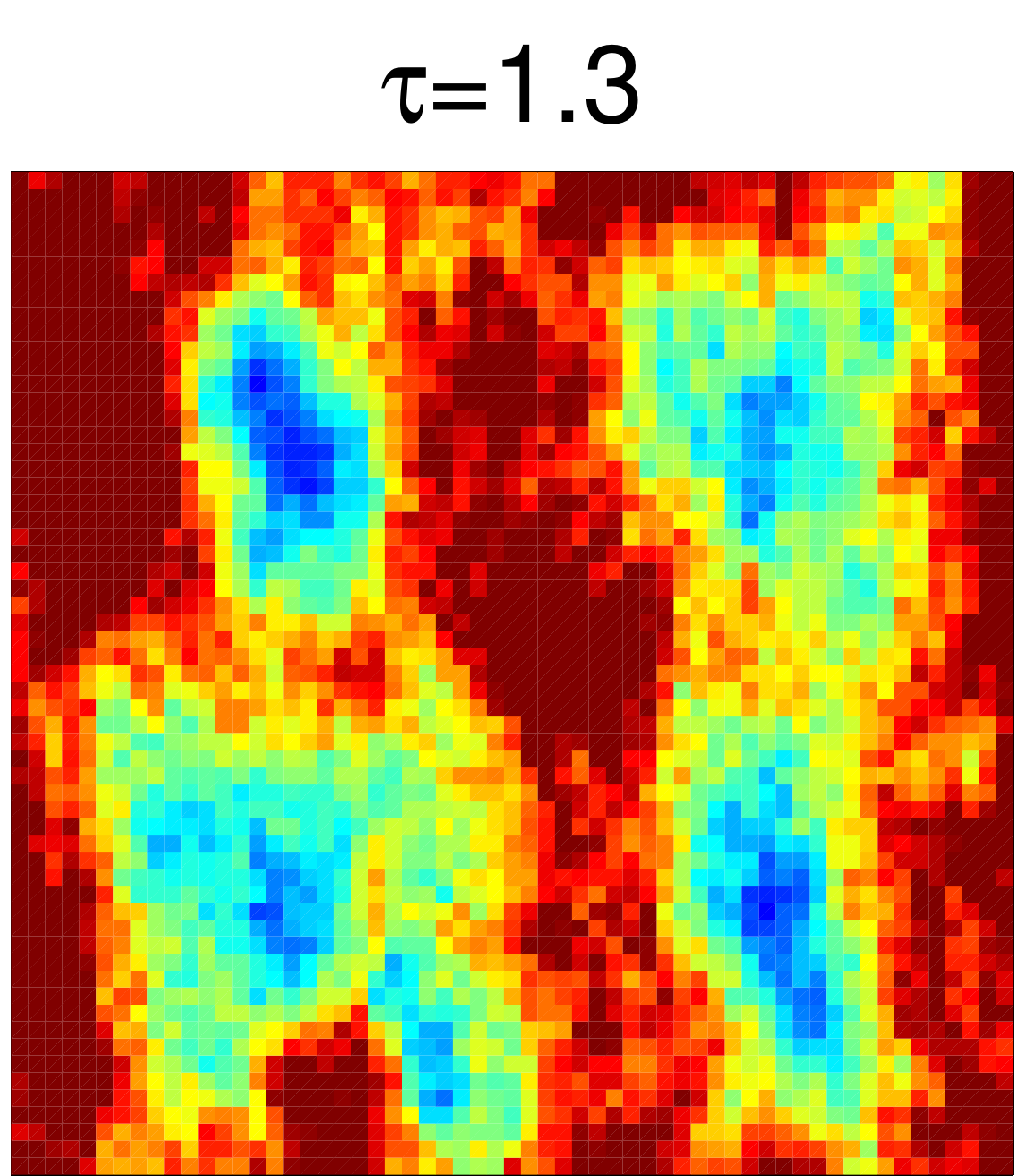}~
\includegraphics[scale=0.15]{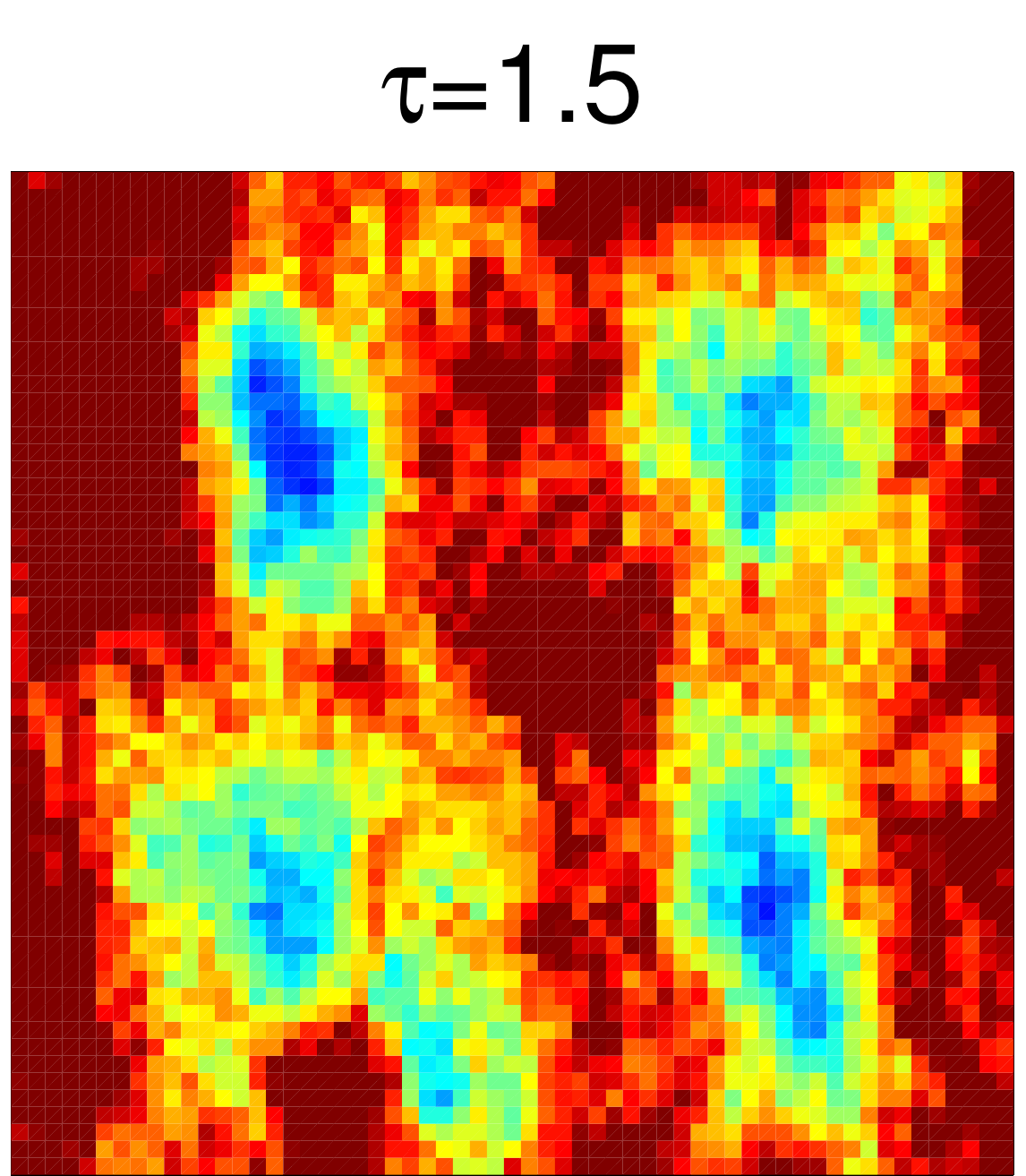}\\
\includegraphics[scale=0.225]{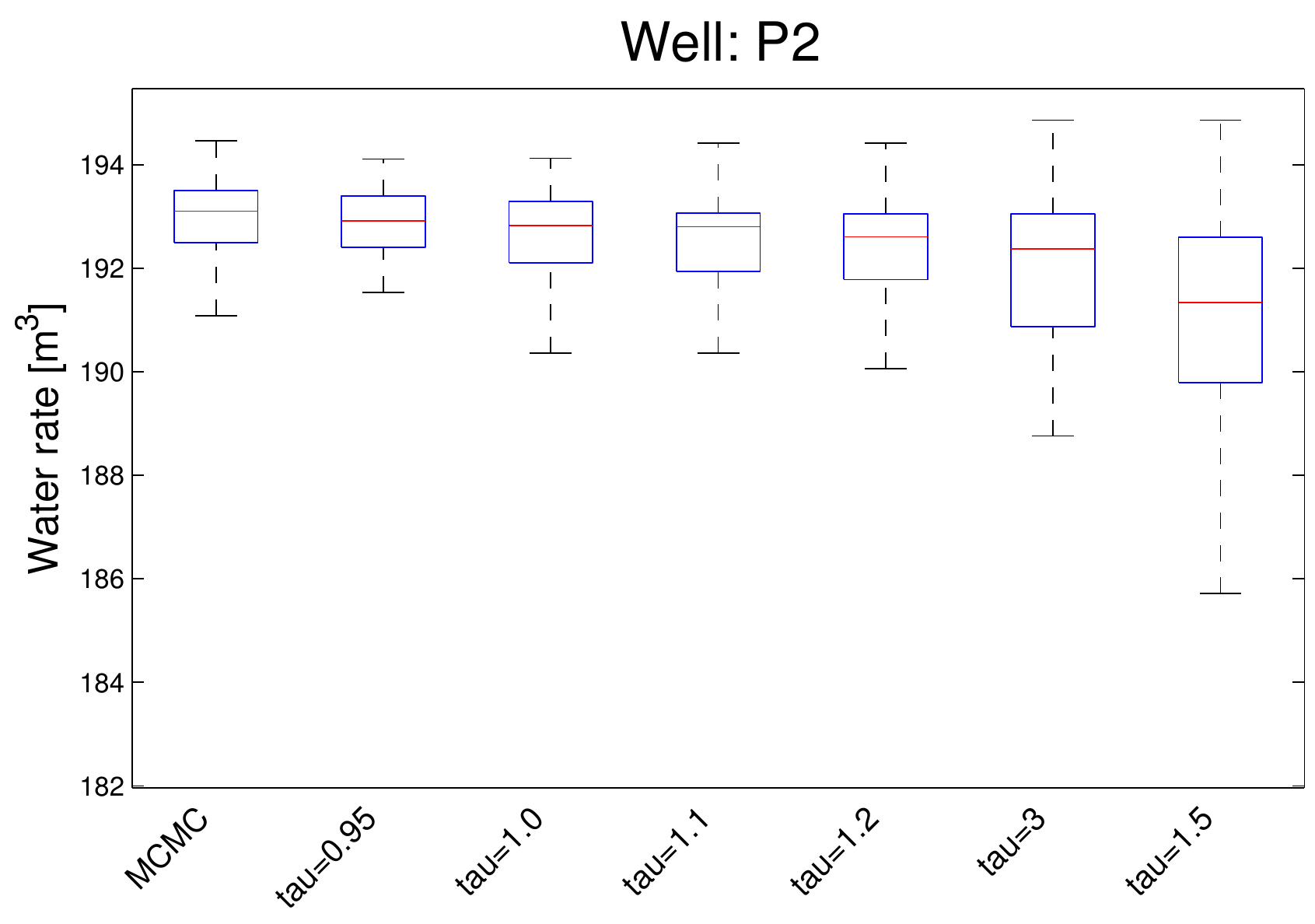}
\includegraphics[scale=0.225]{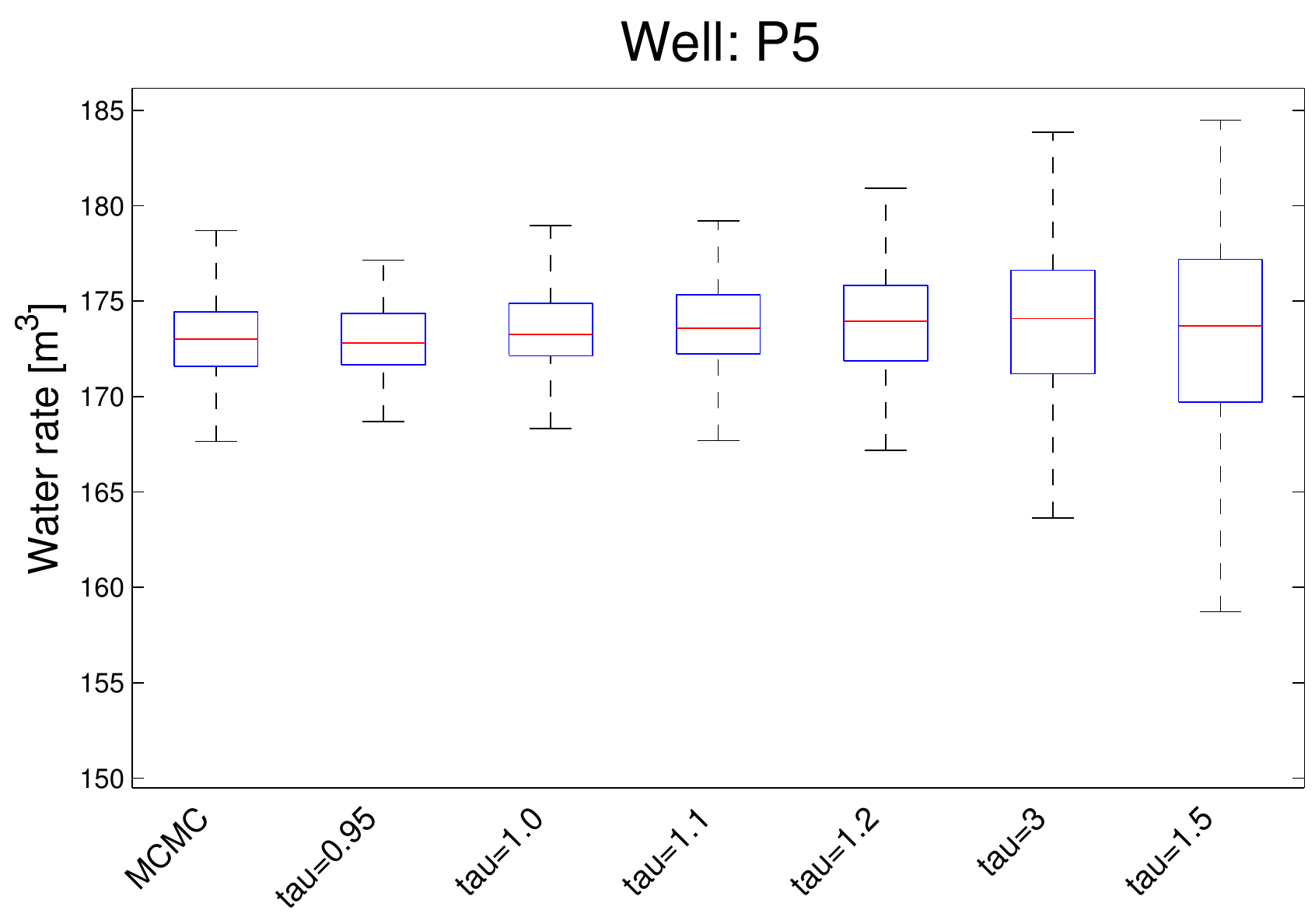}
\includegraphics[scale=0.225]{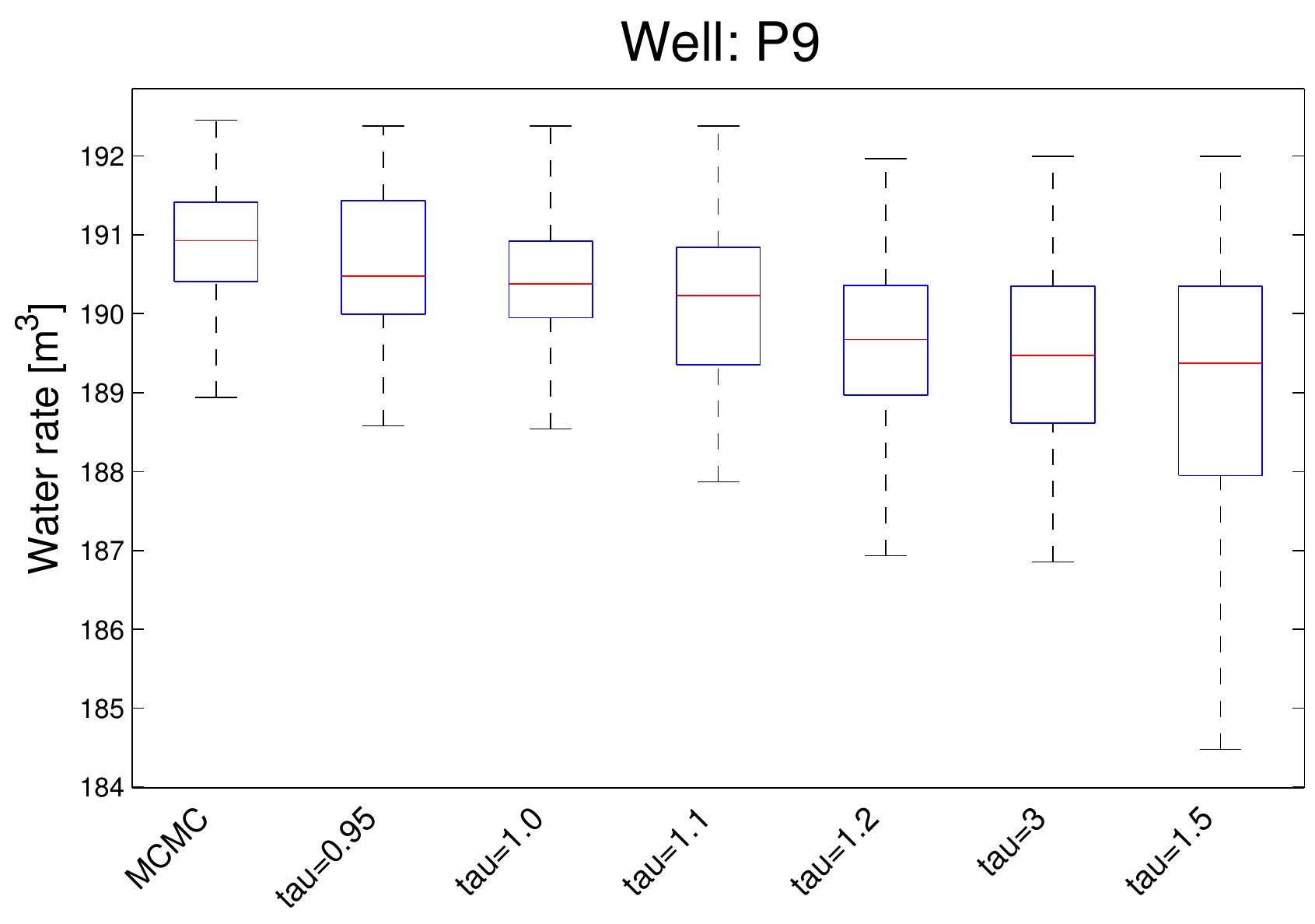}
\includegraphics[scale=0.225]{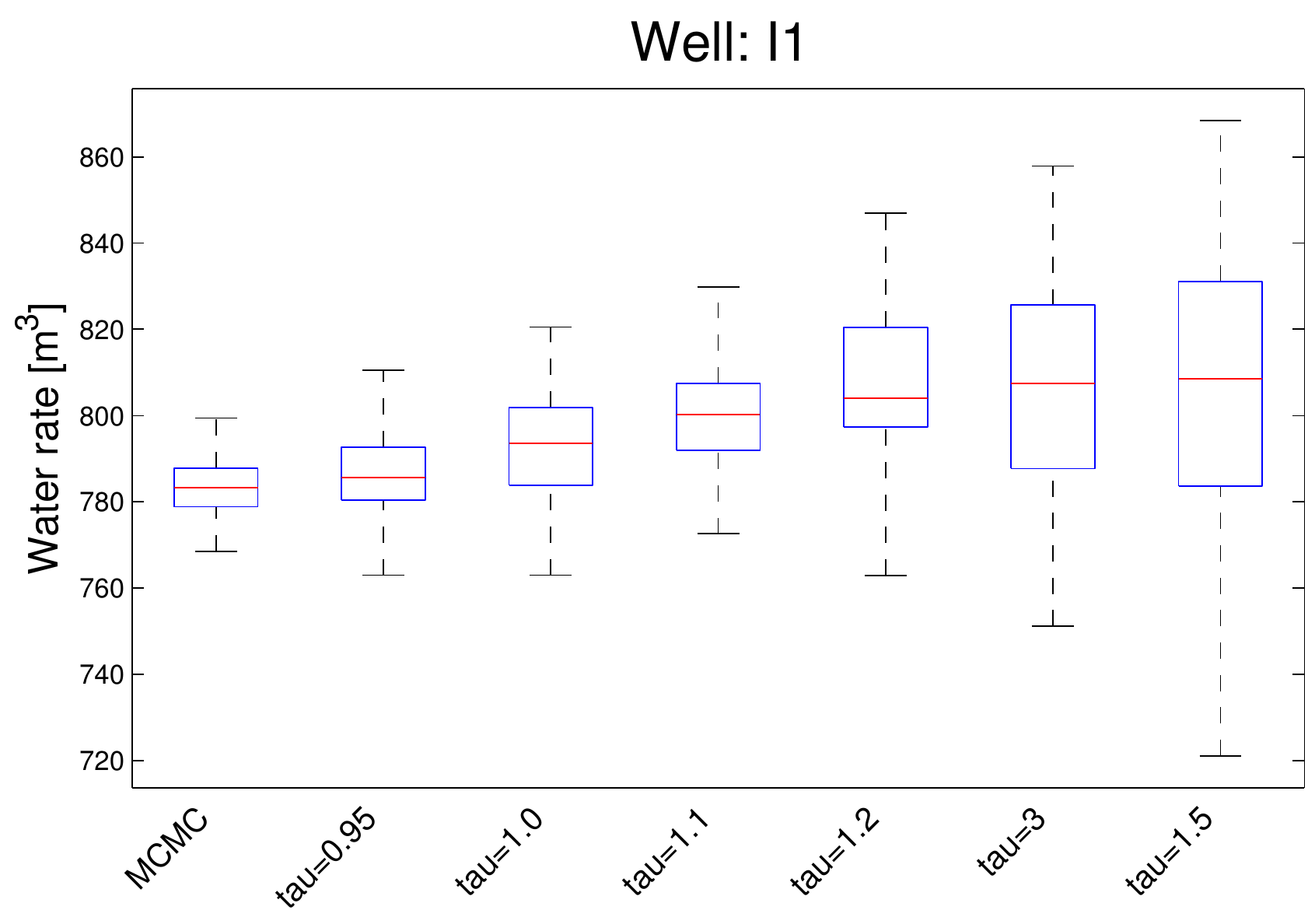}
\includegraphics[scale=0.225]{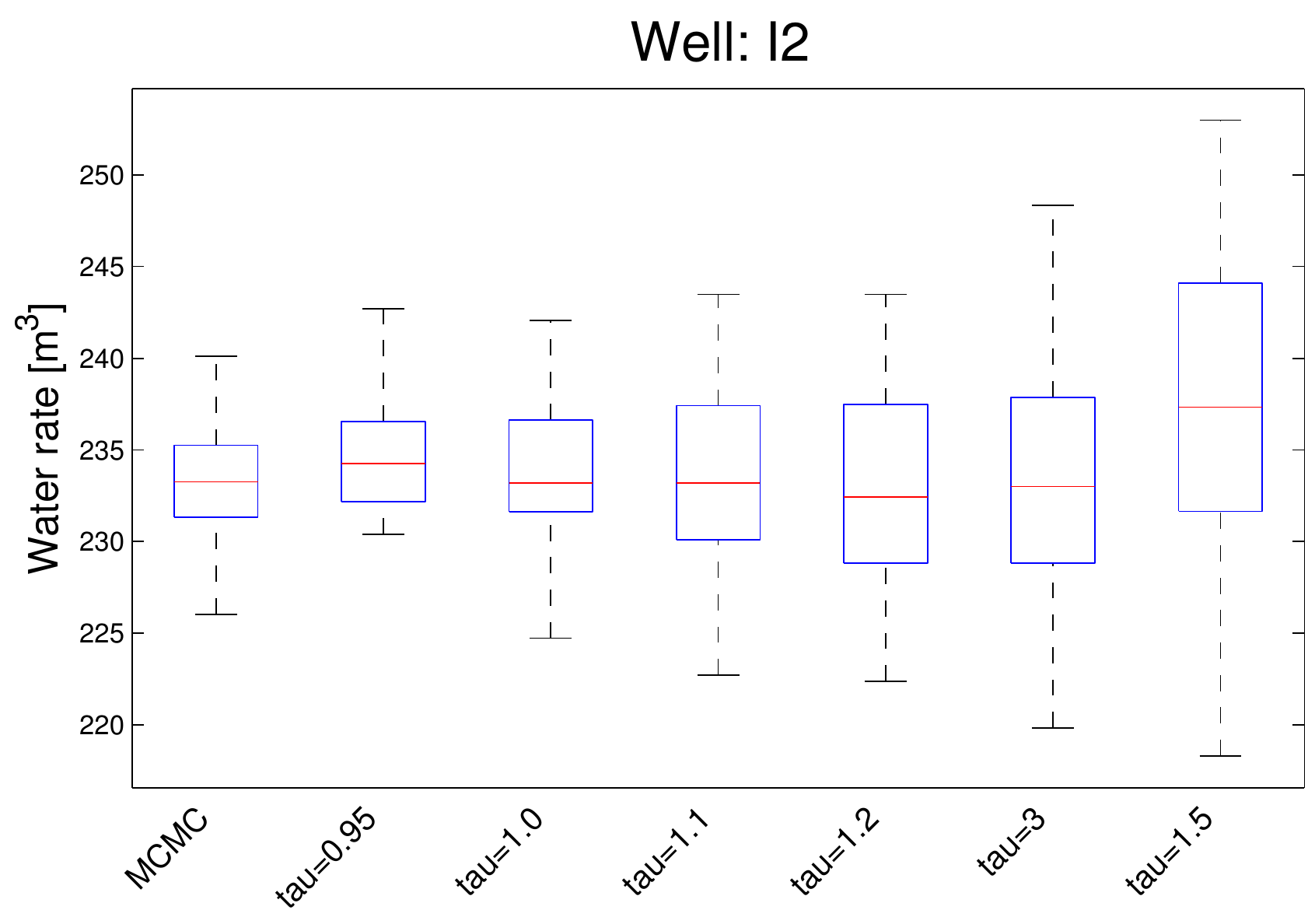}
\includegraphics[scale=0.225]{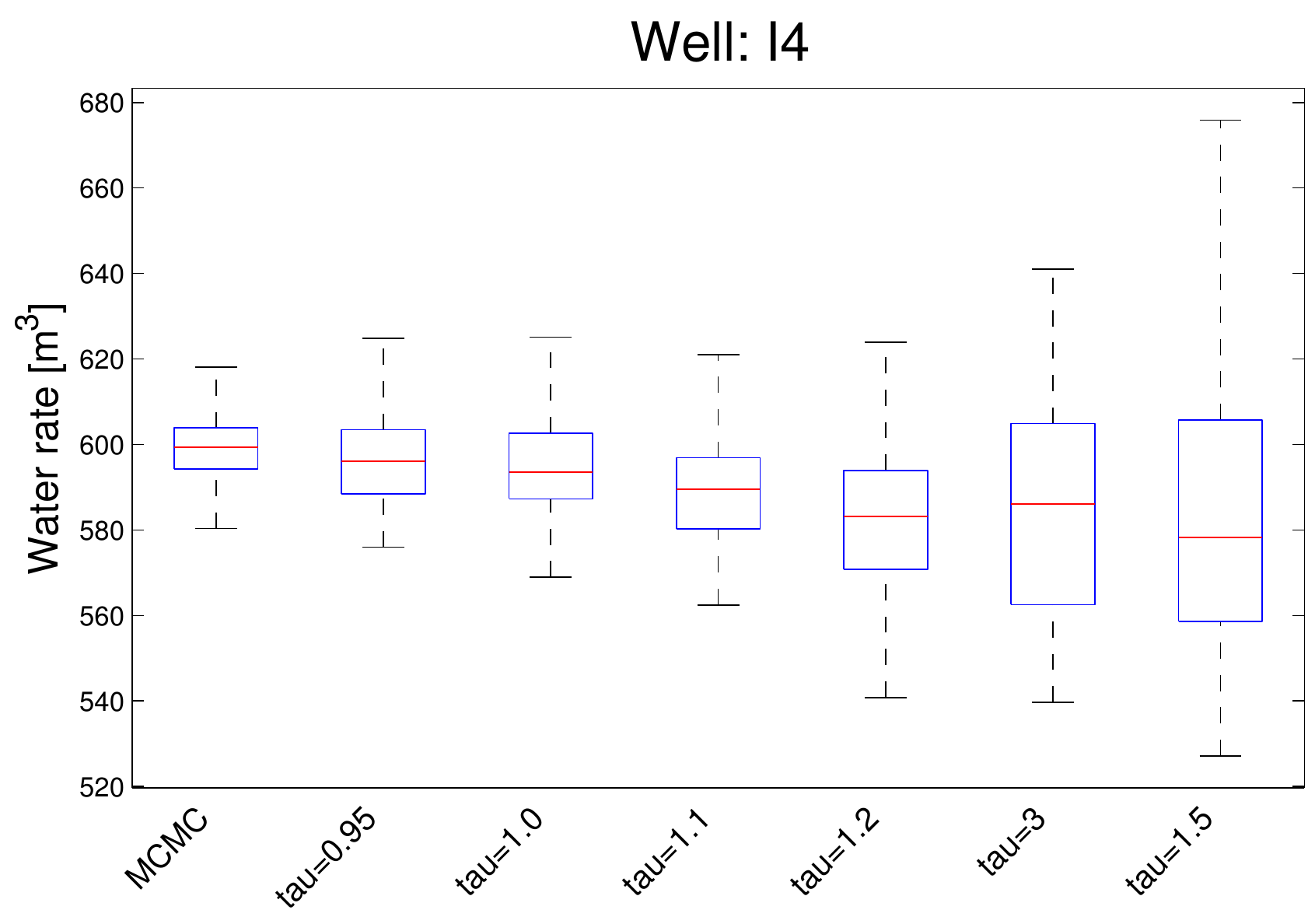}
\caption{Top row (mean of $\mu_B$) and Top-middle row (variance of $\mu_B$) from left to right: pcn-MCMC, IR-enLM approximations with $\rho=0.8$, $N_{e}=50$ and $\tau=0.95$, $\tau=1.1$, $\tau=1.2$, $\tau=1.3$, $\tau=1.5$ (color scales are the same as in the third row of Figure \ref{Figure1}). Middle-bottom and bottom row: Box plots of water rates from production wells $P_{2},P_{5}, P_{9}$ and PBH from injections wells $I_{1},I_{2},I_{4}$ after 6 years of water flood simulated from $\mu_{B}$ (with MCMC ) and the ensemble approximation with IR-enLM for $\rho=0.8$ and different choices of $\tau$} \label{Figure5}

\end{center}

\end{figure}

\subsection{IR-ES}\label{sec:numIR-ES}

Similar to the previous subsection, our objective is now to assess the performance of IR-ES at capturing the Bayesian posterior. More importantly, we aim at understanding the role of the tunable parameters $\rho$, $\tau$ and $M_{ES}$ in the proposed method. In Table \ref{Table2A} we display the approximations of the posterior mean and variance with IR-ES for different ensemble sizes $N_{e}$ and values of $\tau$ for fixed parameters $\rho=0.8$ and $M_{ES}=10$. The value $\tau$ highlighted in blue corresponds to the cutoff $\tau=1/\rho$. Recall from Section \ref{sec:IR-ES} that the tunable parameter $\rho$ appears when we apply the discrepancy principle (\ref{eq:1.38}) for the selection of the regularization parameter $\alpha$. Moreover, the parameter $\tau$ in (\ref{eq:2.2}) should satisfy $\tau>1/\rho>1$ for consistency with assumption (\ref{eq:1.42}) as discussed in Remark \ref{remaIR}. From Table \ref{Table2A} we observe that, for the approximation corresponding to the value $\tau=1/\rho$ (in blue), the accuracy in the approximation of the mean and variance increases with ensemble size. In contrast to IR-enLM where the ensemble size did not play a fundamental role, here we note that the ensemble size has a substantial influence on the accuracy in the ensemble approximation of the posterior obtained with IR-ES. This, of course, follows from the fact that the (Tikhonov) regularization of IR-ES is defined in terms a norm that is weighted by covariance constructed from the ensemble (see expression (\ref{eq:1.45})). Indeed, Table \ref{Table2A} shows that for small ensembles  ($N_{e}\leq 50$)  the relative errors in the mean and the variance starts increasing before the cut off value $\tau=1/\rho$ is reached. Note that it is only for larger ensembles ($N_{e}\ge 75$) where we observe the reduction of the error in the mean and variance as we reduce $\tau$ to the value of approximately $\tau=1/\rho$. Below this critical value, the errors of the ensemble approximation increase. For $\rho=0.8$, $M_{ES}=10$ and $N_{e}=50$ fixed, in Figure \ref{Figure6} we compare the mean (top row) and variance (top-middle row) generated with IR-ES for approximating $\mu_{A}$ with different choices of $\tau$. In Figure \ref{Figure6} (bottom-middle and bottom rows) we show box plots of relevant quantities at some wells. Similarly, in Figure \ref{Figure7} we present analogous results for the approximation of $\mu_{B}$.

When the ensemble is sufficiently large ($N_{e}\ge 75$ for the present case) we note that better approximations are obtained when, for a given $\tau$, the ensemble  approximation is stopped according to (\ref{eq:2.2}) with $\tau=1/\rho$. Recall that the parameter $\rho$ in IR-ES controls each update of the ensemble similar to IR-enLM. A larger $\rho$ in  (\ref{eq:2.6}) results in larger (but adaptively changing) $\alpha_{m}$ in (\ref{eq:2.6}) and therefore more controlled ensemble updates at each iteration; this is reflected in more accurate estimates of the posterior in terms of mean and variance. In Table \ref{Table2B} we show the performance of IR-ES for different values of $\rho$ (for fixed $M_{ES}=10$ and $N_{e}=75$). Note that for a given $\rho$, decreasing the value of $\tau$ reduces the error in the mean and variance decreases until approximately the value values of $\tau=1/\rho$. Below this value the aforementioned error will increase as expected from the lack of stability in the computation of the ensemble. Smaller $\rho$ corresponds to smaller $\alpha_{m}$'s with may not stabilize the scheme and thus the early stopping is required by having a larger $\tau$ selected according to $\tau>1/\rho$. Similar to the IR-enLM, the results from this subsection suggest that the value of $\rho=0.8$ seems a reasonable selection in terms of cost and accuracy for a given $N_{e}$ and $\tau$ fixed. Moreover, for the experiments under consideration a size of $N_e=75$ was sufficient to exhibit the regularizing properties that we would expect in the limit of a large ensemble.


In Table \ref{Table2C} we display the performance of IR-ES for different choices of $M_{ES}$ (for $\rho=0.7, \rho=0.8$ and $N_{e}=75$). As we discussed in subsection \ref{rem3}, $M_{ES}$ controls the amount of iterations where an approximation to the forward model outputs is computed from the ensemble. The value $M_{ES}=1$ corresponds to an evaluation of the forward model per iteration for each ensemble member. This corresponds to the most accurate but extremely expensive case. However, as we increase $M_{ES}$ we observe a dramatic reduction in the computational cost with a reasonable increase in the accuracy of the ensemble approximation. This is particularly clear for the case with larger $\rho$ which, as stated earlier, corresponds to smaller increments in the regularizing LM scheme and so the ensemble updates is better approximated by the updates of the measurement predictions. Finally, we notice that, in contrast to IR-enLM, the performance of IR-ES was similar for the both approximation of $\mu_{A}$ and $\mu_{B}$. In fact, note that the approximation of the variance of $\mu_{B}$ was better than the one of $\mu_{A}$. 

\begin{table}[H!]                                                 
\centering                                                        
\caption{Performance of IR-ES for different choices of $\tau$ and $N_{e}$ (for $\rho=0.8$ and $M_{ES}=10$)}
\label{Table2A}       
\begin{tabular}{cc|ccc|ccc}
\hline\noalign{\smallskip}
    & & &Approx. of $\mu_{A}$& & & Approx. of $\mu_{B}$ & \\
\noalign{\smallskip}\hline\noalign{\smallskip}
 $\tau$ &$N_{e}$& $\epsilon_{u}$&  $\epsilon_{\sigma}$ &aver. iter. &  $\epsilon_{u}$&  $\epsilon_{\sigma}$ & aver. iter.\\
\noalign{\smallskip}\hline\noalign{\smallskip}
\hline                                                             
1.6 & 25.000 & 1.001 & 0.668 & 11.500 & 0.932 & 0.508 & 12.050 \\                             
1.4 & 25.000 & 1.051 & 0.767 & 12.250 & 1.004 & 0.585 & 12.750 \\  
1.3 & 25.000 & 1.107 & 0.803 & 12.550 & 1.090 & 0.664 & 13.150 \\  
\Blue{1.25} &\Blue{ 25.000} &\Blue{ 1.144} &\Blue{ 0.824} & \Blue{12.700} & \Blue{1.210} &\Blue{ 0.718} &\Blue{ 13.450} \\ 
1.2 & 25.000 & 1.184 & 0.845 & 12.900 & 1.311 & 0.799 & 14.050 \\  
\hline                                                             \hline                                                             
1.6 & 50.000 & 0.717 & 0.325 & 10.950 & 0.702 & 0.279 & 11.700 \\  
1.5 & 50.000 & 0.711 & 0.341 & 11.200 & 0.699 & 0.278 & 12.050 \\ 
1.4 & 50.000 & 0.696 & 0.352 & 11.550 & 0.700 & 0.285 & 12.300 \\  
1.3 & 50.000 & 0.684 & 0.396 & 12.000 & 0.718 & 0.335 & 12.750 \\  
\Blue{1.25} &\Blue{ 50.000} &\Blue{ 0.685} &\Blue{ 0.421} &\Blue{ 12.150} &\Blue{ 0.742} &\Blue{ 0.375} & \Blue{13.000} \\ 
1.2 & 50.000 & 0.714 & 0.460 & 12.400 & 0.791 & 0.419 & 13.200 \\  
\hline                                                             \hline                                                             
1.5 & 75.000 & 0.651 & 0.295 & 10.700 & 0.618 & 0.247 & 11.550 \\  
1.4 & 75.000 & 0.635 & 0.280 & 11.050 & 0.606 & 0.223 & 12.050 \\  
1.3 & 75.000 & 0.600 & 0.267 & 11.500 & 0.609 & 0.234 & 12.450 \\  
\Blue{1.25} &\Blue{ 75.000} &\Blue{ 0.583} & \Blue{0.283} &\Blue{ 11.700} &\Blue{ 0.611} & \Blue{0.242} & \Blue{12.600 }\\ 
1.2 & 75.000 & 0.575 & 0.310 & 11.900 & 0.640 & 0.291 & 13.000 \\  
1.1 & 75.000 & 0.598 & 0.366 & 12.300 & 0.672 & 0.316 & 13.150 \\  
\hline                                                             \hline                                                             
1.5 & 100.000 & 0.614 & 0.316 & 10.500 & 0.563 & 0.259 & 11.550 \\ 
1.4 & 100.000 & 0.586 & 0.268 & 10.900 & 0.547 & 0.222 & 12.050 \\ 
1.3 & 100.000 & 0.558 & 0.236 & 11.250 & 0.541 & 0.208 & 12.350 \\ 
\Blue{1.25} & \Blue{100.000} &\Blue{ 0.544} &\Blue{ 0.235} &\Blue{ 11.400} &\Blue{ 0.539} &\Blue{ 0.204} &\Blue{ 12.500 }\\
1.2 & 100.000 & 0.540 & 0.238 & 11.550 & 0.549 & 0.215 & 12.850 \\ 
1.1 & 100.000 & 0.540 & 0.278 & 12.050 & 0.601 & 0.249 & 13.200 \\ 
\hline                                                             
1.4 & 250.000 & 0.534 & 0.407 & 10.267 & 0.469 & 0.283 & 11.714 \\ 
1.3 & 250.000 & 0.454 & 0.251 & 11.067 & 0.452 & 0.246 & 12.071 \\ 
\Blue{1.25} &\Blue{ 250.000} &\Blue{ 0.454} & \Blue{0.251} & \Blue{11.067} &\Blue{ 0.448} & \Blue{0.234} &\Blue{ 12.214} \\
1.2 & 250.000 & 0.437 & 0.237 & 11.133 & 0.440 & 0.186 & 12.571 \\ 
\Red{1.1} & \Red{250.000} & \Red{0.389} &\Red{ 0.178} & \Red{11.667} & \Red{0.450} &\Red{ 0.159} & \Red{12.857} \\ 
\hline                                                             
\end{tabular}                                                     
\end{table}

\begin{table}[H!]                     
\caption{Performance of IR-ES for different choices of $\rho$ and $\tau$ (for $M_{ES}=10$ and $N_{e}=75$)}
\label{Table2B}       
\begin{tabular}{cc|ccc|ccc}
\hline\noalign{\smallskip}
    & & &Approx. of $\mu_{A}$& & & Approx. of $\mu_{B}$ & \\
\noalign{\smallskip}\hline\noalign{\smallskip}
 $\rho$ &$\tau$& $\epsilon_{u}$&  $\epsilon_{\sigma}$ &aver. iter. &  $\epsilon_{u}$&  $\epsilon_{\sigma}$ & aver. iter.\\
\noalign{\smallskip}\hline\noalign{\smallskip}
 0.5 & 3.500 & 0.766 & 0.610 & 3.000 & 0.716 & 0.480 & 3.500 \\           
0.5 & 3.000 & 0.719 & 0.466 & 3.550 & 0.693 & 0.408 & 3.850 \\  
0.5 & 2.800 & 0.695 & 0.397 & 3.800 & 0.684 & 0.388 & 3.950 \\  
0.5 & 2.200 & 0.681 & 0.344 & 4.000 & 0.689 & 0.368 & 4.050 \\  
\Blue{0.5} &\Blue{ 2.000} &\Blue{ 0.681} & \Blue{0.344 }& \Blue{4.000} &\Blue{ 0.741} &\Blue{ 0.339} & \Blue{4.400 }\\  
0.5 & 1.800 & 0.737 & 0.353 & 4.200 & 0.817 & 0.331 & 4.650 \\  
0.5 & 1.600 & 0.806 & 0.366 & 4.600 & 0.909 & 0.346 & 4.900 \\  
\hline                                                          \hline                                                          
0.6 & 2.200 & 0.679 & 0.362 & 5.150 & 0.685 & 0.358 & 5.000 \\  
0.6 & 2.000 & 0.676 & 0.345 & 5.250 & 0.687 & 0.333 & 5.200 \\  
0.6 & 1.900 & 0.668 & 0.314 & 5.400 & 0.686 & 0.326 & 5.250 \\  
0.6 & 1.800 & 0.666 & 0.307 & 5.500 & 0.709 & 0.299 & 5.550 \\  
\Blue{0.6} &\Blue{ 1.700} &\Blue{ 0.666} &\Blue{ 0.292} & \Blue{5.700} &\Blue{ 0.749} & \Blue{0.290} & \Blue{5.850 }\\  
0.6 & 1.600 & 0.673 & 0.296 & 5.900 & 0.769 & 0.291 & 6.000 \\  
0.6 & 1.500 & 0.695 & 0.303 & 5.950 & 0.769 & 0.291 & 6.000 \\  
0.6 & 1.400 & 0.714 & 0.318 & 6.050 & 0.769 & 0.291 & 6.000 \\  
\hline                  \hline                  
0.7 & 1.700 & 0.659 & 0.292 & 7.550 & 0.693 & 0.292 & 8.100 \\                                    
0.7 & 1.600 & 0.655 & 0.279 & 7.750 & 0.698 & 0.281 & 8.250 \\  
0.7 & 1.500 & 0.657 & 0.280 & 8.000 & 0.723 & 0.280 & 8.500 \\  
\Blue{0.7} & \Blue{1.450} &\Blue{ 0.657} &\Blue{ 0.284} & \Blue{8.100} & \Blue{0.759} & \Blue{0.290} & \Blue{8.800} \\  
0.7 & 1.420 & 0.669 & 0.296 & 8.250 & 0.763 & 0.293 & 8.850 \\  
0.7 & 1.350 & 0.684 & 0.302 & 8.300 & 0.770 & 0.296 & 8.900 \\  
0.7 & 1.300 & 0.733 & 0.324 & 8.450 & 0.815 & 0.323 & 9.150 \\  
\hline                                                          \hline                                                          
0.8 & 1.400 & 0.635 & 0.280 & 11.050 & 0.606 & 0.223 & 12.050 \\
0.8 & 1.350 & 0.619 & 0.272 & 11.250 & 0.604 & 0.224 & 12.150 \\
0.8 & 1.300 & 0.600 & 0.267 & 11.500 & 0.609 & 0.234 & 12.450 \\
\Blue{0.8} & \Blue{1.250} &\Blue{ 0.583} &\Blue{ 0.283} &\Blue{ 11.700} &\Blue{ 0.611} & \Blue{0.242} &\Blue{ 12.600} \\
0.8 & 1.200 & 0.575 & 0.310 & 11.900 & 0.640 & 0.291 & 13.000 \\
0.8 & 1.100 & 0.598 & 0.366 & 12.300 & 0.672 & 0.316 & 13.150 \\
\hline                                                          \hline                                                          
0.9 & 1.400 & 0.629 & 0.299 & 23.800 & 0.581 & 0.238 & 26.200 \\
0.9 & 1.300 & 0.599 & 0.265 & 24.600 & 0.582 & 0.220 & 26.900 \\
0.9 & 1.200 & 0.563 & 0.244 & 25.550 & 0.590 & 0.206 & 27.900 \\
0.9 & 1.150 & 0.551 & 0.242 & 26.000 & 0.600 & 0.209 & 28.400 \\
\Blue{0.9} &\Blue{ 1.100} &\Blue{ 0.540} & \Blue{0.260} &\Blue{ 26.650} & \Blue{0.628} &\Blue{ 0.229} &\Blue{ 29.000} \\
0.9 & 1.000 & 0.608 & 0.309 & 27.750 & 0.799 & 0.419 & 30.500 \\
\hline 

\end{tabular}                                                        
\end{table}

\begin{table}[H!]                     
\caption{Performance of IR-ES for different choices of $M_{ES}$ and $\rho$ (for $\tau=1/\rho$ and $N_{e}=75$)}
\label{Table2C}       
\begin{tabular}{cc|ccc|cccc}
\hline\noalign{\smallskip}
&&  & &Approx. of $\mu_{A}$& & & Approx. of $\mu_{B}$ & \\
\noalign{\smallskip}\hline\noalign{\smallskip}
$\rho$ & $M_{ES}$& $\epsilon_{u}$&  $\epsilon_{\sigma}$ &aver. cost (model runs) &  $\epsilon_{u}$&  $\epsilon_{\sigma}$ & aver. cost (model runs) \\
\noalign{\smallskip}\hline\noalign{\smallskip}
                          
0.7&1 & 0.598 & 0.297 & 652.500 & 0.576 & 0.278 & 675.000 \\ 
0.7&5& 0.634 & 0.308 & 150& 0.594 & 0.250 & 150 \\ 
0.7&10&  0.654 & 0.293 & 75 & 0.718 & 0.287 & 75 \\  
\hline \hline
0.8&1 & 0.547 & 0.255 & 918.750 & 0.588 & 0.205 & 1057.500 \\                                                 
0.8&5&  0.575 & 0.277 & 225& 0.583 & 0.222 & 225\\
0.8&10& 0.578 & 0.294 & 150 & 0.610 & 0.246 & 150\\
0.8&15 & 0.668 & 0.307 & 75 & 0.751 & 0.279 & 75\\

\hline                                                 
\end{tabular}                                                    
\end{table}

\begin{figure}
\begin{center}
\includegraphics[scale=0.15]{Mean_MCMC1A}~
\includegraphics[scale=0.15]{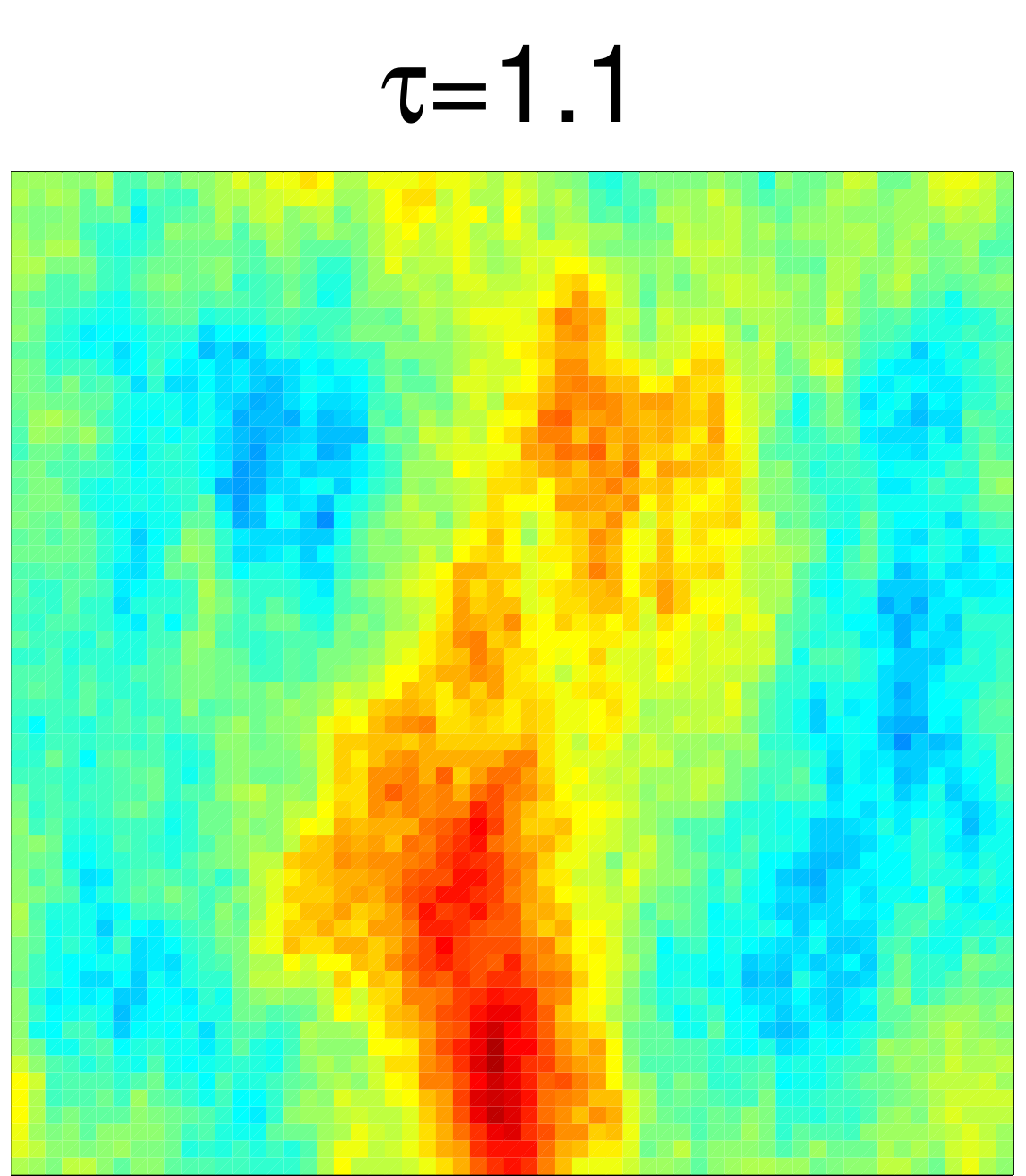}~
\includegraphics[scale=0.15]{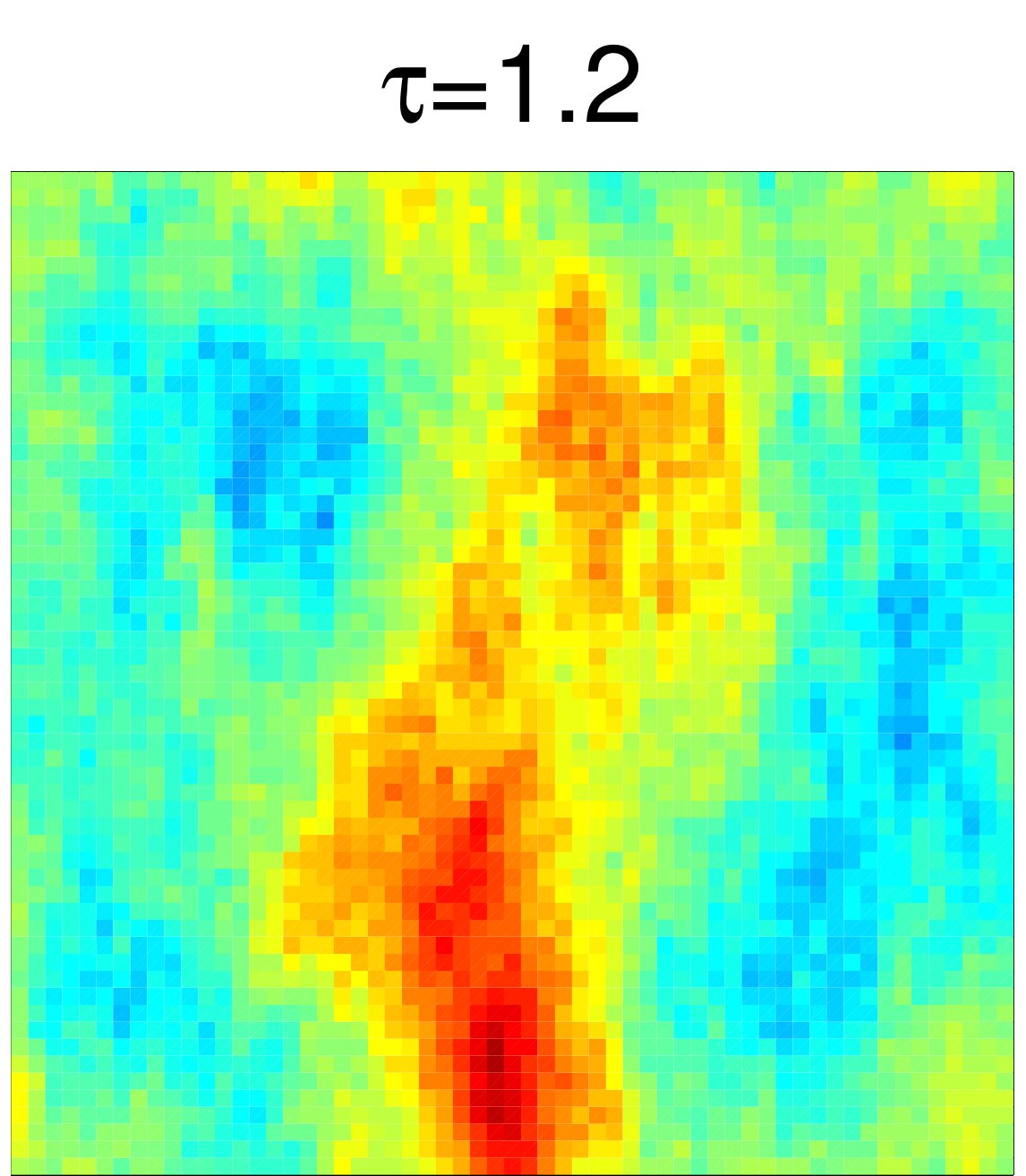}~
\includegraphics[scale=0.15]{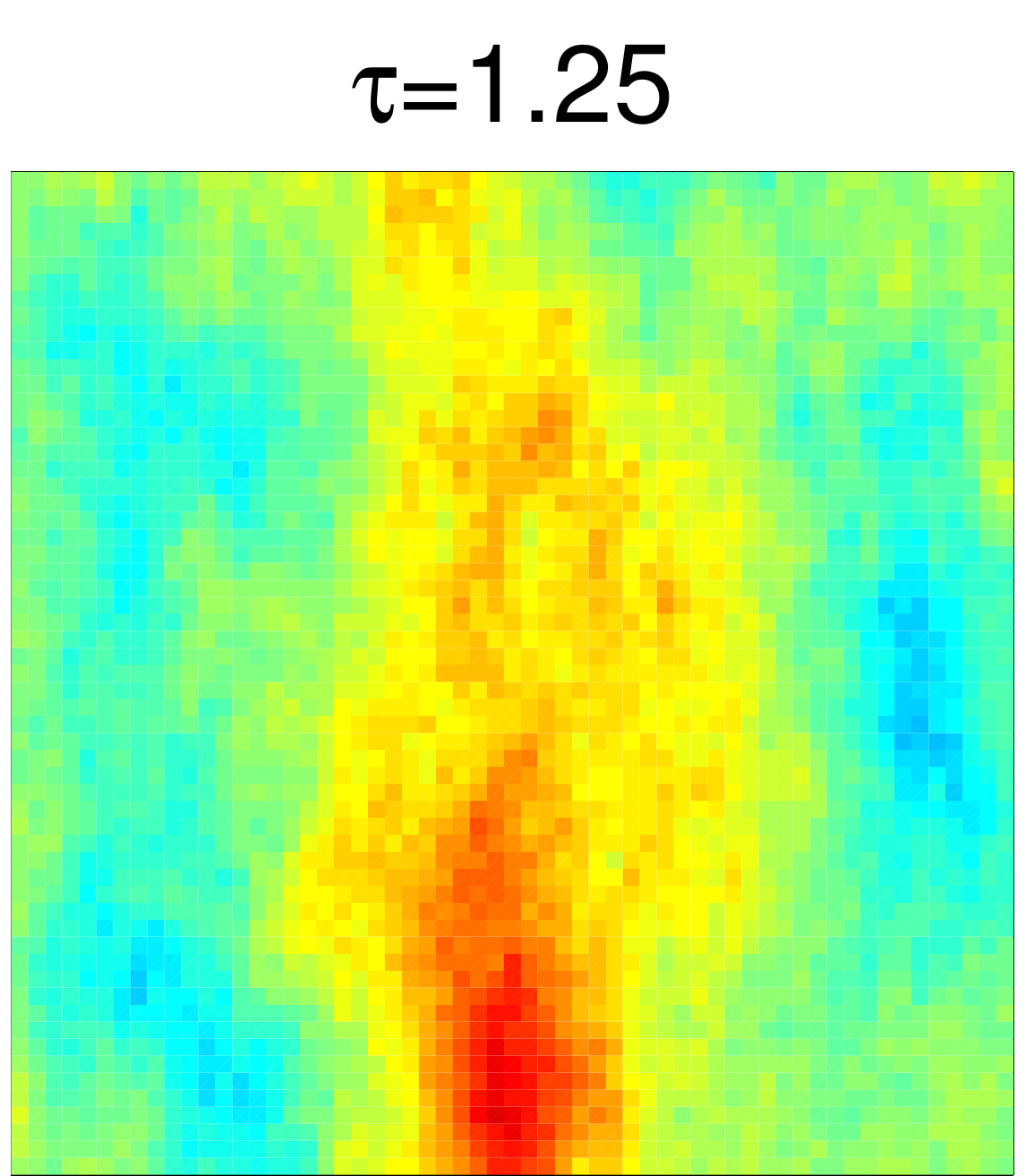}~
\includegraphics[scale=0.15]{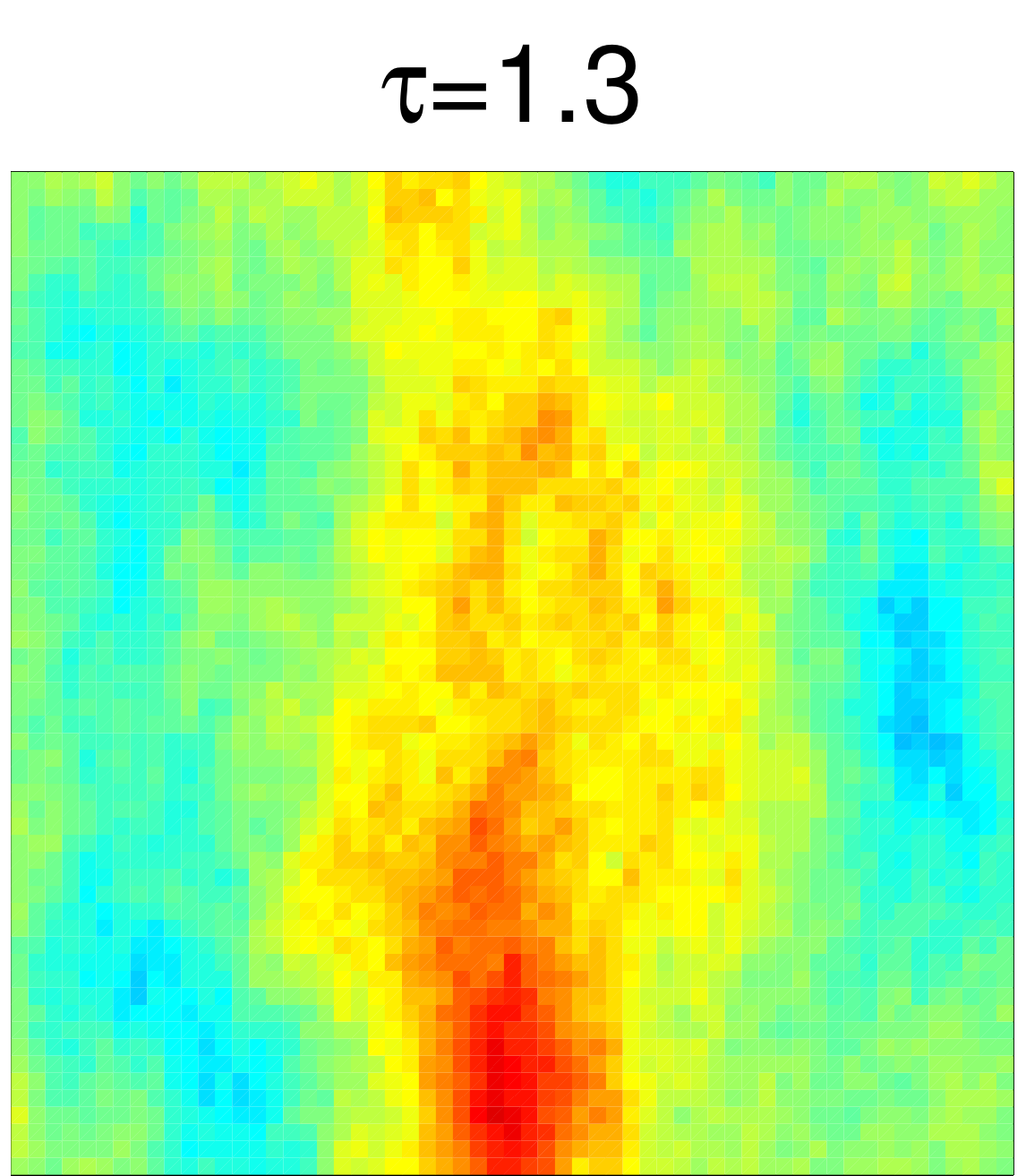}~
\includegraphics[scale=0.15]{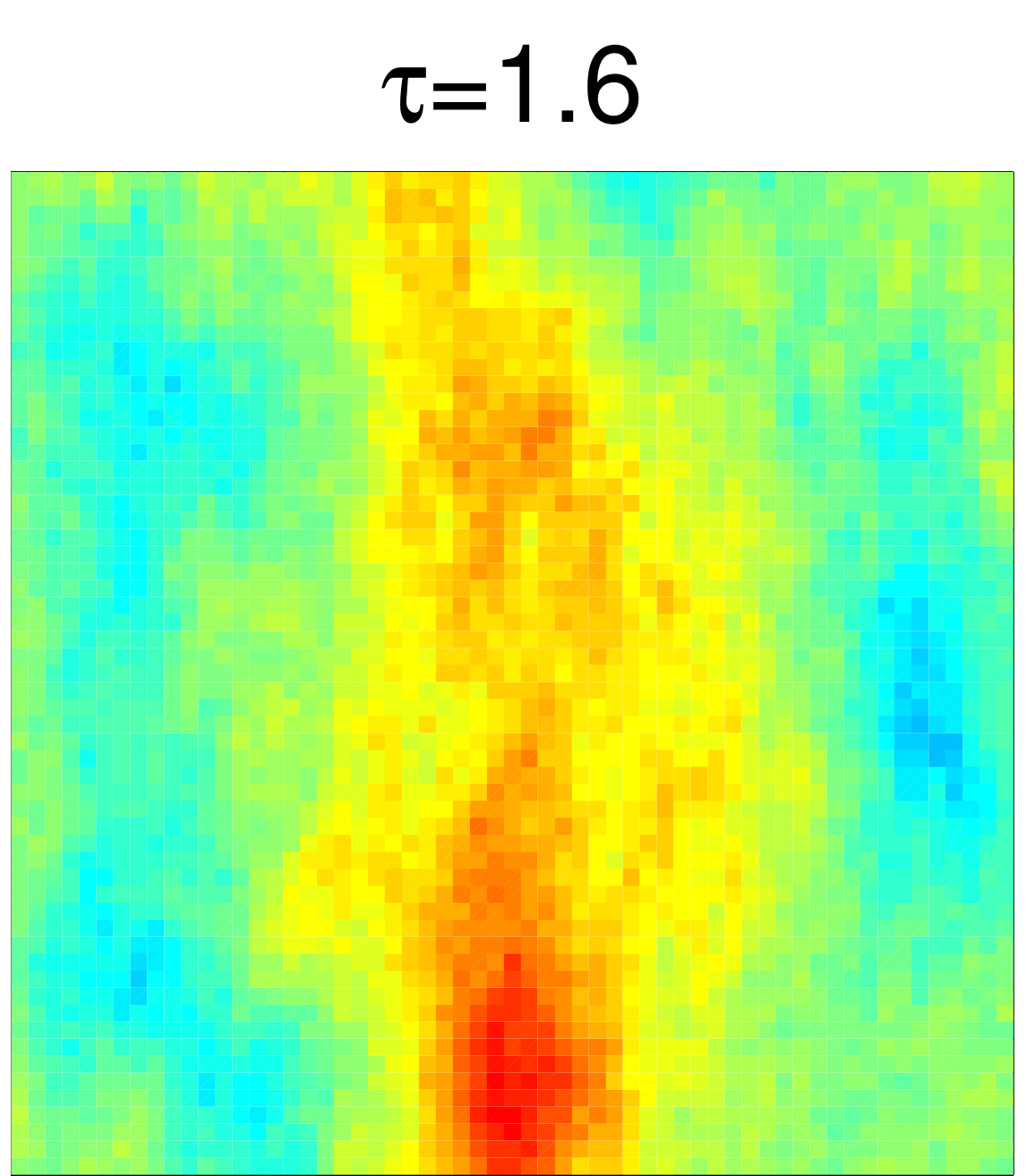}\\
\includegraphics[scale=0.15]{Var_MCMC1A}~
\includegraphics[scale=0.15]{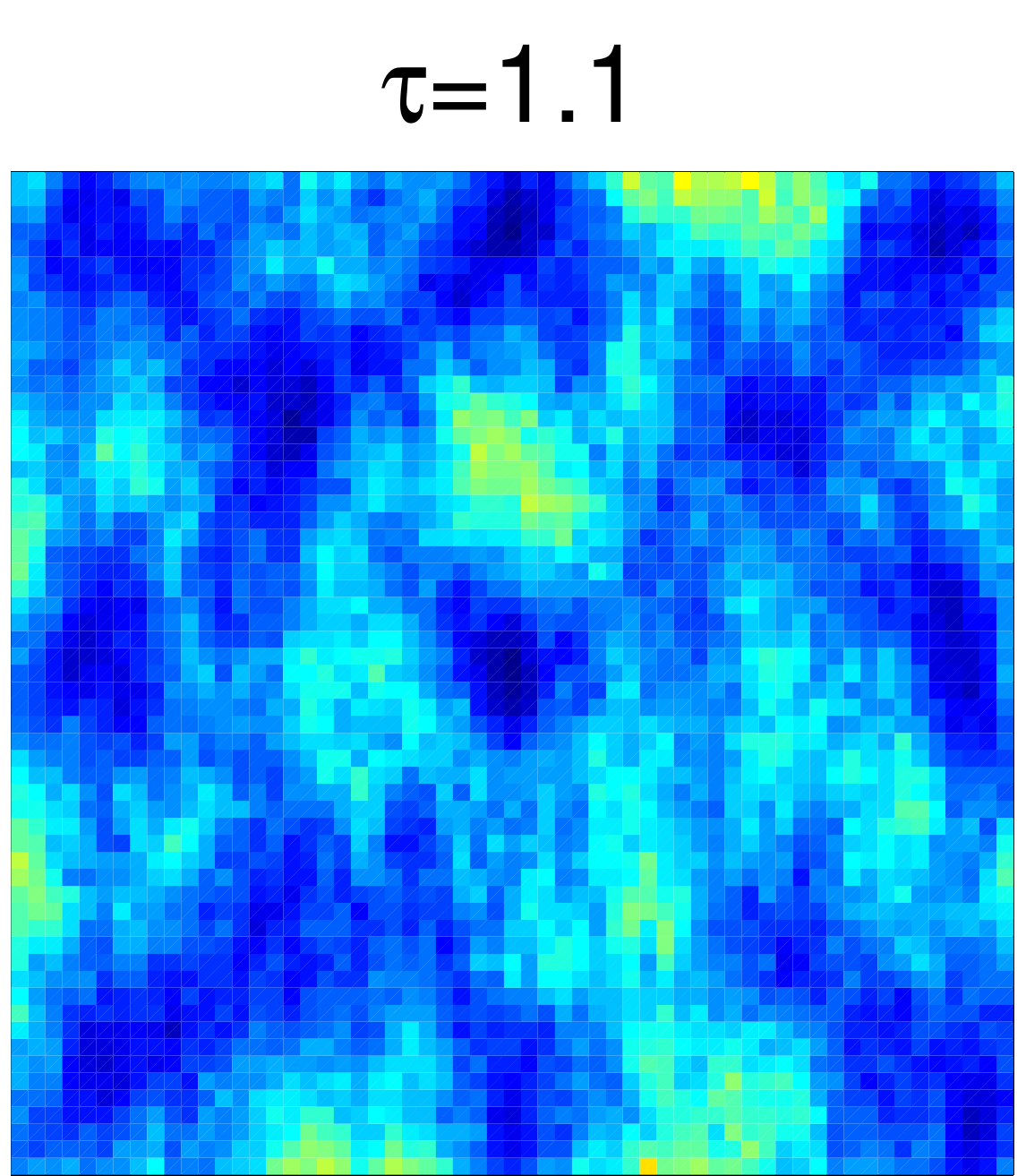}~
\includegraphics[scale=0.15]{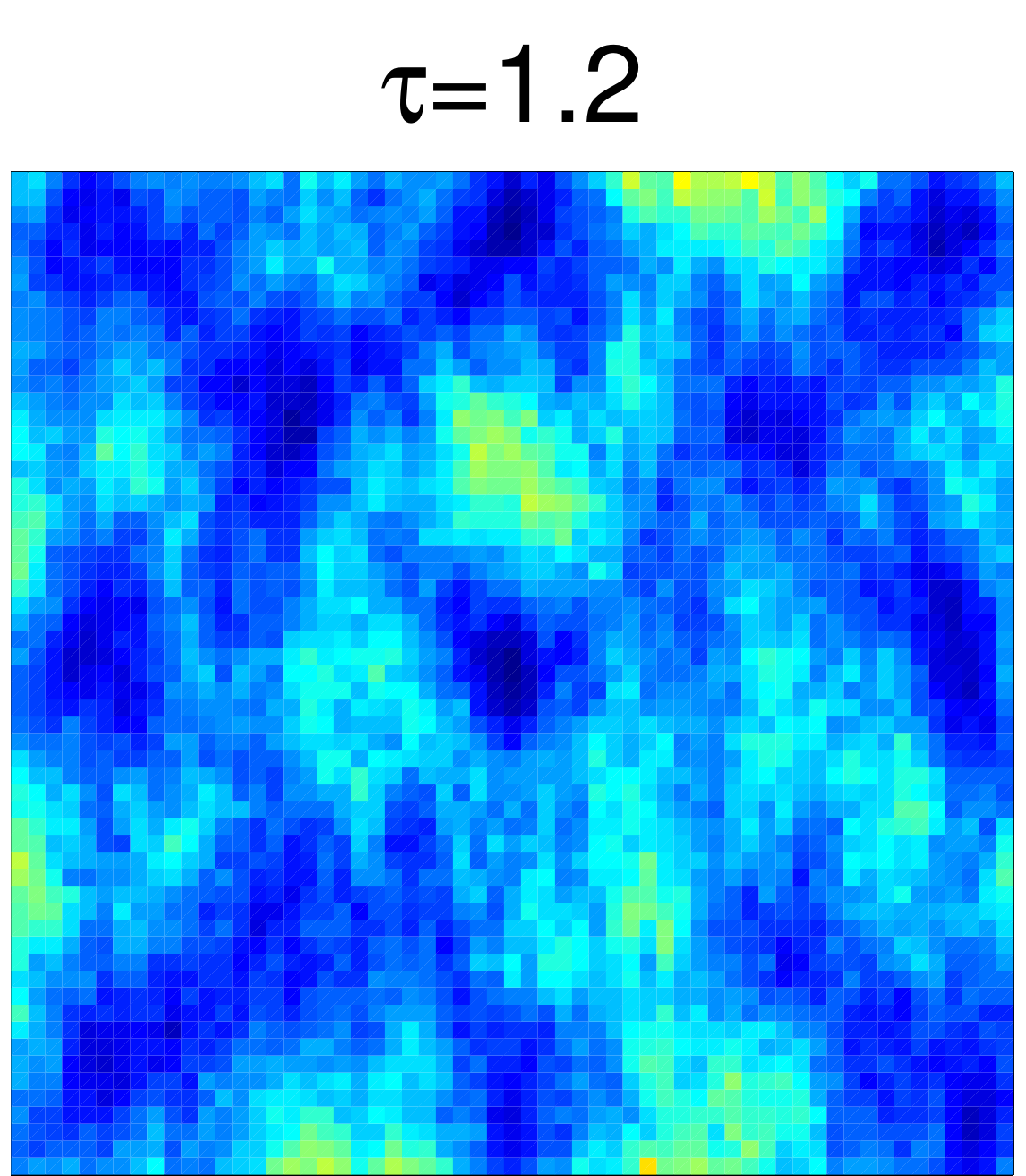}~
\includegraphics[scale=0.15]{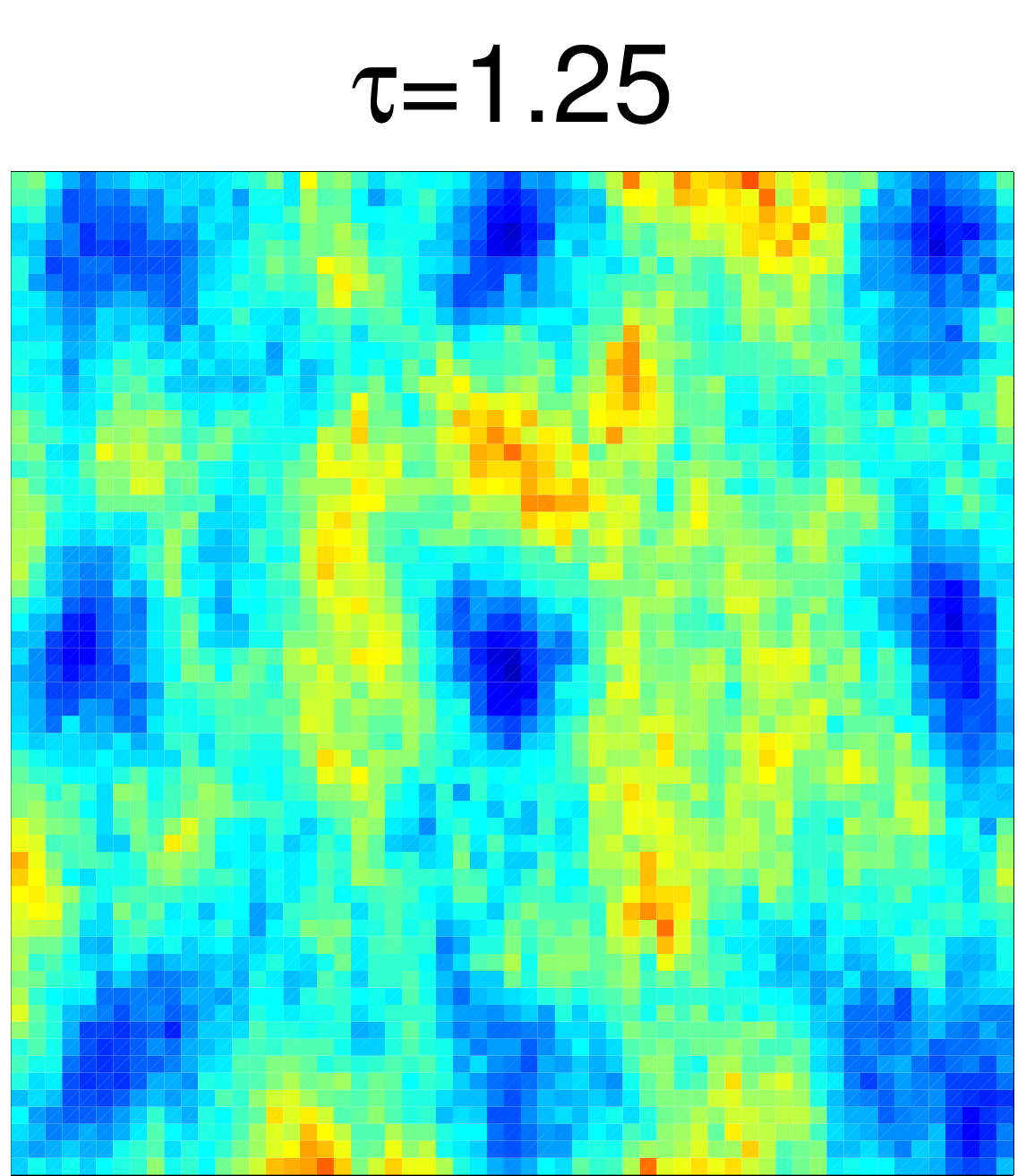}~
\includegraphics[scale=0.15]{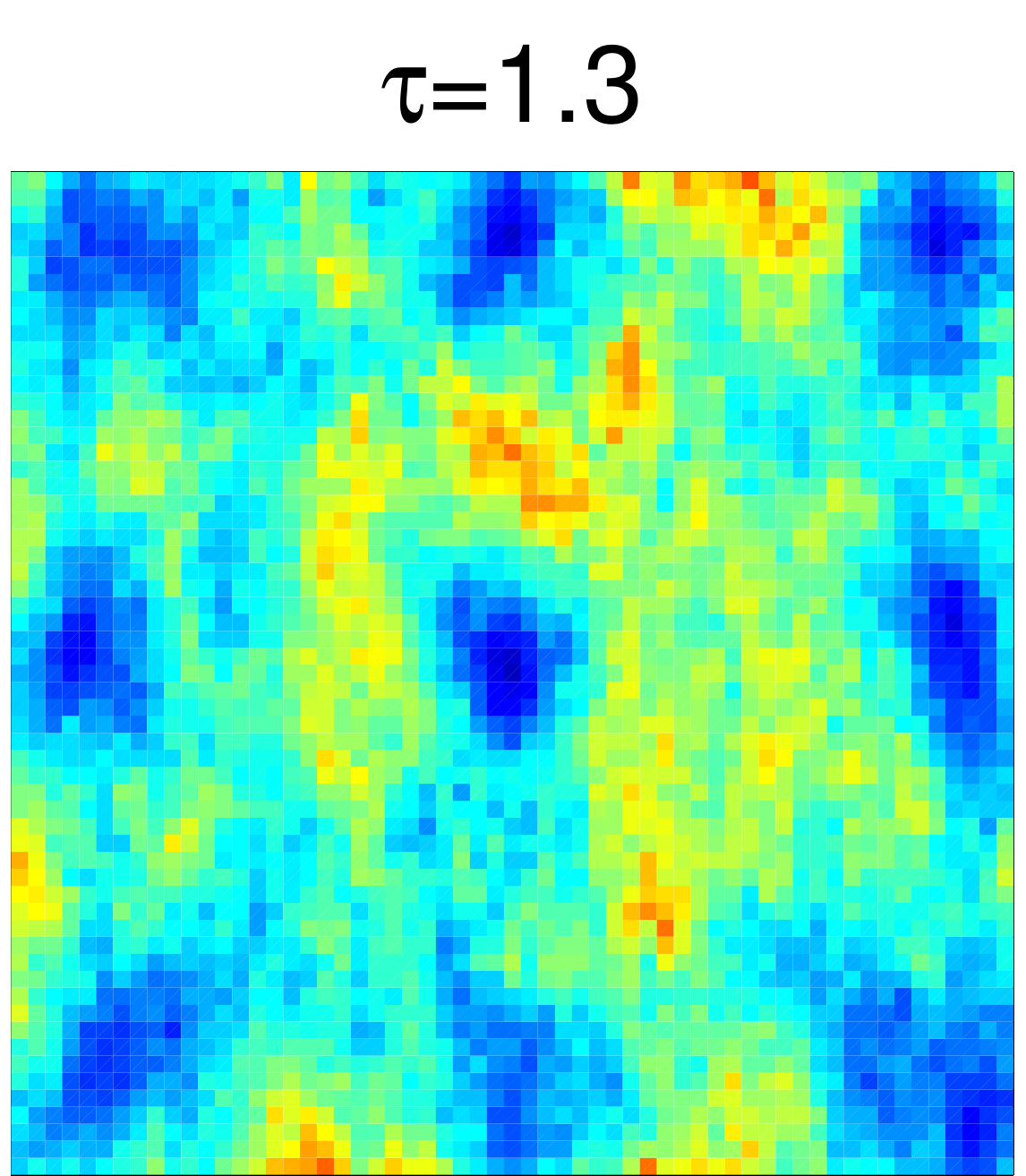}~
\includegraphics[scale=0.15]{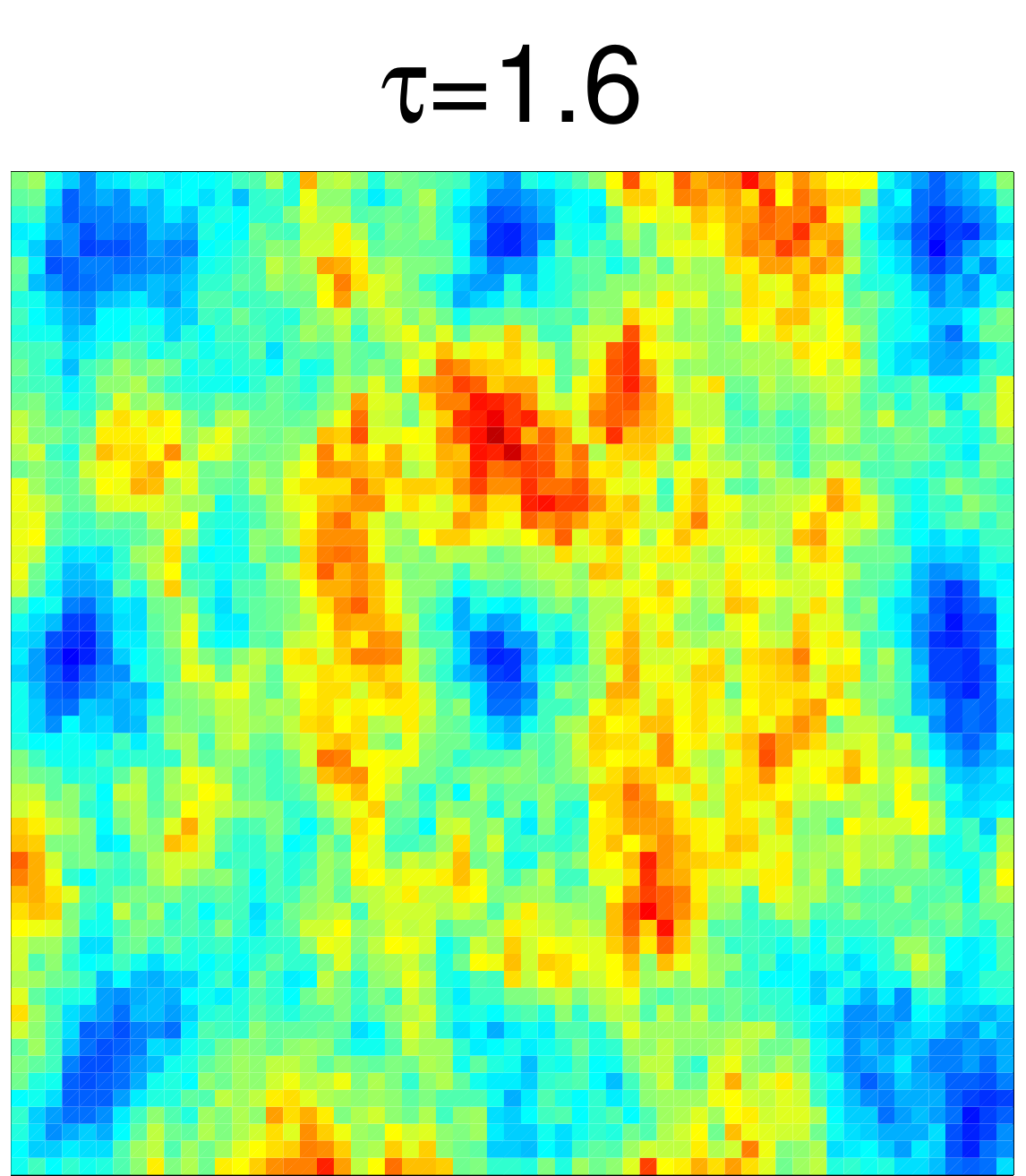}\\
\includegraphics[scale=0.225]{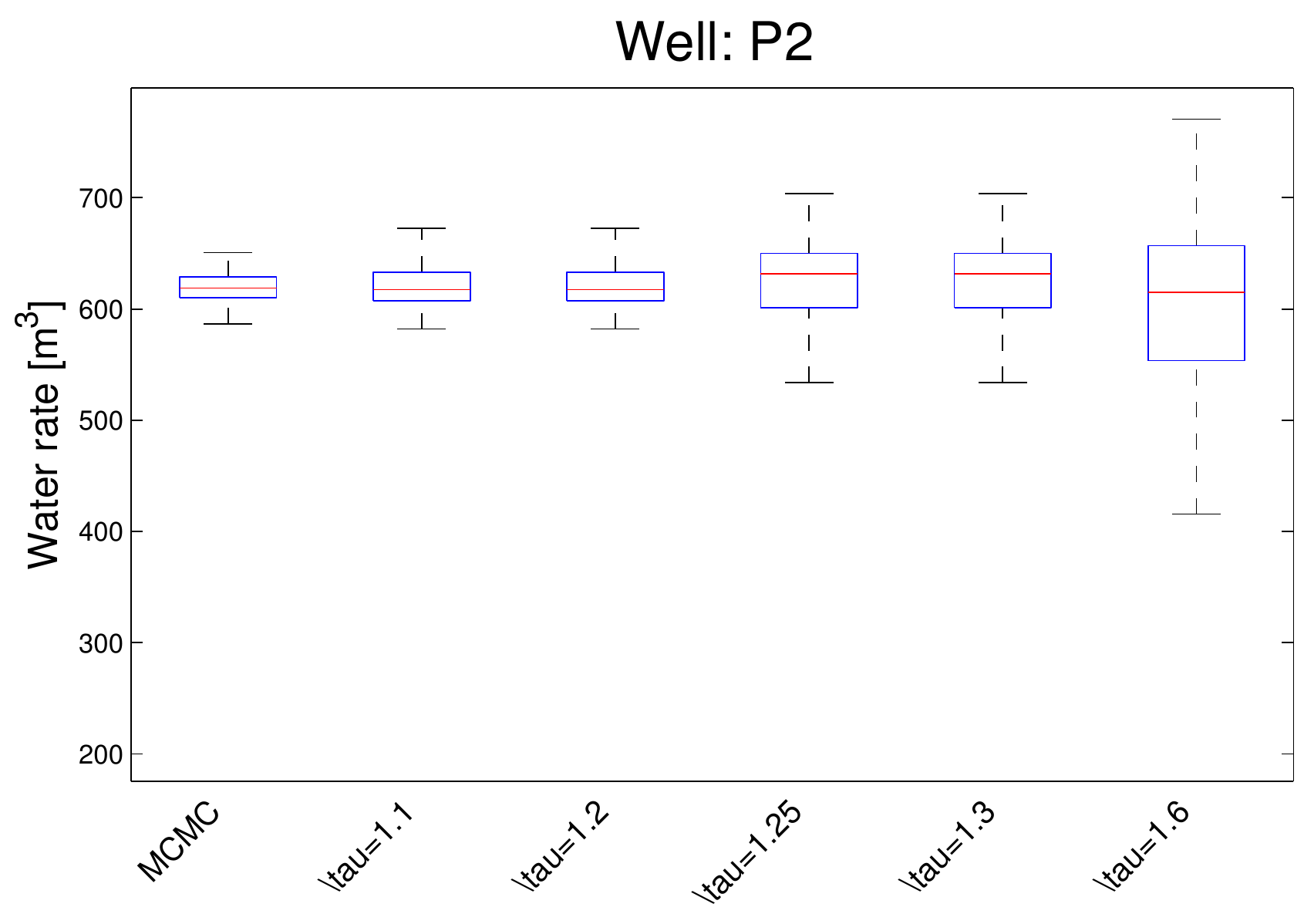}
\includegraphics[scale=0.225]{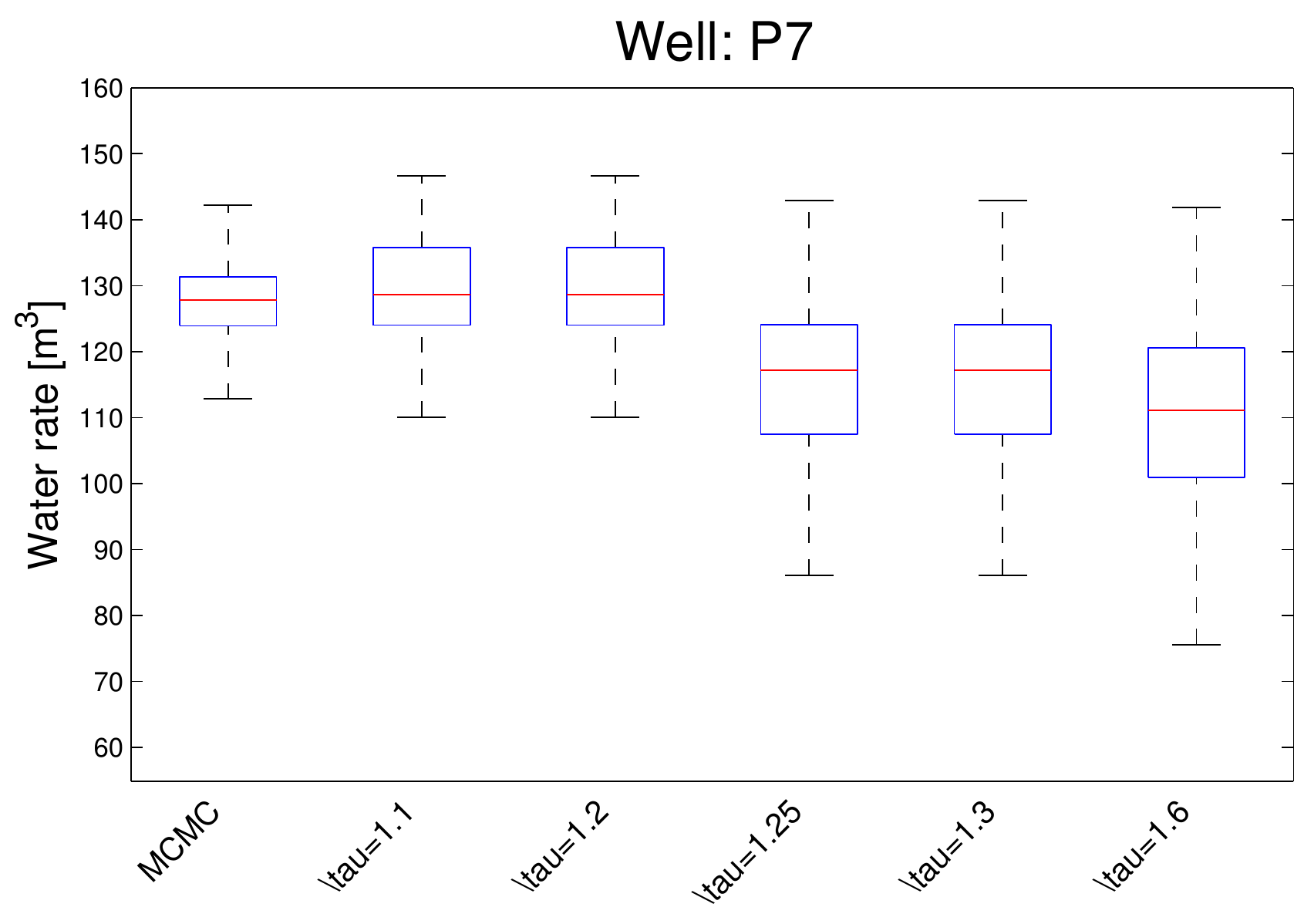}
\includegraphics[scale=0.225]{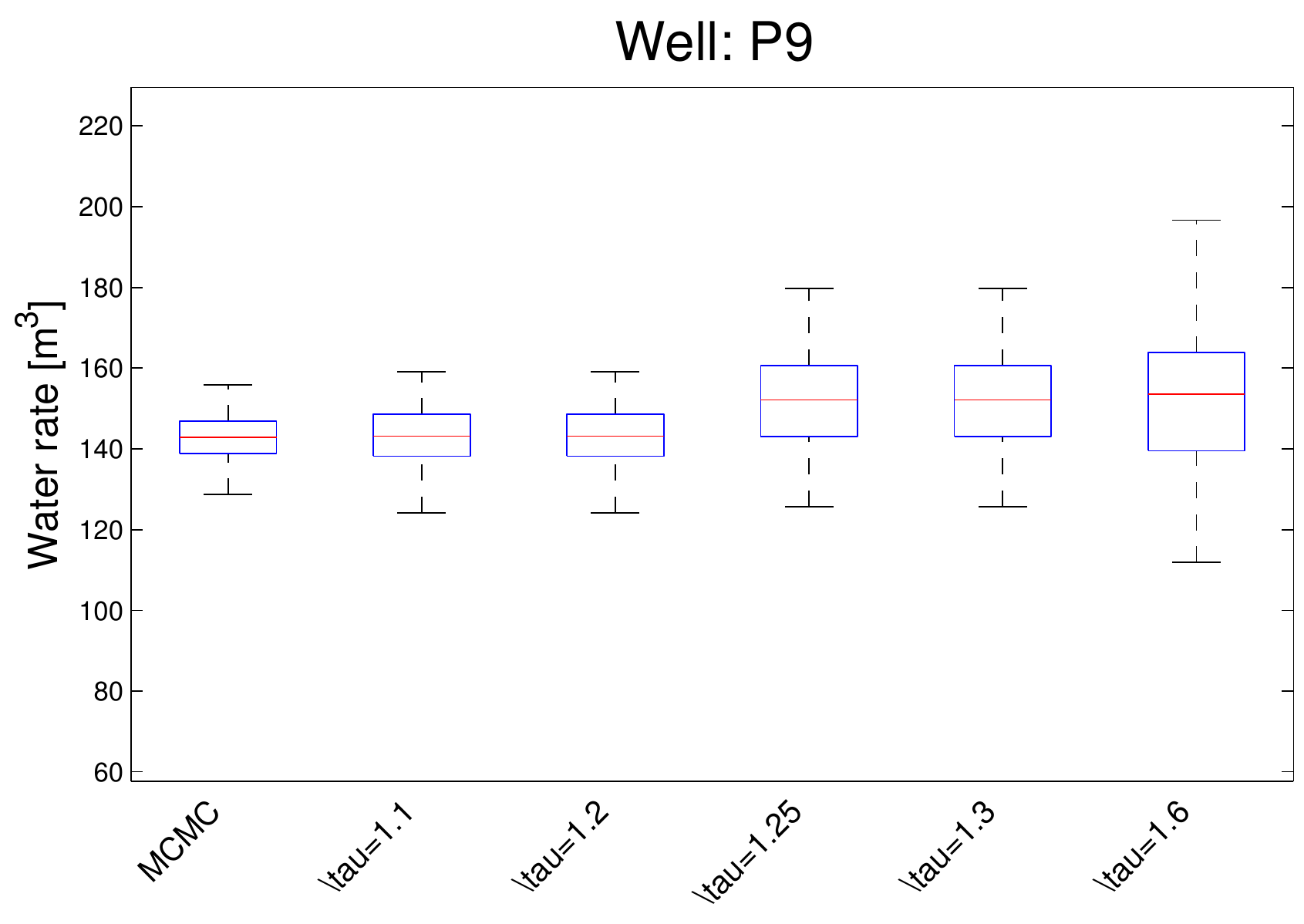}
\includegraphics[scale=0.225]{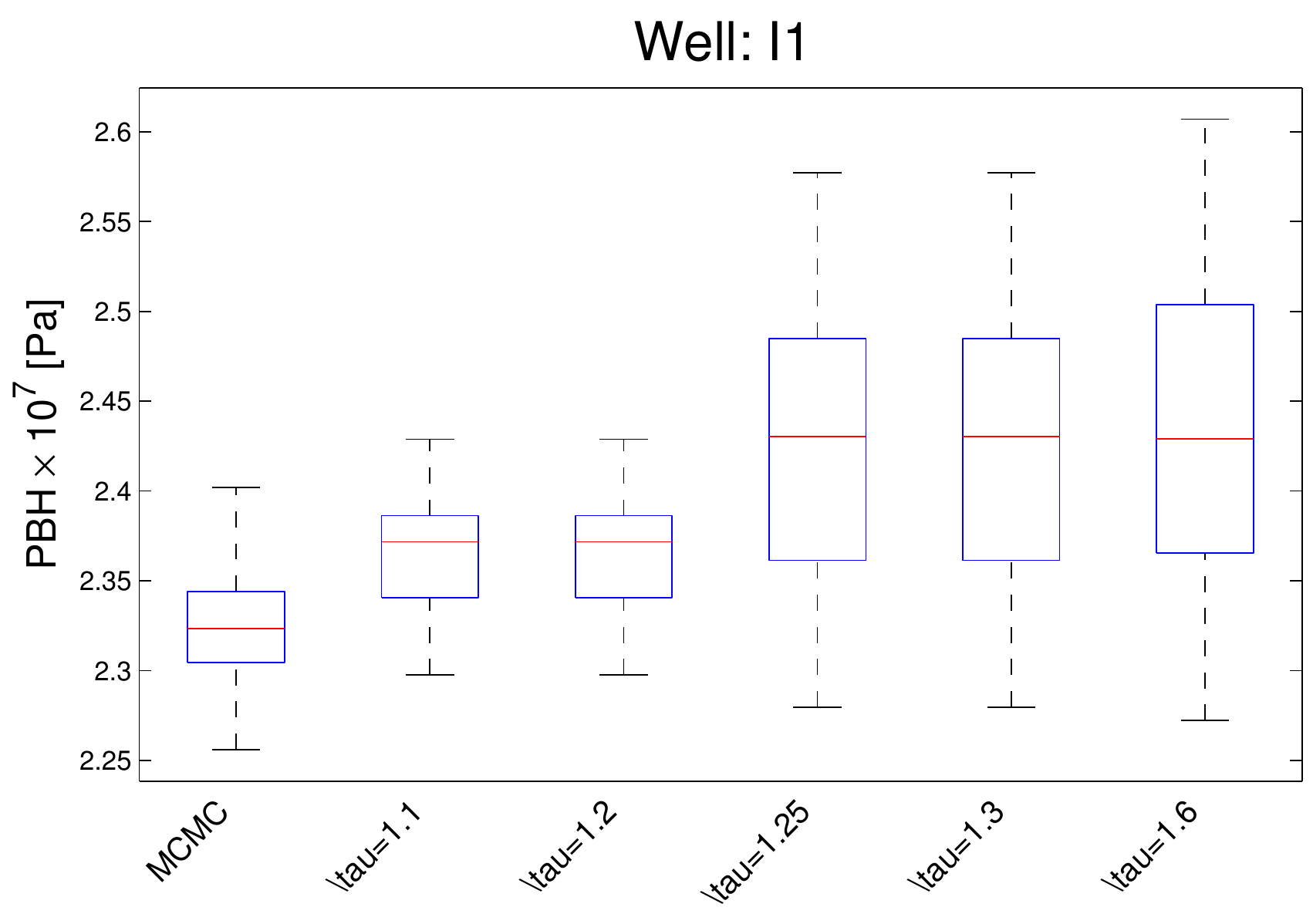}
\includegraphics[scale=0.225]{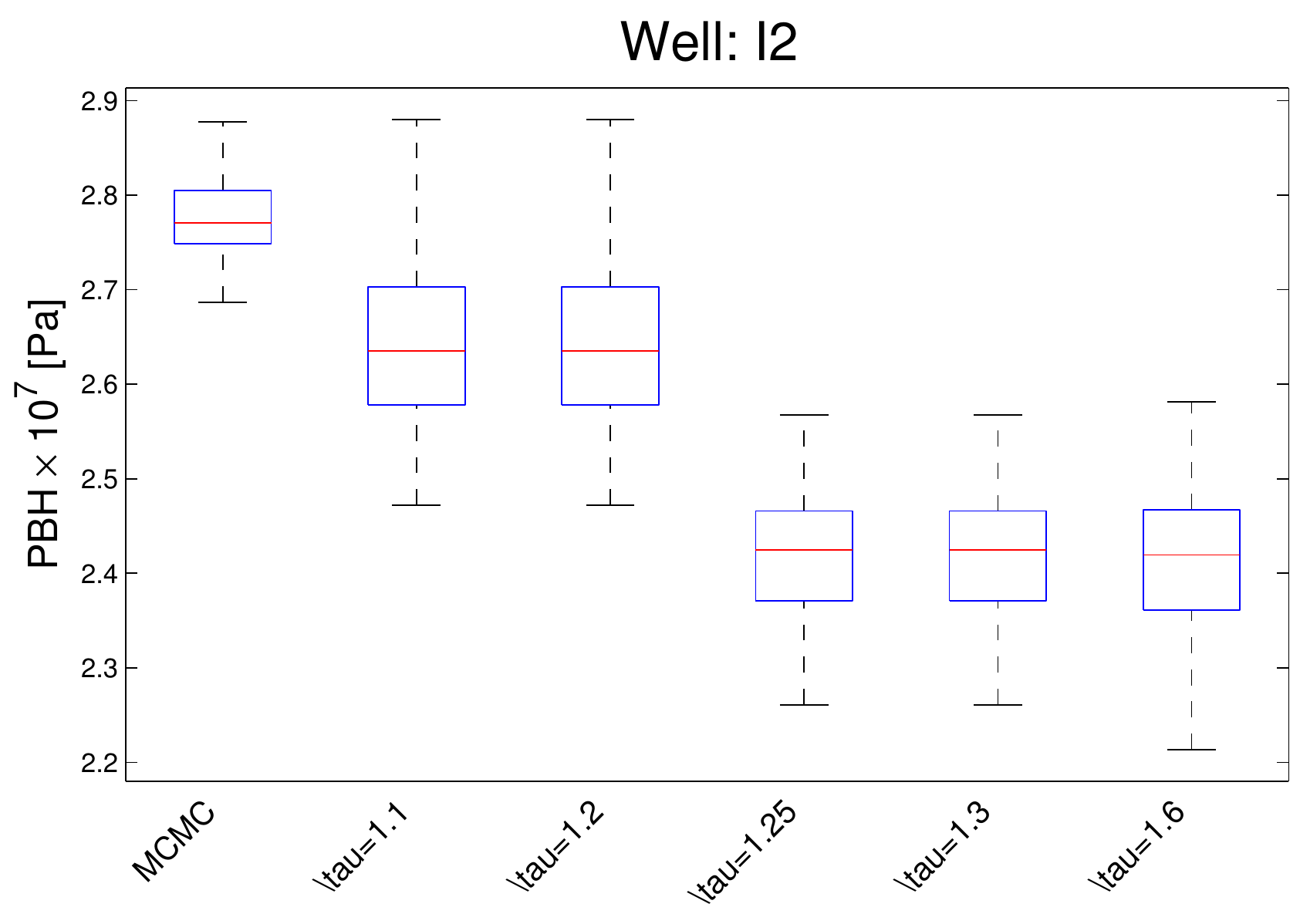}
\includegraphics[scale=0.225]{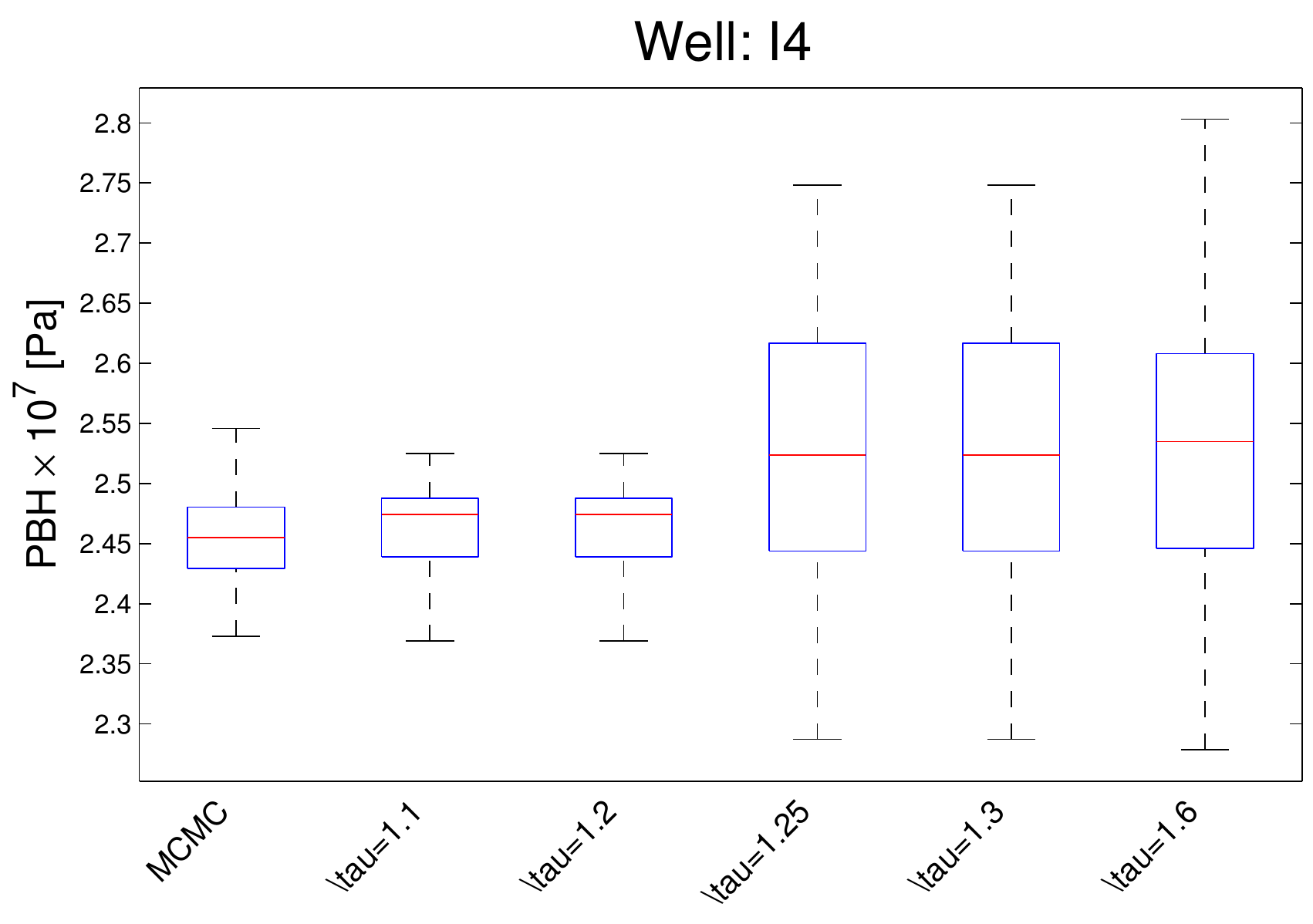}
\caption{Top row (mean of $\mu_A$) and Top-Middle row (variance of $\mu_A$) from left to right: pcn-MCMC, IR-ES approximations with $\rho=0.8$, $N_{e}=75$ and $\tau=1.1$, $\tau=1.2$, $\tau=1.25$, $\tau=1.3$, $\tau=1.6$. Middle-bottom and bottom row: Box plots of water rates from production wells $P_{2},P_{7}, P_{9}$ and PBH from injections wells $I_{1},I_{2},I_{4}$ after 6 years of water flood simulated from $\mu_{A}$ (with MCMC ) and the ensemble approximation with IR-ES for $\rho=0.8$ and different choices of $\tau$}\label{Figure6}
\end{center}
\end{figure}

\begin{figure}
\begin{center}
\includegraphics[scale=0.15]{Mean_MCMC2BB}~
\includegraphics[scale=0.15]{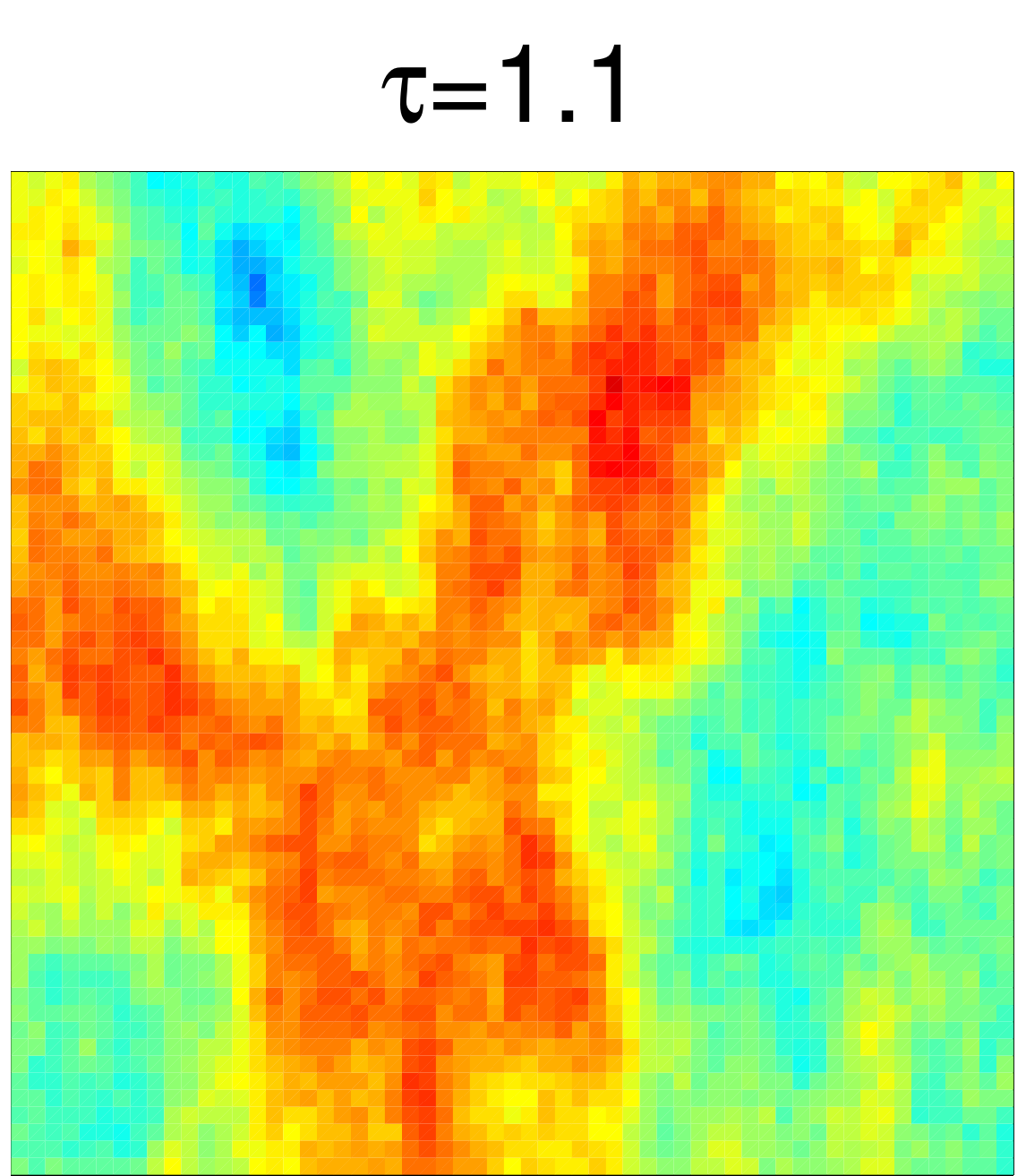}~
\includegraphics[scale=0.15]{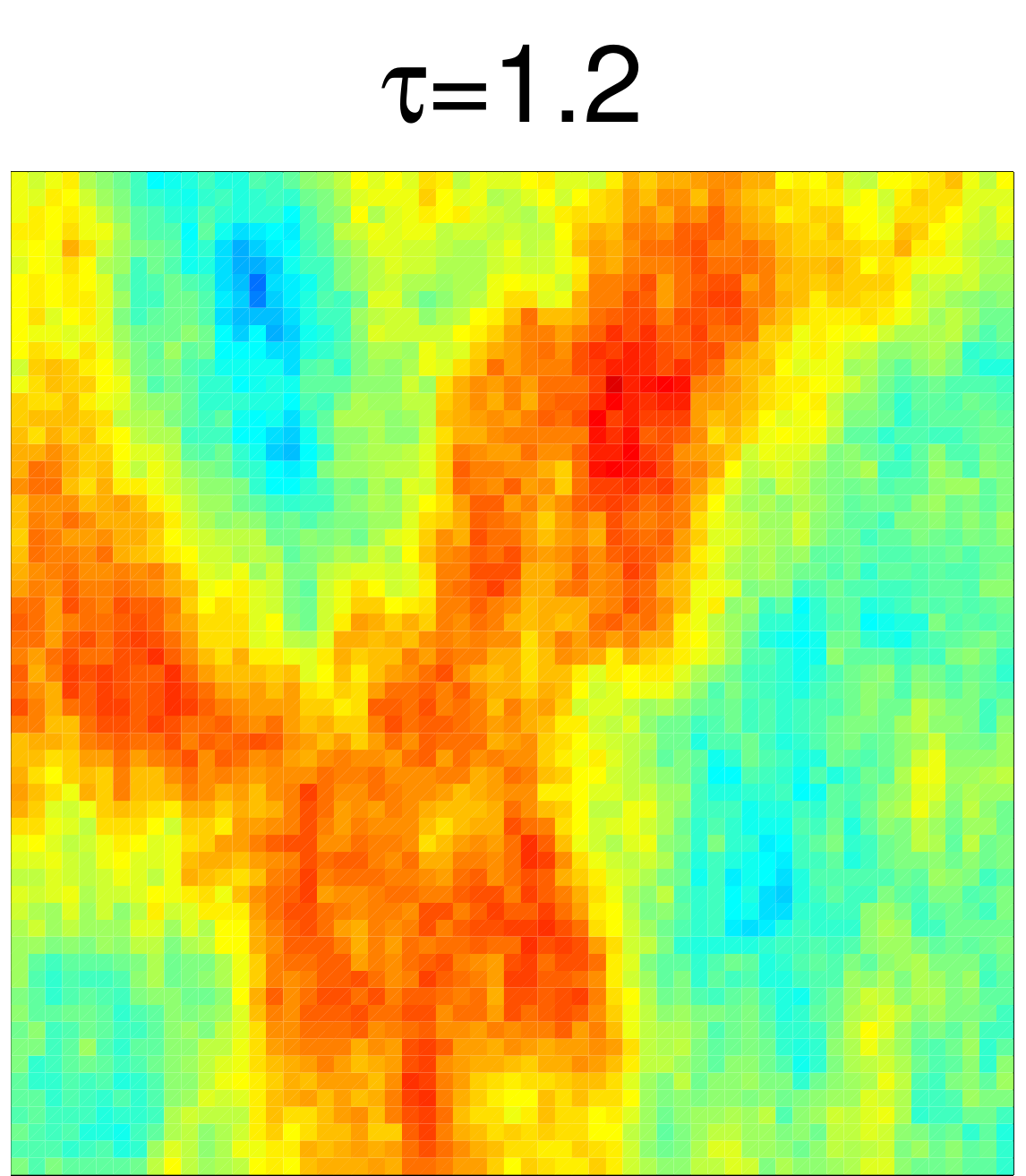}~
\includegraphics[scale=0.15]{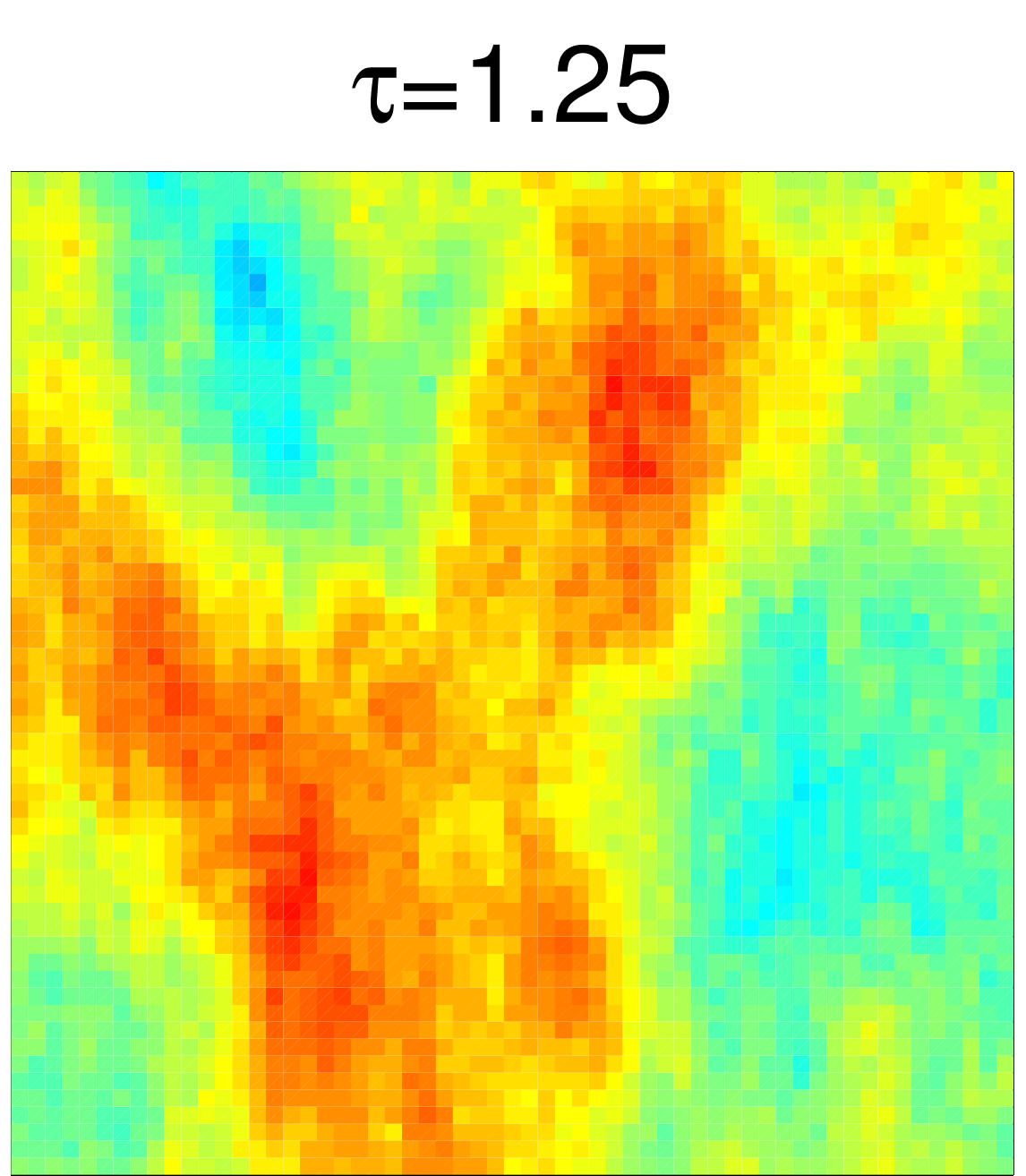}~
\includegraphics[scale=0.15]{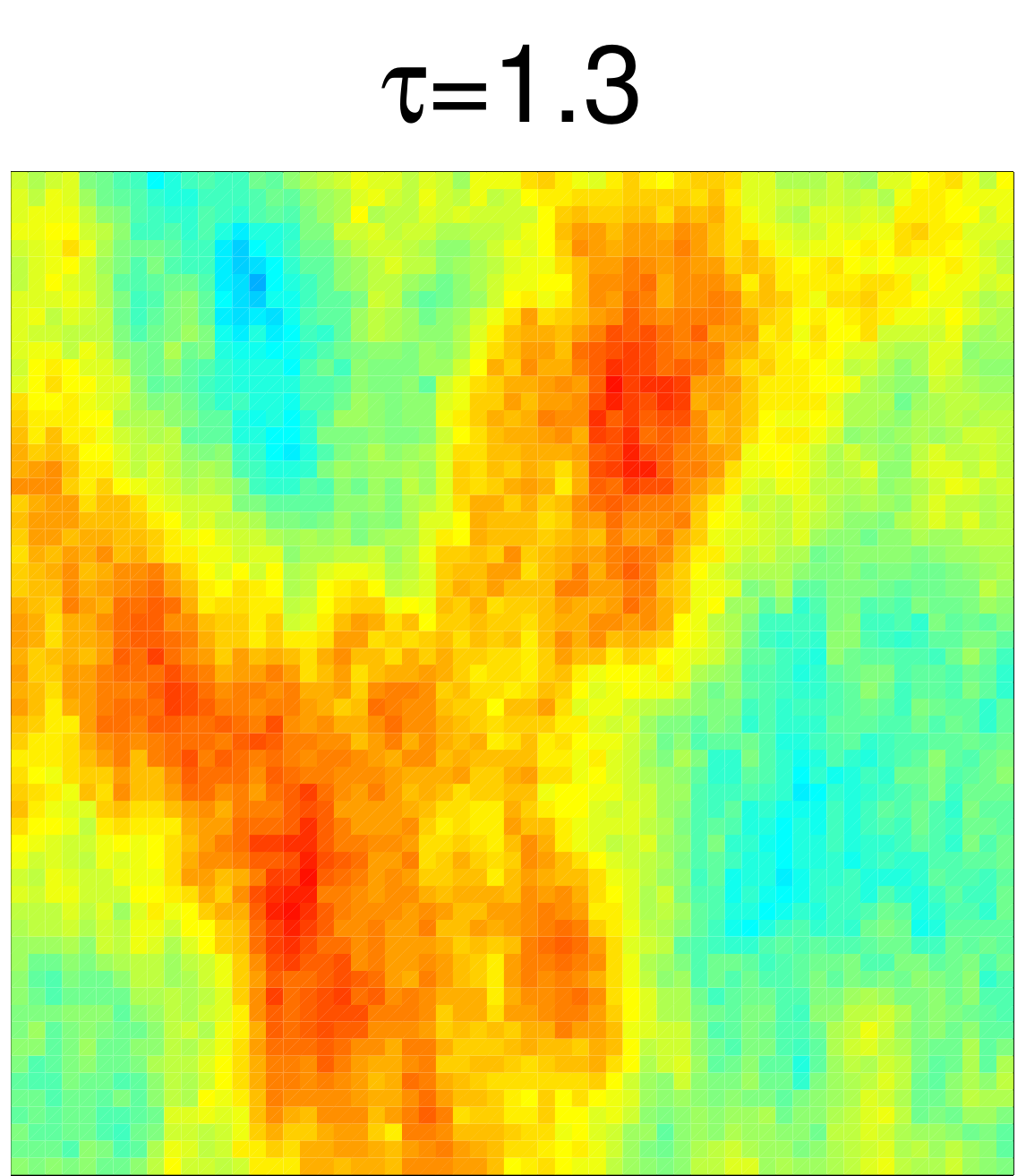}~
\includegraphics[scale=0.15]{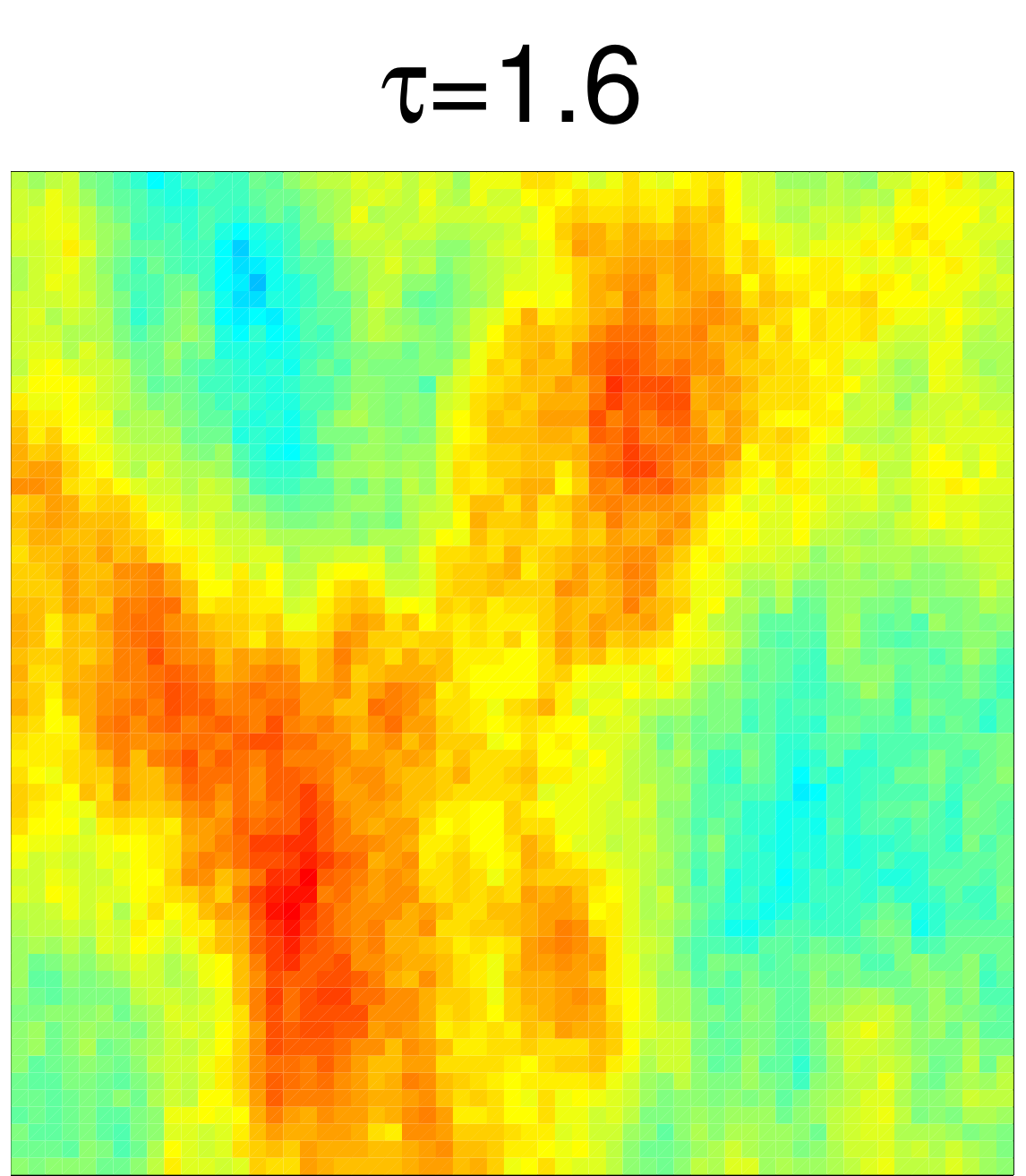}\\
\includegraphics[scale=0.15]{Var_MCMC2BB}~
\includegraphics[scale=0.15]{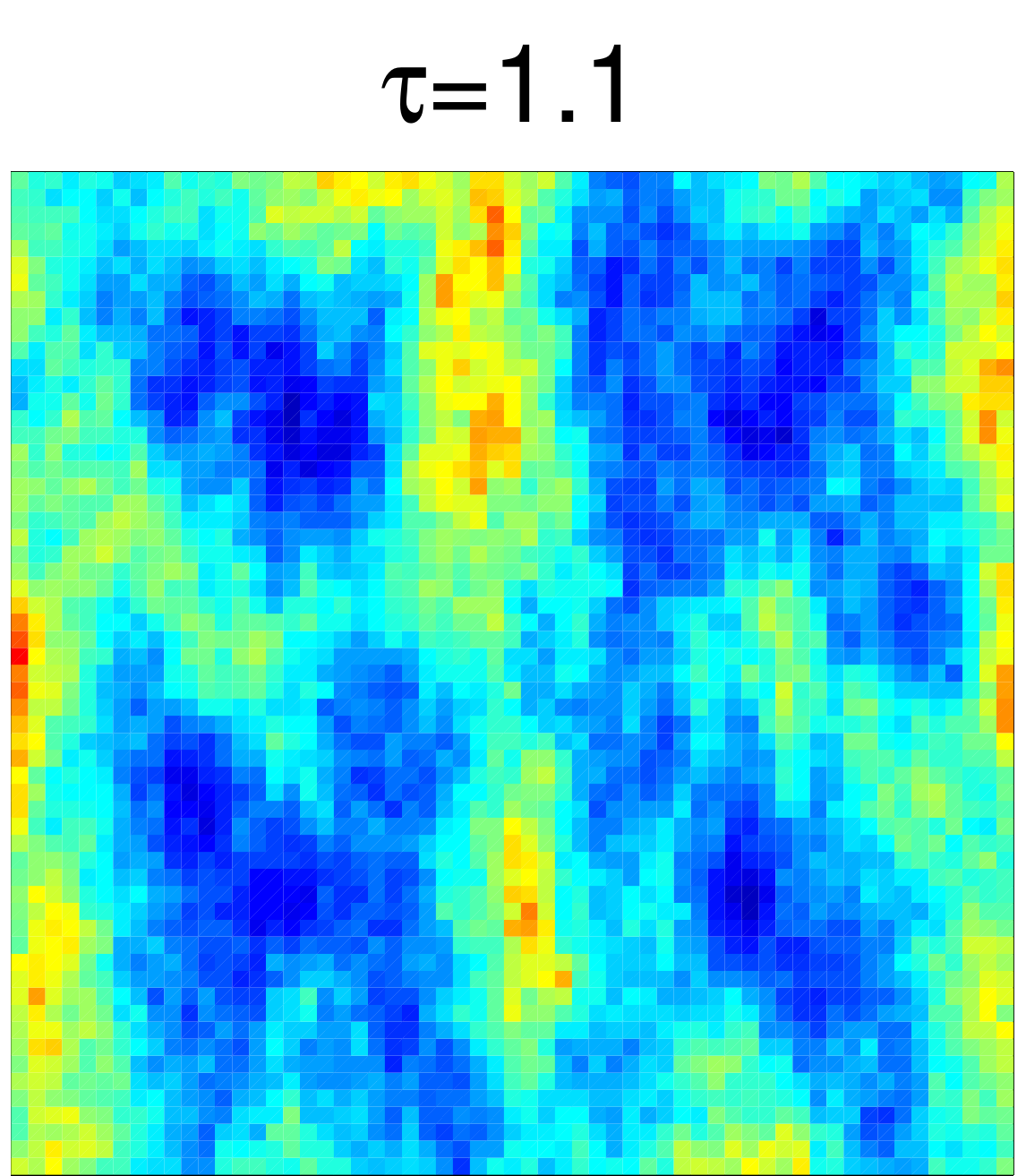}~
\includegraphics[scale=0.15]{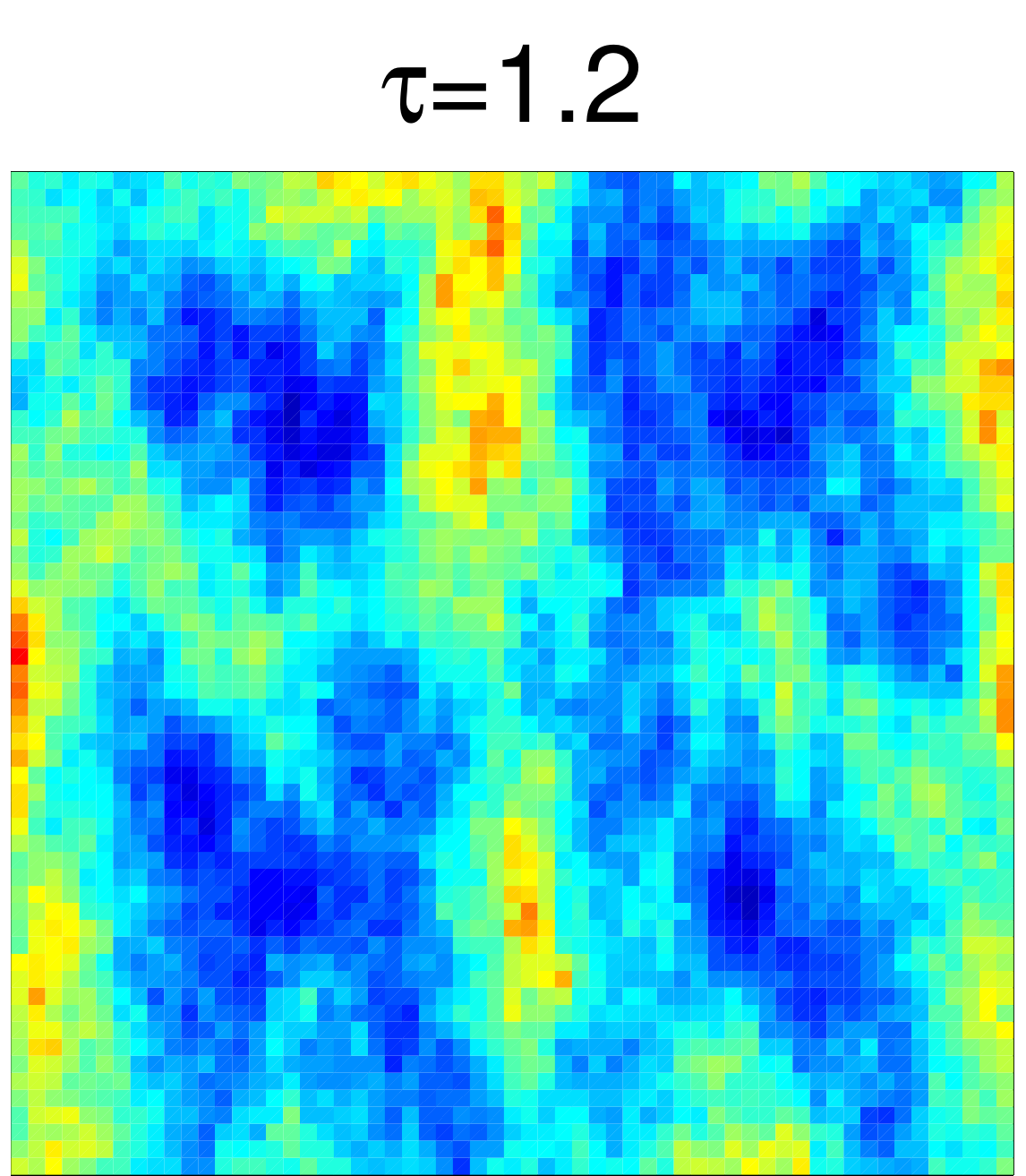}~
\includegraphics[scale=0.15]{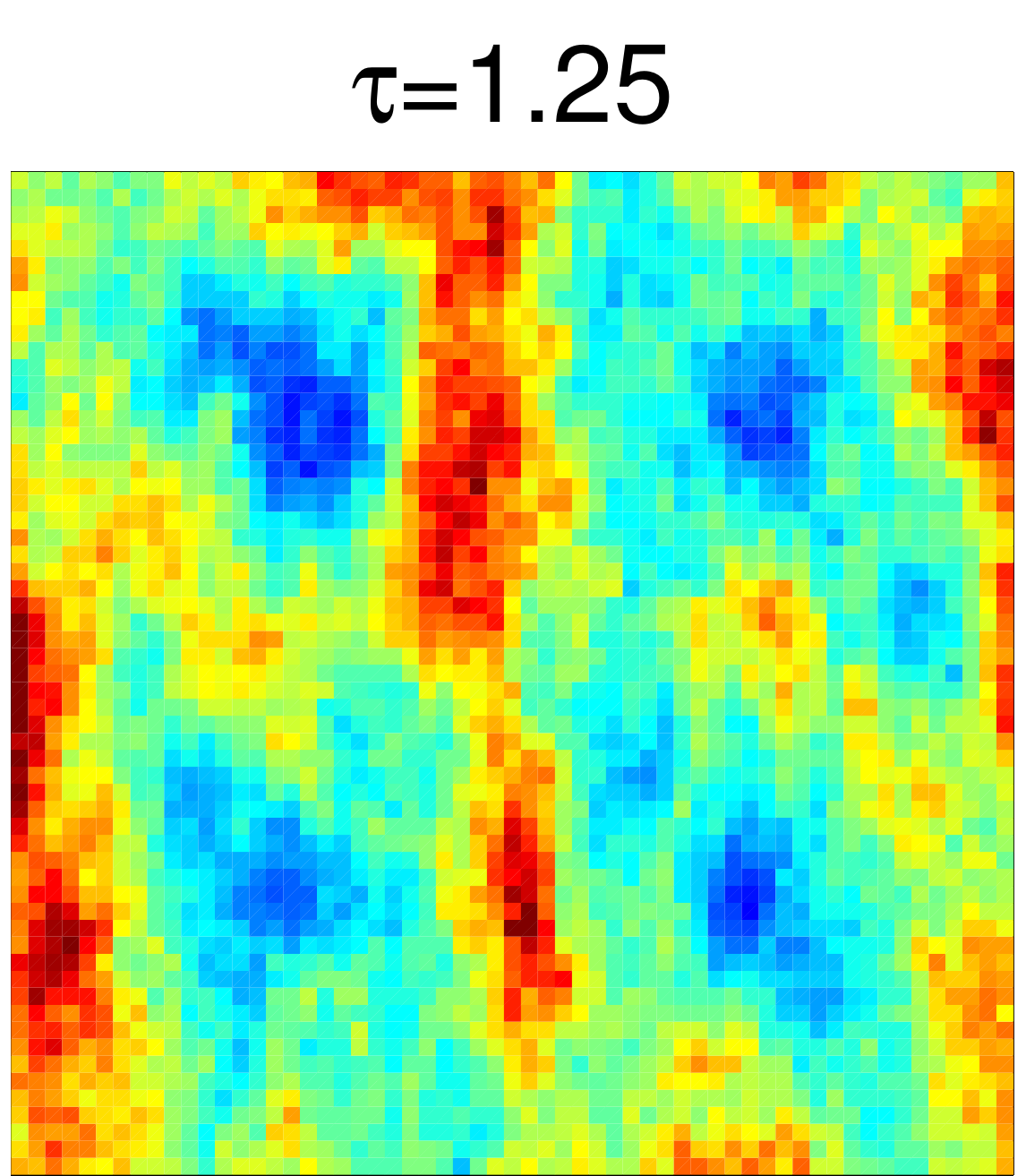}~
\includegraphics[scale=0.15]{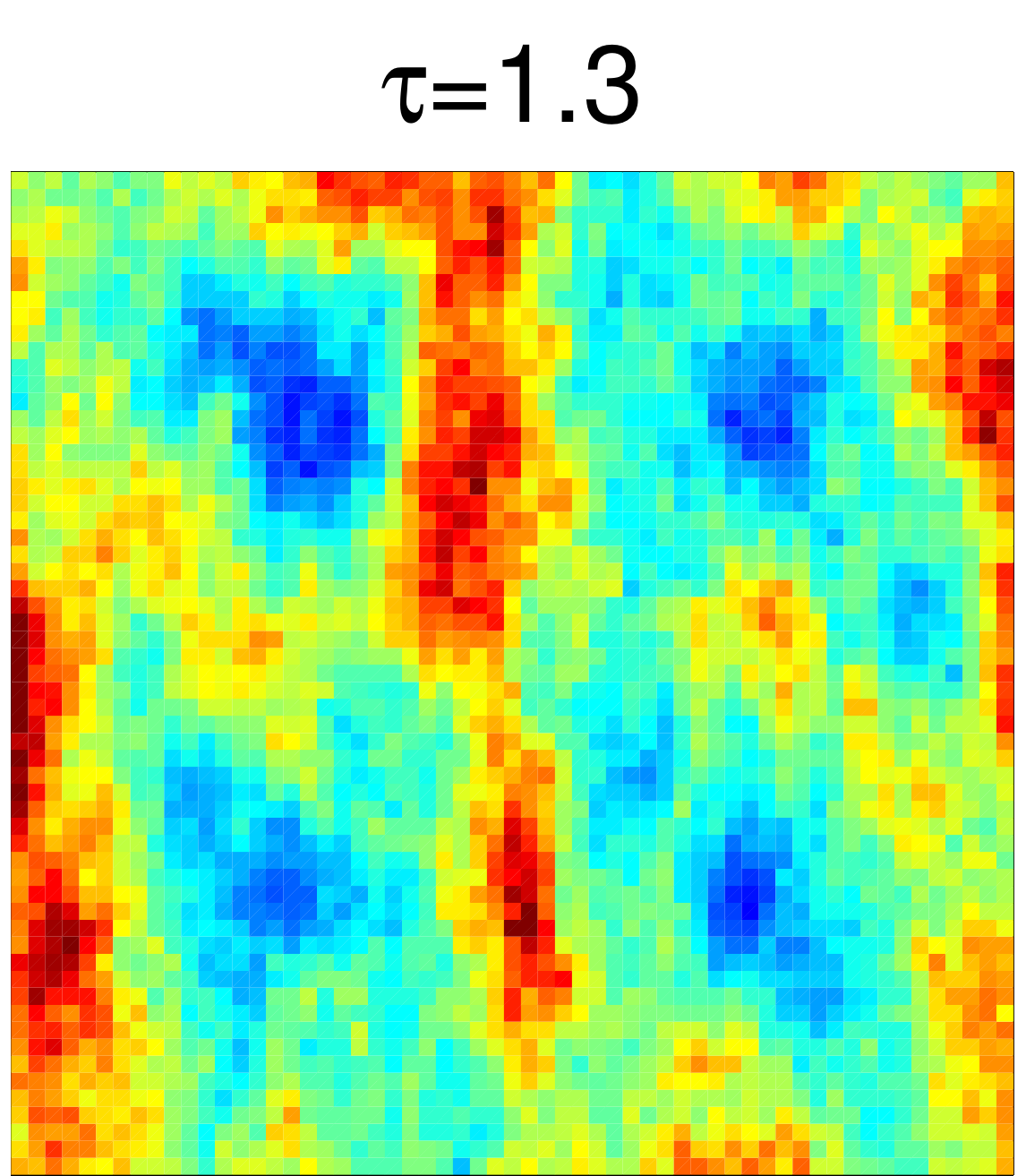}~
\includegraphics[scale=0.15]{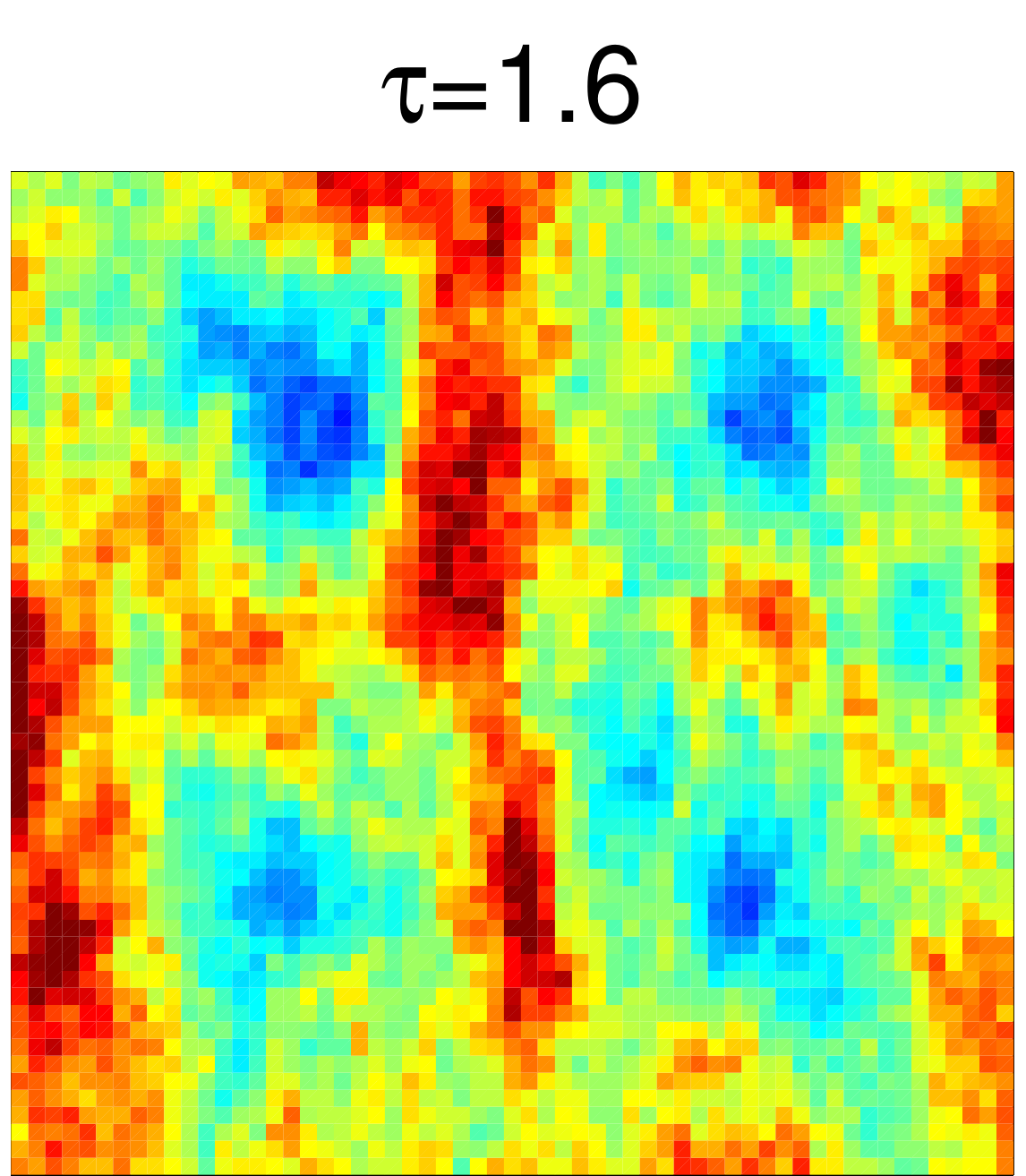}\\
\includegraphics[scale=0.225]{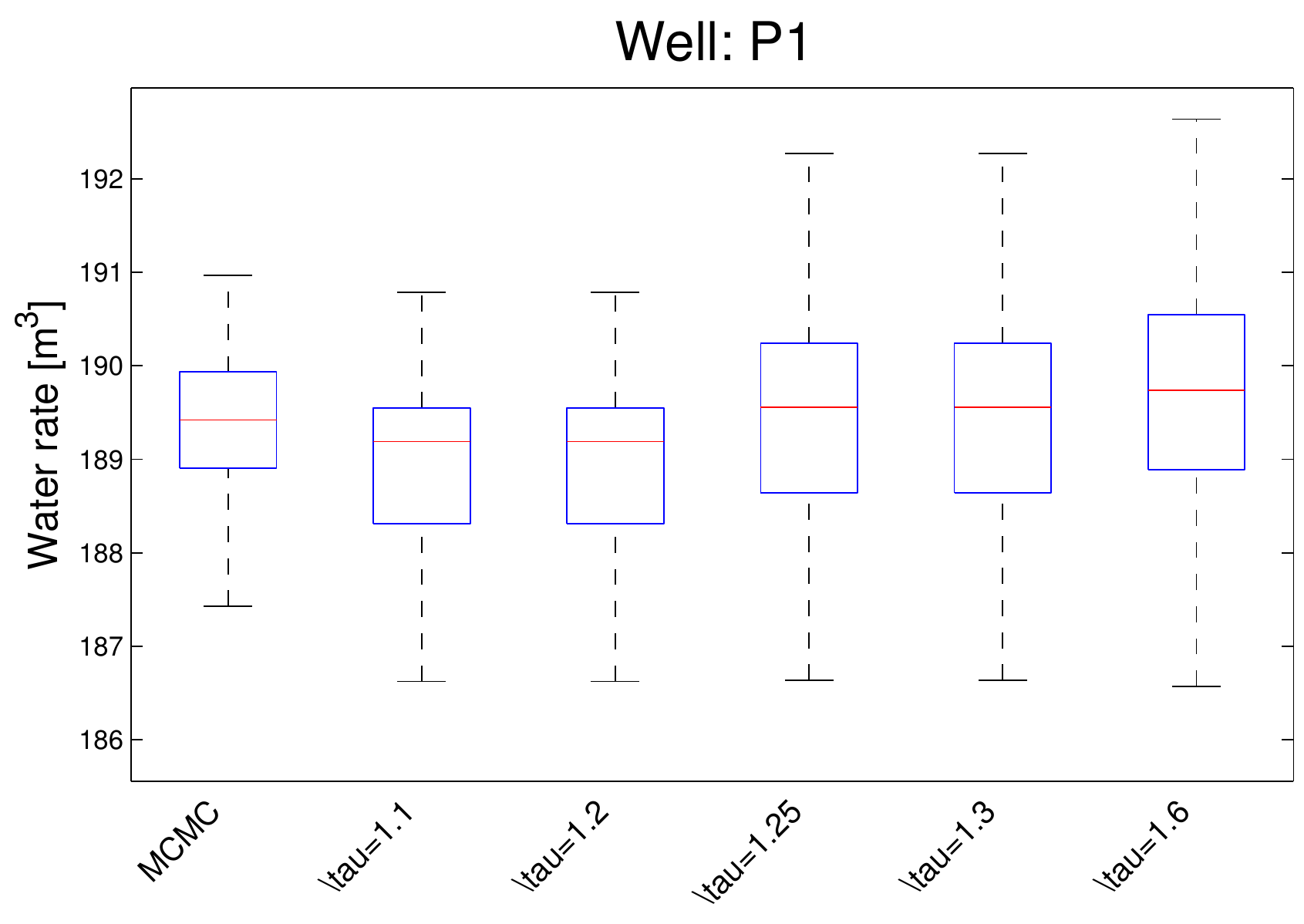}
\includegraphics[scale=0.225]{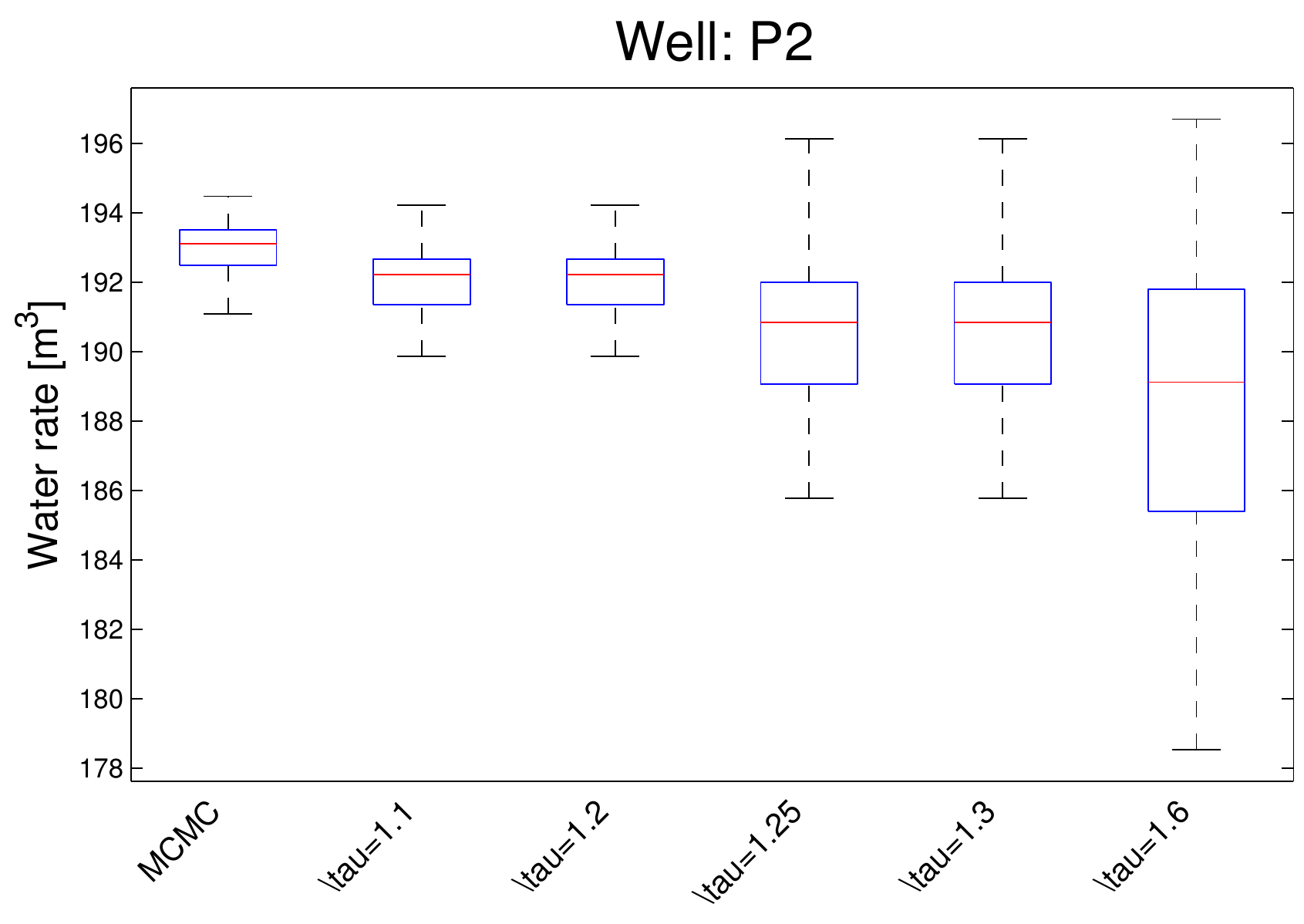}
\includegraphics[scale=0.225]{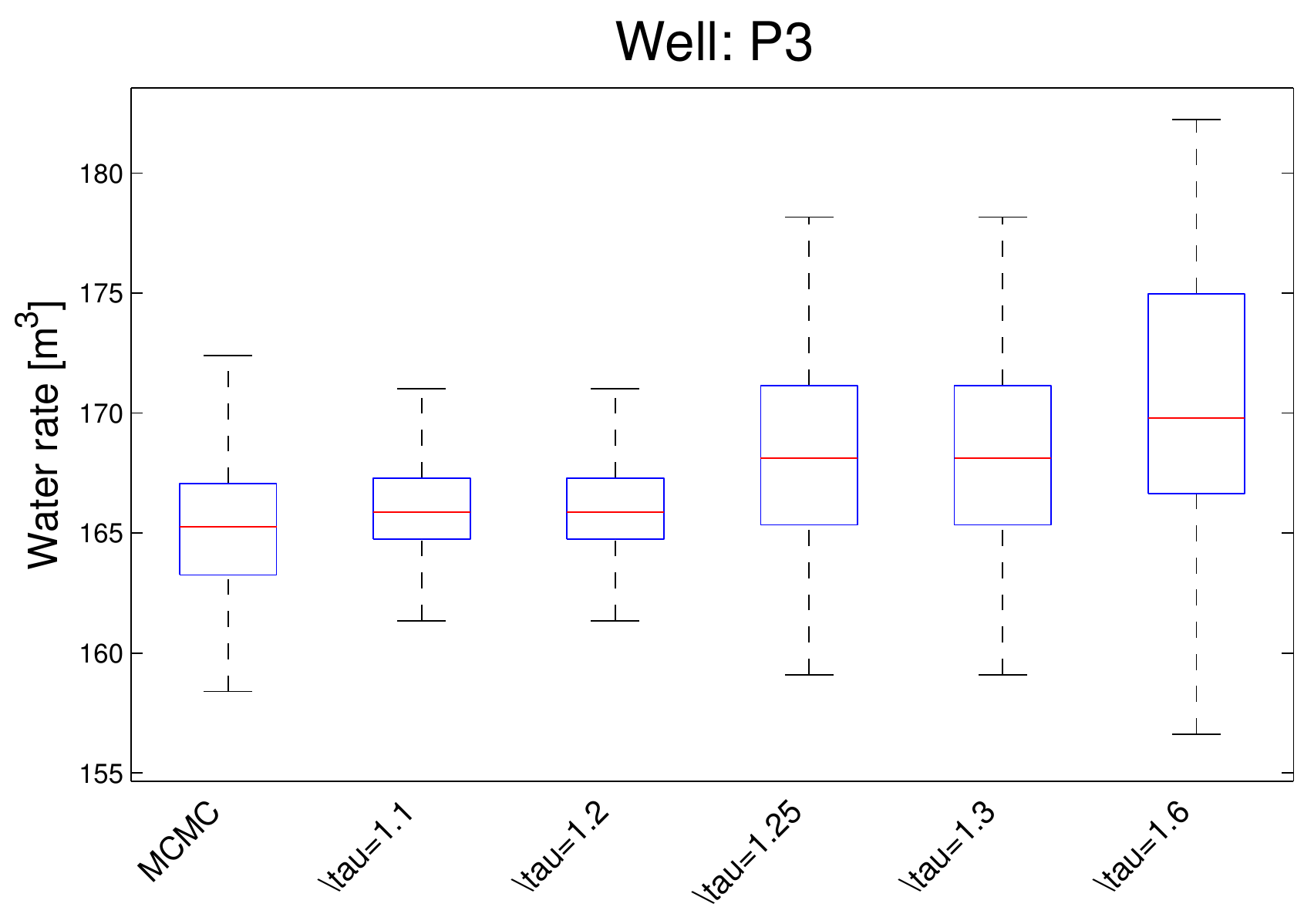}
\includegraphics[scale=0.225]{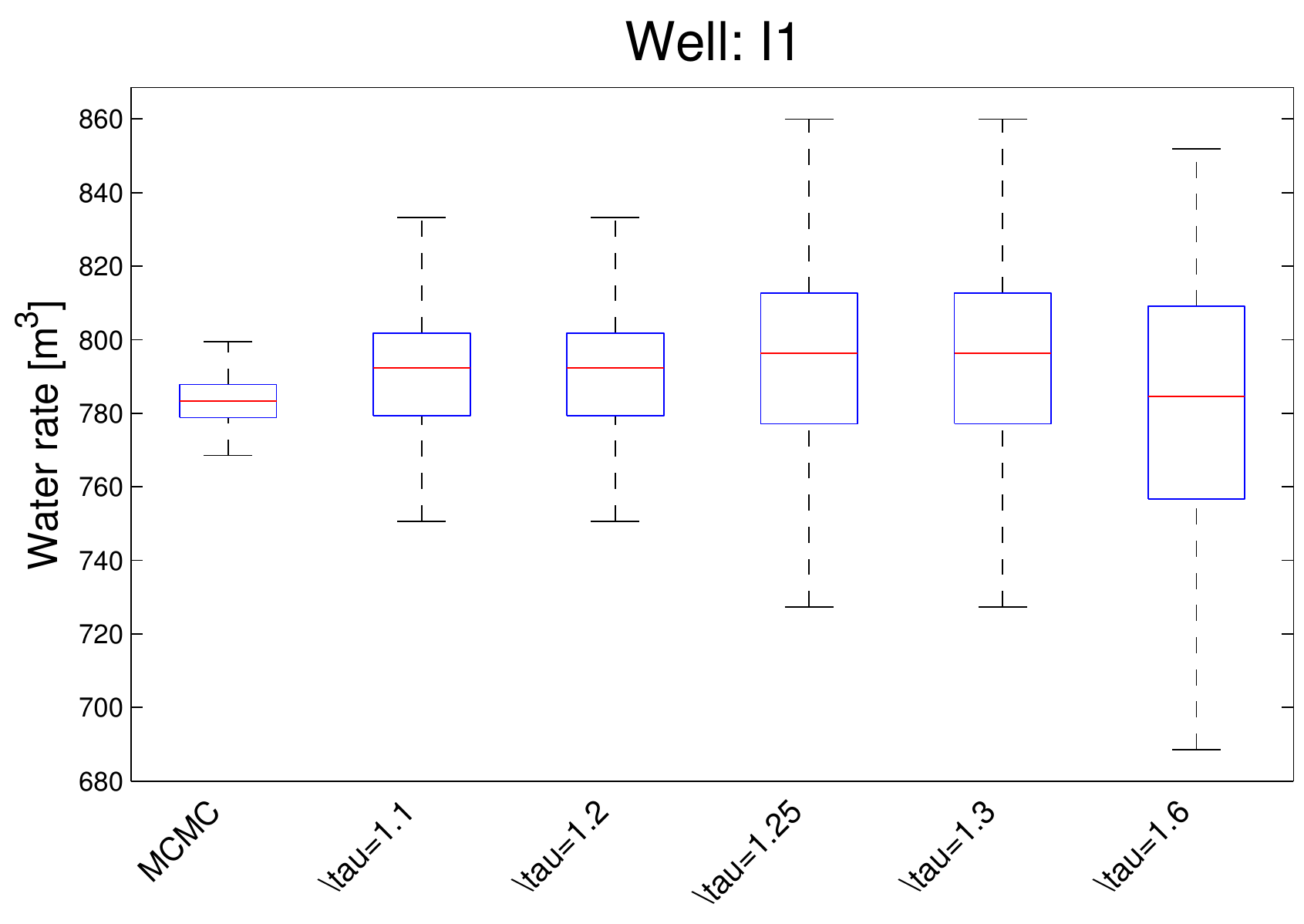}
\includegraphics[scale=0.225]{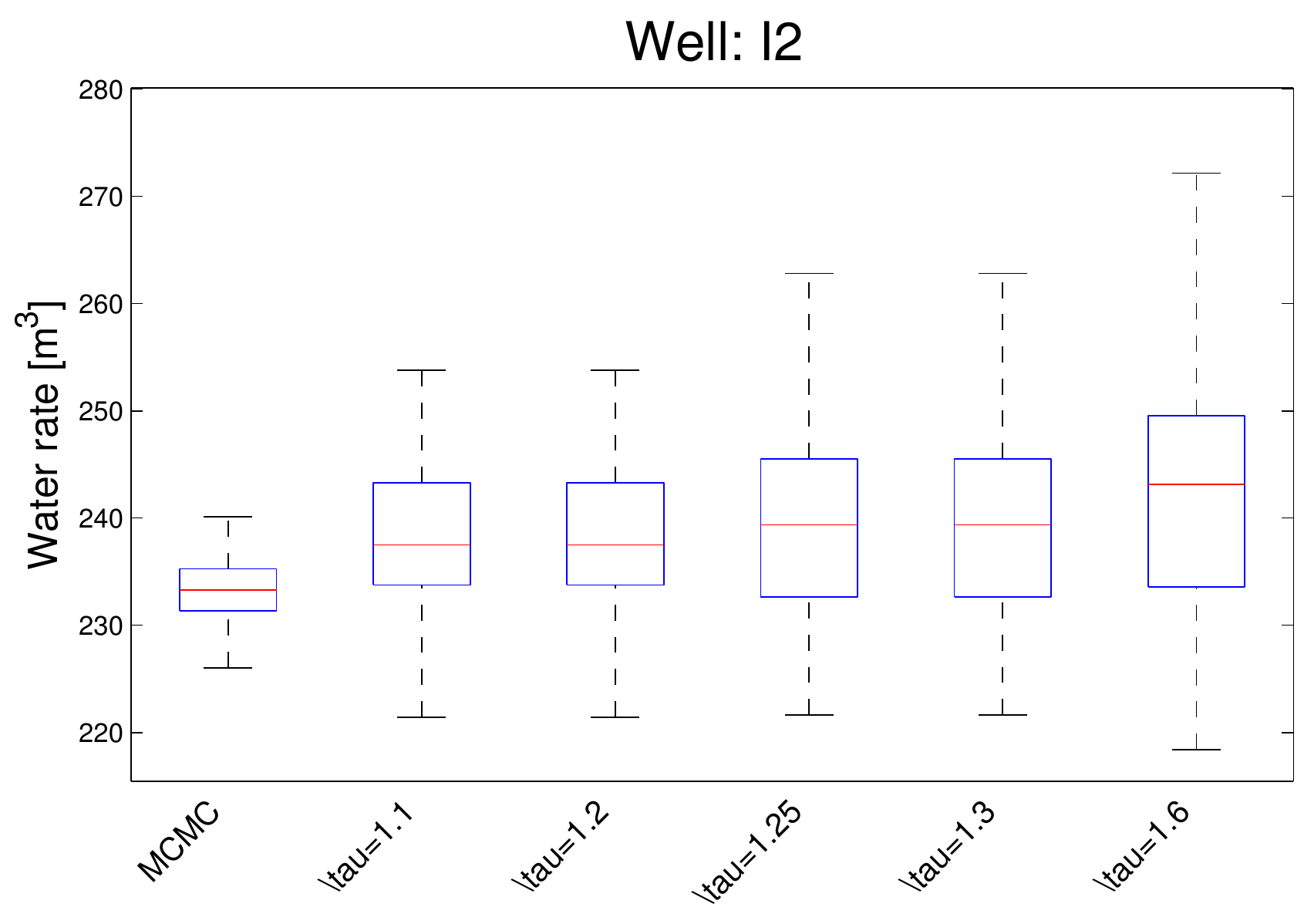}
\includegraphics[scale=0.225]{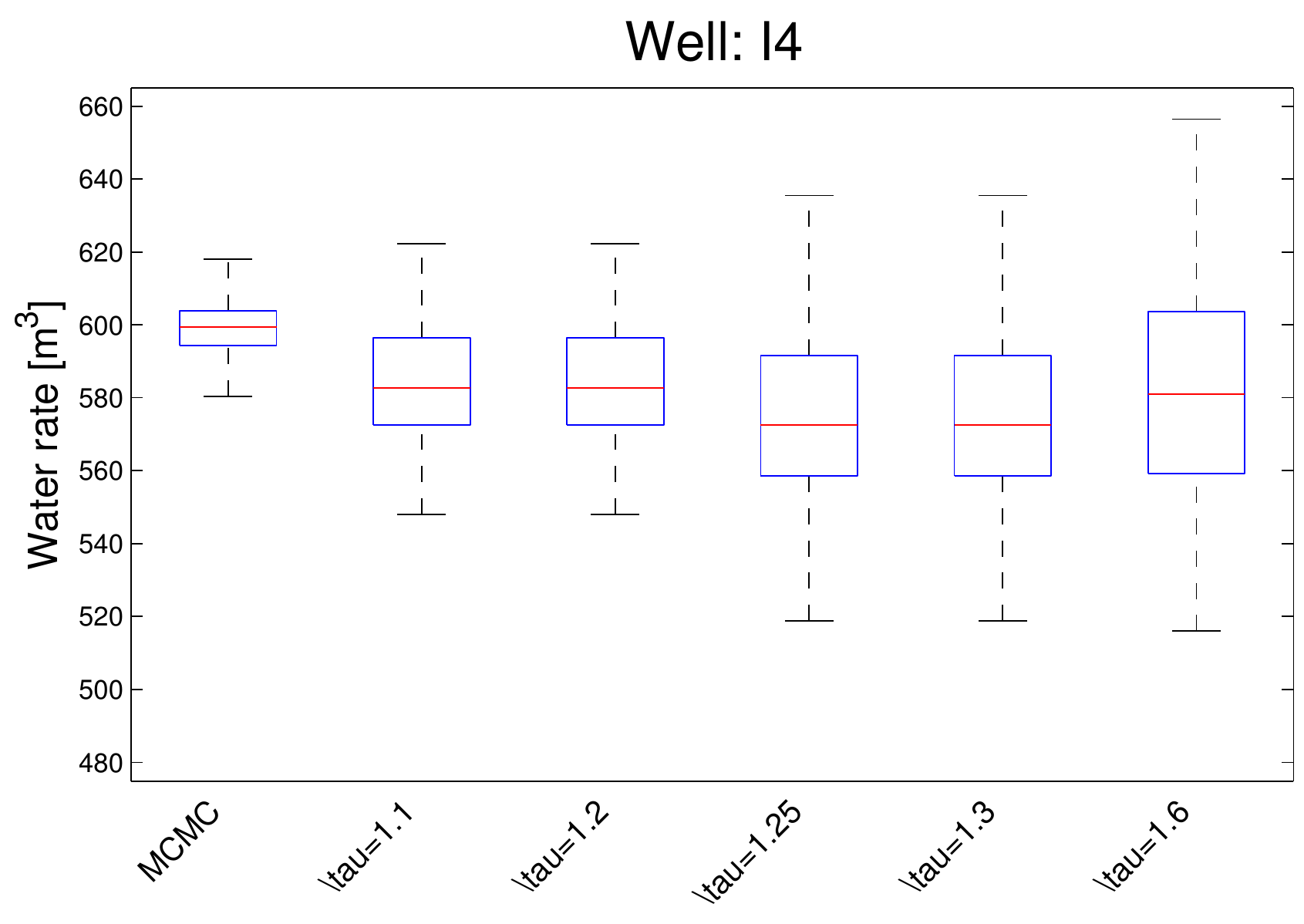}
\caption{Top row (mean of $\mu_B$) and Top-Middle row (variance of $\mu_B$) from left to right: pcn-MCMC, IR-ES approximations with $\rho=0.8$, $N_{e}=75$ and $\tau=1.1$, $\tau=1.2$, $\tau=1.25$, $\tau=1.3$, $\tau=1.6$. Middle-bottom and bottom row: Box plots of water rates from production wells $P_{1},P_{2}, P_{3}$ and PBH from injections wells $I_{1},I_{2},I_{4}$ after 6 years of water flood simulated from $\mu_{B}$ (with MCMC ) and the ensemble approximation with IR-ES for $\rho=0.8$ and different choices of $\tau$}\label{Figure7}
\end{center}
\end{figure}

\subsection{Comparison of the proposed methods versus unregularized standard methods}\label{comp}

In this section we compare the performance of IR-enLM and IR-ES with some standard unregularized methods. In particular, for different ensemble sizes $N_{e}$, in Table \ref{Table3A} we compare the proposed IR-enLM with a implementation of RML where each ensemble is computed with the unregularized LM algorithm described in subsection \ref{sec:RML-LM}. For the latter we consider the stopping criteria (\ref{eq:3.9D}) with $\epsilon_{0}=10^{-3}$ and $\epsilon_{1}=10^{-2}$. In addition, we use the standard recommendations \cite{Oliver,svdRML,YanChenLM} where the selection of $\lambda$ is given by (\ref{eq:1.26}) with $\lambda_{0}\equiv \Lambda_{0}=J(u_{0}^{(j)})/N_{D}$ and $\kappa=10$. For the IR-enLM results from Table \ref{Table3A}, different choices $\rho$ where consider with a selection of $\rho\ge 0.7$ and $\tau=1.0$ suggested from our discussion of subsection \ref{sec:numIR-enLM}. Table \ref{Table3A} shows clearly that the proposed methods, with our recommendations for the selection of the tunable parameters, outperformed RML (with standard unregularized LM method) in terms of approximating mean and variance of the posteriors $\mu_{A}$ and $\mu_{B}$. For each ensemble size, the IR-enLM method provides a similar level of approximation of the mean and variance for $\rho=0.7,0.8,0.9$. However, as discussed earlier, the  computational cost increases significantly as we increase $\rho$. Visual comparisons can be conducted from Figure \ref{Figure10A} (for $\mu_{A}$) and Figure \ref{Figure11A} (for $\mu_{B}$).

In Table \ref{Table3B} we present a comparison of the proposed IR-ES with the standard unregularized ES implementation described in subsection \ref{sec:ES}. For a fixed choice $M_{ES}=10$, we use different choices of the ensemble size $N_{e}$ and tunable parameter $\rho$. From our discussion of subsection \ref{sec:numIR-ES}, for each $\rho$, the IR-ES results from Table \ref{Table3B} correspond to the tunable parameter $\tau=1/\rho$. For each ensemble size, the IR-ES produces a better approximation that the corresponding standard ES. However, similar to the IR-enLM, the computational cost of IR-ES increases with $\rho$. For some of these $\rho$'s, we may be able to use an ensemble size $N_{e}$ for which the cost of ES matches the one of IR-ES but with increased accuracy. For example, the computational cost of IR-ES for $\rho=0.9$ and $N_{e}=25$ is 75 forward model runs which matches the cost of ES with $N_{e}=75$. In this case, ES provides more accurate approximations of the mean that our proposed method. However, the computational cost of IR-ES for $\rho=0.7$ and $N_{e}=75$ is also 75 forward model runs and provides more accurate results both in terms of mean and variance. Therefore, even though IR-ES provides better approximations for a given ensemble size, the selection of $\rho$ is fundamental for the computational efficiency of the method. Table \ref{Table3B} suggests that $\rho=0.7$, $M_{ES}=10$, offers a reasonable compromise between accuracy and cost for each ensemble size. Some visual comparisons are displayed in Figure \ref{Figure10B} (for $\mu_{A}$) and Figure \ref{Figure11B} (for $\mu_{B}$).

\begin{figure}
\begin{center}
\includegraphics[scale=0.2]{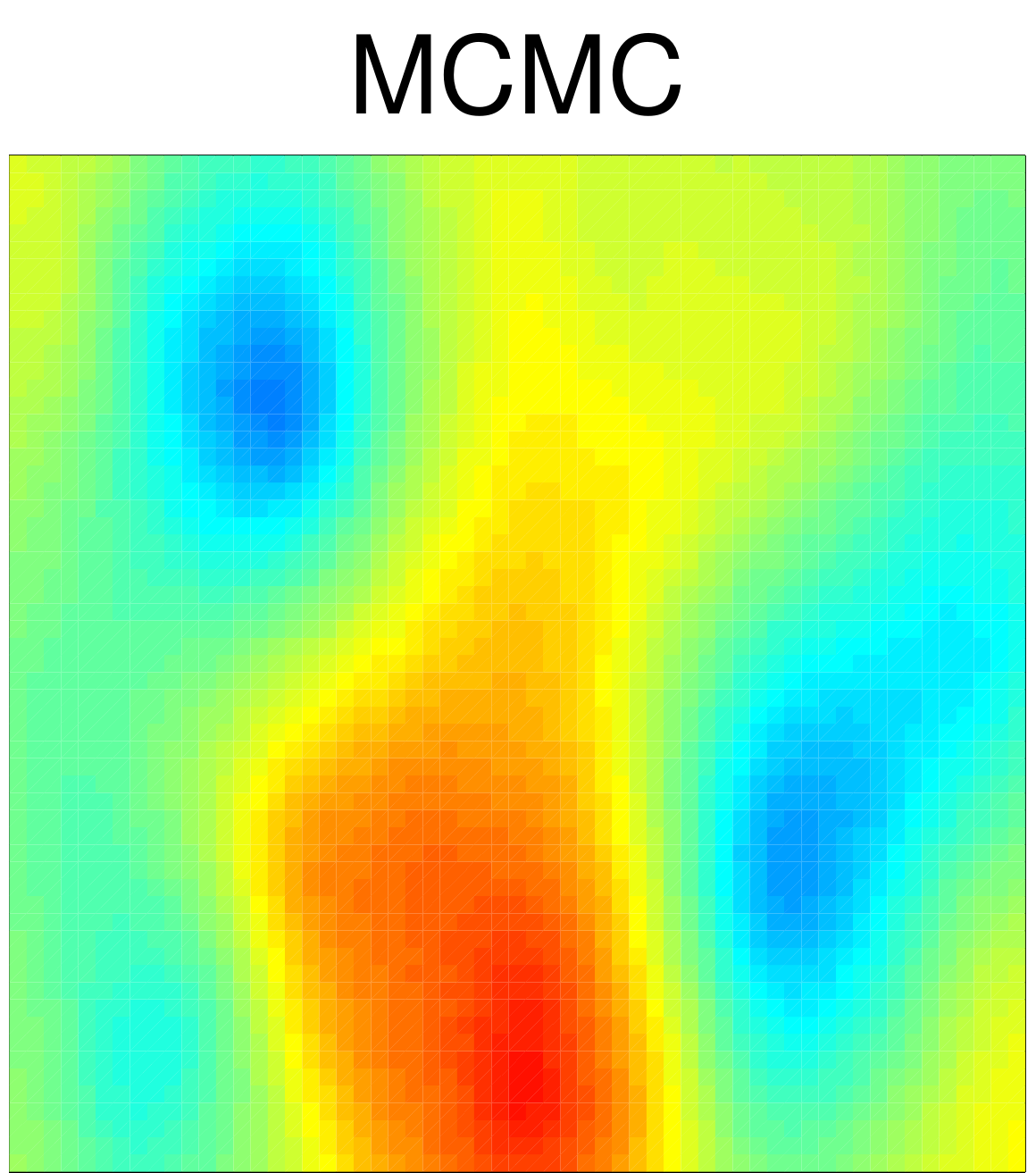}~
\includegraphics[scale=0.2]{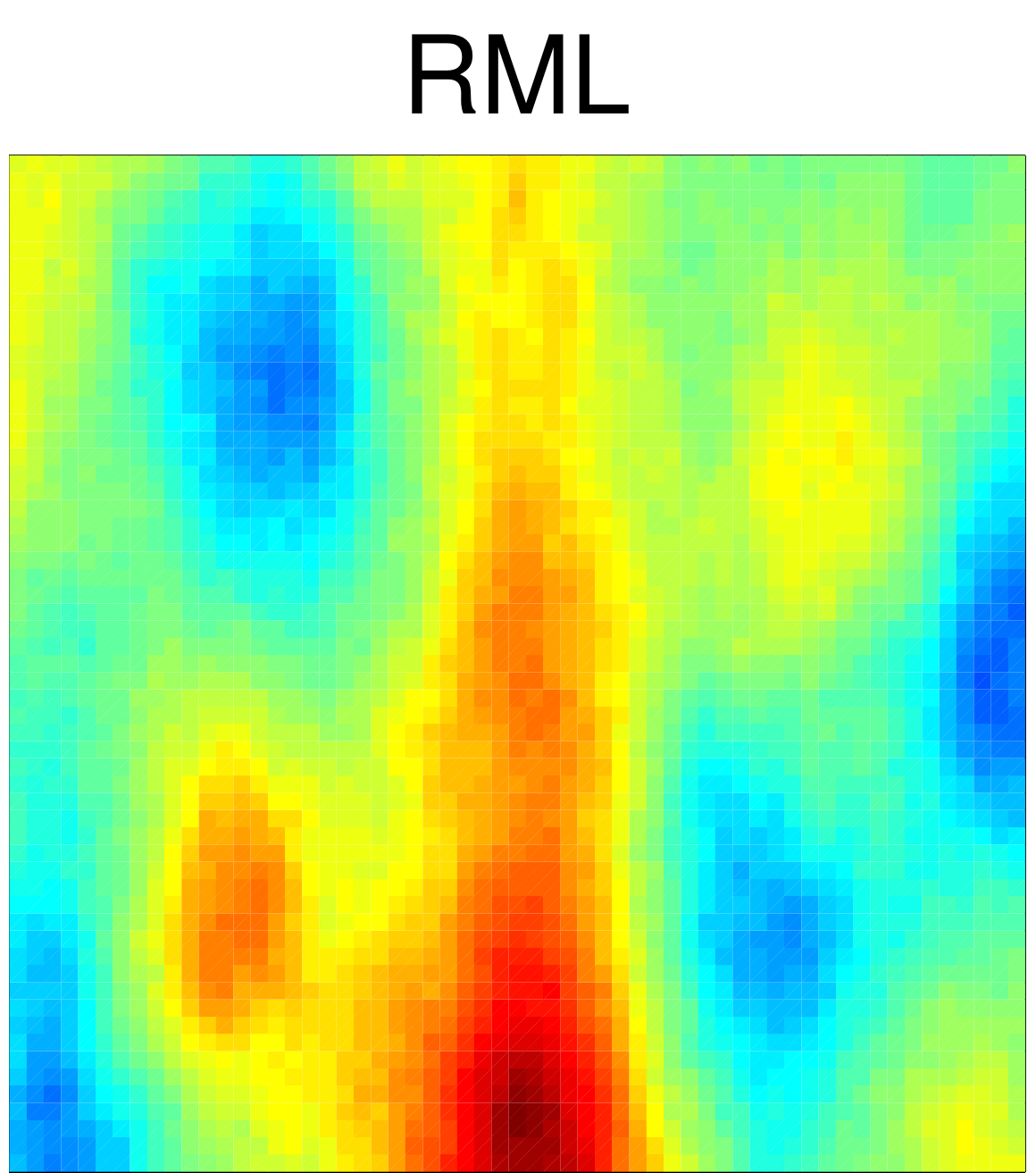}~
\includegraphics[scale=0.2]{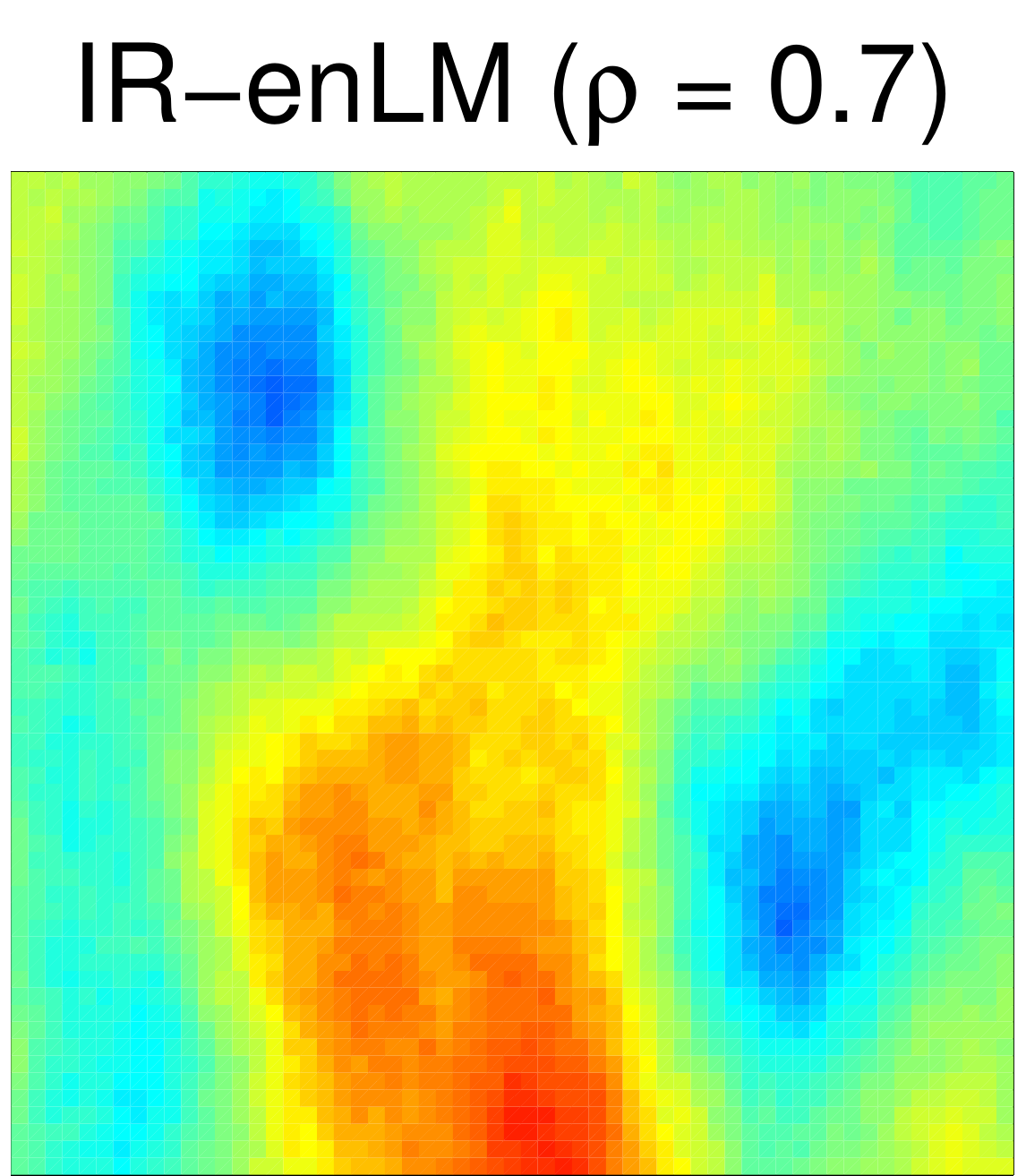}~
\includegraphics[scale=0.2]{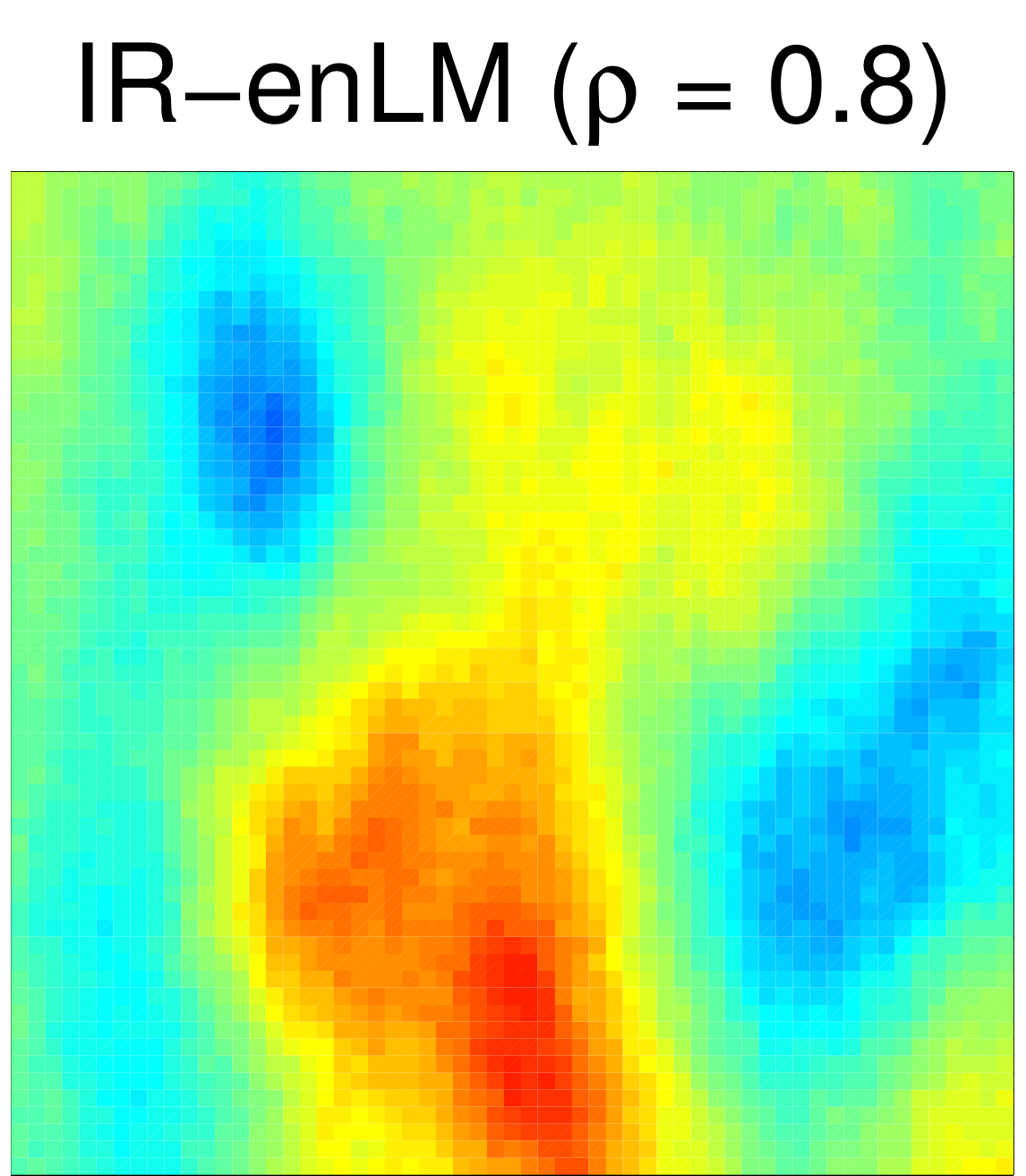}~
\includegraphics[scale=0.2]{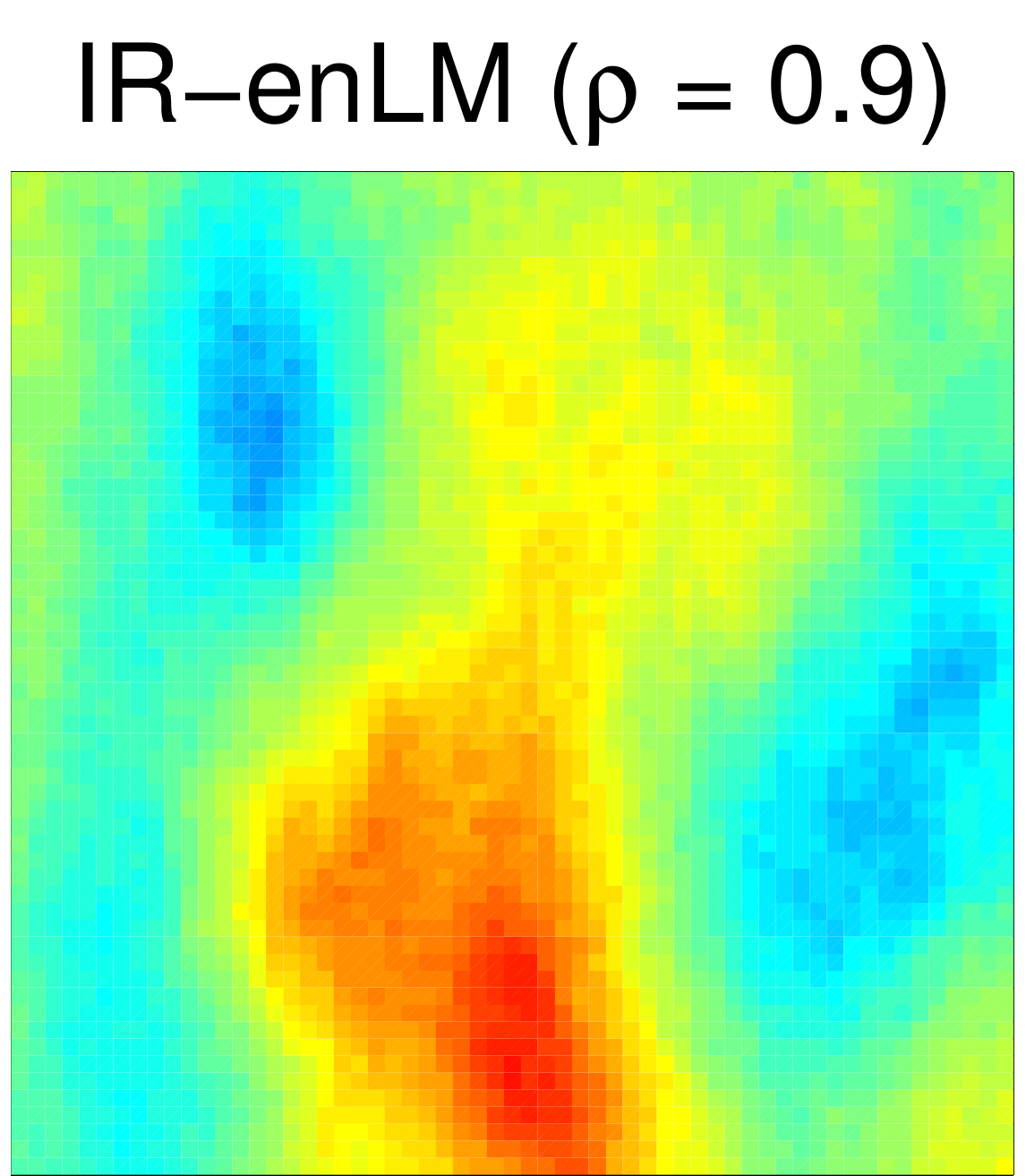}\\
\includegraphics[scale=0.2]{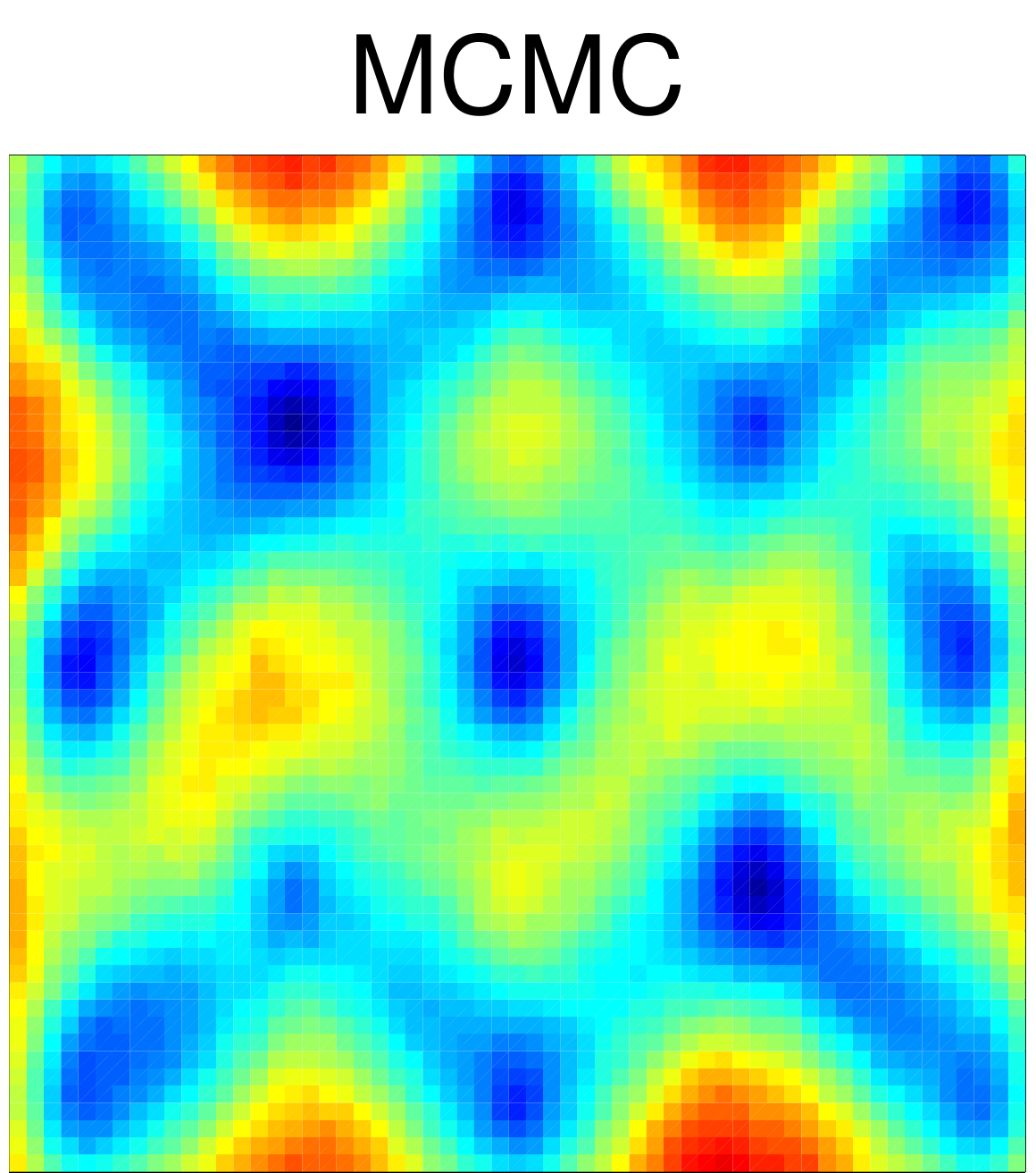}~
\includegraphics[scale=0.2]{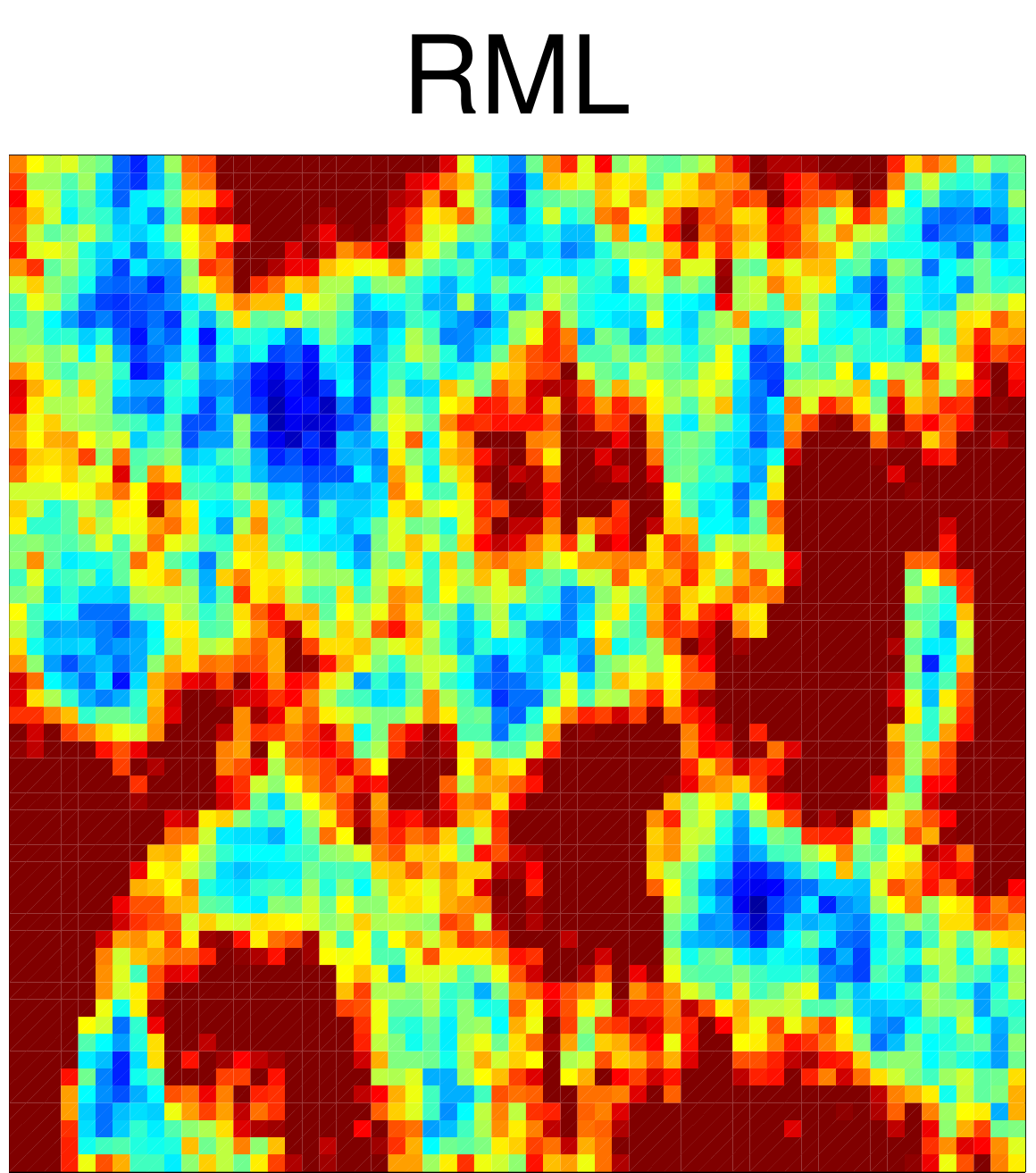}~
\includegraphics[scale=0.2]{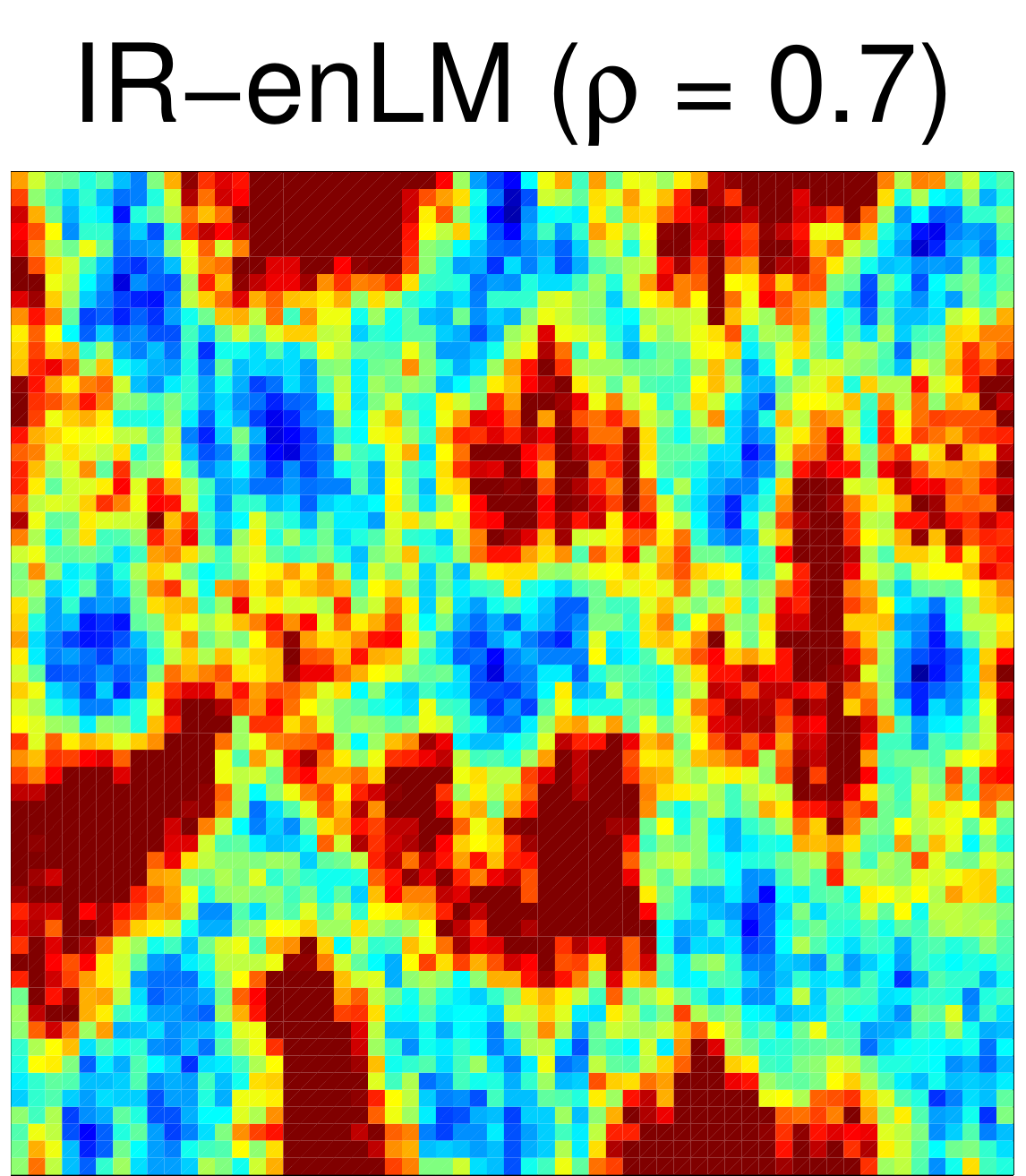}~
\includegraphics[scale=0.2]{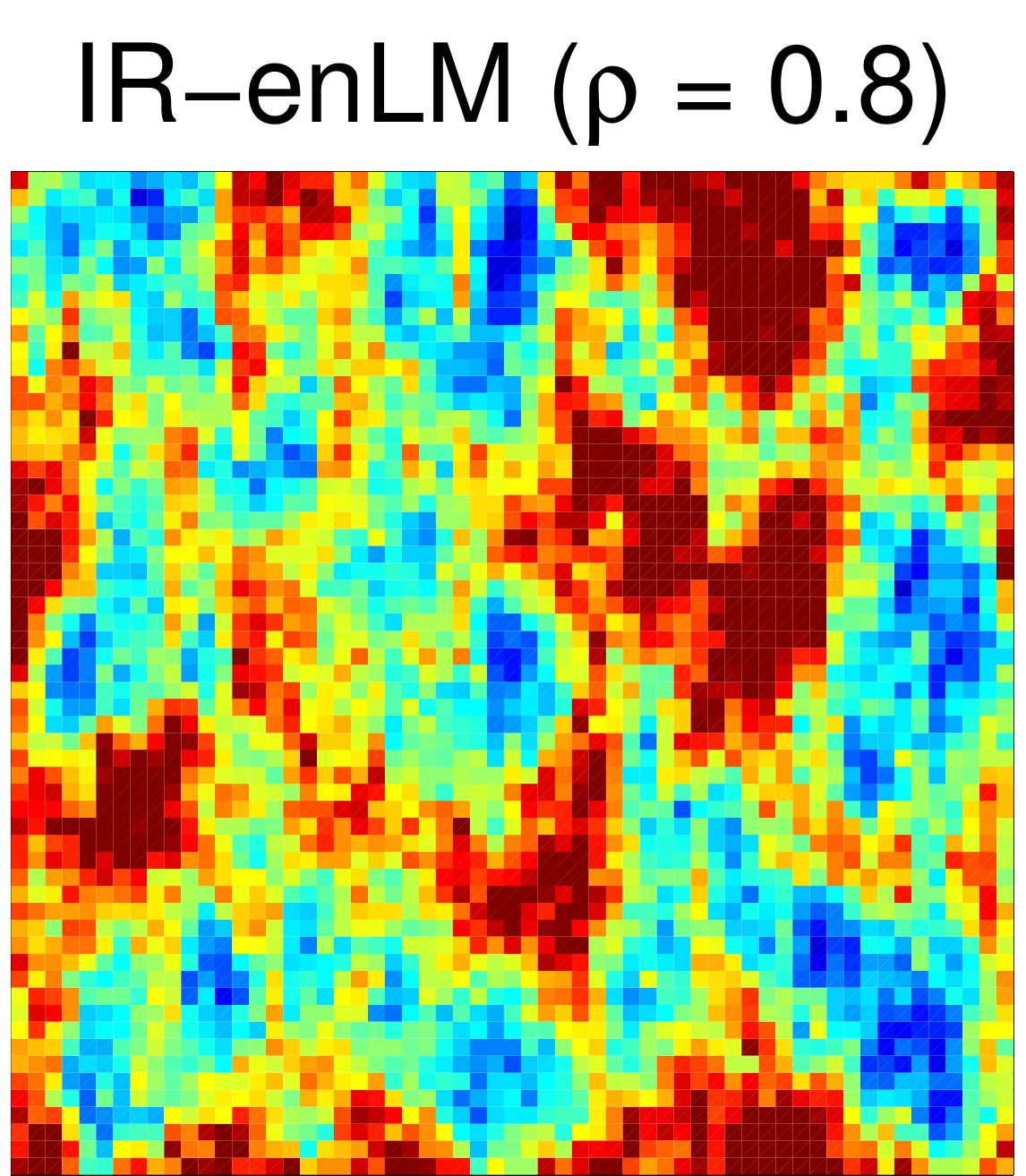}~
\includegraphics[scale=0.2]{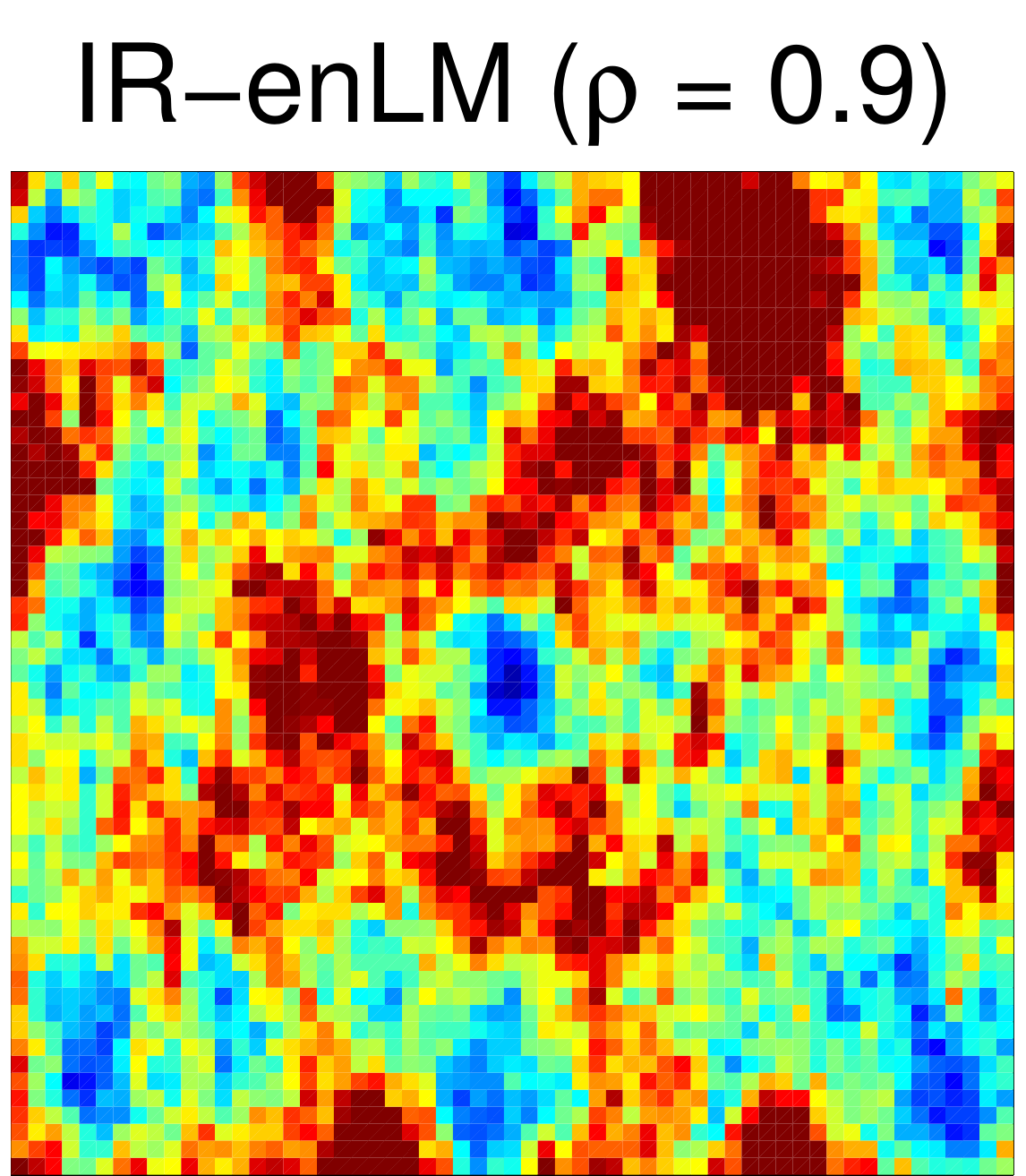}
\end{center}
\caption{ Mean (top) and Variance (bottom) of the posterior distribution $\mu_{A}$ (characterized with MCMC) and ensemble approximations RML and IR-enLM with $N_{e}=50$.}  
\label{Figure10A}
\end{figure}

\begin{figure}
\begin{center}
\includegraphics[scale=0.2]{Mean_MCMC1BA}~
\includegraphics[scale=0.2]{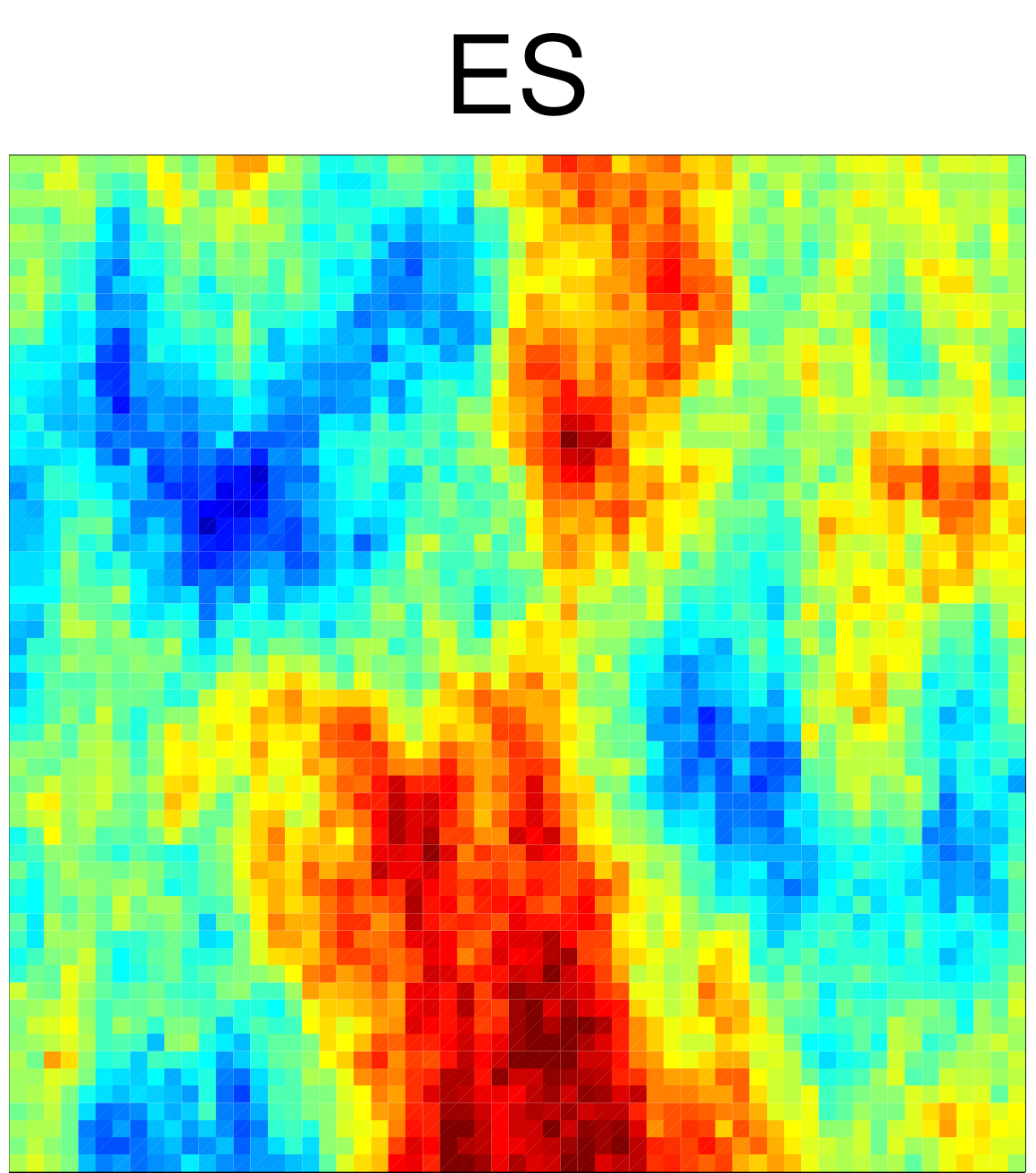}~
\includegraphics[scale=0.2]{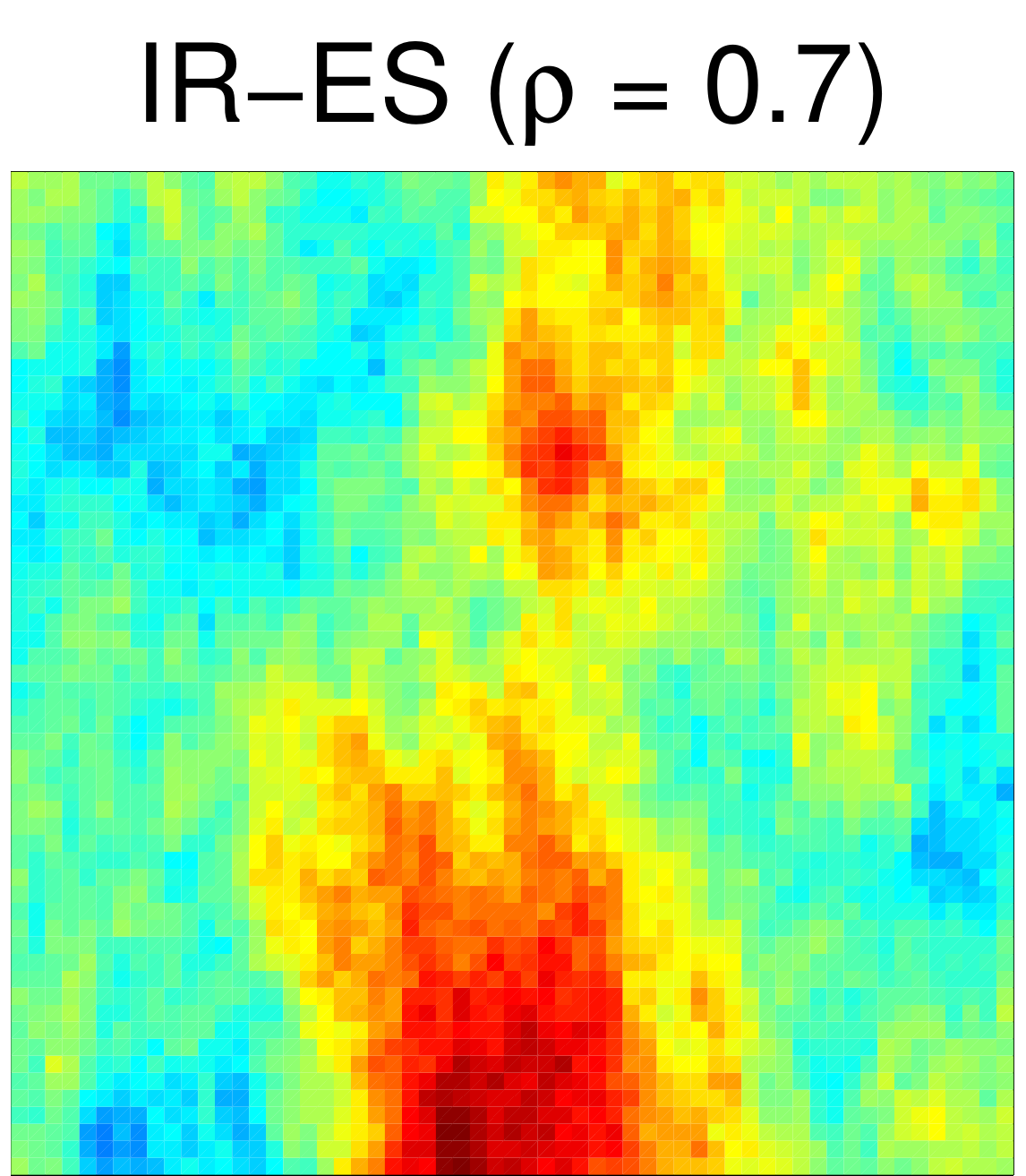}~
\includegraphics[scale=0.2]{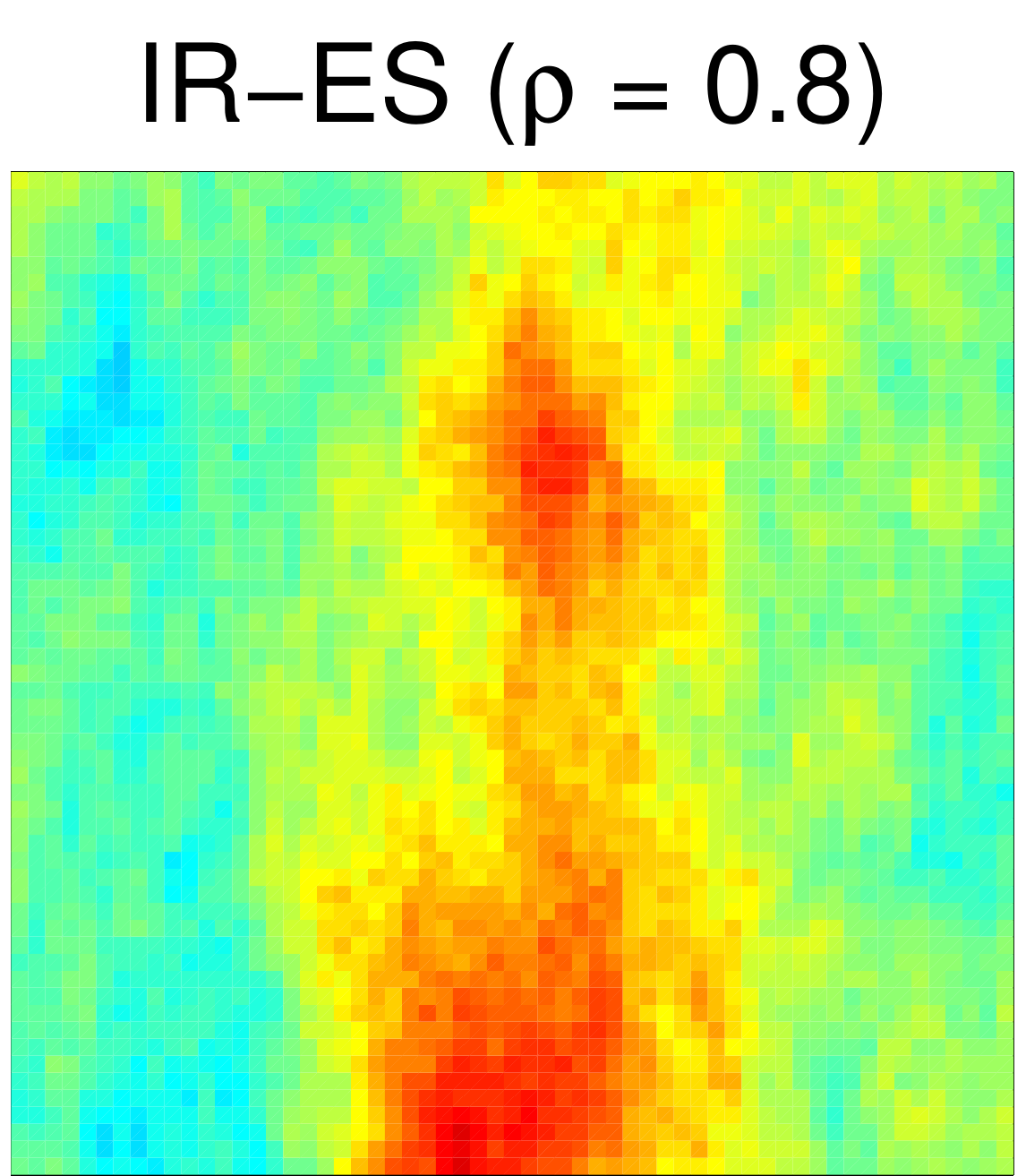}~
\includegraphics[scale=0.2]{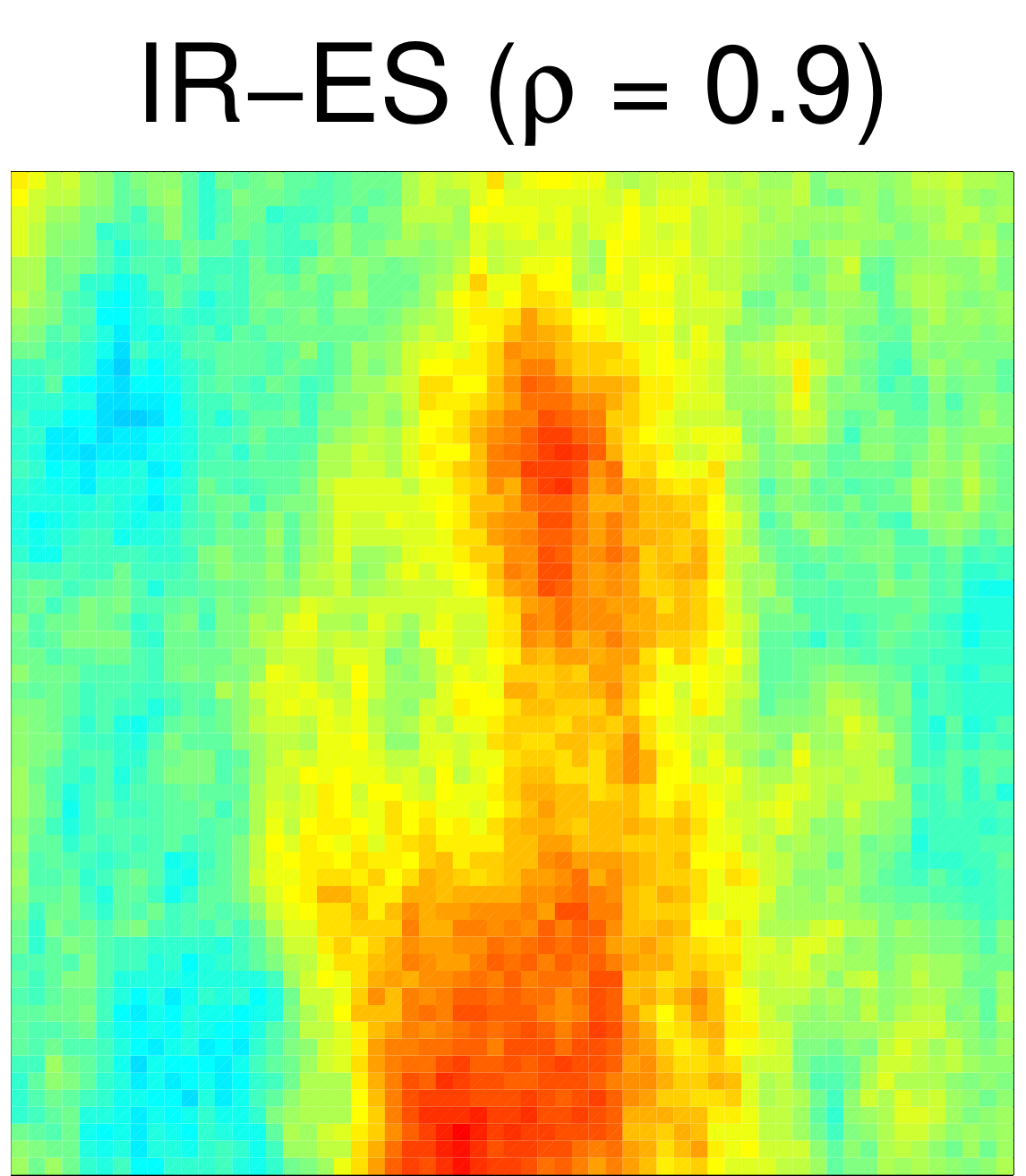}~\\
\includegraphics[scale=0.2]{Var_MCMC1BA}~
\includegraphics[scale=0.2]{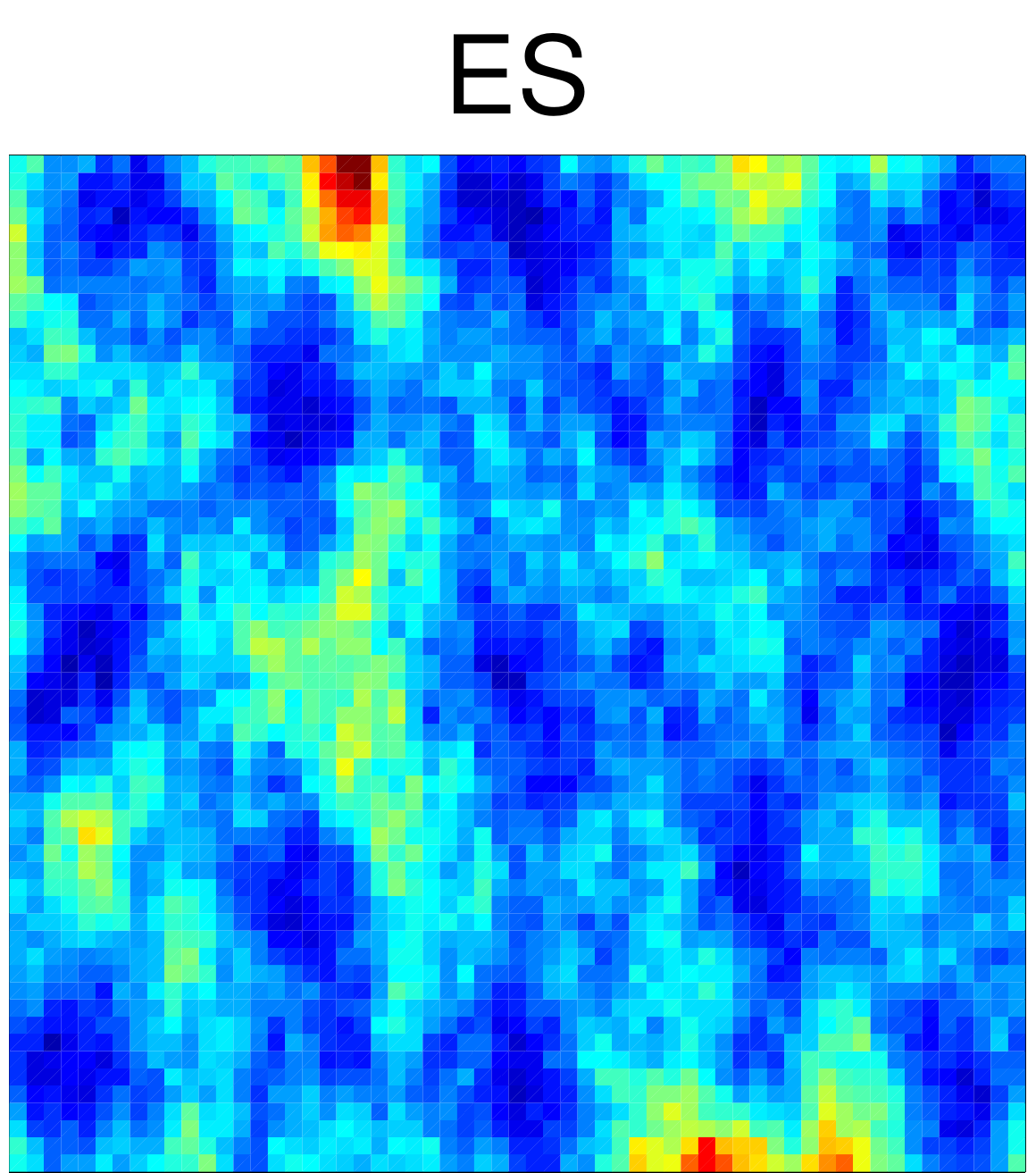}~
\includegraphics[scale=0.2]{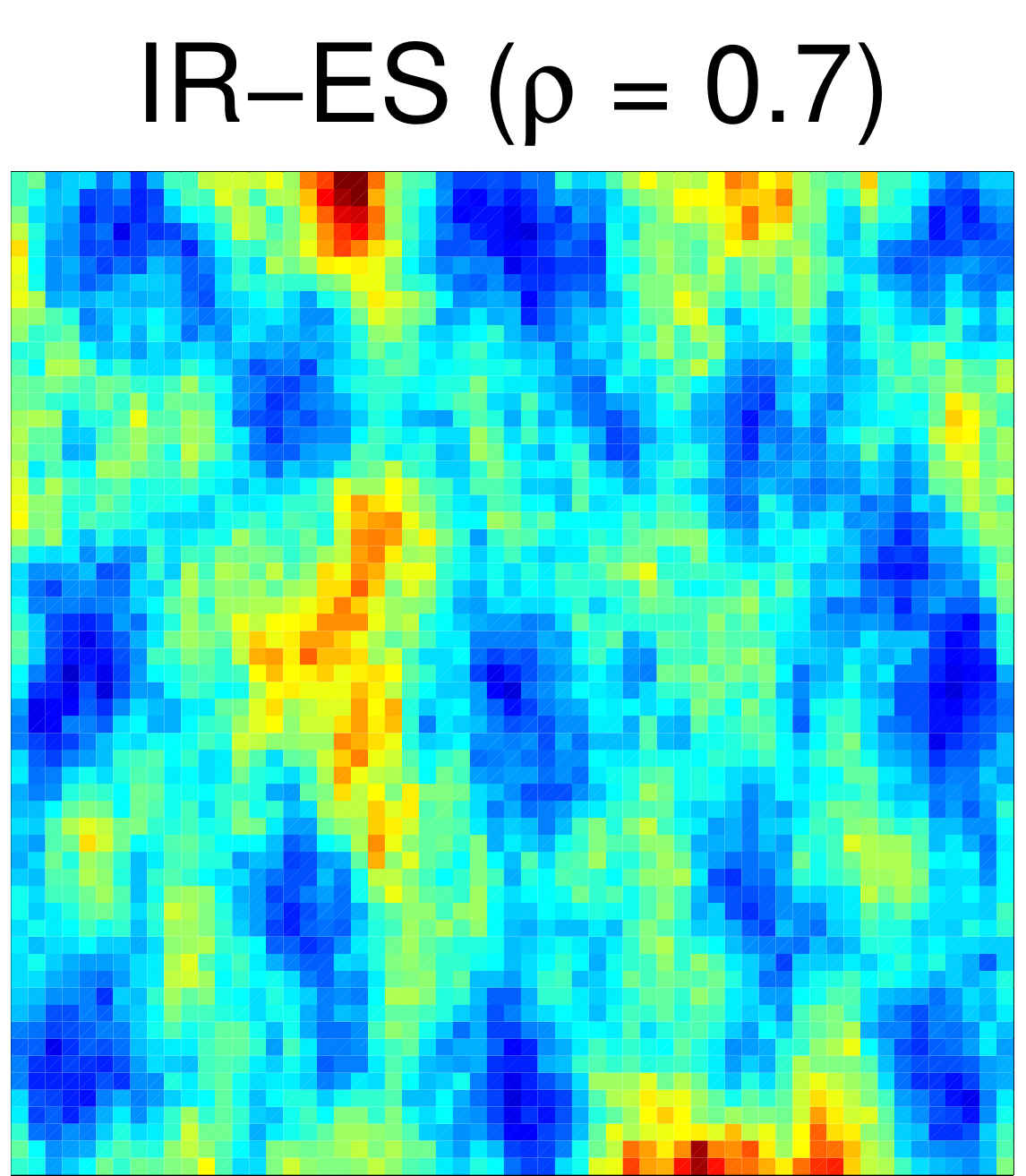}~
\includegraphics[scale=0.2]{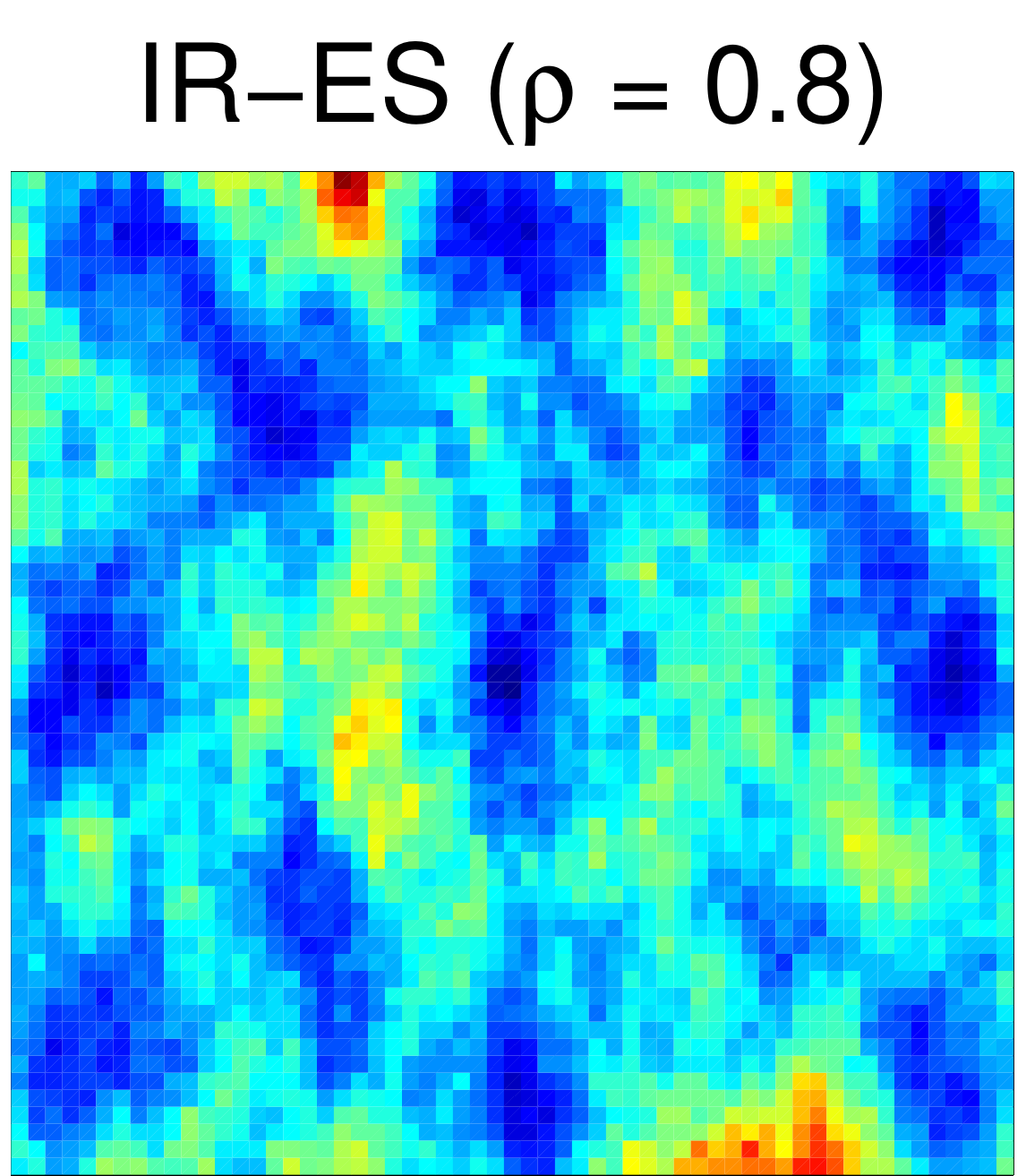}~
\includegraphics[scale=0.2]{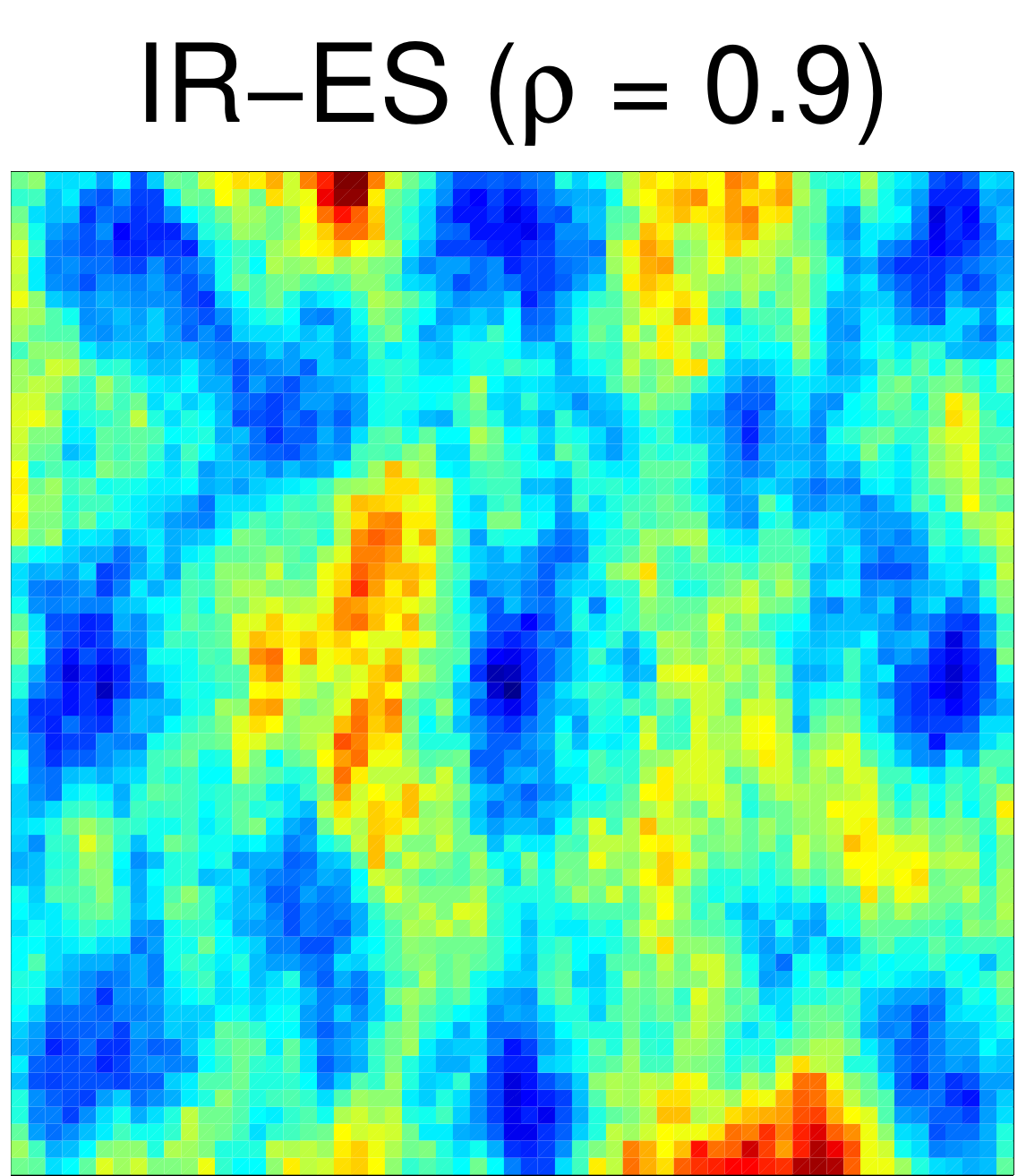}~
\end{center}
\caption{ Mean (top) and Variance (bottom) of the posterior distribution $\mu_{A}$ (characterized with MCMC) and ensemble approximations ES and IR-ES with $N_{e}=75$.}  
\label{Figure10B}
\end{figure}

\begin{figure}
\begin{center}
\includegraphics[scale=0.2]{Mean_MCMC2BB}~
\includegraphics[scale=0.2]{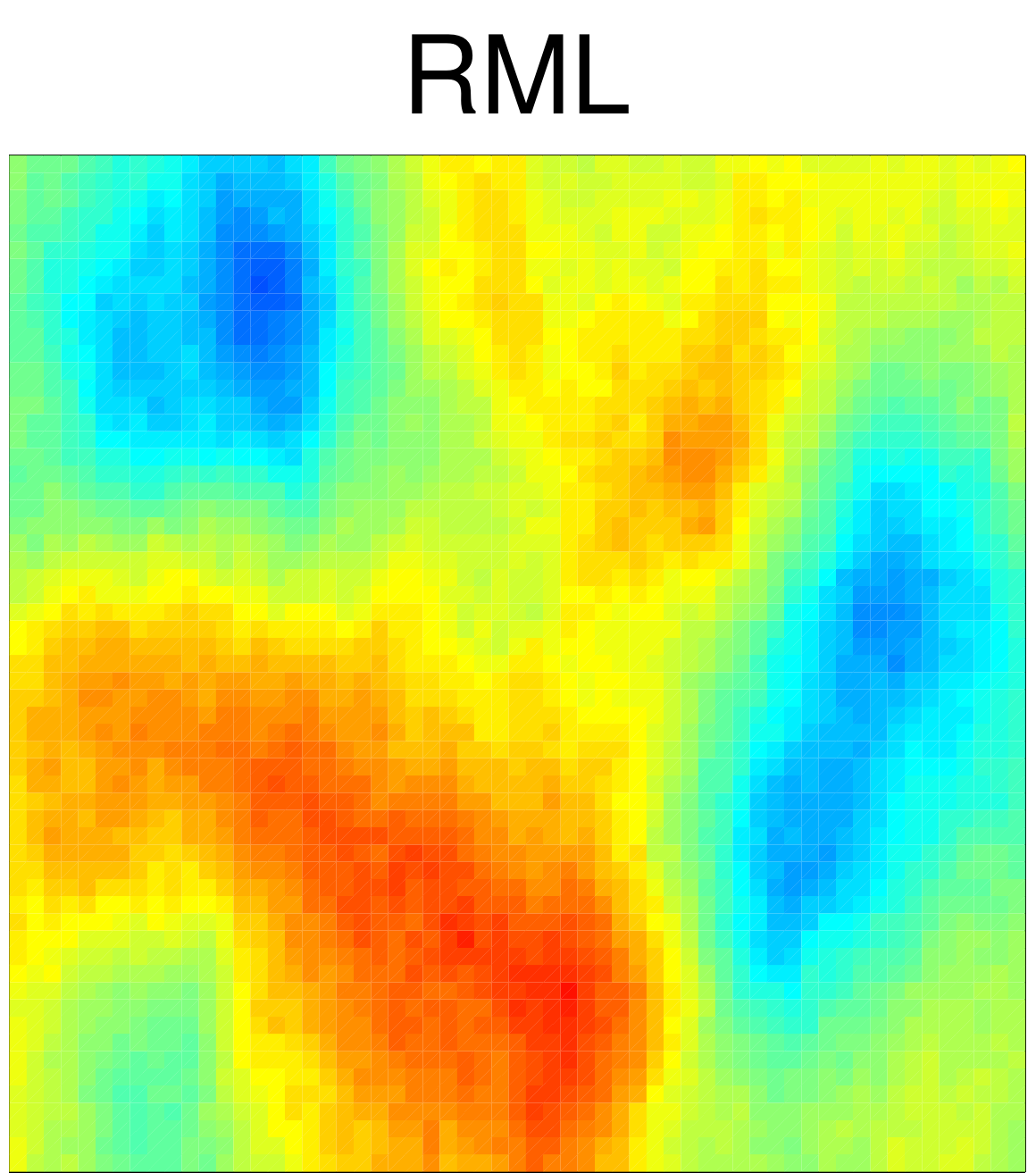}~
\includegraphics[scale=0.2]{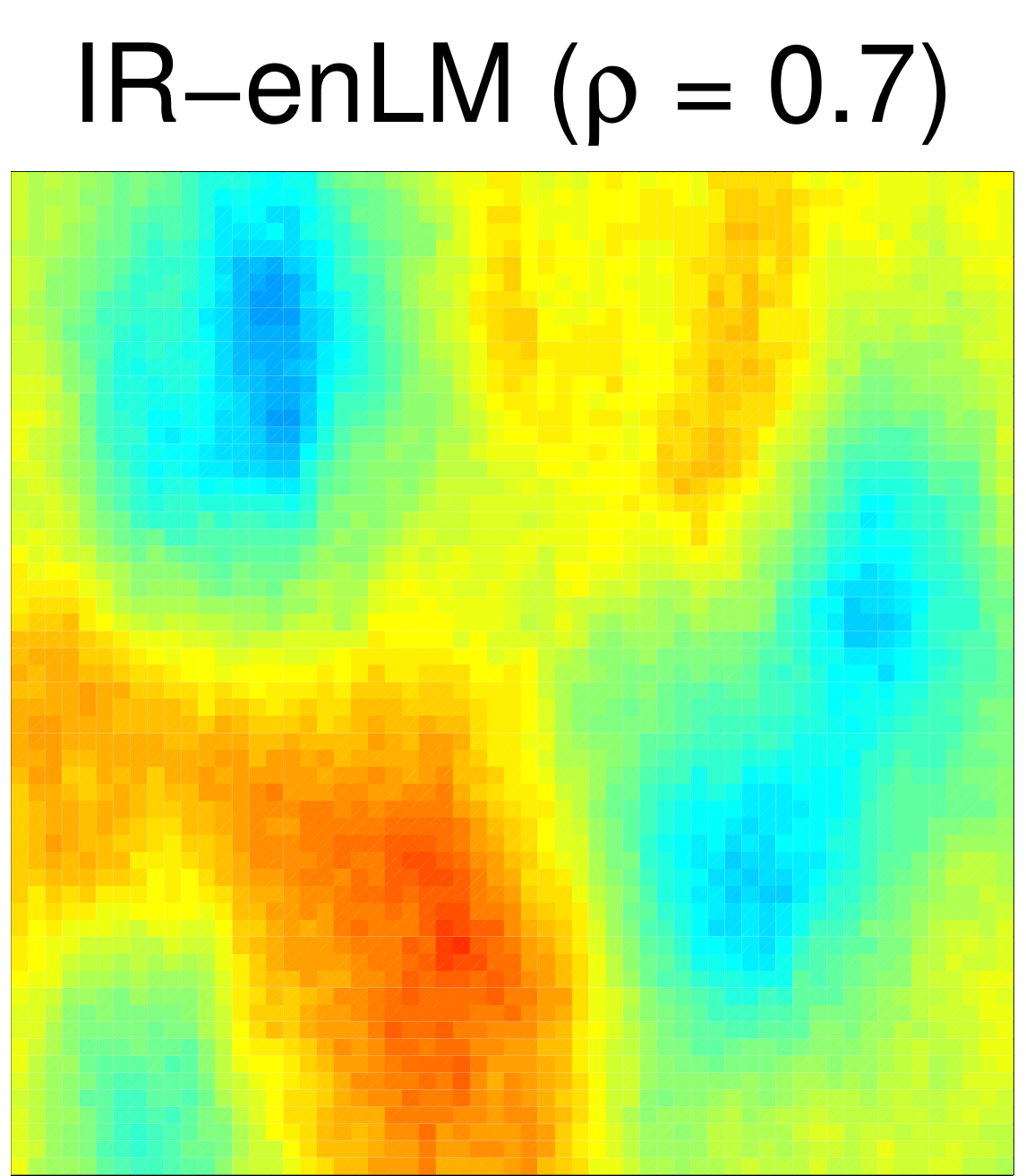}~
\includegraphics[scale=0.2]{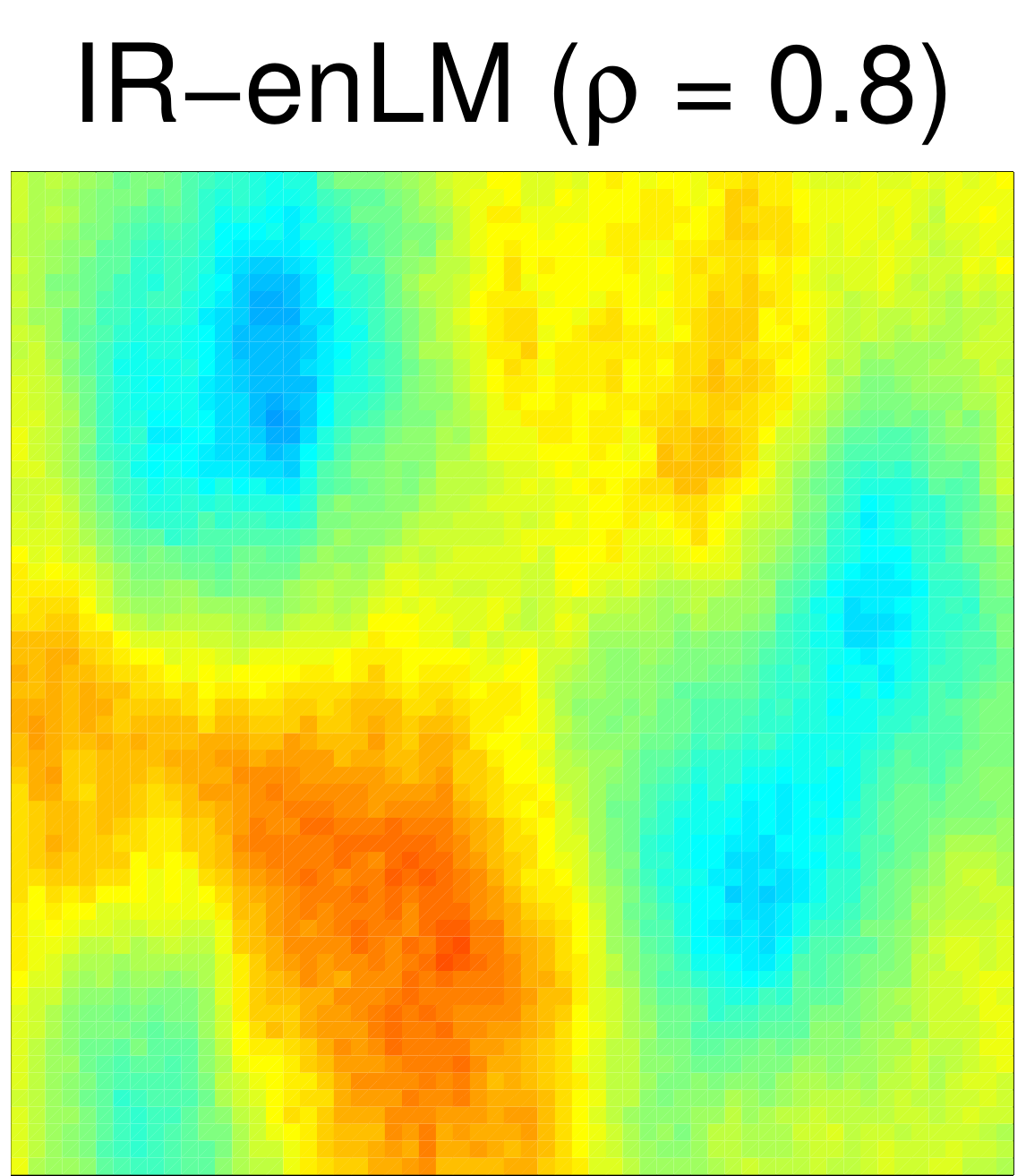}~
\includegraphics[scale=0.2]{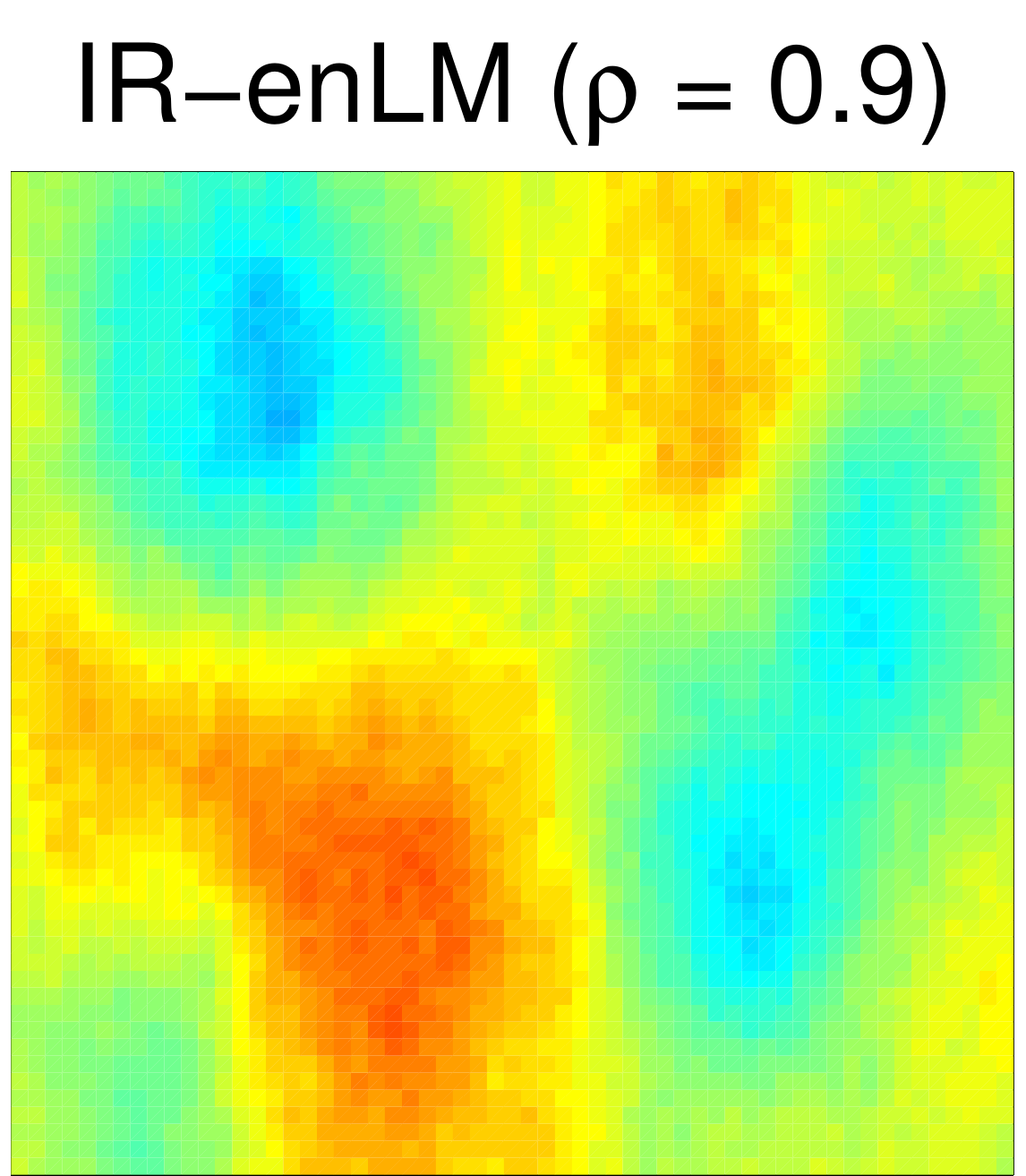}\\
\includegraphics[scale=0.2]{Var_MCMC2BB}~~
\includegraphics[scale=0.2]{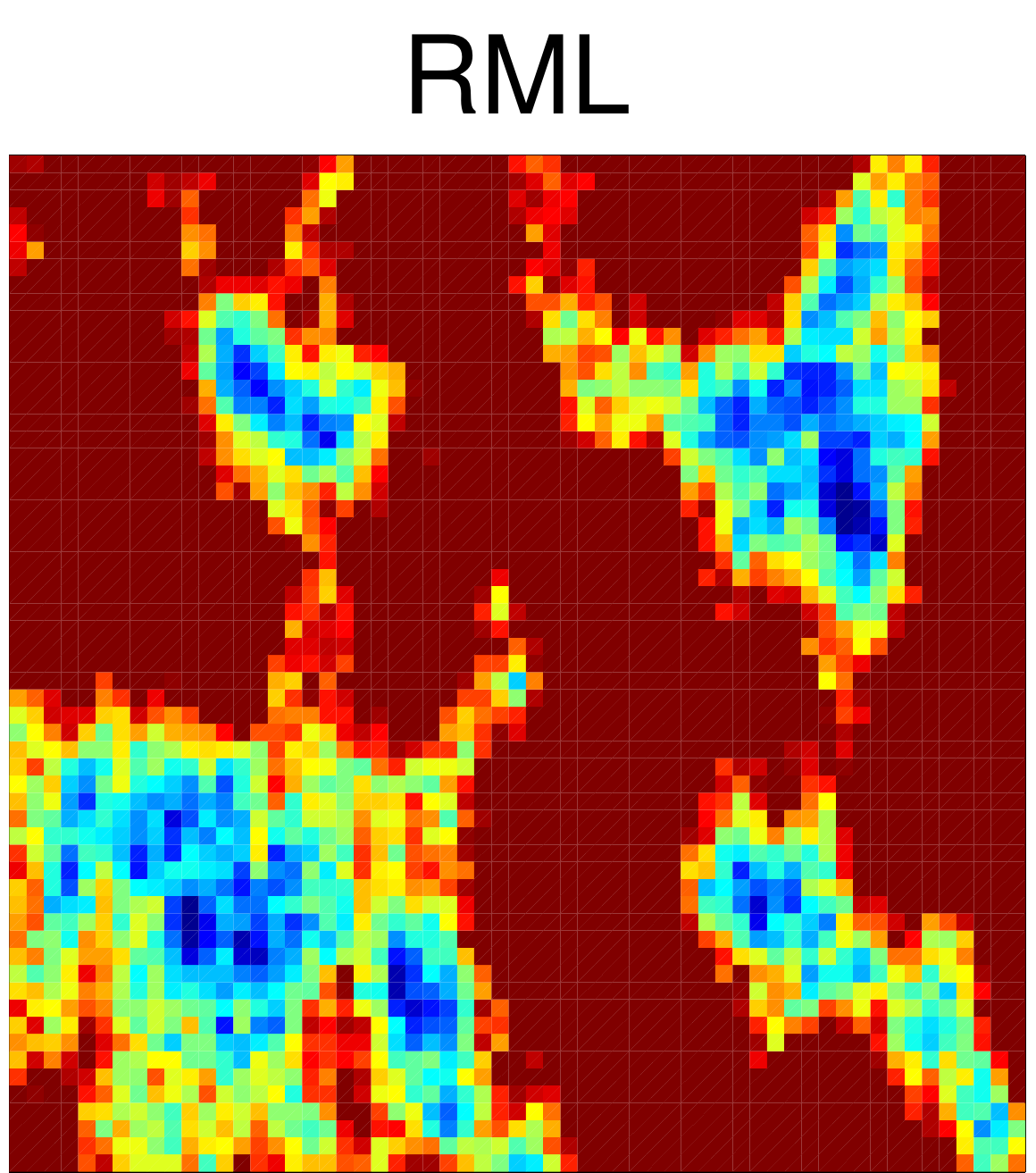}~
\includegraphics[scale=0.2]{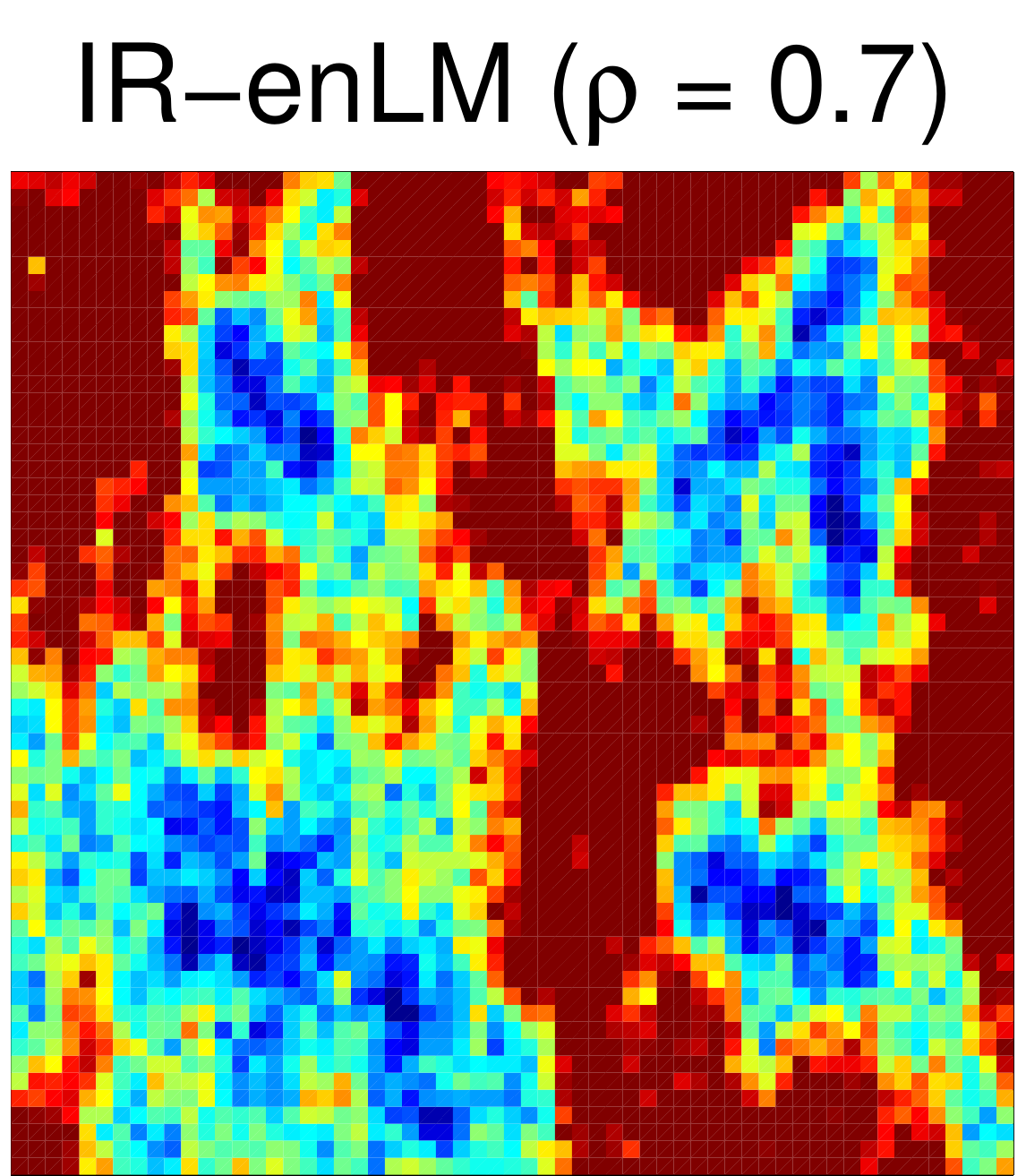}~
\includegraphics[scale=0.2]{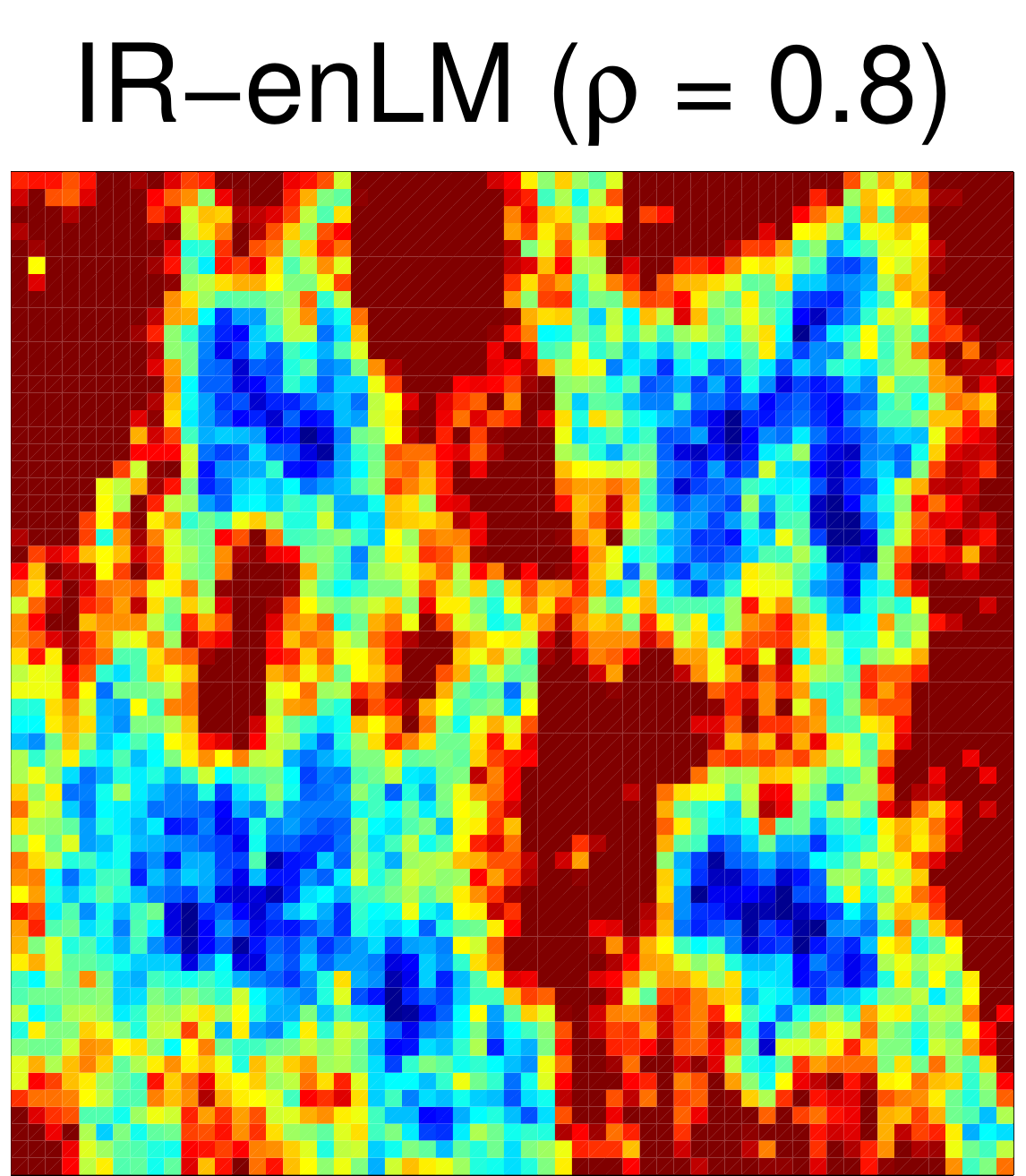}~
\includegraphics[scale=0.2]{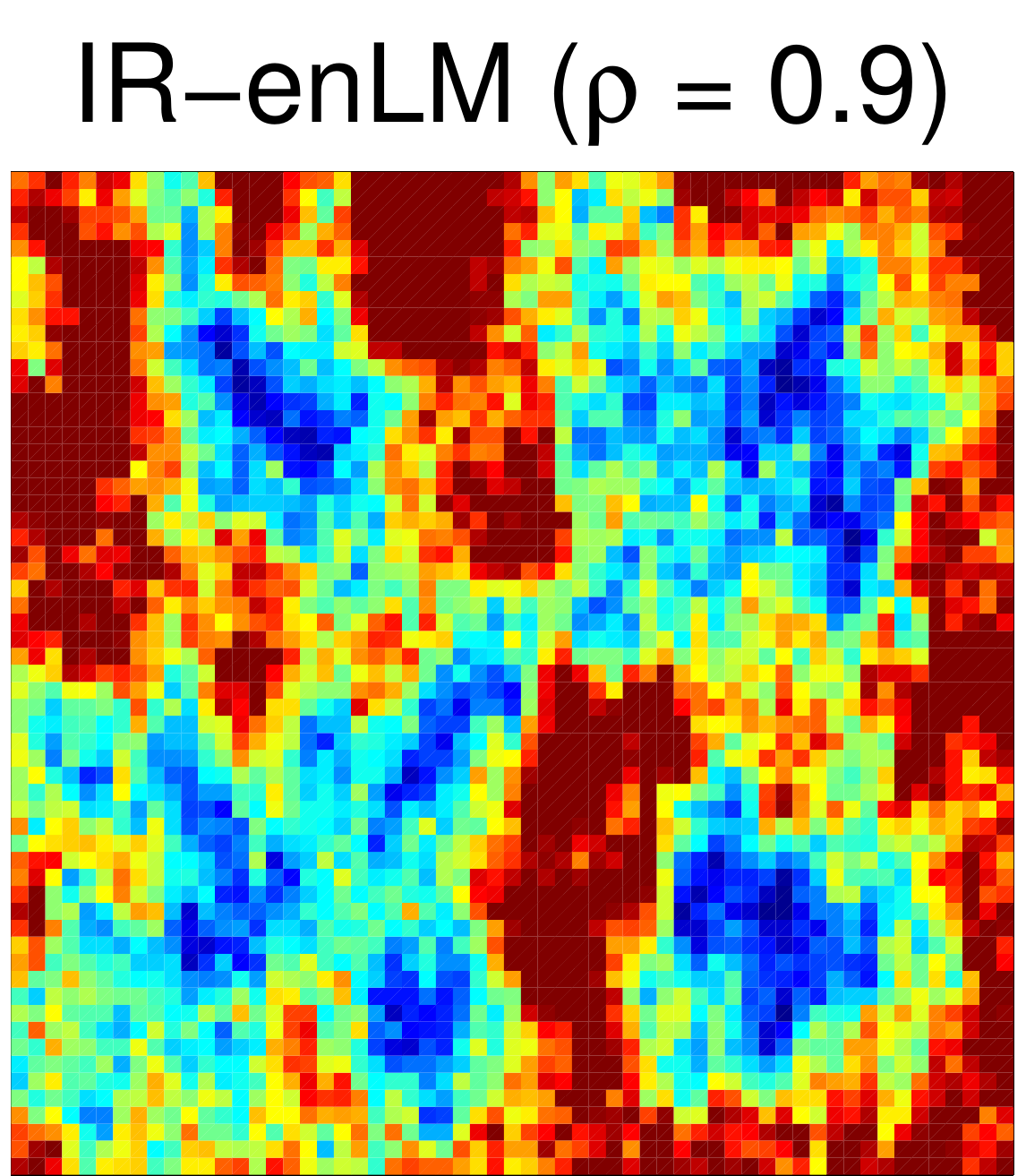}
\end{center}

\caption{ Mean (top) and Variance (bottom) of the posterior distribution $\mu_{B}$ (characterized with MCMC) and ensemble approximations RML, IR-enLM with $N_{e}=50$.}  

\label{Figure11A}
\end{figure}

\begin{figure}
\begin{center}
\includegraphics[scale=0.2]{Mean_MCMC2BB}~
\includegraphics[scale=0.2]{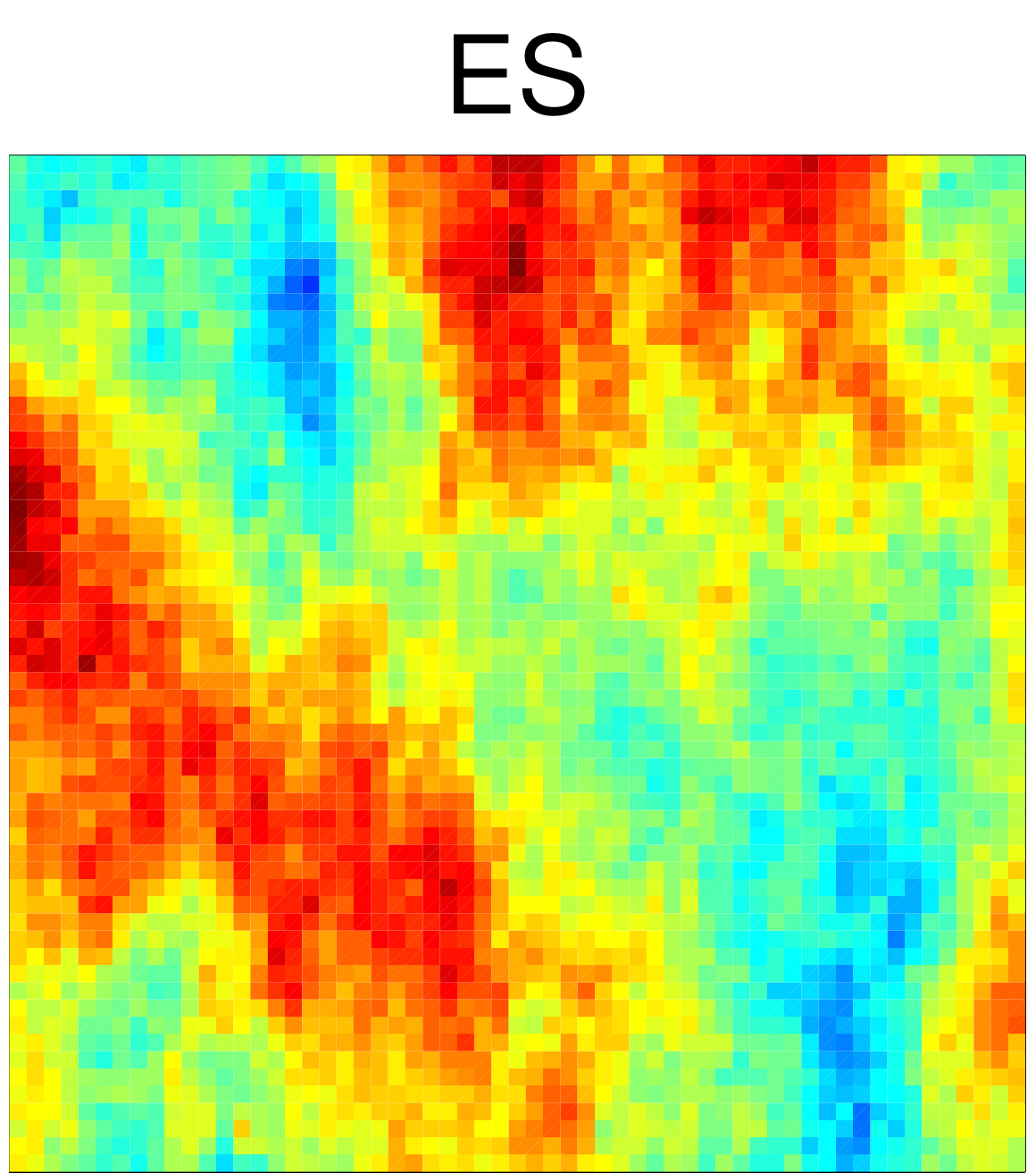}~
\includegraphics[scale=0.2]{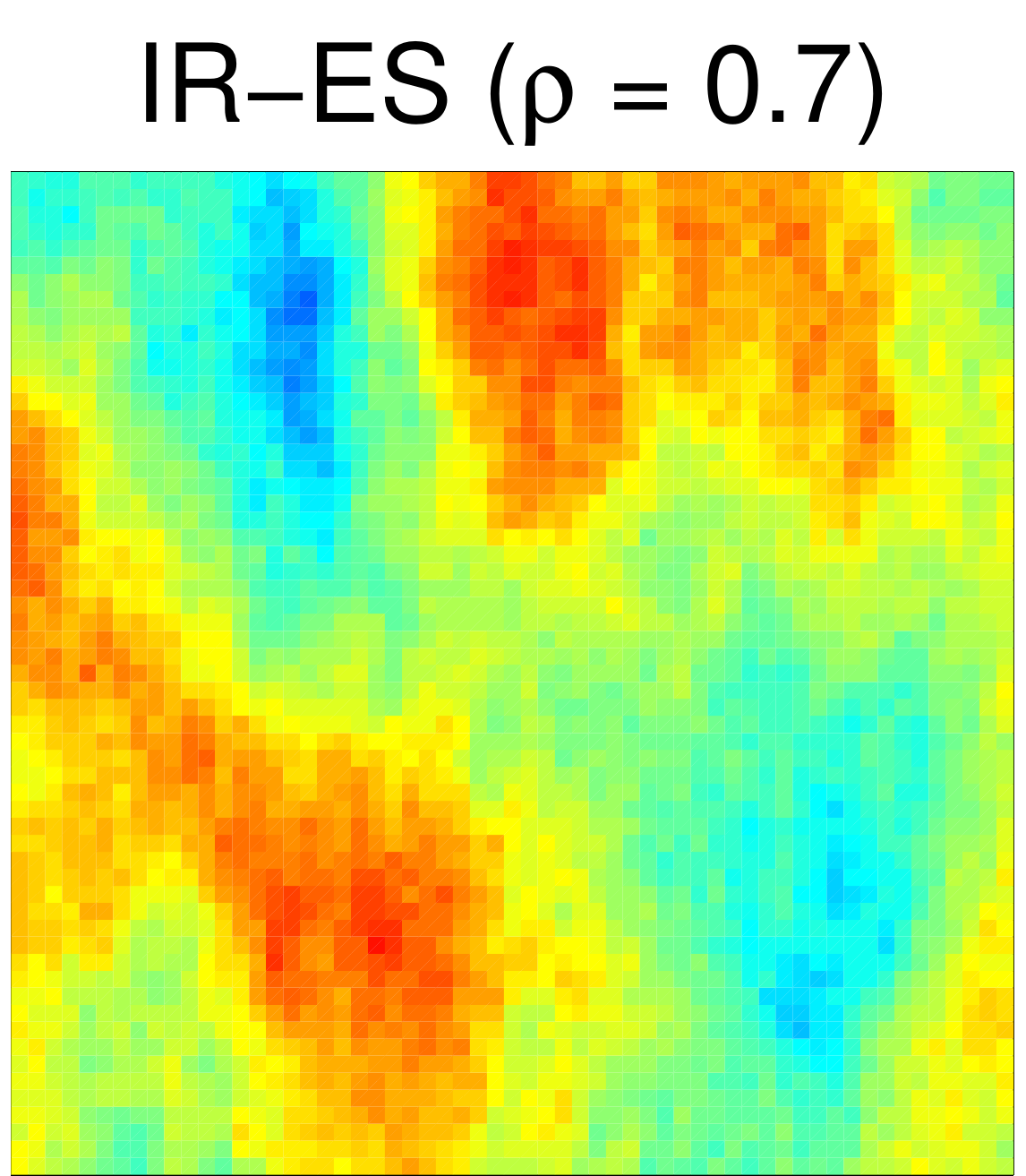}~
\includegraphics[scale=0.2]{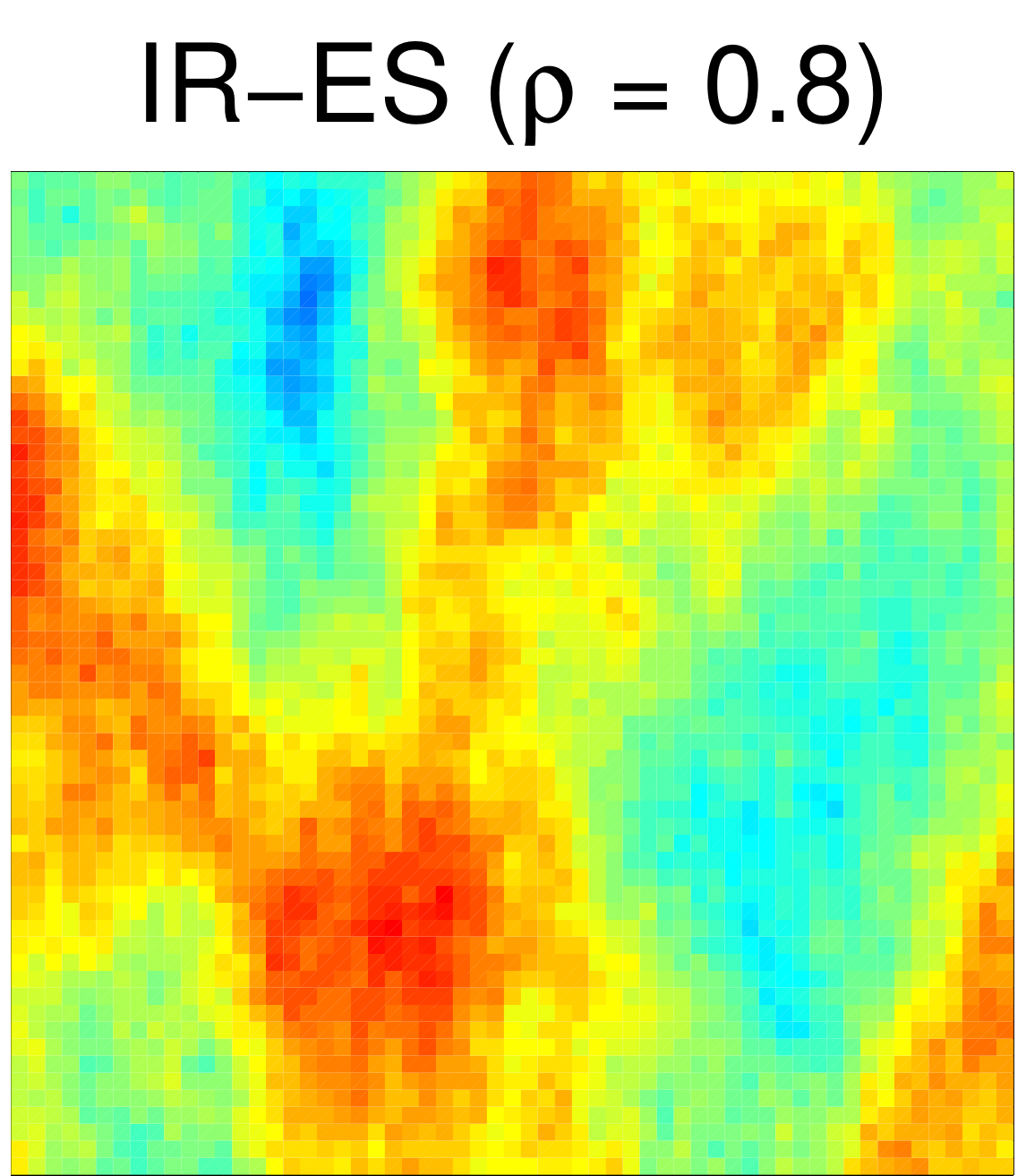}~
\includegraphics[scale=0.2]{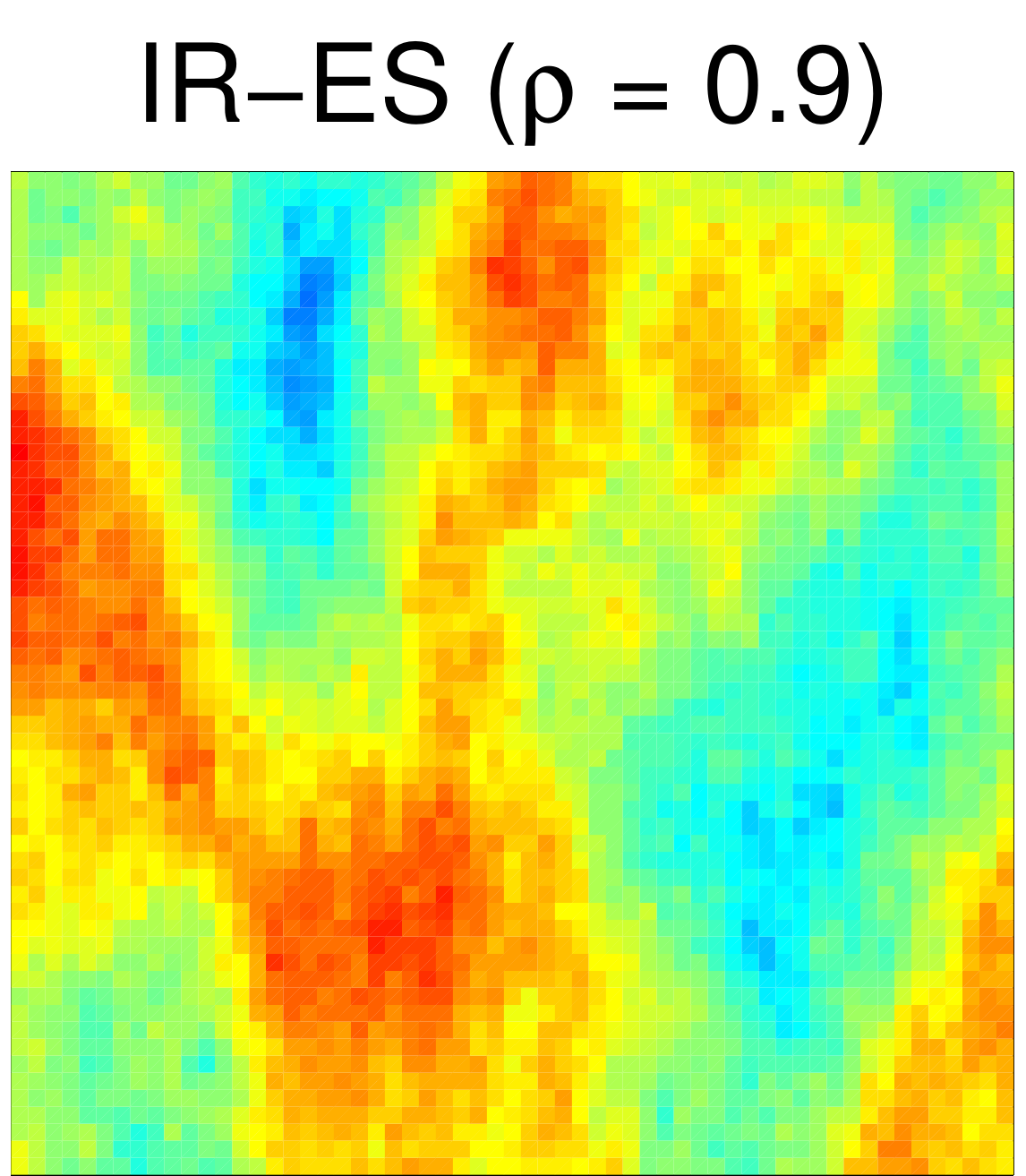}\\
\includegraphics[scale=0.2]{Var_MCMC2BB}~~
\includegraphics[scale=0.2]{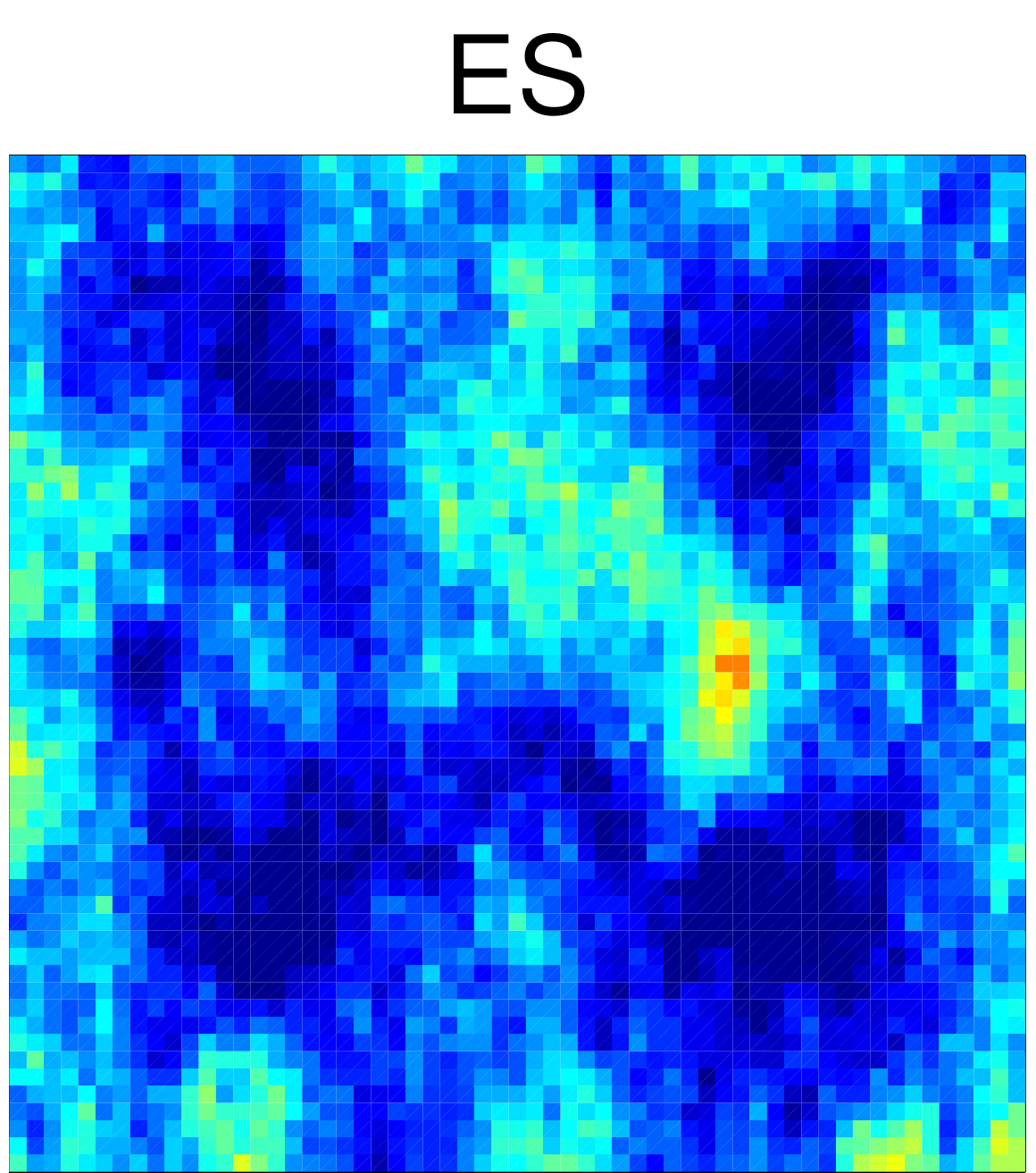}~
\includegraphics[scale=0.2]{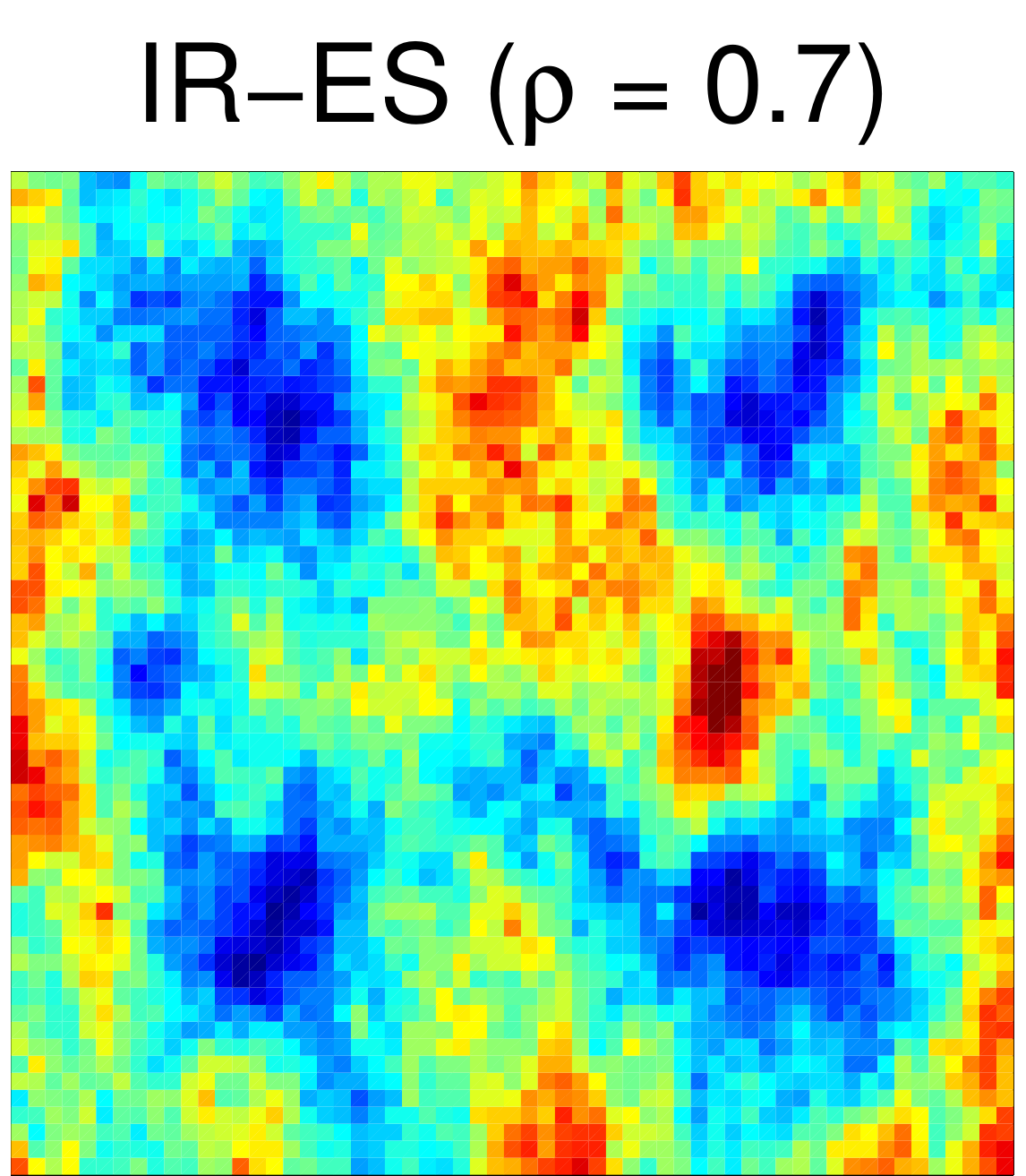}~
\includegraphics[scale=0.2]{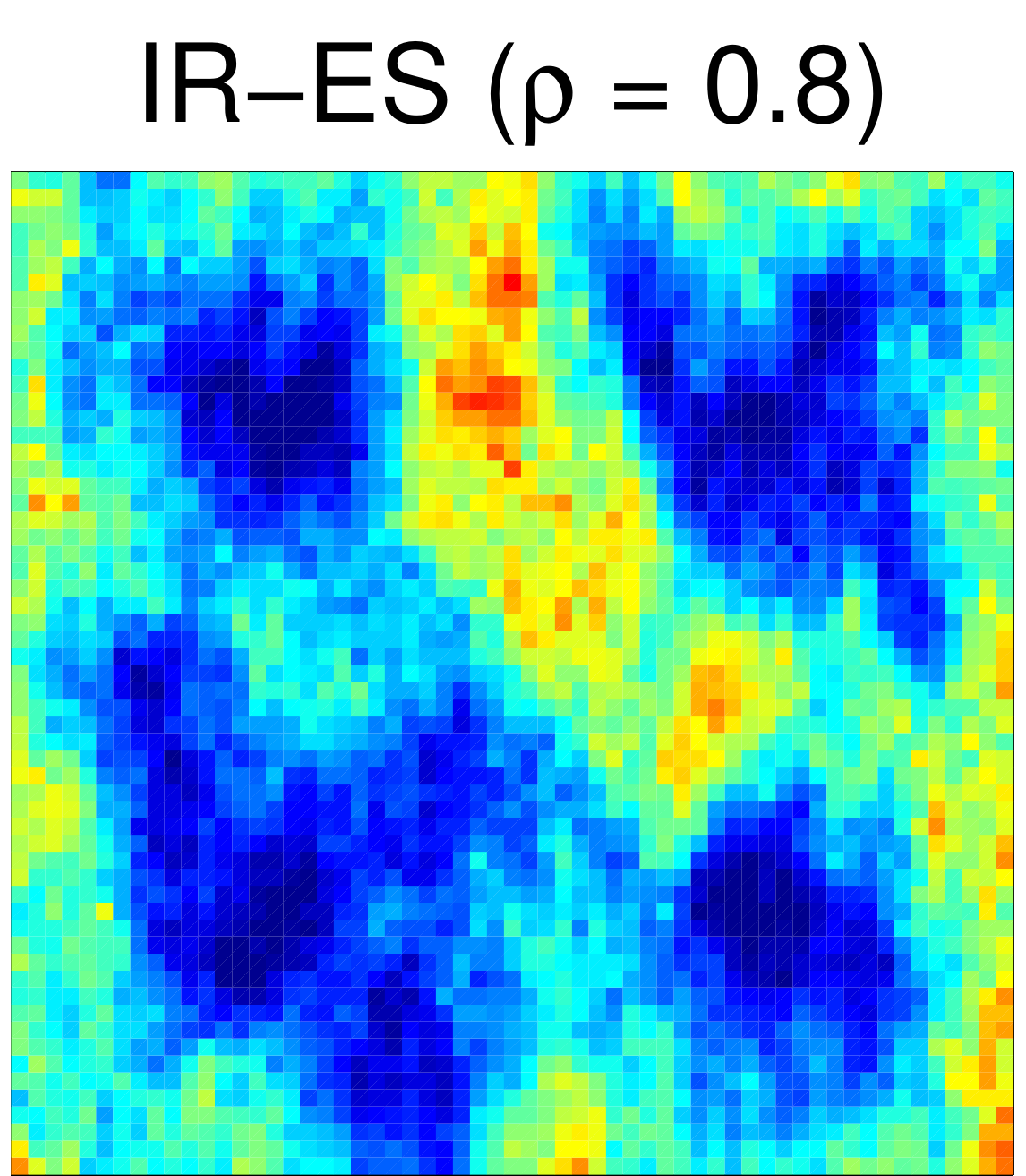}~
\includegraphics[scale=0.2]{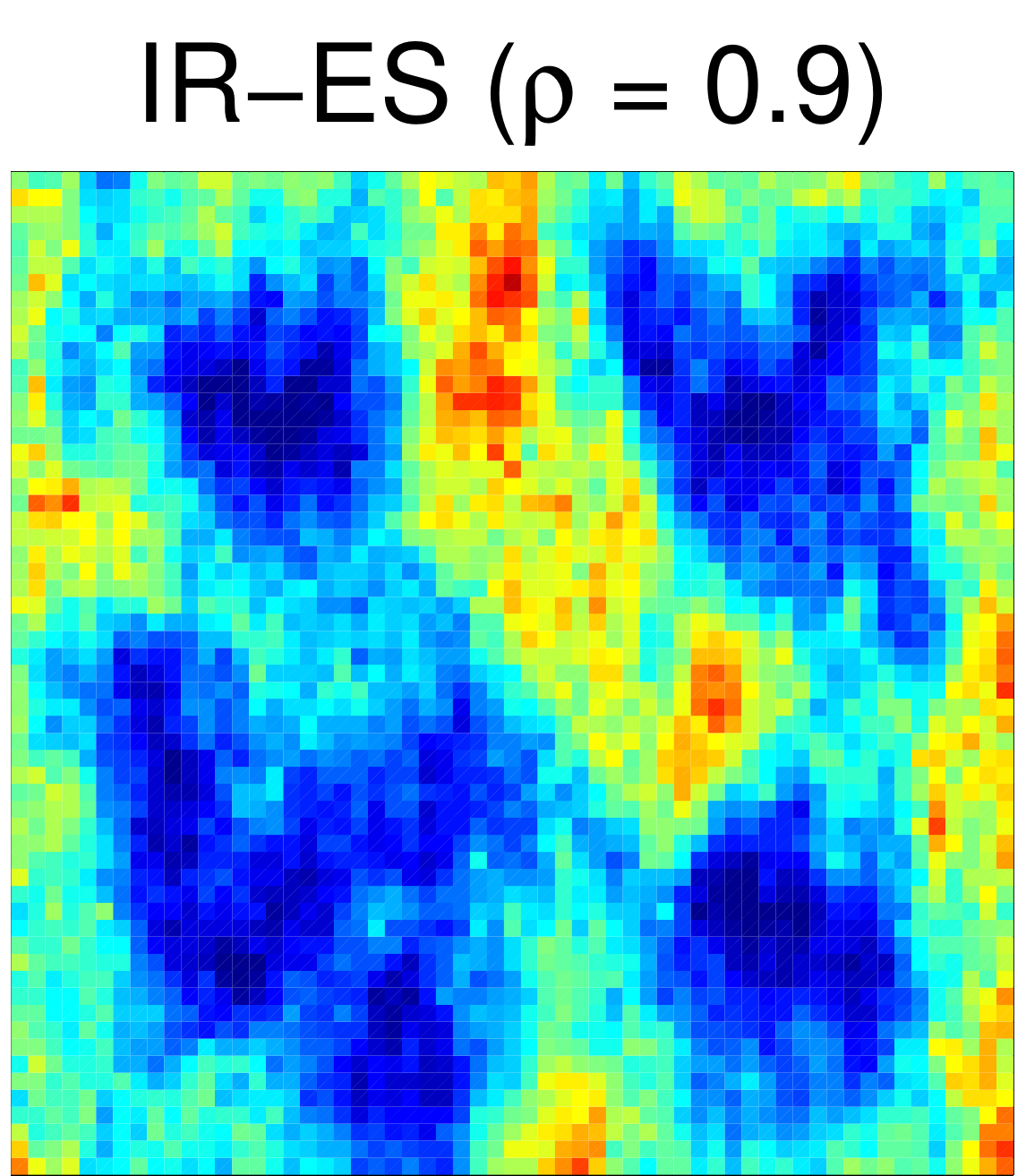}

\end{center}
\caption{ Mean (top) and Variance (bottom) of the posterior distribution $\mu_{B}$ (characterized with MCMC) and ensemble approximations ES and IR-ES with $N_{e}=75$.}  
\label{Figure11B}
\end{figure}

In Table \ref{Table4} we repeat the experiments (applying RLM with unregularized LM) for different choices of $\lambda_{0}$ and different choices of $\kappa$ in (\ref{eq:1.26}). Note that the standard recommendations of parameters $\lambda_{0}=J(u_{0}^{(j)})/N_{D}$ and $\kappa=10$ (that we reported above) corresponds to the most inaccurate approximations of the posterior. In fact, a different selection of the parameters $\lambda$ and $\kappa$ (e.g. $\lambda_{0}=J(u_{0}^{(j)})/N_{D}\times 10^2$, $\kappa=10$) can be selected to produce more accurate approximations of the posterior. Note that the selection of the standard methods are based on the assumption that the computation of the minimizer of (\ref{eq:RMLB}) is well posed. However, in \cite{LM} we have demonstrated that this is not necessarily the case and here we clearly observe that this instabilities in the computation of (\ref{eq:RMLB}) are detrimental to the accuracy in terms of approximating the lean and variance of the posterior. Note, for example, that changing from $\lambda_{0}=J(u_{0}^{(j)})/N_{D}\times 10$ to $\lambda_{0}=J(u_{0}^{(j)})/N_{D}$ has a dramatic (unstable) increase of the error in the mean and variance of the posterior.

Finally, we observe that the criteria (\ref{eq:1.22}) often used to assess the performance of ensemble methods in terms of data mismatch could be misleading. In Figure \ref{Figure8} we present the box plots of the value of the normalized (by the number of measurements) objective function (\ref{eq:RMLB}), evaluated for the ensemble members (for a fixed ensemble) and for different parameters in the RML implementation described above. We note that the standard recommendations for the choices of parameters lead to ensembles that satisfy (\ref{eq:1.22}). However, each of these ensembles provides different and even quite inaccurate approximations of the posterior (see Table \ref{Table4}). In Figure \ref{Figure9} we display the boxplot of the same function evaluated at the ensemble generated with IR-enLM (top row) and IR-ES (bottom) for the approximation of $\mu_{A}$ (left) and $\mu_{B}$ (right), respectively. It is clear that our recommendations for $\tau$ corresponds to small values of the normalized objective functions. However, even though the value of the normalized objective function with the proposed methods was not as small as the one obtained with RML (see Figure \ref{Figure8}), the approximation of the posterior provided by the proposed ensemble methods was, in general, more accurate than the one with RML (computed with an unregularized standard LM method). The numerical results of this subsection suggest that, by avoiding the data overfitting, the proposed methods produce better regularized ensembles which result in more accurate and stable approximations of the posterior.

\begin{table}
\caption{Comparison of methods. $\ast$: ($\tau=1.0$), $\dagger: (\kappa=10,~\lambda_{0}=\Lambda_{0}$)}
\label{Table3A}       
\begin{tabular}{cc|ccc|ccc}
\hline\noalign{\smallskip}
& &  &  Approx. of $\mu_{A}$& & & Approx. of $\mu_{B}$ & \\
\noalign{\smallskip}\hline\noalign{\smallskip}
Method &$N_{e}$&  $\epsilon_{u}$&  $\epsilon_{\sigma}$ & aver. iter. &  $\epsilon_{u}$&  $\epsilon_{\sigma}$ &aver. iter.\\
\noalign{\smallskip}\hline\noalign{\smallskip}
RML$^{\dagger}$ &   25& 0.488 & 0.702 &  10.560 &0.611 & 1.355 &  10.232\\
RML$^{\dagger}$  &  50& 0.466&  0.637 &  10.350 & 0.575&   1.205&10.580 \\
RML$^{\dagger}$  &  75& 0.456&  0.596 &  10.345 & 0.556 &    1.180& 10.636 \\
RML$^{\dagger}$  &  100& 0.452&  0.603 &  10.311 &0.542 & 1.156 & 10.675\\

\hline                                                                               
IR-enLM$^{\ast}$ ($\rho=0.7$) & 25 & 0.295 & 0.467 & 8.616 & 0.381 & 0.640 & 7.364 \\    
IR-enLM$^{\ast}$ ($\rho=0.8$)  & 25 & 0.300 & 0.449 & 13.152 & 0.372 & 0.512 & 11.792 \\    
IR-enLM$^{\ast}$ ($\rho=0.9$) & 25  & 0.321 & 0.452 & 27.796 & 0.395 & 0.469 & 25.344 \\
\hline

IR-enLM ($\rho=0.7$) & 50 &  0.240 & 0.405 & 8.597 & 0.335 & 0.528 & 7.335 \\  
IR-enLM ($\rho=0.8$) & 50 & 0.249 & 0.386 & 13.088 & 0.329 & 0.425 & 11.755 \\
IR-enLM ($\rho=0.9$)  & 50 &  0.277 & 0.371 & 27.476 & 0.362 & 0.381 & 25.511 \\

\hline                              

IR-enLM ($\rho=0.7$) & 75 & 0.224 & 0.365 & 8.523 & 0.324 & 0.500 & 7.347 \\    
IR-enLM ($\rho=0.8$)  & 75 & 0.231 & 0.339 & 12.988 & 0.319 & 0.386 & 11.767 \\    
IR-enLM ($\rho=0.9$) & 75 &  0.265 & 0.341 & 27.461 & 0.347 & 0.351 & 25.443 \\
\hline           

IR-enLM ($\rho=0.7$) & 100 & 0.217 & 0.361 & 8.597 & 0.312 & 0.481 & 7.335 \\  
IR-enLM ($\rho=0.8$)  & 100 & 0.226 & 0.336 & 13.084 & 0.307 & 0.378 & 11.755 \\
IR-enLM ($\rho=0.9$) & 100 & 0.258 & 0.327 & 27.476 & 0.343 & 0.332 & 25.511 \\
\noalign{\smallskip}\hline

\end{tabular}
\end{table}

\begin{table}
\caption{Comparison of methods. $\ast$: ($M_{ES}=10$, $\tau=1/\rho$), $\dagger$: in forward model runs }
\label{Table3B}       
\begin{tabular}{cc|ccc|ccc}
\hline\noalign{\smallskip}
& &  &  Approx. of $\mu_{A}$& & & Approx. of $\mu_{B}$ & \\
\noalign{\smallskip}\hline\noalign{\smallskip}
Method &$N_{e}$&  $\epsilon_{u}$&  $\epsilon_{\sigma}$ & average cost$^{\dagger}$ &  $\epsilon_{u}$&  $\epsilon_{\sigma}$ &average cost$^{\dagger}$ \\
\noalign{\smallskip}\hline\noalign{\smallskip}
ES& 25&  1.617&0.913& 25&1.863& 0.866 &25\\
ES  &   50 &1.185 & 0.617 & 50&1.472& 0.581 &50\\
ES  &   75 &0.914& 0.420 & 75&1.071& 0.412 &75\\
ES  &   100 &0.779 & 0.310&100&0.969& 0.306 &100\\
ES  &   150 &0.630 & 0.220&150&0.795& 0.234 &150\\
ES  &   200 &0.555&  0.193&200&0.701& 0.226 &200\\
\noalign{\smallskip}\hline\hline\noalign{\smallskip}
IR-ES$^{\ast}$ ($\rho=0.7$) &25  & 1.522 & 0.869 &25& 1.520 & 0.741 &25 \\                                           
IR-ES$^{\ast}$ ($\rho=0.8$) &25 & 1.107 & 0.803 & 50& 1.090 & 0.664 & 50\\  
IR-ES$^{\ast}$ ($\rho=0.9$) &25   & 0.969 & 0.672 &75& 1.053 & 0.579 & 75\\
\noalign{\smallskip}\hline\hline\noalign{\smallskip}                                      
IR-ES$^{\ast}$ ($\rho=0.7$) &50  & 0.789 & 0.404 &50 & 0.942 & 0.374 & 50 \\  
IR-ES$^{\ast}$ ($\rho=0.8$) &50 & 0.684 & 0.396 & 100 & 0.718 & 0.335 & 100 \\
IR-ES$^{\ast}$ ($\rho=0.9$) &50  & 0.637 & 0.321 & 150 & 0.689 & 0.272 & 150 \\
\noalign{\smallskip}\hline\hline\noalign{\smallskip}
IR-ES$^{\ast}$ ($\rho=0.7$) &75 & 0.657 & 0.280 & 75 & 0.723 & 0.280 & 75 \\  
IR-ES$^{\ast}$ ($\rho=0.8$) &75 & 0.600 & 0.267 & 150 & 0.609 & 0.234 & 150 \\  
IR-ES$^{\ast}$ ($\rho=0.9$) &75  & 0.563 & 0.244 & 225 & 0.590 & 0.206 & 225 \\
\noalign{\smallskip}\hline\hline\noalign{\smallskip}
IR-ES$^{\ast}$ ($\rho=0.7$) &100 & 0.586 & 0.267 &100 & 0.659 & 0.290 & 100 \\                                            
IR-ES$^{\ast}$ ($\rho=0.8$) &100 & 0.558 & 0.236 & 200 & 0.541 & 0.208 & 200\\
IR-ES$^{\ast}$ ($\rho=0.9$) &100& 0.525 & 0.232 & 300 & 0.525 & 0.206 & 300\\
\noalign{\smallskip}\hline

\end{tabular}
\end{table}

\begin{table}
\caption{Performance of RML (with a standard unregularized LM) for different choices of $\lambda_{0}$ and $\kappa$.}
\label{Table4}       
\begin{tabular}{ccc|ccc|ccc}
\hline\noalign{\smallskip}
& &  & & Approx. of $\mu_{A}$& & & Approx. of $\mu_{B}$ & \\

\noalign{\smallskip}\hline\noalign{\smallskip}
$\lambda_{0}$ & $\kappa$ &  $N_{e}$&$\epsilon_{u}$&  $\epsilon_{\sigma}$ & aver. iter. &  $\epsilon_{u}$&  $\epsilon_{\sigma}$ & aver. iter.\\
\noalign{\smallskip}\hline\noalign{\smallskip}
$\Lambda_{0}\times 10^5$ & 10 & 50 & 0.864 & 1.264 & 3.363 & 0.973 & 0.919 & 2.077 \\   
$\Lambda_{0},\times 10^4$ & 10 & 50 & 0.210 & 0.288 & 12.100 & 0.352 & 0.484 & 13.335 \\
$\Lambda_{0}\times 10^3$ & 10 & 50 & 0.209 & 0.283 & 11.119 & 0.361 & 0.448 & 12.664 \\ 
$\Lambda_{0}\times 10^2$ & 10 & 50 & 0.212 & 0.284 & 10.049 & 0.362 & 0.456 & 11.665 \\ 
$\Lambda_{0}\times 10$ & 10 & 50 & 0.246 & 0.305 & 8.876 & 0.393 & 0.553 & 10.668 \\    
$\Lambda_{0}$ & 10 & 50 & 0.464 & 0.626 & 10.345 & 0.567 & 1.209 & 10.636 \\            
$\Lambda_{0}$ & 5 & 50 & 0.408 & 0.539 & 10.403 & 0.537 & 1.130 & 11.293 \\             
$\Lambda_{0}$ & 2 & 50 & 0.200 & 0.272 & 11.919 & 0.510 & 1.045 & 12.376 \\             
$\Lambda_{0}$ & 1.5 & 50 & 0.201 & 0.267 & 13.915 & 0.502 & 1.012 & 12.613 \\      
\hline
\end{tabular}                                                                
\end{table}

\begin{figure}
\begin{center}
\includegraphics[scale=0.34]{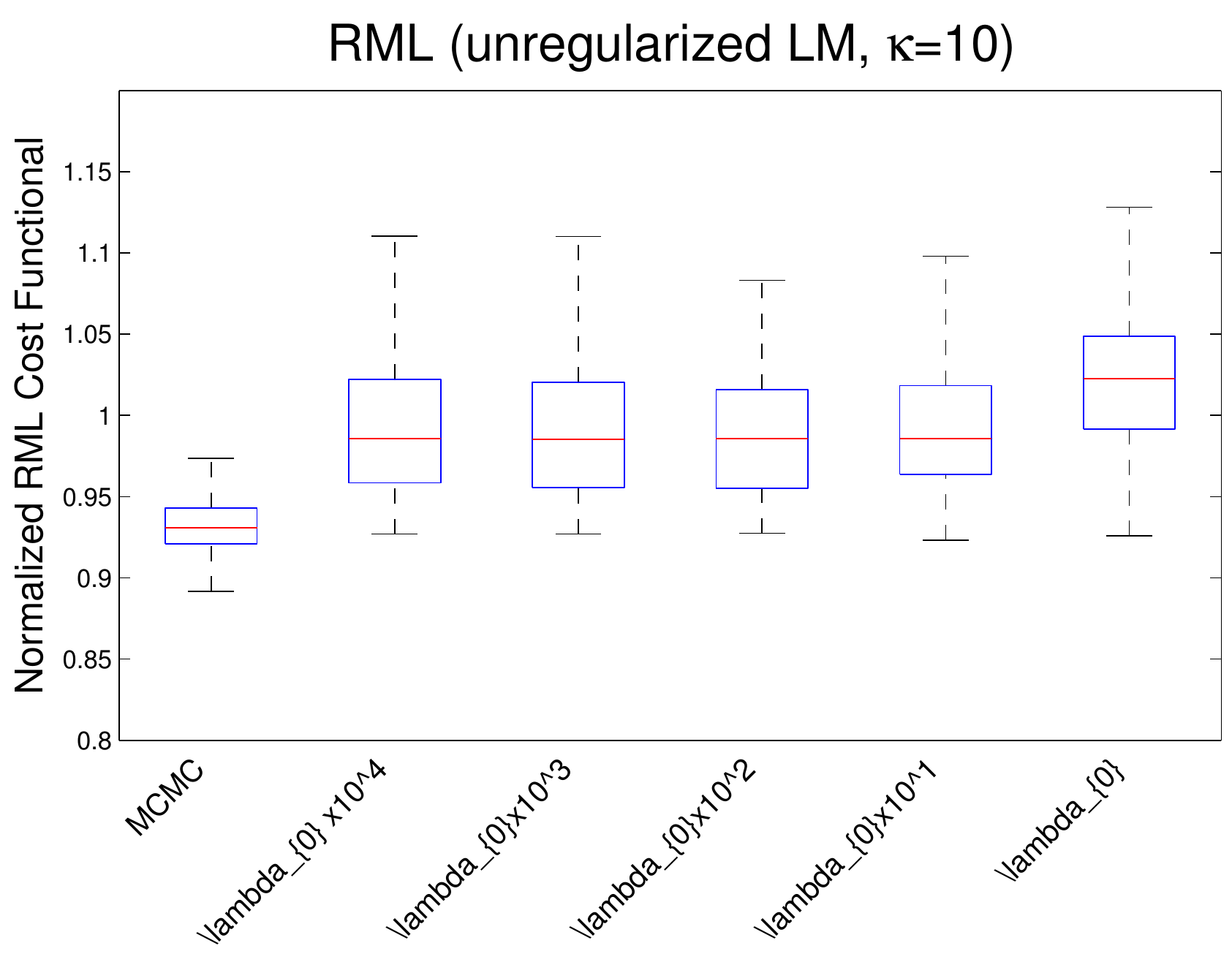}
\includegraphics[scale=0.34]{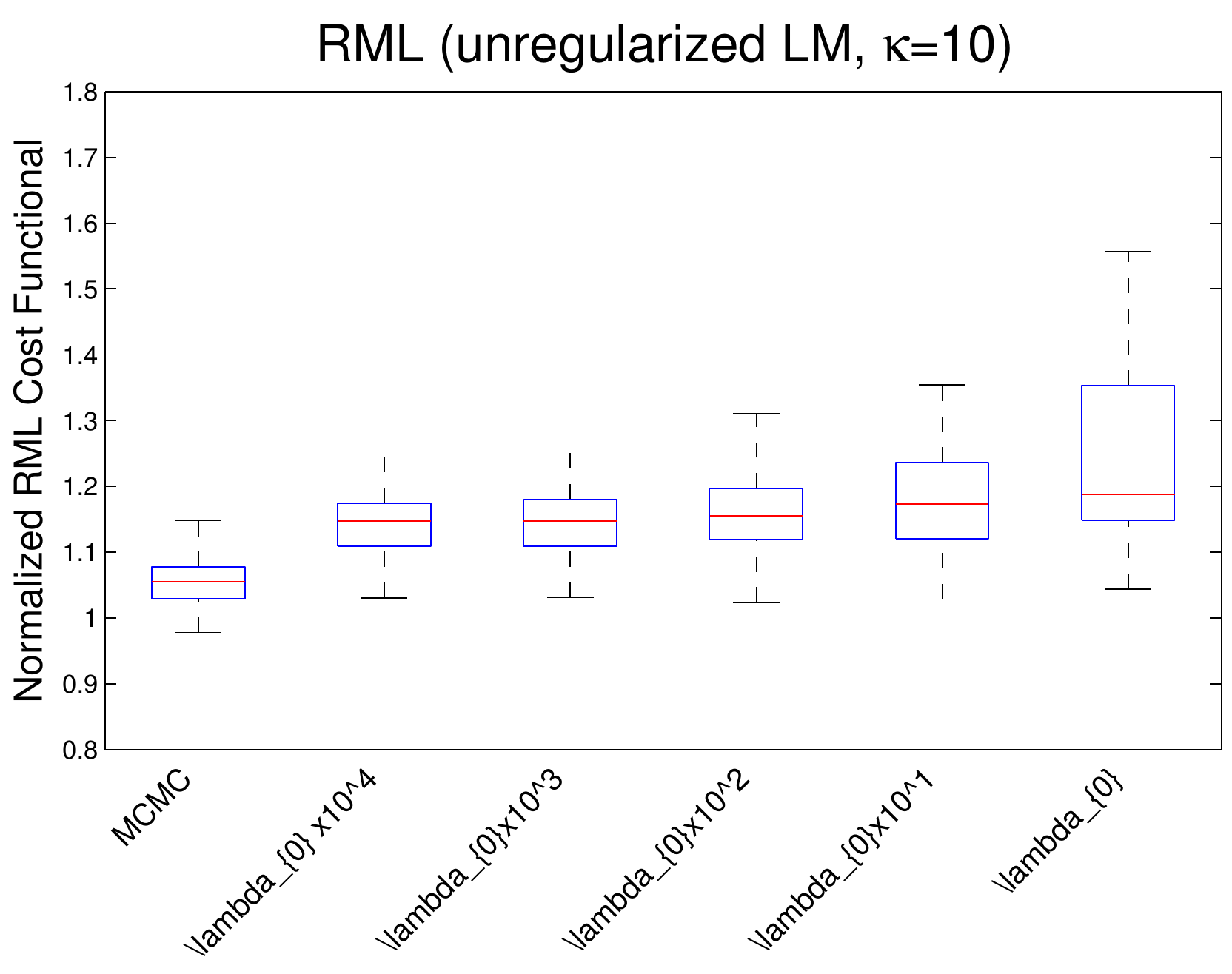}\\
\includegraphics[scale=0.34]{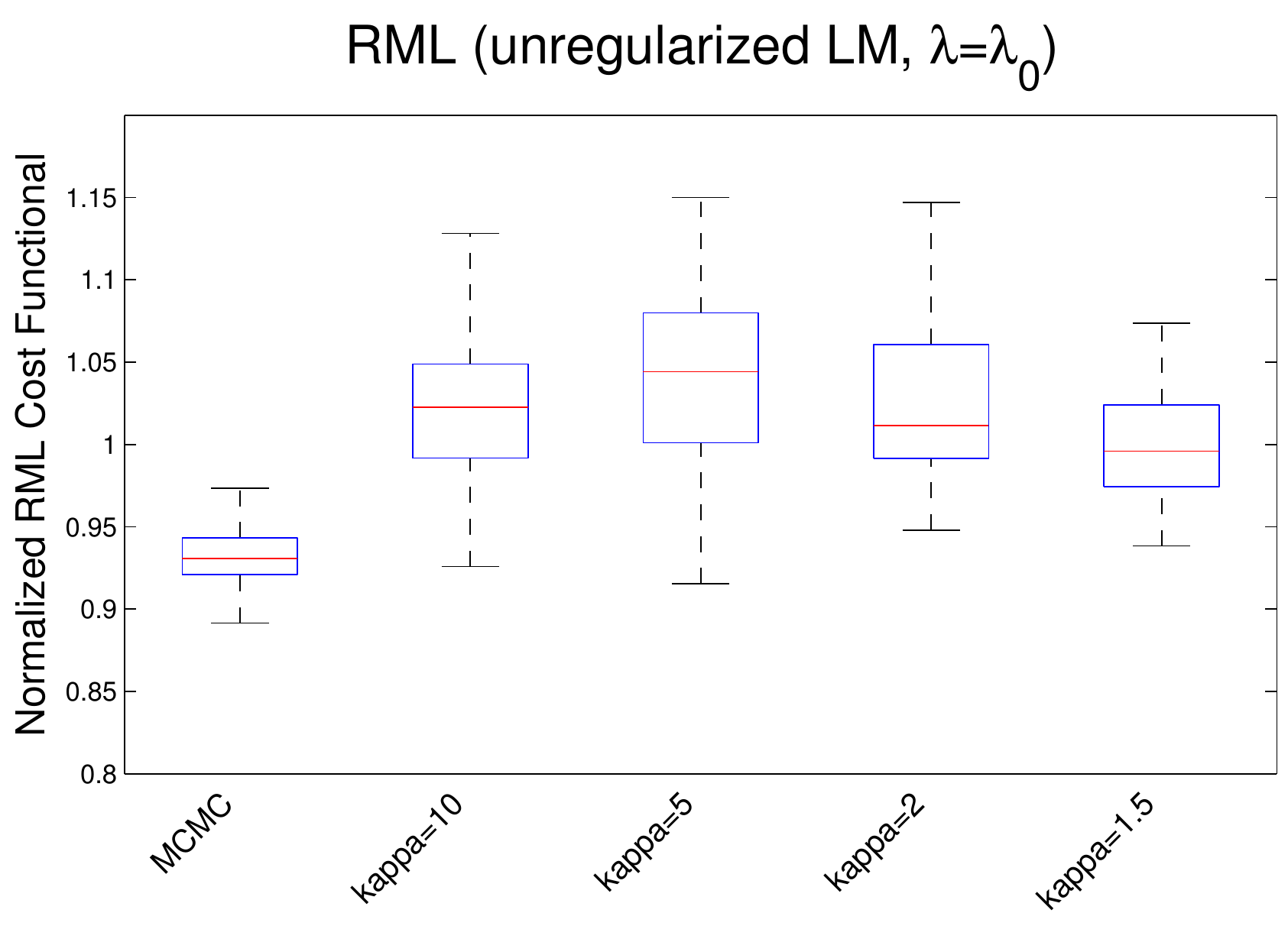}
\includegraphics[scale=0.34]{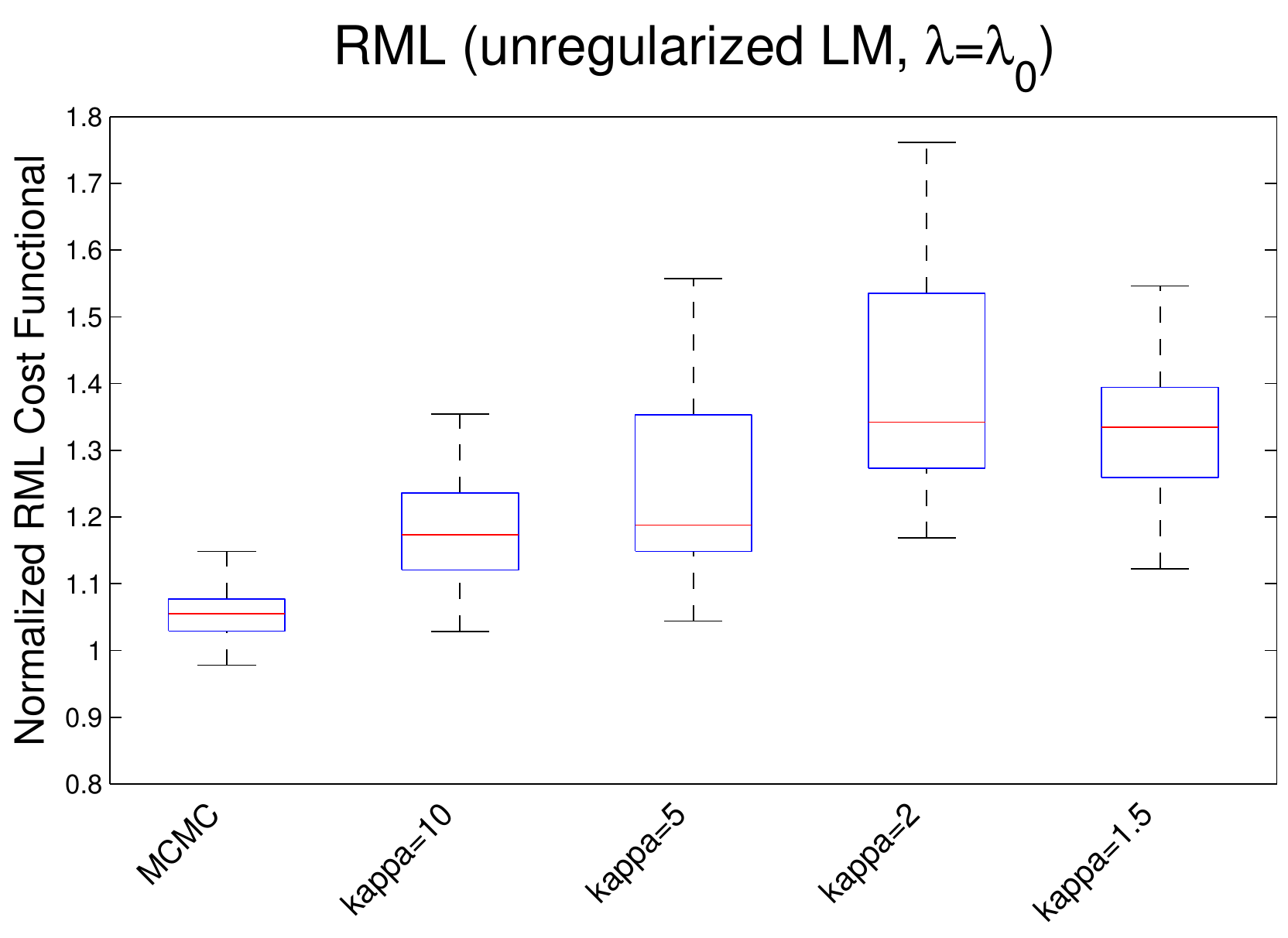}

\caption{Box plots of the normalized data mismatch simulated from $\mu_{A}$ (left) and $\mu_{B}$ (right) with MCMC  and RML  (with unregularized LM) for $N_{e}=50$ and different choices of the regularization parameters} 
 \label{Figure8}
\end{center}
\end{figure}

\begin{figure}
\begin{center}
\includegraphics[scale=0.34]{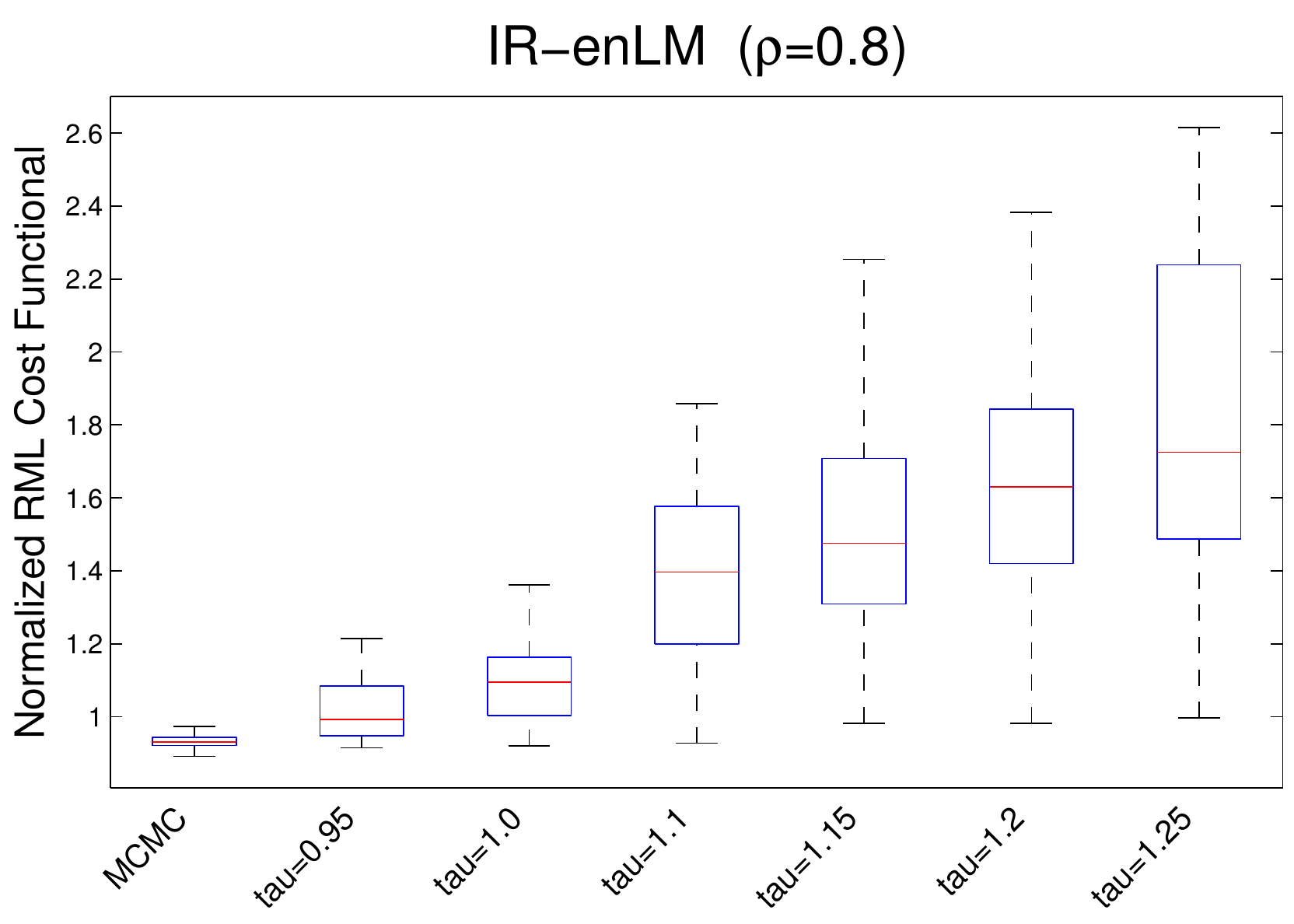}
\includegraphics[scale=0.34]{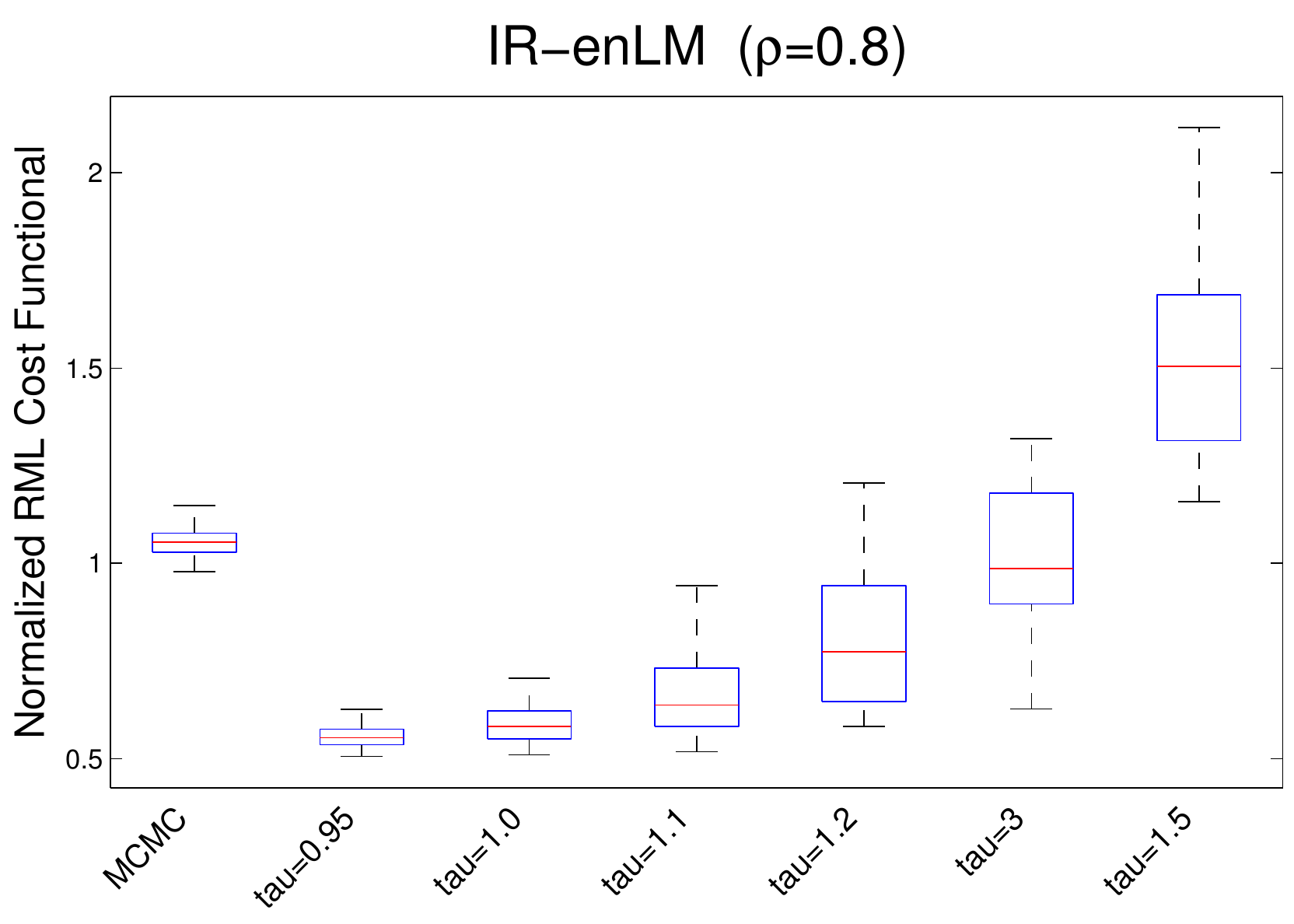}\\
\includegraphics[scale=0.34]{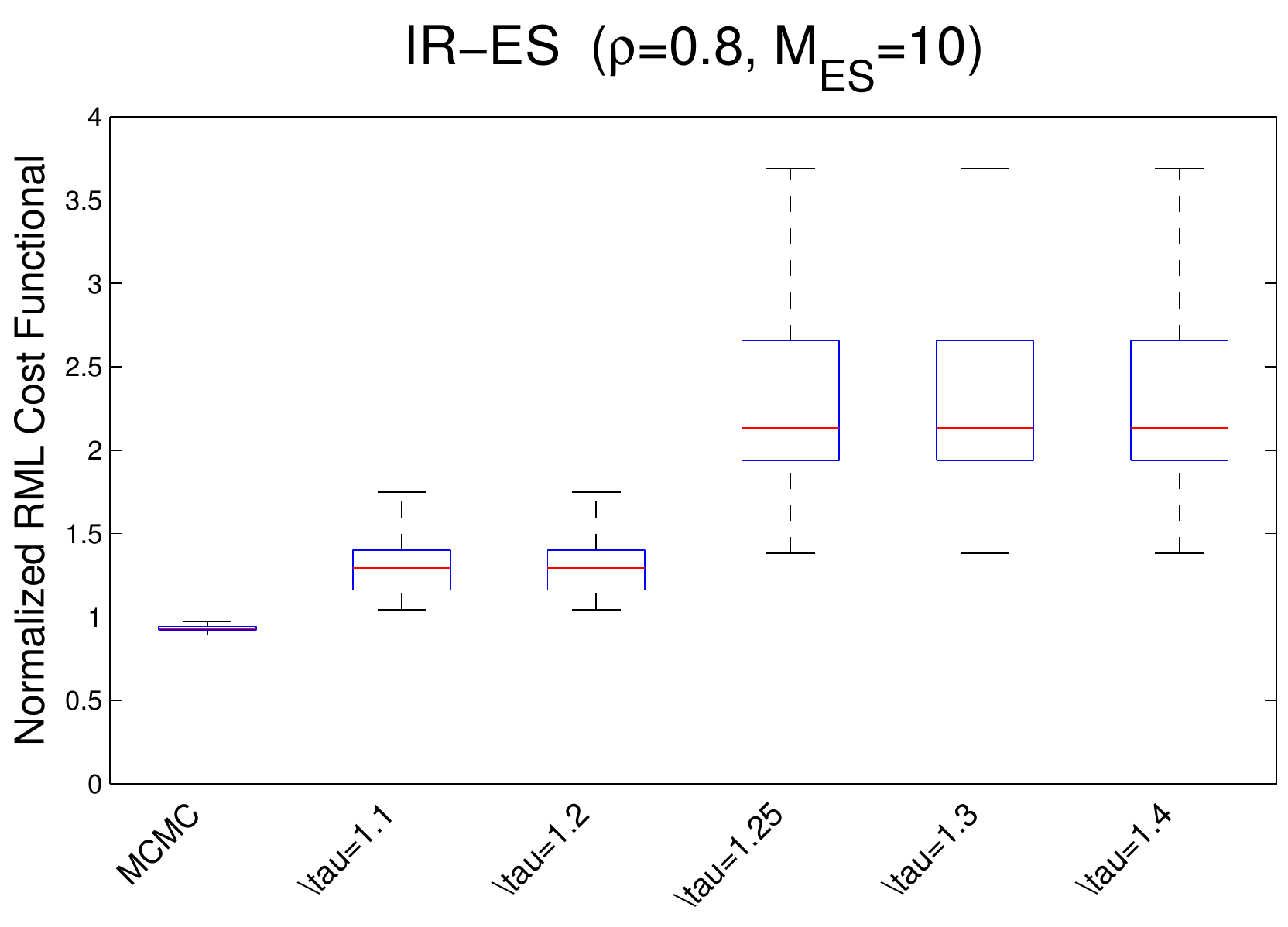}
\includegraphics[scale=0.34]{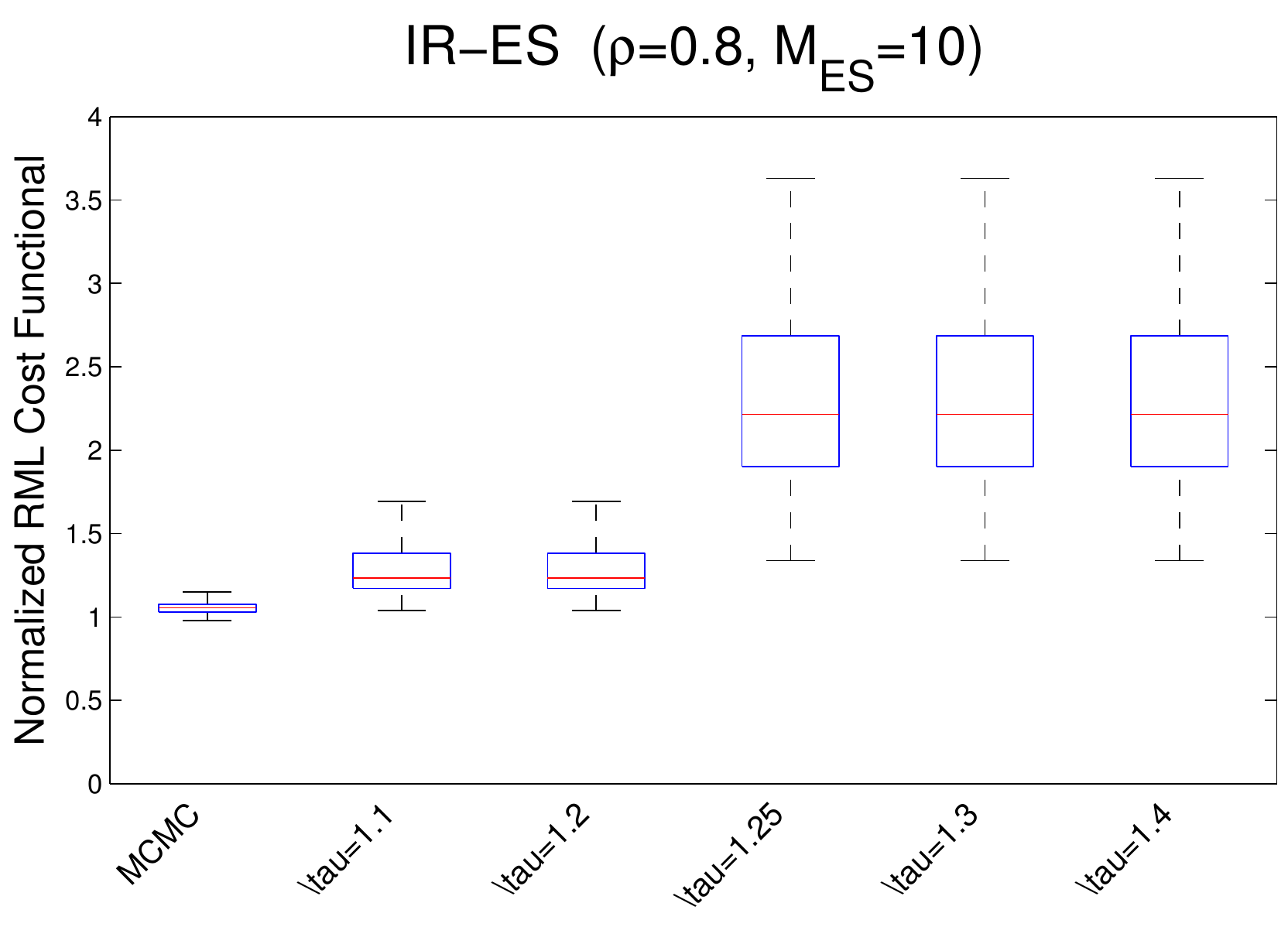}

\caption{Box plots of the normalized data mismatch simulated from $\mu_{A}$ (left) and $\mu_{B}$ (right) with MCMC  and the proposed methods with $N_{e}=50$ and for different choices of the regularization parameters. Top-row: ensemble approximation with IR-enLM for $\rho=0.8$ and different choices of $\tau$. Bottom-row: ensemble approximation with IR-ES for $\rho=0.8$, $M_{ES}=10$ and different choices of $\tau$} \label{Figure9}
\end{center}
\end{figure}

\section{Conclusions}\label{Conclusions}

The proposed IR-enLM and IR-ES are ensemble methods within the RML and Kalman-based frameworks, respectively. The aim of these methods is to provide useful information of the posterior. Both IR-enLM and IR-ES can be derived as an ensemble of iterative solutions to linear inverse problems posed as the minimization of a Tikhonov-regularized functional. The novel feature of the proposed methods is the use of the discrepancy principle for the selection of the Tikhonov parameter that adaptively changes as the iteration progresses. Additionally, the aforementioned principle is applied to control the early termination of the proposed ensemble methods. Crucial to the proposed methods is a set of tunable parameters that arise from the application of the discrepancy principle. In this work we provide extensive numerical experiments to demonstrate the effect of these parameters in the approximation properties of the proposed IR-enLM and IR-ES. Our numerical investigations are based on synthetic experiments where we use a state-of-the-art MCMC method to fully resolve the mean and variance of the resulting Bayesian posterior. For both methods, the tunable parameter $\rho\in (0,1)$ controls the adaptive selection of the regularization parameter which, in turn, determines the size of the increment in the iterative scheme. In standard LM methods for RML-based applications this regularization parameter is typically decreased by an arbitrary factor selected heuristically. In contrast, the regularization parameter in the proposed methods is adaptively chosen according to the discrepancy principle in order to stabilize the computation of each ensemble member. In addition, the size of the parameter $\rho$ defines an additional parameter $\tau$ that controls the early termination of the scheme in the computation of each ensemble member. The theory of iterative regularization methods indicates that the selection of $\tau>1/\rho$ ensures the aforementioned stabilization. On the one hand, smaller values of $\rho$ are associated with smaller values of the regularization parameter and therefore with larger $\tau$ that ensures the early termination of the scheme to avoid unstable computations. On the other hand, larger values of $\rho$ are associated to smaller and more controlled steps and therefore recommended for highly ill-posed nonlinear inverse problems. 

The numerical results indicate that the stability in the computation of each ensemble is reflected in the accuracy of the ensemble approximation of the posterior. While for the IR-enLM the ensemble size had limited effect, the approximation properties of the IR-ES were substantially affected by the ensemble size. Our results suggest that with a size of $N_{e}=75$, IR-ES inherits the regularization properties of the regularizing LM scheme \cite{Hanke}. In addition, for both proposed methods, larger values of the tunable parameter $\rho$ correspond to more accurate approximations of the mean and variance of the posterior although at a larger computational cost. For the present examples we have found that values $\rho=0.7$ and $\rho = 0.8$ are reasonable compromise between accuracy and cost. Finally, for IR-enLM it is important to reiterate that improved estimates of the mean and variance of the posterior were obtained where values of $\tau=1\leq 1/\rho$ were selected. However, the condition of $\tau>1/\rho$ is only sufficient for the stability in the computation of each ensemble member and it is not conclusive of the quality of the ensemble approximation of the Bayesian posterior. However, the present numerical study may potentially provide theoreticians with insight to develop further analysis of the ensemble methods for the solution of Bayesian inverse problems.

\section{Appendix: The forward operator}\label{ReservoirModels}

We recall from the discussion of Section \ref{Numerics} that $G_{A}$ and $G_{B}$ are the forward operators that arise from two different reservoir models that we now describe. We denoted by $D$ the physical domain of the reservoir. We consider an incompressible oil-water reservoir. Water is injected at $N_I$ injection wells located at $\{x_I^{l}\}_{l=1}^{N_{I}}$. $N_P$ production wells are located at$\{x_P^{l}\}_{l=1}^{N_{P}}$. The absolute permeability and porosity are denoted by $K$ and $\phi$ respectively. We consider a waterflood during an interval  of time denoted by $[0,T]$ ($T > 0$). As stated in Section \ref{Numerics}, for simplicity we assume that the only unknown parameter is $u=\log{K}$. The pressure $p(x,t)$ and the saturation $s(x,t)$ ($(x, t ) \in D \times [0, T ]$) are the solutions to  \cite{Chen}
\begin{eqnarray}\label{eq:2.7}
-\nabla \cdot \lambda(s) e^{u}\nabla p= \sum_{l=1}^{N_{I}}q_{I}^l \delta(x-x_I^l)+\sum_{l=1}^{N_{P}}q_{P}^{l}\delta(x-x_P^l)\\
\phi \frac{\partial s}{\partial t} -\nabla \cdot \lambda_w(s) e^{u}\nabla p= \sum_{l=1}^{N_{I}}q_{I}^l \delta (x-x_I^l)+\sum_{l=1}^{N_{P}}\frac{\lambda_{w}}{\lambda}q_{P}^{l}\delta(x-x_P^l)\label{eq:2.7B}
\end{eqnarray}
in $D\times(0,T]$, where $\delta(x-x_P^l)$ and $\delta(x-x_I^l)$ are (possibly mollified) Dirac deltas.  We consider the following expressions
\begin{eqnarray}\label{eq:2.8}
\lambda_w(s)=\frac{0.3s^2}{\mu_{w}},\qquad \lambda(s)=\frac{(1-s)^2}{\mu_{o}}+\lambda_w(s)
\end{eqnarray}
for the water and total mobility, respectively. Furthermore, we choose that $\mu_{w}=5\times 10^{-4}[\textrm{Pa s}]$ and $\mu_{0}= 10^{-2}[\textrm{Pa s}]$. Constant initial conditions for pressure and water saturation are imposed
\begin{eqnarray}\label{eq:2.9}
p=2.5\times 10^{7}\textrm{Pa}, \qquad s=0\qquad \textrm{in }   D\times\{0\} 
\end{eqnarray}
Expressions (\ref{eq:2.7})-(\ref{eq:2.7}) are furnished with no-flow boundary conditions. In addition, we assume that there are $N_{M}$ measurement times denoted by $\{t_{n}\}_{n=1}^{N_{M}}$. $q_{I}^l$ and $q_{P}^{l}$ in (\ref{eq:2.7})-(\ref{eq:2.7}) are chosen according to the well constraints defined for each of the following models.

\subsection{Model A}
For this reservoir model we assume that injection wells are operated under prescribed $q_{I}^l$ rates while production wells are constrained to bottom-hole pressure $P_{bh}^{l}$. In particular, we consider $\{q_{I}^l(t)= 2.6\times 10^3 \textrm{m}^{3}/\textrm{day} \}_{l=1}^{N_I}$ and $\{P_{bh}^{l}(t)=2.7\times 10^{7} \textrm{Pa}\}_{l=1}^{N_P}$. The expression for the total flow rate $q_{P}^{l}$ (at the production wells) in terms of $P_{bh}^{l}$ is given by the following well model \cite{Chen}
\begin{eqnarray}\label{eq:2.10}
q_{P}^{l}(t)=\omega_{P}^{l}\exp(u(x_{P}^{l}))\lambda(s(x_{P}^{l},t))(P_{bh}^{l}(t)-p(x_{P}^l,t)),
\end{eqnarray}
where $\omega_{P}^{l}$ is the well index of the $l$-th production well. 

For Model A we assume that measurements of BHP are collected at the injection wells at $\{t_{n}\}_{n=1}^{N_{M}}$. According to Peaceman well-model \cite{Chen}, BHP is defined by
\begin{eqnarray}\label{eq:2.11}
M_{A,n}^{l,I}(p,s)=\frac{q_{I}^l(t_{n})}{\omega_{I}^{l}\exp(u(x_{I}^{l}))\lambda(s(x_{I}^{l},t_{n}))}+p(x_{I}^l,t_{n})
\end{eqnarray}
for $l=1,\dots,N_{I}$ and $n=1,\dots, N_M$. At the production wells, we collect measurements of water rate
\begin{eqnarray}\label{eq:2.12}
M_{A,n}^{l,P_{w}}(p,s)=q_{P}^{l}(t_{n})=\omega_{P}^{l}\exp(u(x_{P}^{l}))\lambda_{w}(s(x_{P}^{l},t_{n}))(P_{bh}^{l}(t_{n})-p(x_{P}^l,t_{n}))
\end{eqnarray}
for $l=1,\dots, N_{P}$ and $n=1,\dots, N_M$. Let us define the $N_{M}N_{I}$-dimensional vector 
\begin{eqnarray}\label{eq:2.13}
G_{A,I}(u)\equiv (M_{A,1}^{1,I}(p,s), \dots, M_{A,N_{M}}^{1,I}(p,s),\dots,M_{A,1}^{N_{I},I}(p,s), \dots, M_{A,N_{M}}^{N_{I},I}(p,s))
\end{eqnarray}
as well as the $N_{M}N_{P}$-dimensional vector
\begin{eqnarray}\label{eq:2.13B}
G_{A,w}(u) \equiv (M_{A,1}^{1,P_{w}}(p,s), \dots, M_{A,N_{M}}^{1,P_{w}}(p,s)\dots,M_{A,1}^{N_{P},P_{w}}(p,s), \dots, M_{A,N_{M}}^{N_P,P_{w}}(p,s))\nonumber\\
\end{eqnarray}
The total number of measurements is $N=[N_{P}+N_{I}]N_{M}$ and the forward map $G_{A}:X\to \mathbb{R}^{N}$ for Model A is defined by
expression
\begin{eqnarray}\label{eq:2.14}
G_{A}(u)=(G_{A,I}(u),G_{A,w}(u))
\end{eqnarray}
which comprises the production data obtained from production and injection wells at the measurement times.

\subsection{Model B}
In this case the injection wells are operated under prescribed $P_{bh,I}^{l}$ rates and the production wells are constrained to total flow rate $q_{P}^{l}$. In concrete, we select $\{q_{I}^l(t)= 2.6\times 10^3 \textrm{m}^{3}/\textrm{day} \}_{l=1}^{N_I}$ and $\{P_{bh,I}^{l}(t)=2.7\times 10^{7} \textrm{Pa}\}_{l=1}^{N_P}$. The expression for the injection rate $q_{I}^{l}$ in terms of $P_{bh,I}^{l}$ is given by the following expression
\begin{eqnarray}\label{eq:2.15}
q_{I}^{l}(t)=\omega_{I}^{l}\exp(u(x_{I}^{l}))\lambda_{w}(s(x_{I}^{l},t))(P_{bh,I}^{l}(t)-p(x_{I}^l,t)),
\end{eqnarray}
where $\omega_{I}^{l}$ is the well index of the $l$-th injection well. Measurements of water rate are collected the injection wells at $\{t_{n}\}_{n=1}^{N_{M}}$. Therefore, 
\begin{eqnarray}\label{eq:2.16}
M_{B,n}^{l,I}(p,s)=q_{I}^{l}(t_{n})=\omega_{I}^{l}\exp(u(x_{I}^{l}))\lambda_{w}(s(x_{I}^{l},t_{n}))(P_{bh,I}^{l}(t_{n})-p(x_{I}^l,t_{n})),
\end{eqnarray}
for $l=1,\dots,N_{I}$ and $n=1,\dots, N_M$. The production wells are operated under prescribed total flow rates $\{q^{l}(t)\}_{l=1}^{N_{P}}$. At these wells, we collect measurements of water rate
\begin{eqnarray}\label{eq:2.17}
M_{B,n}^{l,P_{w}}(p,s)=\frac{\lambda_{w}(s(x_{P}^{l},t_{n}))}{\lambda(s(x_{P}^{l},t_{n}))}q^{l}(t_{n})
\end{eqnarray}
for $l=1,\dots, N_{P}$ and $n=1,\dots, N_M$. The forward map $G_{B}:X\to \mathbb{R}^{N}$ is defined by
expression
\begin{eqnarray}\label{eq:2.20}
G_{B}(u)=(G_{B,I}(u),G_{B,w}(u))
\end{eqnarray}
with $G_{B,I}(u)$ and $G_{B,w}(u)$ defined with expressions analogous to (\ref{eq:2.13})-(\ref{eq:2.13B}).

\begin{acknowledgements}

The author would like to thank Andrew Stuart for helpful discussions and his generous feedback on the content and structure of the manuscript. 
\end{acknowledgements}

\bibliographystyle{plain}
\bibliography{Ensemble_bib}   

\end{document}